\documentclass[reqno]{amsart}

\usepackage{latexsym,amsfonts,amssymb,exscale,enumerate,comment}
\usepackage{amsmath,amsthm,amscd}
\usepackage{mathrsfs}

\usepackage{graphicx}
\usepackage{dutchcal}

\usepackage{url}

\usepackage{hyperref}

\usepackage{enumitem}

\usepackage[capitalize]{cleveref}
\crefname{theorem}{Theorem}{Theorems}
\crefname{fact}{Fact}{Facts}
\crefname{note}{Note}{Notes}
\crefname{lemma}{Lemma}{Lemmas}
\crefname{alg}{Algorithm}{Algorithms}
\crefname{remark}{Remark}{Remarks}
\crefname{example}{Example}{Examples}
\crefname{prop}{Proposition}{Propositions}
\crefname{conj}{Conjecture}{Conjectures}
\crefname{cor}{Corollary}{Corollaries}
\crefname{definition}{Definition}{Definitions}
\crefname{Relation}{Relation}{Relations}
\crefname{equation}{\!\!}{\!\!} 

\input xy
\usepackage[all]{xy}
\xyoption{line}
\xyoption{arrow}
\xyoption{color}
\SelectTips{cm}{}

\usepackage{tikz}
\usetikzlibrary{shapes,snakes}
\usetikzlibrary{decorations.markings}
\usetikzlibrary{decorations.pathreplacing}
\usetikzlibrary{cd}
\tikzstyle directed=[postaction={decorate,decoration={markings,
    mark=at position #1 with {\arrow{>}}}}]

\newcommand{\hackcenter}[1]{
 \xy (0,0)*{#1}; \endxy}

\tikzset{->-/.style={decoration={
  markings,
  mark=at position #1 with {\arrow{>}}},postaction={decorate}}}

\tikzset{middlearrow/.style={
        decoration={markings,
            mark= at position 0.5 with {\arrow{#1}} ,
        },
        postaction={decorate}
    }
}

\usepackage{graphicx}
\usepackage{color}

\usetikzlibrary{shapes.geometric}

\def\UglnA{\dot{{\bf U}}(\mathfrak{gl}_n(A))}
\def\UglmA{\dot{{\bf U}}(\mathfrak{gl}_m(A))}

\def\UglnAsop{\dot{{\bf U}}(\mathfrak{gl}_n(A^{\sop}))}

\newcommand{\con}{\ensuremath{con}}
\renewcommand{\exp}{\ensuremath{exp}}

\theoremstyle{plain}
\newtheorem{theorem}{Theorem}

\newtheorem{corollary}[theorem]{Corollary}
\newtheorem{proposition}[theorem]{Proposition}
\newtheorem{lemma}[theorem]{Lemma}

\theoremstyle{definition}
\newtheorem{example}[theorem]{Example}
\newtheorem{definition}[theorem]{Definition}

\theoremstyle{definition}
\newtheorem{remark}[theorem]{Remark}

\numberwithin{equation}{subsection}
\numberwithin{theorem}{subsection}

\newcommand{\refequal}[1]{\xy {\ar@{=}^{#1}
(-1,0)*{};(1,0)*{}};
\endxy}

\newcommand{\clr}{rgb:black,1;blue,4;red,1}

\newcommand{\Hom}{{\rm Hom}}

\renewcommand{\to}{\rightarrow}

\def\sModR{R\text{-}{\mathbf{sMod}}}
\def\smodR{R\text{-}{\mathbf{smod}}}
\def\sModk{\k \text{-}{\mathbf{sMod}}}

\def\Web{{\mathbf{Web}}}
\def\TT{{\mathbf{T}}}
\def\TTAa{{\mathbf{T}}^{\catA}_{\a}}
\def\TTAaI{{\mathbf{T}}^{A,a}_{I}}
\def\TTAaOne{{\mathbf{T}}^{A,a}_{\{1\}}}

\def\gl{\mathfrak{gl}}

\def\WebAaIntro{\mathbf{Web}^{A,a}}
\def\AffWebAaIntro{\mathbf{AffWeb}^{A,a}}
\def\WebAaI{\mathbf{Web}^{A,a}_I}
\def\WebAaIthin{\mathbf{Web}^{A,a}_{I,\textup{thin}}}
\def\WAand{W^{A,a}_{n,d}}
\def\WebAaOne{\mathbf{Web}^{A,a}_{\{1\}}}
\def\WebAaOnen{\mathbf{Web}^{A,a}_{\{1\},n}}
\def\WebAaIAlt{\mathbf{Web}^{A,a}_{I,\textup{alt}}}
\def\WebAaOne{\mathbf{Web}^{A,a}_{\{1\}}}

\def\WebAaQ{\mathbf{Web}^{\operatorname{Cl}_1,\mathbb{C}}_{\{1\}}}

\def\WrA{\mathbf{Wr}_{A}}
\def\WrAI{\mathbf{Wr}^{A}_I}

\def\catA{\mathbf{A}}

\def\catB{\mathbf{B}}
\def\catC{\mathbf{C}}
\def\catD{\mathbf{D}}

\def\a{\mathbcal{a}}

\let\hat=\widehat
\let\tilde=\widetilde

\let\epsilon=\varepsilon

\usepackage{bbm}
\def\C{{\mathbb{C}}}
\def\N{{\mathbbm N}}

\def\Z{{\mathbbm Z}}

\def\1{\mathbbm{1}}%
\usepackage{bbm}

\def \KK {\mathbbm{K}}
\def \LL {\mathbbm{L}}

\def \k {\mathbbm{k}}
\def \Z {\mathbbm{Z}}

\def \N {\mathbbm{N}}

\def \Ob{\operatorname{Ob}}

\newcommand\nc{\newcommand}
\nc\rnc{\renewcommand}
\nc\Kar{\operatorname{Kar}}
\nc\modQ {{\mathbb Q}}
\nc\modZ {{\mathbb Z}}
\nc\simeqto{\overset{\simeq}{\longrightarrow }}
\nc\K{\mathcal {K}}
\nc\CC{\mathbf{C}}

\nc\tE{\check{E}}
\nc\qh{\mathcal{H}}
\nc\hbm{\mathcal{B}}
\nc\bu{\mathbf{u}}
\nc\bZ{\mathbf{Z}}
\nc\theirs{\mathrm{theirs}}
\nc\ours{\mathrm{ours}}
\nc\hE{\mathcal{\hat E}}
\nc\bK{\mathbf{K}}
\nc\bw{\mathbf{w}}
\nc\bh{\mathbf{h}}
\nc\ba{\mathbf{a}}
\nc\bb{\mathbf{b}}
\nc\bbx{\mathbf{x}}
\nc\bby{\mathbf{y}}

\nc{\urcap}{\textrm{uRCap}}
\nc{\urcup}{\textrm{uRCup}}
\nc{\ulcap}{\textrm{uLCap}}
\nc{\ulcup}{\textrm{uLCup}}
\nc{\drcap}{\textrm{dRCap}}
\nc{\drcup}{\textrm{dRCup}}
\nc{\dlcap}{\textrm{dRCap}}
\nc{\dlcup}{\textrm{dRCup}}

\nc{\cl}{\mathrm{cl}}
\nc{\cala}{\mathcal{A}}
\nc{\calb}{\mathcal{B}}
\nc{\calc}{\mathcal{C}}
\nc{\bs}{\mathbf{s}}
\nc{\br}{\mathbf{r}}
\nc{\bz}{\mathbf{z}}
\nc{\bm}{\mathbf{m}}
\nc{\bmu}{\boldsymbol{\mu}}
\nc{\blam}{\boldsymbol{\lambda}}
\nc{\bnu}{\boldsymbol{\nu}}
\nc{\bom}{\boldsymbol{\omega}}
\nc{\bOne}{\boldsymbol{1}}
\nc{\bx}{\mathbf{x}}
\nc{\by}{\mathbf{y}}
\nc{\bbbf}{\mathbf{f}}
\nc{\bE}{\mathbbm{E}}
\nc{\bElong}{{\mathbbm E^{\bullet, \geq}}}

\nc{\Sk}{\mathrm{Sk}}
\nc{\Hilb}{\mathrm{Hilb}}

\nc\col{\colon\thinspace}

\allowdisplaybreaks

\usepackage{amsbsy,nccmath}
\usepackage{todonotes}
\DeclareMathAlphabet{\mathpzc}{OT1}{pzc}{m}{it}

\theoremstyle{theorem}
\newtheorem*{IntroTheoremA}{Theorem A}
\newtheorem*{IntroTheoremB}{Theorem B}
\newtheorem*{IntroTheoremC}{Theorem C}

\newcommand{\End}{\operatorname{End}}
\newcommand{\F}{\mathbbm{F}}
\newcommand{\sop}{\text{sop}}
\renewcommand{\SS}{\mathfrak{S}}
\newcommand{\bi}{\mathbf{i}}
\newcommand{\bj}{\mathbf{j}}
\newcommand{\bk}{\mathbf{k}}
\newcommand{\balpha}{\boldsymbol{\alpha}}

\newcommand{\bn}{\mathbf{n}}
\newcommand{\Wr}{\mathbf{Wr}}
\newcommand{\fh}{\mathfrak{h}}
\newcommand{\fin}{\text{fin}}
\newcommand{\Fun}{{\mathcal{F}\hspace{-0.02in}\mathpzc{un}}}
\newcommand{\smodA}{\mathbf{sMod}\text{-}A}
\newcommand{\Asmod}{A\text{-}\mathbf{sMod}}
\newcommand{\BasisB}{\mathtt{B}}
\newcommand{\basisb}{\mathtt{b}}
\newcommand{\ksmod}{{\k}\text{-}{\mathbf{Mod}}}
\newcommand{\ksMod}{{\k}\text{-}{\mathbf{Mod}}}
\renewcommand{\star}{\circledast}
\newcommand{\tuplambda}{\boldsymbol{\lambda}}
\newcommand{\tupmu}{\boldsymbol{\mu}}

\newcommand{\Cliff}{\operatorname{Cl}_{1}}
\newcommand{\Ser}{\operatorname{Ser}}
\newcommand{\bd}{\mathbf{d}}
\newcommand{\typeM}{\mathtt{M}}
\newcommand{\typeQ}{\mathtt{Q}}

\newcommand{\Awreathsmod}{\mathbf{smod}\text{-}\SS_{d}\wr A }
\newcommand{\eAewreathsmod}{\mathbf{smod}\text{-}\mathcal{e}(\SS_{d}\wr A)\mathcal{e}}

\newcommand{\sdim}{\operatorname{grdim}}
\newcommand{\SAmn}{\mathcal{S}_{m,n}}

\newcommand{\modglnAS}{\textup{mod}_{\mathcal{S}}\textup{-}\gl_n(A)}
\newcommand{\modglnAT}{\textup{mod}_{\mathcal{T}}\textup{-}\gl_n(A)}

\title{
Superalgebra deformations of web categories: Finite webs}

\begin{document}
\setcounter{tocdepth}{2}

\author{Nicholas Davidson}
\email{davidsonnj@cofc.edu}
\address{Department of Mathematics\\ College of Charleston \\ Charleston, SC, USA}

\author{Jonathan R. Kujawa}
\email{kujawa@ou.edu}
\address{Department of Mathematics\\ University of Oklahoma \\ Norman, OK, USA}

\author{Robert Muth}
\email{muthr@duq.edu}
\address{Department of Mathematics and Computer Science \\ Duquesne University \\ Pittsburgh, PA, USA}

\author{Jieru Zhu}
\email{jieruzhu699@gmail.com}
\address{Research Institute in Mathematics and Physics, Universit\'e catholique de Louvain, Louvain-la-Neuve, Belgium}

\thanks{The second author was supported in part by Simons Collaboration Grant for Mathematicians No.\ 525043.}

\date{\today}

\begin{abstract}  Let \(\k\) be a characteristic zero domain.
For a locally unital \(\k\)-superalgebra \(A\) with distinguished idempotents \(I\) and even subalgebra \(a \subseteq A_{\bar 0}\), we define and study an associated diagrammatic monoidal \(\k\)-linear supercategory \(\WebAaI\). This supercategory yields a diagrammatic description of the generalized Schur algebras \(T^A_a(n,d)\). We also show there is an asymptotically faithful functor from \(\WebAaI\) to the monoidal supercategory of \(\gl_n(A)\)-modules generated by symmetric powers of the natural module.  When this functor is full, the single diagrammatic supercategory \(\WebAaI\) provides a combinatorial description of this module category for all $n \geq 1$. We also use these results to establish Howe dualities between \(\gl_{m}(A)\) and \(\gl_{n}(A)\) when \(A\) is semisimple.
\end{abstract}

\maketitle

\tableofcontents

\section{Introduction}

\subsection{Generalized Schur algebras} Let $\k$ be an integral domain and let $A = A_{\bar{ 0}} \oplus A_{\bar{1}}$ be a unital $\k$-superalgebra which is free as a $\k$-supermodule.  The prefix ``super'' means that objects are $\Z_{2}=\Z/2\Z$-graded and one should take graded versions of the usual definitions; see \cref{rings} for details.  While our results hold and are already interesting for ordinary $\k$-algebras (i.e., when $A_{\bar{1}}=0$), the grading plays an important role in several of our motivating examples.

For a fixed $n,d \geq 1$, let $V_{n}= A^{\oplus n}$ and let $\SS_{d}\wr A$ be the wreath product superalgebra of the symmetric group $\SS_{d}$ and $A$.  There is a right action of $\SS_{d}\wr A$ on $V_{n}^{\otimes d}$ where $\SS_{d}$ acts by place permutations,  $A^{\otimes d}$ acts by right multiplication, and where both actions depend on the $\Z_{2}$-grading.  Here and elsewhere $\otimes = \otimes_{\k}$.  The \emph{generalized Schur superalgebra} of $A$ is defined to be
\[
S^{A}(n,d) = \End_{\SS_{d}\wr A}\left(V_{n}^{\otimes d} \right).
\]
Also associated to this data is the Lie superalgebra $\gl_{n}(A)$ of $n \times n$ matrices with entries in $A$ and the Lie bracket given by the graded commutator. This framework includes a  variety of interesting algebras.  When $A=\k$ this gives the general linear Lie algebra $\gl_{n}(\k)$ and Schur algebra $S(n,d)$.  When $A$ is the group algebra of a cyclic group of order $r$, $C_{r}$, then $\SS_{d} \wr C_{r}$ is a complex reflection group and $S^{A}(n,d)$ is the cyclotomic Schur algebra.  When $A=k[t]$ or $A=\k [t, t^{-1}]$, $\gl_{n}(A)$ is a current or loop algebra, respectively.  More generally, if $A$ is the coordinate ring of a variety (or scheme), then $\gl_{n}(A)$ is the corresponding map algebra.  If $A=\Cliff $ is the Clifford superalgebra on one odd generator, then $\gl_{n}(A)$ is the Lie superalgebra $\mathfrak{q}_{n}(\k)$ and $S^{A}(n,d)$ is the type $Q$ Schur superalgebra. If $\Lambda$ is the exterior superalgebra on one odd generator, then $\mathfrak{sl}_{2}(\Lambda)$ acts on sutured annular Khovanov homology \cite{GLW}.

There is a refined version of the generalized Schur algebra which recently came to prominence.  Fix a unital subalgebra $a \subseteq A_{\bar{0}}$ which is a free $\k$-supermodule with a free direct complement.  To the pair $(A, a)$ there is an associated full rank $\k$-subalgebra, 
\[
T^{A}_{a}(n,d) \subseteq S^{A}(n,d),
\] which is sometimes called the \emph{Schurification} of $A$.   Note that $S^{A}(n,d)$ is not lost since $T^{A}_{A_{\bar{0}}}(n,d) = S^{A}(n,d)$.

Given a ring extension $\k \subseteq \KK$ one can extend scalars to define  $T^{A}_{a}(n,d)_{\KK}$ and $S^{A}(n,d)_{\KK}$.  These will, in general, be non-isomorphic $\KK$-superalgebras of the same rank.  Both algebras are interesting, but $T^{A}_{a}(n,d)_{\KK}$ is now understood to play an especially significant role in representation theory.  Namely, if $\KK$ is a field of positive characteristic, then for a suitable zigzag superalgebra $\bar{Z}$ and subalgebra $\bar{z} \subseteq \bar{Z}_{\bar{0}}$, Evseev--Kleshchev proved the algebra $T^{\bar{Z}}_{\bar{z}}(n,d)_{\KK}$ is Morita equivalent to weight \(d\) RoCK blocks of symmetric groups  \cite{EKRoCK} --- thereby proving a conjecture of Turner \cite{Turner}.  In particular, this means that via the derived block equivalences of Chuang--Rouquier \cite{CR} $T^{\bar{Z}}_{\bar{z}}(n,d)_{\KK}$ can be viewed as the ``local'' object for a block of the symmetric group with non-abelian defect group.  It is expected other choices of $(A, a)$ will give similar results in other settings.  For example, it is conjectured that for certain explicit pairs of Brauer tree algebras $(Z,z)$ and $(Z',z')$ the algebra $T^{Z}_{z}(n,d)$ will be Morita equivalent to the weight \(d\) RoCK blocks of Schur algebras \cite[Conjecture 7.58]{KMqh} and $T^{Z'}_{z'}(n,d)$ will be Morita equivalent to weight \(d\) spin RoCK blocks \cite[Conjecture 1]{KleLiv}.

The algebra $T^{A}_{a}(n,d)_{\KK}$  also seems to be better behaved than the more naively defined $S^{A}(n,d)_{\KK}$.  For example, Kleshchev--Muth showed if $A$ is based quasi-hereditary, then $T^{A}_{a}(n,d)_{\KK}$ is again based quasi-hereditary whenever $n \geq d$ \cite{KM2}.  This need not be true for $S^{A}(n,d)_{\KK}$.  Similarly, under suitable assumptions, if $(A, a)$ is symmetric or cellular, than so is $T^{A}_{a}(n,d)$.  See \cite{AxtellgenSchur,EK,KM2,KM,KMqh} for details.

\subsection{\texorpdfstring{$(A,a)$}{(A,a)}-webs}  Our goal in this paper is to show the language of diagrammatic supercategories can naturally be used to describe $T^{A}_{a}(n,d)$, $U(\gl_{n}(A))$, and $\SS_{d}\wr A$ and their representation theories.

Let $(A,a)$ be a pair of unital superalgebras as above.  In fact, in the body of the paper we will work in the greater generality of a locally unital superalgebra \(A\) (or equivalently a small $\k$-linear supercategory).  While this level of generality is needed for applications, for ease of exposition we focus on the unital case in the introduction. 

To the data of $(A,a)$ we associate a strict monoidal supercategory defined by generators and relations, $\WebAaIntro$, which we call \emph{$(A,a)$-webs}.  The objects will be tuples of nonnegative integers with concatenation of words providing the monoidal product.  The generating morphisms of \(\WebAaIntro\) are given by the diagrams:
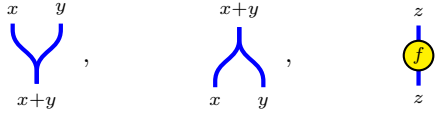
\begin{align*}
\hackcenter{
{}
}
\hackcenter{
\begin{tikzpicture}[scale=.8]
  \draw[ultra thick,blue] (0,0)--(0,0.2) .. controls ++(0,0.35) and ++(0,-0.35) .. (-0.4,0.9)--(-0.4,1);
  \draw[ultra thick,blue] (0,0)--(0,0.2) .. controls ++(0,0.35) and ++(0,-0.35) .. (0.4,0.9)--(0.4,1);
      \node[above] at (-0.4,1) {$ \scriptstyle x$};
      \node[above] at (0.4,1) {$ \scriptstyle y$};
      \node[below] at (0,0) {$ \scriptstyle x+y $};
\end{tikzpicture}} \; ,
\qquad
\qquad
\hackcenter{
\begin{tikzpicture}[scale=.8]
  \draw[ultra thick,blue ] (-0.4,0)--(-0.4,0.1) .. controls ++(0,0.35) and ++(0,-0.35) .. (0,0.8)--(0,1);
\draw[ultra thick, blue] (0.4,0)--(0.4,0.1) .. controls ++(0,0.35) and ++(0,-0.35) .. (0,0.8)--(0,1);
      \node[below] at (-0.4,0) {$ \scriptstyle x$};
      \node[below] at (0.4,0) {$ \scriptstyle y $};
      \node[above] at (0,1) {$ \scriptstyle x+y$};
\end{tikzpicture}} \; ,
\qquad
\qquad
\hackcenter{
\begin{tikzpicture}[scale=.8]
  \draw[ultra thick, blue] (0,0)--(0,0.5);
   \draw[ultra thick, blue] (0,0.5)--(0,1);
   \draw[thick, fill=yellow]  (0,0.5) circle (7pt);
    \node at (0,0.5) {$ \scriptstyle f$};
     \node[below] at (0,0) {$ \scriptstyle z $};
      \node[above] at (0,1) {$ \scriptstyle z $};
\end{tikzpicture}} \; ,
\end{align*}
for \(x,y \in \Z_{\geq 0}\), \(z \in \Z_{>0}\), and \(f \in A\) if $z=1$ and $f \in a$ if $z \geq 2$.  Our convention is to read diagrams from bottom to top.  Composition is given by vertical concatenation and the monoidal product is given by horizontal concatenation.  Morphisms in $\WebAaIntro$ are then linear combinations of diagrams built by repeated concatenation of generating diagrams, subject to a fairly simple set of local relations.  See \cref{defwebaa} for details.  In particular, there the reader will see if $A$ is  locally unital, then the strands are colored by the distinguished idempotents and there is an additional family of ``crossing'' morphisms as generators.  In the body of the paper the results described below are usually proven in this greater generality. 

The category \(\WebAaIntro\) generalizes a number of web category constructions. In particular when $A=\k$ we recover the symmetric webs for $\gl_{n}(\k)$, which are non-quantized versions of those which appear in the literature in \cite{CKM, QS, ST, TVW}, and are isomorphic to categories defined in \cite{BEPO, DKM}.  When $A = \Cliff$ is a rank one Clifford algebra, our constructions recover those of \cite{BrKuWebs}, which are a non-quantum version of those given in \cite{BrKu}.    
 See \cref{S:Examples} for more details on these and other examples.

\subsection{Main results: Webs}\label{}
Our first set of results study features of the monoidal supercategory $\WebAaIntro$.  If we view $V_{n}= A^{\oplus n}$ as row vectors of length $n$, then $\gl_{n}(A)$ acts on the right by matrix multiplication.   Thanks to the cocommutative coproduct of the universal enveloping algebra $U(\gl_{n}(A))$  and the permutation action of $\SS_{d}$, for all $d \geq 1$ there are $\gl_{n}(A)$-supermodules $V_{n}^{\otimes d}$ and $S^{d}V_{n}$, the $d$th tensor and symmetric power, respectively.  The category of right $\gl_{n}(A)$-supermodules is a monoidal supercategory. Let $\modglnAS$ be the full monoidal subcategory of $\gl_{n}(A)$-supermodules consisting of objects of the form 
\[
S^{x_{1}}V_{n} \otimes \dotsb \otimes S^{x_{t}}V_{n},
\] where $t \in \Z_{\geq 1}$ and $x_{1}, \dotsc , x_{t} \in \Z_{\geq 0}$.  

In \cref{Gthm,AsymFaith} we establish the following result.
\begin{IntroTheoremA}
For all $n \geq 1$ there is a functor of monoidal supercategories, 
\[
G_{n}: \WebAaIntro \to \modglnAS,
\] which is essentially surjective and is asymptotically faithful in the sense that for any fixed morphism space the functor defines an injective map for $n \gg 0$.
\end{IntroTheoremA}   We call this the \emph{defining representation} of $\WebAaIntro$.
A fundamental tool in understanding a category given by generators and relations is explicit bases for the morphism spaces.  To this end, in \cref{SpanLem,BasisThm} we use the defining representation to give bases for the morphism spaces of $\WebAaIntro$.

\subsection{Main results: Schurification and wreath products}  
In \cref{S:webificationofSchurification} we relate $\WebAaIntro$ to the Schurification $T^{A}_{a}(n,d)$. Taking \(\Omega(n,d)\) to be the set of \(n\)-part compositions of \(d\), there exists a set of orthogonal idempotents \(\{ \xi_\bx \mid \bx \in \Omega(n,d)\} \subseteq T^A_a(n,d)\) which sum to the identity.  Therefore, 
\begin{align*}
\TT^{A,a} := \bigoplus_{\substack{n,d \geq 0\\ \bx, \by \in \Omega(n,d)}} \xi_{\by}T^A_a(n,d)\xi_{\bx} = \bigoplus_{n,d \geq 0} T^A_a(n,d)
\end{align*}
may be considered as a \(\k\)-linear supercategory with objects indexed by \(\Omega(n,d)\) and composition of morphisms given by multiplication in \(T^A_a(n,d)\). In fact, the {\em shifted star product} defined in \cref{SS:SchurificationofaGoodPair} gives \(\TT^{A,a}\) the structure of a monoidal supercategory.

We establish the following in \cref{sameSchur,moneq1}.

\begin{IntroTheoremB}  
There is an equivalence of monoidal supercategories
\begin{align*}
\TT^{A,a} \to \WebAaIntro,
\end{align*}
and for all $n,d \geq 0$, there is an explicit superalgebra isomorphism
\[
 T^{A}_{a}(n,d)  \xrightarrow{\cong} \bigoplus_{\bx, \by \in \Omega(n,d)} \WebAaIntro(\bx,\by).
\]
\end{IntroTheoremB}  

This result describes \(T^A_a(n,d)\) as a diagrammatic algebra.  This allows new tools and techniques to be brought to the study of $T^{A}_{a}(n,d)$ and its representation theory. It is perhaps worth remarking that $T^{A}_{a}(n,d)$ is defined in the literature as a subalgebra of $S^{A}(n,d)$.  In contrast, here it appears in $\WebAaIntro$ with no reference to $S^{A}(n,d)$ needed.

In a similar spirit, in \cref{S:wreathcategory} we show that the wreath product superalgebras $\SS_{d} \wr A$ also naturally appear as endomorphism algebras in $\WebAaIntro$.  Namely, we define a diagrammatic monoidal supercategory $\Wr^{A}$ by generators and relations, prove in \cref{P:WreathtoThin} that it is equivalent to the full monoidal subcategory of $\WebAaIntro$ consisting of objects of the form $\left\{1^{d}:=(1, \dotsc , 1) \mid d \geq 0 \right\}$, and prove in \cref{C:WreathsAreIsomorphic} that for every $d \geq 1$ there is a superalgebra isomorphism: 
\[
\WebAaIntro(1^{d},1^d) \xrightarrow{\cong} \SS_{d} \wr A.
\]

\subsection{Main results: Howe duality}  A fundamental question is to determine when the defining representation $G_{n}$ is full.  When it is, $\WebAaIntro$ gives a complete combinatorial model for the supercategory $\modglnAS$.  Web-slingers know to expect fullness of $G_{n}$ ought to be intimately related to the existence of a Howe duality for $\gl_{n}(A)$.  In \cref{S:HoweDuality} this expectation is confirmed.

If $V_{m}=A^{\oplus m}$ is written as column vectors and $V_{n} = A^{\oplus n}$ is written as row vectors, then there are commuting actions,
\[
U(\gl_{m}(A)) \curvearrowright V_{m} \otimes_{A} V_{n} \curvearrowleft U(\gl_{n}(A)),
\] given by left and right matrix multiplication. Hence, there are commuting actions,
\begin{equation}\label{E:HoweIntro}
U(\gl_{m}(A)) \curvearrowright S^{\bullet}(V_{m} \otimes_{A} V_{n}) \curvearrowleft U(\gl_{n}(A)).
\end{equation} Having a symmetric Howe duality for the pair $(\gl_{m}(A), \gl_{n}(A))$ means, very roughly, that these two actions should mutually centralize.  In \cref{SS:HoweDuality} we formulate Howe duality in this setting and show it holds if and only if the corresponding defining representations are full.

Our strongest results in this direction are under the assumption that $\k$ is a field of characteristic zero.  In this case fullness of the defining representation $G_{n}$ is equivalent to fullness of the defining representation for the supercategory $\Wr^A$. This, in turn, is equivalent to the surjectivity of the natural map, 
\[
\SS_{d} \wr A \to \End_{S^{A}(n,d)}\left(V_{n}^{\otimes d} \right),
\]  for all $d \geq 1$.  We say \emph{Schur--Weyl duality holds for $A$} if this map is surjective for all $n,d \geq 1$.    When $\k$ is an algebraically closed field of characteristic zero and $A$ is a finite-dimensional semisimple $\k$-superalgebra, then by the double-centralizer theorem for semisimple superalgebras it follows that Schur--Weyl duality holds.  See  \cref{L:SchurWeylDuality}.  This already covers a number of examples of interest. The question remains to determine exactly when Schur--Weyl duality holds.

In the case when $\k$ is an algebraically closed field of characteristic zero and $A$ is a finite-dimensional semisimple superalgebra, we can give a stronger version of Howe duality.  Under these conditions, in \cref{L:HDMultiplicityFree} we prove the following strong multiplicity-free result.
\begin{IntroTheoremC}
 If $m,n \geq 1$ and $r$ is the minimum of $m$ and $n$, then for all $d \geq 0$,
\[
S^{d} \left(V_{m} \otimes_{A} V_{n} \right) \cong \bigoplus_{\tuplambda \in \Lambda^{A}_{+}(r,d)} L_{m}(\tuplambda) \star L_{n}(\tuplambda)^{*}
\] as $(U(\gl_{m}(A)), U(\gl_{n}(A)))$-bisupermodules.
Here $L_{t}(\tuplambda)$ denotes a certain simple $\gl_{t}(A)$-supermodule indexed by an explicit set of multipartitions of $d$, $\Lambda^{A}_{+}(r,d)$.
\end{IntroTheoremC}

\subsection{Future directions}
In a sequel to this paper we plan to focus on the case where \(A\) is a Frobenius algebra and study the {\em affine} web category \(\AffWebAaIntro\). This is a diagrammatic monoidal supercategory given by generators and relations as in \cref{defwebaa} coupled with an additional affine generator coupon.  Much like \(\WebAaIntro\) can be viewed as a thick-strand calculus which encompasses the wreath algebra $\SS_{d} \wr A$, \(\AffWebAaIntro\) is a thick-strand calculus which encompasses the affine wreath product algebra \(\mathcal{A}_n(A)\) defined in \cite{SavageAff}. We expect \(\AffWebAaIntro\) to play an essential role in establishing certain Morita equivalence results in general which are known to hold for \(\mathcal{A}_n(A)\) under restrictions on the characteristic of the ground field (such as in \cite{KMaffzig}).

\subsection{Acknowledgments}
A portion of this research was completed during the AIM SQuaRE collaboration {\em `Superalgebra deformations of web categories'}. The authors wish to thank the American Institute of Mathematics for their support and hospitality.

\section{ Preliminaries}
\subsection{Ground rings, superalgebras, and supercategories}\label{rings}
\subsubsection{Supermodules and supercategories}\label{supermodulesandsupercategories}
Let \(\k\) be a characteristic zero domain with field of fractions \(\KK\).   For \(R \in \{\k, \KK\}\), write \(\sModR\) for the category of left \(R\)-supermodules, and \(\smodR\) for the full subcategory of finitely generated \(R\)-supermodules. Here the prefix ``super'' indicates that we are working in the $\Z_{2}=\Z/2\Z$-graded setting.   For brevity we often leave the prefix super implicit.   We view $R$ as $\Z_{2}$-graded and concentrated in parity $\bar{0} \in \Z_{2}$ and an $R$-supermodule is simply a $\Z_{2}$-graded $R$-module.  Given a $\Z_{2}$-graded object $X = X_{\bar{0}}\oplus X_{\bar{1}}$ we write $\bar{x} \in \Z_{2}$ if $x \in X_{\bar{x}}$.  We sometimes say $x$ is \emph{even} (resp., \emph{odd}) if $\bar{x}=\bar{0}$ (resp., $\bar{x}=\bar{1}$).   Given a free $\k$-supermodule $M$ of finite rank, we write $\dim_{\k}(M)$ for the rank of $M$ and $\sdim_{\k}(M) = p|q$ for the graded rank of $M$; that is, when the rank of $M_{\bar{0}}$ is $p$ and the rank of $M_{\bar{1}}$ is $q$. 

We presume the reader is familiar with superalgebras, supermodules, (monoidal) supercategories, and the like.  In particular, a supercategory is a category which is enriched over a category of supermodules (e.g., $\sModR$).  We generally follow the conventions of \cite{BE} for (monoidal) supercategories, superfunctors, and the diagrammatic calculus for (monoidal) supercategories.

One convention is worth making explicit: we choose to write supermodule homomorphisms on the opposite side of the action.  For example, if $M$ and $N$ are left $A$-supermodules, then a supermodule homomorphism $f: M \to N$ is a $\k$-linear map which satisfies $(am)f = a(m)f$ for all $a \in A$, $m \in M$. If we instead chose to write them on the left, then a homogeneous supermodule homomorphism $g: M \to N$ would be a $\k$-linear map which satisfies $g(am) = (-1)^{\bar{a}\bar{g}}ag(m)$ for all homogeneous $a \in A$, $m\in M$.  Given a homogeneous $\k$-linear map $f:M \to N$, define $g:M \to N$ by $g(x) = (-1)^{\bar{f}\bar{x}}(x)f$. An easy check verifies this gives a bijection between the two conventions and there is no loss in choosing one over the other.

Given a set of objects $X$ in a monoidal supercategory, we write $X^{\otimes d}$ for the set of all $d$-fold tensor products of objects from $X$.  Given an object $x$ in a monoidal supercategory we sometimes write $x^{d}$ or $x^{\otimes d}$ for the $d$-fold tensor product of the object with itself. 

\subsubsection{Locally unital superalgebras}\label{locunicov}
We say that a \(\k\)-superalgebra \(A\) is {\em locally unital} (see \cite[Section 5]{BS}) if there exists a designated system \(I\) of mutually orthogonal even idempotents such that \(A = \bigoplus_{i,j \in I} j A i\).  Given a locally unital $\k$-superalgebra $A$, a locally unital left $A$-supermodule is a left $A$-supermodule which satisfies $M = \bigoplus_{i \in I} iM$.  There is an obvious notion of locally unital for right supermodules and bisupermodules.  Whenever a superalgebra is locally unital we assume its supermodules are also locally unital.

Given such $A$ and $I$ we may define an associated small \(\k\)-linear supercategory \(\catA\) such that \(\Ob(\catA) = I\) and \(\catA(i,j) = jAi\), with composition of morphisms \(f \circ g\) given by multiplication \(fg\).
In the other direction, given a small \(\k\)-linear supercategory \(\catA\) with \(\Ob(\catA) = I\), we may define a locally unital \(\k\)-superalgebra \(A = \bigoplus_{i,j \in I} \catA(i,j)\) with designated system of idempotents \(\{i \mid i \in I\}\) and multiplication given by linearly extending morphism composition.
This association \(A \leftrightarrow \catA\) gives a correspondence between locally unital \(\k\)-superalgebras and small \(\k\)-linear supercategories. 

\subsubsection{Representations of supercategories}

A \emph{(left) representation} of a $\k$-linear supercategory $\catA$ is a $\k$-linear functor $F: \catA \to \ksmod$. Supernatural transformations provide the morphisms between representations of $\catA$.   We write $\catA\text{-Rep}$ for the supercategory of representations of $\catA$.

The correspondence between locally unital \(\k\)-superalgebras and small \(\k\)-linear supercategories extends to representations.  Namely, given a left representation $F: \catA  \to \ksMod$, $\bigoplus_{i \in I} F(i)$ is naturally a left $A$-supermodule.  Conversely, if $M$ is a locally unital left $A$-supermodule with decomposition $M = \bigoplus_{i\in I} iM$ as $\k$-supermodules, then there is an associated representation $F_{M}: \catA \to \ksMod$ given on objects by $F_{M}(i) = iM$ and on morphisms by $F_{M}(a)(m) = am$ for $a \in \catA (i,j) = jAi$ and $m \in iM$.

\subsection{Good pairs} 

\begin{definition}
Let \(A\) be a locally unital \(\k\)-superalgebra with distinguished idempotent set \(I\). Let \(I \subseteq a \subseteq A_{\bar 0}\) be an even subalgebra of \(A\). We say that \((A,a)_I\) is a {\em good pair} provided that for all \(i,j \in I\), there exists a homogeneous \(\k\)-basis \({}_j\BasisB_i\) for \(jAi\) such that \({}_j\basisb_i:={}_j\BasisB_i \cap a\) is a \(\k\)-basis for \(jai\), and \(i \in {}_i \basisb_i \subseteq {}_i \BasisB_i\).
\end{definition}

We remark that this definition of good pair is slightly stronger than that in \cite[\S2A]{KM}, which mandates only that \(A/a\) and \(a\) be \(\k\)-free. Our definition includes the supposition that \(a/ \k I\) is \(\k\)-free as well. This additional requirement holds in all interesting examples known to the authors. 

In particular, note if \((A,a)_I\) is a good pair, then \(a\) is a locally unital algebra with distinguished idempotent set \(I\).  When we work with a good pair we will always presume to have fixed a basis \(\BasisB = \bigsqcup_{i,j \in I} {}_j\BasisB_i\), \(\basisb = \bigsqcup_{i,j \in I} {}_j\basisb_i\) as described above. We additionally fix some total order \(<\) on \(\BasisB\). It will be convenient to write \(A^{(1)} =A\) and \(A^{(z)} = a\) for \(z > 1\).

If \((A,a)_I\), \((A,a')_I\) are good pairs, we write \((A, a')_I \subseteq (A, a)_I\) if \((a, a')_I\) is a good pair. In particular, for any good pair \((A,a)_I\), we always have a chain of good pairs
\(
 (A, \k I)_I \subseteq (A, a)_I \subseteq (A, A_{\overline 0})_I
\).

\subsection{Combinatorics}
\subsubsection{Counting and sets}
We use the generalized binomial coefficient: 
\begin{align*}
{n \choose k}:= \frac{n(n-1) \cdots (n-k+1)}{k!} \in \Z,
\end{align*}
for any \(n \in \Z\), \(k \in \Z_{\geq 0}\).  Given integers $a,b$ with $a \leq b$, let $[a,b] = \{a, a+1, \dotsc , b \}$.  Given sets $X$ and $Y$, we write $X^{Y}$ for the set of tuples indexed by $Y$ consisting of elements from $X$. 

\subsubsection{Compositions}
For \(\bx = (x_1, \ldots, x_t) \in \Z_{\geq 0}^t\), we will set
\begin{align*}
 |\bx| = t,\qquad\lVert \bx \rVert = \sum_{k=1}^t x_k,\qquad  \bx! =  \prod_{k=1}^t (x_k!).
 \end{align*}
Set
\begin{align*}
\Omega(n,d) = \{ \bx = (x_1, \ldots, x_n) \in \Z_{\geq 0}^n \mid \lVert \bx \rVert = d\} \qquad \text{and} \qquad 
\Omega =  \bigsqcup_{n,d \in \Z_{\geq 0}}  \Omega(n,d).
\end{align*}

\subsubsection{Colored compositions}
For \(n,d \in \Z_{\geq 0}\) and a set \(I\), an {\em \(n\)-part \(I\)-colored composition of \(d\)} is the data of a composition \(\bx = (x_1, \ldots, x_n) \in \Omega(n,d)\), and a tuple \(\bi = (i_1, \ldots, i_n) \in I^n\), which we write as \(\bi^{(\bx)} = (i_1^{(x_1)}, \ldots, i_n^{(x_n)})\). For two such \(I\)-colored compositions, we will set \(\bi^{(\bx)} = \bj^{(\bx)}\) provided \(i_t = j_t\) whenever \(x_t \neq 0\); i.e., the coloring of the zero components of a composition is irrelevant. Set
\begin{align*}
\Omega_I(n,d) = \{ \bi^{(\bx)} = (i_1^{(x_1)}, \ldots, i_n^{(x_n)}) \mid \bi \in I^n, \, \bx \in \Omega(n,d)\} \qquad \text{and} \qquad 
\Omega_I =  \bigsqcup_{n,d \in \Z_{\geq 0}}  \Omega_I(n,d).
\end{align*}
We consider \(\Omega_I\) as a monoid under concatenation, with identity the empty composition \(() \in \Omega_I(0,0)\). We define the monoid of {\em reduced} \(I\)-colored compositions \(\widehat{\Omega}_I\) by imposing the following additional relation on \(\Omega_I\): \(() = (i^{(0)})\) for all \(i \in I\). In other words, we have \(\bi^{(\bx)} = \bj^{(\by)}\) in \(\widehat{\Omega}_I\) provided the colored compositions are identical after deleting all the zero components. We write
\begin{align*}
\hat\Omega_I(n,d) = \{ \bi^{(\bx)} = (i_1^{(x_1)}, \ldots, i_n^{(x_n)}) \mid \bi \in I^n, \, \bx \in \Omega(n,d)\} \subseteq \hat \Omega_I.
\end{align*}

\begin{example}
If \(I = \{\alpha, \beta\}\), then \(\alpha^{(2)} \beta^{(0)} \beta^{(3)} = \alpha^{(2)} \alpha^{(0)} \beta^{(3)}\) in both \( \Omega_I\) and \(\widehat{\Omega}_I\), while \(\alpha^{(2)} \beta^{(0)} \beta^{(3)} =  \alpha^{(2)} \beta^{(3)}\) in \(\widehat{\Omega}_I\) but not in \(\Omega_I\).
\end{example}

Recall that a \emph{partition} (resp., \emph{multi-partition}) of $d$ is a composition of $d$ with weakly decreasing parts (resp., a tuple of partitions whose total sum is $d$).  Starting in \cref{S:SWDualityandPolyReps} we will introduce a veritable zoo of partitions and multi-partitions.  We defer their introduction until then.

\section{The \texorpdfstring{$\WebAaI$}{WebAaI} category}

\subsection{Defining the \texorpdfstring{$\WebAaI$}{WebAaI} category}\label{defweb}  Let \((A, a)_I\) be a good pair.
\begin{definition}\label{defwebaa}
 Let \(\WebAaI\) be the strict monoidal \(\k\)-linear supercategory defined as follows.  
We have \(\Ob(\WebAaI)=\widehat{ \Omega}_I\). The monoidal structure on objects in \(\WebAaI\) is inherited from the monoid \(\widehat{ \Omega}_I\); in particular the objects of \(\WebAaI\) are monoidally generated by objects \(i^{(x)}\) for \(i \in I\), \(x \in \Z_{\geq 0}\), with \(i^{(0)}\) equating to the unit object \(\mathbbm{1} \in \WebAaI\) for all \(i \in I\). 

The generating morphisms of \(\WebAaI\) are given by the diagrams:
\begin{align}\label{WebAaGens}
\hackcenter{
{}
}
\hackcenter{
\begin{tikzpicture}[scale=.8]
  \draw[ultra thick,blue] (0,0)--(0,0.2) .. controls ++(0,0.35) and ++(0,-0.35) .. (-0.4,0.9)--(-0.4,1);
  \draw[ultra thick,blue] (0,0)--(0,0.2) .. controls ++(0,0.35) and ++(0,-0.35) .. (0.4,0.9)--(0.4,1);
      \node[above] at (-0.4,1) {$ \scriptstyle i^{\scriptstyle (x)}$};
      \node[above] at (0.4,1) {$ \scriptstyle i^{\scriptstyle (y)}$};
      \node[below] at (0,0) {$ \scriptstyle i^{\scriptstyle (x+y)} $};
\end{tikzpicture}}
\qquad
\qquad
\hackcenter{
\begin{tikzpicture}[scale=.8]
  \draw[ultra thick,blue ] (-0.4,0)--(-0.4,0.1) .. controls ++(0,0.35) and ++(0,-0.35) .. (0,0.8)--(0,1);
\draw[ultra thick, blue] (0.4,0)--(0.4,0.1) .. controls ++(0,0.35) and ++(0,-0.35) .. (0,0.8)--(0,1);
      \node[below] at (-0.4,0) {$ \scriptstyle i^{ \scriptstyle (x)}$};
      \node[below] at (0.4,0) {$ \scriptstyle i^{ \scriptstyle (y)}$};
      \node[above] at (0,1) {$ \scriptstyle i^{ \scriptstyle (x+y)}$};
\end{tikzpicture}}
\qquad
\qquad
\hackcenter{
\begin{tikzpicture}[scale=.8]
  \draw[ultra thick,red] (0.4,0)--(0.4,0.1) .. controls ++(0,0.35) and ++(0,-0.35) .. (-0.4,0.9)--(-0.4,1);
  \draw[ultra thick,blue] (-0.4,0)--(-0.4,0.1) .. controls ++(0,0.35) and ++(0,-0.35) .. (0.4,0.9)--(0.4,1);
      \node[above] at (-0.4,1) {$ \scriptstyle j^{ \scriptstyle (y)}$};
      \node[above] at (0.4,1) {$ \scriptstyle i^{ \scriptstyle (x)}$};
       \node[below] at (-0.4,0) {$ \scriptstyle i^{ \scriptstyle (x)}$};
      \node[below] at (0.4,0) {$ \scriptstyle j^{ \scriptstyle (y)}$};
\end{tikzpicture}}
\qquad
\qquad
\hackcenter{
\begin{tikzpicture}[scale=.8]
  \draw[ultra thick, blue] (0,0)--(0,0.5);
   \draw[ultra thick, red] (0,0.5)--(0,1);
   \draw[thick, fill=yellow]  (0,0.5) circle (7pt);
    \node at (0,0.5) {$ \scriptstyle f$};
     \node[below] at (0,0) {$ \scriptstyle i^{ \scriptstyle (z)}$};
      \node[above] at (0,1) {$ \scriptstyle j^{ \scriptstyle (z)}$};
\end{tikzpicture}}
\end{align}
for \(i,j \in I\), \(x,y \in \Z_{\geq 0}\), \(z \in \Z_{>0}\), \(f \in jA^{(z)}i\), where diagrams are to be read from bottom to top. We call these morphisms `{\em split}', `{\em merge}', `{\em crossing}' and `{\em coupon}' respectively. Splits, merges and crossings have parity \(\bar 0\), and the parity of the \(f\) coupon is \(\bar{f}\). We say a strand labeled by \(i^{(x)}\) has {\em color} \(i\) and {\em thickness} \(x\). We refer to strands of thickness 1 as {\em thin} strands, otherwise we call them {\em thick}.  Recall that \(A^{(1)} =A\) and \(A^{(z)} = a\) for \(z > 1\). Consequently, thick strands may only be decorated by coupons in the even subalgebra \(a\), whereas thin strands may be decorated with arbitrary coupons in \(A\).

Going forward, we will use the following conventions:
\begin{itemize}
\item Strands of thickness 0 (and any coupons thereon) are to be deleted;
\item Diagrams containing a strand of negative thickness are to be read as zero.
\end{itemize}

The relations imposed on the morphisms in \(\WebAaI\) are given in \cref{AssocRel} - \cref{AaIntertwine}:

{\em Web-associativity.} For all \(i \in I\), \(x,y,z \in \Z_{\geq 0}\):
\begin{align}\label{AssocRel}
\hackcenter{
{}
}
\hackcenter{
\begin{tikzpicture}[scale=.8]
  \draw[ultra thick,blue] (0,0)--(0,0.2) .. controls ++(0,0.35) and ++(0,-0.35) .. (-0.4,0.9)--(-0.4,1) 
  .. controls ++(0,0.35) and ++(0,-0.35) .. (0,1.7)--(0,1.8); 
    \draw[ultra thick,blue] (0,0)--(0,0.2) .. controls ++(0,0.35) and ++(0,-0.35) .. (-0.4,0.9)--(-0.4,1) 
  .. controls ++(0,0.35) and ++(0,-0.35) .. (-0.8,1.7)--(-0.8,1.8); 
  \draw[ultra thick,blue] (0,0)--(0,0.2) .. controls ++(0,0.5) and ++(0,-0.5) .. (0.8,1.5)--(0.8,1.8);
      \node[above] at (-0.8,1.8) {$ \scriptstyle i^{ \scriptstyle (x)}$};
      \node[above] at (0,1.8) {$ \scriptstyle i^{ \scriptstyle (y)}$};
      \node[above] at (0.8,1.8) {$ \scriptstyle i^{ \scriptstyle (z)}$};
      \node[below] at (0,0) {$ \scriptstyle i^{ \scriptstyle (x+y+z)}$};
      \node[left] at (-0.4,0.8) {$ \scriptstyle i^{ \scriptstyle (x+y)}$};
\end{tikzpicture}}
=
\hackcenter{
\begin{tikzpicture}[scale=.8]
  \draw[ultra thick,blue] (0,0)--(0,0.2) .. controls ++(0,0.35) and ++(0,-0.35) .. (0.4,0.9)--(0.4,1) 
  .. controls ++(0,0.35) and ++(0,-0.35) .. (0,1.7)--(0,1.8); 
    \draw[ultra thick,blue] (0,0)--(0,0.2) .. controls ++(0,0.35) and ++(0,-0.35) .. (0.4,0.9)--(0.4,1) 
  .. controls ++(0,0.35) and ++(0,-0.35) .. (0.8,1.7)--(0.8,1.8); 
  \draw[ultra thick,blue] (0,0)--(0,0.2) .. controls ++(0,0.5) and ++(0,-0.5) .. (-0.8,1.5)--(-0.8,1.8);
      \node[above] at (-0.8,1.8) {$ \scriptstyle i^{ \scriptstyle (x)}$};
      \node[above] at (0,1.8) {$ \scriptstyle i^{ \scriptstyle (y)}$};
      \node[above] at (0.8,1.8) {$ \scriptstyle i^{ \scriptstyle (z)}$};
      \node[below] at (0,0) {$ \scriptstyle i^{ \scriptstyle (x+y+z)}$};
      \node[right] at (0.4,0.8) {$ \scriptstyle i^{ \scriptstyle (y+z)}$};
\end{tikzpicture}}
\qquad
\qquad
\hackcenter{
\begin{tikzpicture}[scale=.8]
  \draw[ultra thick,blue] (0,0)--(0,-0.2) .. controls ++(0,-0.35) and ++(0,0.35) .. (-0.4,-0.9)--(-0.4,-1) 
  .. controls ++(0,-0.35) and ++(0,0.35) .. (0,-1.7)--(0,-1.8); 
    \draw[ultra thick,blue] (0,0)--(0,-0.2) .. controls ++(0,-0.35) and ++(0,0.35) .. (-0.4,-0.9)--(-0.4,-1) 
  .. controls ++(0,-0.35) and ++(0,0.35) .. (-0.8,-1.7)--(-0.8,-1.8); 
  \draw[ultra thick,blue] (0,0)--(0,-0.2) .. controls ++(0,-0.5) and ++(0,0.5) .. (0.8,-1.5)--(0.8,-1.8);
      \node[below] at (-0.8,-1.8) {$ \scriptstyle i^{ \scriptstyle (x)}$};
      \node[below] at (0,-1.8) {$ \scriptstyle i^{ \scriptstyle (y)}$};
      \node[below] at (0.8,-1.8) {$ \scriptstyle i^{ \scriptstyle (z)}$};
      \node[above] at (0,0) {$ \scriptstyle i^{ \scriptstyle (x+y+z)}$};
      \node[left] at (-0.4,-0.8) {$ \scriptstyle i^{ \scriptstyle (x+y)}$};
\end{tikzpicture}}
=
\hackcenter{
\begin{tikzpicture}[scale=.8]
  \draw[ultra thick,blue] (0,0)--(0,-0.2) .. controls ++(0,-0.35) and ++(0,0.35) .. (0.4,-0.9)--(0.4,-1) 
  .. controls ++(0,-0.35) and ++(0,0.35) .. (0,-1.7)--(0,-1.8); 
    \draw[ultra thick,blue] (0,0)--(0,-0.2) .. controls ++(0,-0.35) and ++(0,0.35) .. (0.4,-0.9)--(0.4,-1) 
  .. controls ++(0,-0.35) and ++(0,0.35) .. (0.8,-1.7)--(0.8,-1.8); 
  \draw[ultra thick,blue] (0,0)--(0,-0.2) .. controls ++(0,-0.5) and ++(0,0.5) .. (-0.8,-1.5)--(-0.8,-1.8);
      \node[below] at (-0.8,-1.8) {$ \scriptstyle i^{ \scriptstyle (x)}$};
      \node[below] at (0,-1.8) {$ \scriptstyle i^{ \scriptstyle (y)}$};
      \node[below] at (0.8,-1.8) {$ \scriptstyle i^{ \scriptstyle (z)}$};
      \node[above] at (0,0) {$ \scriptstyle i^{ \scriptstyle (x+y+z)}$};
      \node[right] at (0.4,-0.8) {$ \scriptstyle i^{ \scriptstyle (y+z)}$};
\end{tikzpicture}}
\end{align}

{\em Merge-split relation.}
For all \(x,y,z,w \in \Z_{\geq 0}\) such that \(x+y = z+w\):
\begin{align}\label{MSrel}
\hackcenter{}
\hackcenter{
\begin{tikzpicture}[scale=0.8]
  \draw[ultra thick, blue] (0,0)--(0,0.2) .. controls ++(0,0.35) and ++(0,-0.35)  .. (0.7,0.8)--(0.7,1.2)
  .. controls ++(0,0.35) and ++(0,-0.35)  .. (0,1.8)--(0,2); 
    \draw[ultra thick, blue] (1.4,0)--(1.4,0.2) .. controls ++(0,0.35) and ++(0,-0.35)  .. (0.7,0.8)--(0.7,1.2)
  .. controls ++(0,0.35) and ++(0,-0.35)  .. (1.4,1.8)--(1.4,2); 
     \node[below] at (0,0) {$\scriptstyle i^{\scriptstyle (x)}$};
     \node[below] at (1.4,0) {$\scriptstyle i^{\scriptstyle (y)}$};
      \node[above] at (0,2) {$\scriptstyle i^{\scriptstyle (z)}$};
      \node[above] at (1.4,2) {$\scriptstyle i^{\scriptstyle (w)}$};
                  \node[left] at (0.6,1) {$\scriptstyle i^{\scriptstyle (x+y)}$};
\end{tikzpicture}}
=
\sum_{t \in \Z_{\geq 0}}
\hackcenter{
\begin{tikzpicture}[scale=0.8]
\draw[ultra thick, blue] (0,0)--(0,2);
\draw[ultra thick, blue] (1.4,0)--(1.4,2);
  \draw[ultra thick, blue] (0,0)--(0,0.3) .. controls ++(0,0.35) and ++(0,-0.35)  .. (1.4,1.7)--(1.4,2); 
    \draw[ultra thick, blue] (1.4,0)--(1.4,0.3) .. controls ++(0,0.35) and ++(0,-0.35)  .. (0,1.7)--(0,2); 
     \node[below] at (0,0) {$\scriptstyle i^{\scriptstyle (x)}$};
     \node[below] at (1.4,0) {$\scriptstyle i^{\scriptstyle (y)}$};
      \node[above] at (0,2) {$\scriptstyle i^{\scriptstyle (z)}$};
      \node[above] at (1.4,2) {$\scriptstyle i^{\scriptstyle (w)}$};
      \node[below] at (0.5,0.8) {$\scriptstyle i^{\scriptstyle (t)}$};
            \node[left] at (0,1) {$\scriptstyle i^{\scriptstyle (x-t)}$};
      \node[right] at (1.4,1) {$\scriptstyle i^{\scriptstyle (w-t)}$};
\end{tikzpicture}}
\end{align}

{\em Knothole relation.} For all \(x,y \in \Z_{\geq 0}\):
\begin{align}\label{KnotholeRel}
\hackcenter{}
\hackcenter{
\begin{tikzpicture}[scale=0.8]
  \draw[ultra thick, blue] (0,-0.1)--(0,0.1) .. controls ++(0,0.35) and ++(0,-0.35) .. (-0.4,0.6)--(-0.4,0.9);
    \draw[ultra thick, blue]  (-0.4,0.9)--(-0.4,1.2) 
  .. controls ++(0,0.35) and ++(0,-0.35) .. (0,1.7)--(0,1.9);
  \draw[ultra thick, blue] (0,-0.1)--(0,0.1) .. controls ++(0,0.35) and ++(0,-0.35) .. (0.4,0.6)--(0.4,0.9);
    \draw[ultra thick, blue]  (0.4,0.9)--(0.4,1.2) 
  .. controls ++(0,0.35) and ++(0,-0.35) .. (0,1.7)--(0,1.9);
  .. controls ++(0,0.35) and ++(0,-0.35) .. (0,1.7)--(0,1.9);
       \node[above] at (0,1.8) {$\scriptstyle i^{\scriptstyle(x+y)}$};
         \node[below] at (0,-0.1) {$\scriptstyle i^{\scriptstyle(x+y)}$};
             \node[right] at (0.4,0.9) {$\scriptstyle i^{\scriptstyle(y)}$};
               \node[left] at (-0.4,0.9) {$\scriptstyle i^{\scriptstyle(x)}$};
\end{tikzpicture}}
\;
=
\;
{x+y \choose x}
\hspace{-2mm}
\hackcenter{
\begin{tikzpicture}[scale=0.8]
  \draw[ultra thick, blue] (0,-0.1)--(0,0.9);
  \draw[ultra thick, blue] (0,0.9)--(0,1.9);
     \node[below] at (0,-0.1) {$\scriptstyle i^{\scriptstyle(x+y)}$};
      \node[above] at (0,1.9) {$\scriptstyle i^{\scriptstyle(x+y)}$};
\end{tikzpicture}}
\end{align}

{\em Coxeter relations.} For all {\em distinct} \(i,j,k \in I\) and \(x,y,z \in \Z_{\geq 0}\):
\begin{align}\label{Cox}
\hackcenter{}
\hackcenter{
\begin{tikzpicture}[scale=.8]
  \draw[ultra thick, blue] (0,0)--(0,0.1) .. controls ++(0,0.35) and ++(0,-0.35)  .. (0.8,0.8)--(0.8,1) .. controls ++(0,0.35) and ++(0,-0.35)  .. (0,1.7)--(0,1.8); 
    \draw[ultra thick, red] (0.8 ,0)--(0.8, 0.1) .. controls ++(0,0.35) and ++(0,-0.35)  .. (0,0.8)--(0,1) .. controls ++(0,0.35) and ++(0,-0.35)  .. (0.8,1.7)--(0.8,1.8);
     \node[below] at (0,0) {$ \scriptstyle i^{ \scriptstyle (x)}$};
     \node[below] at (0.8,0) {$ \scriptstyle j^{ \scriptstyle (y)}$};
       \node[above] at (0,1.8) {$ \scriptstyle i^{ \scriptstyle (x)}$};
     \node[above] at (0.8,1.8) {$ \scriptstyle j^{ \scriptstyle (y)}$};
\end{tikzpicture}}
\;
=
\;
\hackcenter{
\begin{tikzpicture}[scale=.8]
  \draw[ultra thick, blue] (0,0)--(0,1.8); 
    \draw[ultra thick, red] (0.8 ,0)--(0.8,1.8);
     \node[below] at (0,0) {$ \scriptstyle i^{ \scriptstyle (x)}$};
     \node[below] at (0.8,0) {$ \scriptstyle j^{ \scriptstyle (y)}$};
       \node[above] at (0,1.8) {$ \scriptstyle i^{ \scriptstyle (x)}$};
     \node[above] at (0.8,1.8) {$ \scriptstyle j^{ \scriptstyle (y)}$};
\end{tikzpicture}}
\qquad
\qquad
\hackcenter{
\begin{tikzpicture}[scale=.8]
  \draw[ultra thick, red] (0.2,0)--(0.2,0.1) .. controls ++(0,0.35) and ++(0,-0.35) .. (-0.4,0.9)
  .. controls ++(0,0.35) and ++(0,-0.35) .. (0.2,1.7)--(0.2,1.8);
  \draw[ultra thick, blue] (-0.6,0)--(-0.6,0.1) .. controls ++(0,0.35) and ++(0,-0.35) .. (1,1.7)--(1,1.8);
  \draw[ultra thick, green] (1,0)--(1,0.1) .. controls ++(0,0.35) and ++(0,-0.35) .. (-0.6,1.7)--(-0.6,1.8);
   \node[below] at (0.2,0) {$ \scriptstyle j^{ \scriptstyle (y)}$};
        \node[below] at (-0.6,0) {$ \scriptstyle i^{ \scriptstyle (x)}$};
         \node[below] at (1,0) {$ \scriptstyle k^{\scriptstyle (z)}$};
          \node[above] at (0.2,1.8) {$ \scriptstyle j^{ \scriptstyle (y)}$};
        \node[above] at (-0.6,1.8) {$ \scriptstyle k^{\scriptstyle (z)}$};
         \node[above] at (1,1.8) {$ \scriptstyle i^{ \scriptstyle (x)}$};
\end{tikzpicture}}
=
\hackcenter{
\begin{tikzpicture}[scale=.8]
  \draw[ultra thick, red] (0.2,0)--(0.2,0.1) .. controls ++(0,0.35) and ++(0,-0.35) .. (0.8,0.9)
  .. controls ++(0,0.35) and ++(0,-0.35) .. (0.2,1.7)--(0.2,1.8);
  \draw[ultra thick, blue] (-0.6,0)--(-0.6,0.1) .. controls ++(0,0.35) and ++(0,-0.35) .. (1,1.7)--(1,1.8);
  \draw[ultra thick, green] (1,0)--(1,0.1) .. controls ++(0,0.35) and ++(0,-0.35) .. (-0.6,1.7)--(-0.6,1.8);
   \node[below] at (0.2,0) {$ \scriptstyle j^{ \scriptstyle (y)}$};
        \node[below] at (-0.6,0) {$ \scriptstyle i^{ \scriptstyle (x)}$};
         \node[below] at (1,0) {$ \scriptstyle k^{\scriptstyle (z)}$};
          \node[above] at (0.2,1.8) {$ \scriptstyle j^{ \scriptstyle (y)}$};
        \node[above] at (-0.6,1.8) {$ \scriptstyle k^{\scriptstyle (z)}$};
         \node[above] at (1,1.8) {$ \scriptstyle i^{ \scriptstyle (x)}$};
\end{tikzpicture}}
\end{align}

{\em Split-intertwining relations.} For all {\em distinct} \(i,j \in I\) and \(x,y,z \in \Z_{\geq 0}\):
\begin{align}\label{SplitIntertwineRel}
\hackcenter{}
\hackcenter{
{}
}
\hackcenter{
\begin{tikzpicture}[scale=.8]
\draw[ultra thick, blue] (1.6,0)--(1.6,-0.2) .. controls ++(0,-0.35) and ++(0,0.35) .. (0.8,-1.2)--(0.8,-1.4)
        .. controls ++(0,-0.35) and ++(0,0.35) .. (0.4,-2)--(0.4,-2.2);
 \draw[ultra thick, blue] (0.8,0)--(0.8,-0.2) .. controls ++(0,-0.35) and ++(0,0.35) .. (0,-1.2)--(0,-1.4)
        .. controls ++(0,-0.35) and ++(0,0.35) .. (0.4,-2)--(0.4,-2.2);       
  \draw[ultra thick, red] (0,0)--(0,-0.2) .. controls ++(0,-0.35) and ++(0,0.35) .. (1.4,-1.2)--(1.4,-2.2);        
      \node[above] at (0,0) {$ \scriptstyle j^{\scriptstyle (z)}$};
      \node[above] at (0.8,0) {$ \scriptstyle i^{ \scriptstyle (x)}$};
      \node[above] at (1.6,0) {$ \scriptstyle i^{ \scriptstyle (y)}$};
       \node[below] at (0.4,-2.2) {$ \scriptstyle i^{ \scriptstyle (x+y)}$};
        \node[below] at (1.4,-2.2) {$ \scriptstyle j^{\scriptstyle (z)}$};
\end{tikzpicture}}
=
\hackcenter{
\begin{tikzpicture}[scale=.8]
\draw[ultra thick, red] (0,0)--(0,-0.6) .. controls ++(0,-0.35) and ++(0,0.35) .. (1.4,-2)--(1.4,-2.2);
\draw[ultra thick, blue] (0.8,0)--(0.8,-0.2) .. controls ++(0,-0.35) and ++(0,0.35) .. (1.2,-0.8)--(1.2,-1)
        .. controls ++(0,-0.35) and ++(0,0.35) .. (0.4,-2)--(0.4,-2.2);
        \draw[ultra thick, blue] (1.6,0)--(1.6,-0.2) .. controls ++(0,-0.35) and ++(0,0.35) .. (1.2,-0.8)--(1.2,-1)
        .. controls ++(0,-0.35) and ++(0,0.35) .. (0.4,-2)--(0.4,-2.2);
      \node[above] at (0,0) {$ \scriptstyle j^{\scriptstyle (z)}$};
      \node[above] at (0.8,0) {$ \scriptstyle i^{ \scriptstyle (x)}$};
      \node[above] at (1.6,0) {$ \scriptstyle i^{ \scriptstyle (y)}$};
       \node[below] at (0.4,-2.2) {$ \scriptstyle i^{ \scriptstyle (x+y)}$};
        \node[below] at (1.4,-2.2) {$ \scriptstyle j^{\scriptstyle (z)}$};
\end{tikzpicture}}
\qquad
\qquad
\hackcenter{
\begin{tikzpicture}[scale=.8]
\draw[ultra thick, blue] (-1.6,0)--(-1.6,-0.2) .. controls ++(0,-0.35) and ++(0,0.35) .. (-0.8,-1.2)--(-0.8,-1.4)
        .. controls ++(0,-0.35) and ++(0,0.35) .. (-0.4,-2)--(-0.4,-2.2);
 \draw[ultra thick, blue] (-0.8,0)--(-0.8,-0.2) .. controls ++(0,-0.35) and ++(0,0.35) .. (0,-1.2)--(0,-1.4)
        .. controls ++(0,-0.35) and ++(0,0.35) .. (-0.4,-2)--(-0.4,-2.2);       
  \draw[ultra thick, red] (0,0)--(0,-0.2) .. controls ++(0,-0.35) and ++(0,0.35) .. (-1.4,-1.2)--(-1.4,-2.2);        
      \node[above] at (0,0) {$ \scriptstyle j^{\scriptstyle (z)}$};
      \node[above] at (-0.8,0) {$ \scriptstyle i^{ \scriptstyle (y)}$};
      \node[above] at (-1.6,0) {$ \scriptstyle i^{ \scriptstyle (x)}$};
       \node[below] at (-0.4,-2.2) {$ \scriptstyle i^{ \scriptstyle (x+y)}$};
        \node[below] at (-1.4,-2.2) {$ \scriptstyle j^{\scriptstyle (z)}$};
\end{tikzpicture}}
=
\hackcenter{
\begin{tikzpicture}[scale=.8]
\draw[ultra thick, red] (0,0)--(0,-0.6) .. controls ++(0,-0.35) and ++(0,0.35) .. (-1.4,-2)--(-1.4,-2.2);
\draw[ultra thick, blue] (-0.8,0)--(-0.8,-0.2) .. controls ++(0,-0.35) and ++(0,0.35) .. (-1.2,-0.8)--(-1.2,-1)
        .. controls ++(0,-0.35) and ++(0,0.35) .. (-0.4,-2)--(-0.4,-2.2);
        \draw[ultra thick, blue] (-1.6,0)--(-1.6,-0.2) .. controls ++(0,-0.35) and ++(0,0.35) .. (-1.2,-0.8)--(-1.2,-1)
        .. controls ++(0,-0.35) and ++(0,0.35) .. (-0.4,-2)--(-0.4,-2.2);
      \node[above] at (0,0) {$ \scriptstyle j^{\scriptstyle (z)}$};
      \node[above] at (-0.8,0) {$ \scriptstyle i^{ \scriptstyle (y)}$};
      \node[above] at (-1.6,0) {$ \scriptstyle i^{ \scriptstyle (x)}$};
       \node[below] at (-0.4,-2.2) {$ \scriptstyle i^{ \scriptstyle (x+y)}$};
        \node[below] at (-1.4,-2.2) {$ \scriptstyle j^{\scriptstyle (z)}$};
\end{tikzpicture}}
\end{align}

{\em Merge-intertwining relations.} For all {\em distinct} \(i,j \in I\) and \(x,y,z \in \Z_{\geq 0}\):
\begin{align}\label{MergeIntertwineRel}
\hackcenter{}
\hackcenter{
\begin{tikzpicture}[scale=.8]
\draw[ultra thick, blue] (1.6,0)--(1.6,0.2) .. controls ++(0,0.35) and ++(0,-0.35) .. (0.8,1.2)--(0.8,1.4)
        .. controls ++(0,0.35) and ++(0,-0.35) .. (0.4,2)--(0.4,2.2);
 \draw[ultra thick, blue] (0.8,0)--(0.8,0.2) .. controls ++(0,0.35) and ++(0,-0.35) .. (0,1.2)--(0,1.4)
        .. controls ++(0,0.35) and ++(0,-0.35) .. (0.4,2)--(0.4,2.2);       
  \draw[ultra thick, red] (0,0)--(0,0.2) .. controls ++(0,0.35) and ++(0,-0.35) .. (1.4,1.2)--(1.4,2.2);        
      \node[below] at (0,0) {$ \scriptstyle j^{\scriptstyle (z)}$};
      \node[below] at (0.8,0) {$ \scriptstyle i^{ \scriptstyle (x)}$};
      \node[below] at (1.6,0) {$ \scriptstyle i^{ \scriptstyle (y)}$};
       \node[above] at (0.4,2.2) {$ \scriptstyle i^{ \scriptstyle (x+y)}$};
        \node[above] at (1.4,2.2) {$ \scriptstyle j^{\scriptstyle (z)}$};
\end{tikzpicture}}
=
\hackcenter{
\begin{tikzpicture}[scale=.8]
\draw[ultra thick, red] (0,0)--(0,0.6) .. controls ++(0,0.35) and ++(0,-0.35) .. (1.4,2)--(1.4,2.2);
\draw[ultra thick, blue] (0.8,0)--(0.8,0.2) .. controls ++(0,0.35) and ++(0,-0.35) .. (1.2,0.8)--(1.2,1)
        .. controls ++(0,0.35) and ++(0,-0.35) .. (0.4,2)--(0.4,2.2);
        \draw[ultra thick, blue] (1.6,0)--(1.6,0.2) .. controls ++(0,0.35) and ++(0,-0.35) .. (1.2,0.8)--(1.2,1)
        .. controls ++(0,0.35) and ++(0,-0.35) .. (0.4,2)--(0.4,2.2);
      \node[below] at (0,0) {$ \scriptstyle j^{\scriptstyle (z)}$};
      \node[below] at (0.8,0) {$ \scriptstyle i^{ \scriptstyle (x)}$};
      \node[below] at (1.6,0) {$ \scriptstyle i^{ \scriptstyle (y)}$};
       \node[above] at (0.4,2.2) {$ \scriptstyle i^{ \scriptstyle (x+y)}$};
        \node[above] at (1.4,2.2) {$ \scriptstyle j^{\scriptstyle (z)}$};
\end{tikzpicture}}
\qquad
\qquad
\hackcenter{
\begin{tikzpicture}[scale=.8]
\draw[ultra thick, blue] (-1.6,0)--(-1.6,0.2) .. controls ++(0,0.35) and ++(0,-0.35) .. (-0.8,1.2)--(-0.8,1.4)
        .. controls ++(0,0.35) and ++(0,-0.35) .. (-0.4,2)--(-0.4,2.2);
 \draw[ultra thick, blue] (-0.8,0)--(-0.8,0.2) .. controls ++(0,0.35) and ++(0,-0.35) .. (0,1.2)--(0,1.4)
        .. controls ++(0,0.35) and ++(0,-0.35) .. (-0.4,2)--(-0.4,2.2);       
  \draw[ultra thick, red] (0,0)--(0,0.2) .. controls ++(0,0.35) and ++(0,-0.35) .. (-1.4,1.2)--(-1.4,2.2);        
      \node[below] at (0,0) {$ \scriptstyle j^{\scriptstyle (z)}$};
      \node[below] at (-0.8,0) {$ \scriptstyle i^{ \scriptstyle (y)}$};
      \node[below] at (-1.6,0) {$ \scriptstyle i^{ \scriptstyle (x)}$};
       \node[above] at (-0.4,2.2) {$ \scriptstyle i^{ \scriptstyle (x+y)}$};
        \node[above] at (-1.4,2.2) {$ \scriptstyle j^{\scriptstyle (z)}$};
\end{tikzpicture}}
=
\hackcenter{
\begin{tikzpicture}[scale=.8]
\draw[ultra thick, red] (0,0)--(0,0.6) .. controls ++(0,0.35) and ++(0,-0.35) .. (-1.4,2)--(-1.4,2.2);
\draw[ultra thick, blue] (-0.8,0)--(-0.8,0.2) .. controls ++(0,0.35) and ++(0,-0.35) .. (-1.2,0.8)--(-1.2,1)
        .. controls ++(0,0.35) and ++(0,-0.35) .. (-0.4,2)--(-0.4,2.2);
        \draw[ultra thick, blue] (-1.6,0)--(-1.6,0.2) .. controls ++(0,0.35) and ++(0,-0.35) .. (-1.2,0.8)--(-1.2,1)
        .. controls ++(0,0.35) and ++(0,-0.35) .. (-0.4,2)--(-0.4,2.2);
      \node[below] at (0,0) {$ \scriptstyle j^{\scriptstyle (z)}$};
      \node[below] at (-0.8,0) {$ \scriptstyle i^{ \scriptstyle (y)}$};
      \node[below] at (-1.6,0) {$ \scriptstyle i^{ \scriptstyle (x)}$};
       \node[above] at (-0.4,2.2) {$ \scriptstyle i^{ \scriptstyle (x+y)}$};
        \node[above] at (-1.4,2.2) {$ \scriptstyle j^{\scriptstyle (z)}$};
\end{tikzpicture}}
\end{align}

{\em Coupon relations.} For all \(i,j,k \in I\), \(x \in \Z_{\geq 0}\), \(f,g \in jA^{(x)}i\), \(h \in kA^{(x)}j\), \(\alpha \in \k\):
\begin{align}\label{AaRel1}
\hackcenter{}
\hackcenter{
\begin{tikzpicture}[scale=.8]
  \draw[ultra thick, blue] (0,0)--(0,1.8);
   \draw[thick, fill=yellow]  (0,0.9) circle (9pt);
    \node at (0,0.9) {$ \scriptstyle i$};
     \node[below] at (0,0) {$ \scriptstyle i^{ \scriptstyle (x)}$};
           \node[above] at (0,1.8) {$ \scriptstyle i^{ \scriptstyle (x)}$};
\end{tikzpicture}}
\;
=
\hackcenter{
\begin{tikzpicture}[scale=.8]
  \draw[ultra thick, blue] (0,0)--(0,1.8);
     \node[below] at (0,0) {$ \scriptstyle i^{ \scriptstyle (x)}$};
           \node[above] at (0,1.8) {$ \scriptstyle i^{ \scriptstyle (x)}$};
\end{tikzpicture}}
\qquad
\qquad
\hackcenter{
\begin{tikzpicture}[scale=.8]
  \draw[ultra thick, blue] (0,0)--(0,0.9);
    \draw[ultra thick, red] (0,0.9)--(0,1.8);
    \draw[thick, fill=yellow]  (0,0.9) circle (9pt);
    \node at (0,0.9) {$ \scriptstyle \alpha f$};
     \node[below] at (0,0) {$ \scriptstyle i^{ \scriptstyle (x)}$};
       \node[above] at (0,1.8) {$ \scriptstyle j^{ \scriptstyle (x)}$};
\end{tikzpicture}}
\;
=
\;
\alpha^x
\hackcenter{
\begin{tikzpicture}[scale=.8]
  \draw[ultra thick, blue] (0,0)--(0,0.9);
    \draw[ultra thick, red] (0,0.9)--(0,1.8);
 \draw[thick, fill=yellow]  (0,0.9) circle (7pt);
    \node at (0,0.9) {$ \scriptstyle f$};
     \node[below] at (0,0) {$ \scriptstyle i^{ \scriptstyle (x)}$};
      \node[above] at (0,1.8) {$ \scriptstyle j^{ \scriptstyle (x)}$};
\end{tikzpicture}}
\end{align}
\begin{align}\label{AaRel2}
\hackcenter{
\begin{tikzpicture}[scale=.8]
  \draw[ultra thick, blue] (0,0)--(0,0.5);
    \draw[ultra thick, red] (0,0.5)--(0,1.3);
      \draw[ultra thick, green] (0,1.3)--(0,1.8);
     \draw[thick, fill=yellow]  (0,0.5) circle (7pt);
    \node at (0,0.5) {$ \scriptstyle f$};
   \draw[thick, fill=yellow]  (0,1.3) circle (7pt);
    \node at (0,1.3) {$\scriptstyle h$};
     \node[below] at (0,0) {$ \scriptstyle i^{ \scriptstyle (x)}$};
      \node[above] at (0,1.8) {$ \scriptstyle k^{\scriptstyle (x)}$};
         \node[left] at (-0.1,0.9) {$ \scriptstyle j^{ \scriptstyle (x)}$};
\end{tikzpicture}}
\;
=
\;
\hackcenter{
\begin{tikzpicture}[scale=.8]
  \draw[ultra thick, blue] (0,0)--(0,0.9);
    \draw[ultra thick, green] (0,0.9)--(0,1.8);
      \draw[thick, fill=yellow]  (0,0.9) circle (9pt);
    \node at (0,0.9) {$\scriptstyle hf$};
     \node[below] at (0,0) {$ \scriptstyle i^{ \scriptstyle (x)}$};
      \node[above] at (0,1.8) {$  \scriptstyle k^{\scriptstyle (x)}$};
\end{tikzpicture}}
\qquad
\qquad
\hackcenter{
\begin{tikzpicture}[scale=.8]
  \draw[ultra thick, blue] (0,0)--(0,0.9);
    \draw[ultra thick, red] (0,0.9)--(0,1.8);
     \draw[thick, fill=yellow]  (0,0.9) circle (11pt);
    \node at (0,0.9) {$ \scriptstyle f\hspace{-0.3mm}+\hspace{-0.2mm}g$};
     \node[below] at (0,0) {$ \scriptstyle i^{ \scriptstyle (x)}$};
      \node[above] at (0,1.8) {$ \scriptstyle j^{ \scriptstyle (x)}$};
\end{tikzpicture}}
\;
=
\sum_{t =0}^{x}
\;
\hackcenter{
\begin{tikzpicture}[scale=.8]
  \draw[ultra thick, blue] (0,0)--(0,0.1) .. controls ++(0,0.35) and ++(0,-0.35) .. (-0.4,0.6)--(-0.4,0.9); 
 \draw[ultra thick, red] (-0.4,0.9)--(-0.4,1.2) .. controls ++(0,0.35) and ++(0,-0.35) .. (0,1.7)--(0,1.8);
  \draw[ultra thick, blue] (0,0)--(0,0.1) .. controls ++(0,0.35) and ++(0,-0.35) .. (0.4,0.6)--(0.4,0.9); 
 \draw[ultra thick, red] (0.4,0.9)--(0.4,1.2) .. controls ++(0,0.35) and ++(0,-0.35) .. (0,1.7)--(0,1.8);
 \draw[thick, fill=yellow]  (-0.4,0.9) circle (7pt);
    \draw[thick, fill=yellow]  (0.4,0.9) circle (7pt);
       \node[above] at (0,1.8) {$ \scriptstyle j^{ \scriptstyle (x)}$};
         \node[below] at (0,0) {$ \scriptstyle i^{ \scriptstyle (x)}$};
           \node at (-0.4,0.9) {$ \scriptstyle f$};
             \node at (0.4,0.9) {$ \scriptstyle g$};
             \node[right] at (0.3,0.3) {$ \scriptstyle i^{\scriptstyle (x-t)}$};
               \node[left] at (-0.3,0.3) {$ \scriptstyle i^{\scriptstyle (t)}$};
\end{tikzpicture}}
\end{align}

{\em Odd knothole annihilation.} For all \(i,j \in I\), \(f \in jA_{\bar 1}i\):
\begin{align}\label{OddKnotholeRel}
\hackcenter{}
\hackcenter{
\begin{tikzpicture}[scale=.8]
  \draw[ultra thick, blue] (0,0)--(0,0.1) .. controls ++(0,0.35) and ++(0,-0.35) .. (-0.4,0.6)--(-0.4,0.9); 
 \draw[ultra thick, red] (-0.4,0.9)--(-0.4,1.2) .. controls ++(0,0.35) and ++(0,-0.35) .. (0,1.7)--(0,1.8);
  \draw[ultra thick, blue] (0,0)--(0,0.1) .. controls ++(0,0.35) and ++(0,-0.35) .. (0.4,0.6)--(0.4,0.9); 
 \draw[ultra thick, red] (0.4,0.9)--(0.4,1.2) .. controls ++(0,0.35) and ++(0,-0.35) .. (0,1.7)--(0,1.8);
 \draw[thick, fill=yellow]  (-0.4,0.9) circle (7pt);
    \draw[thick, fill=yellow]  (0.4,0.9) circle (7pt);
       \node[above] at (0,1.8) {$  \scriptstyle j^{ \scriptstyle (2)}$};
         \node[below] at (0,0) {$  \scriptstyle i^{ \scriptstyle (2)}$};
           \node at (-0.4,0.9) {$ \scriptstyle f$};
             \node at (0.4,0.9) {$ \scriptstyle f$};
             \node[right] at (0.3,0.3) {$ \scriptstyle  i^{ \scriptstyle (1)}$};
               \node[left] at (-0.3,0.3) {$ \scriptstyle  i^{ \scriptstyle (1)}$};
\end{tikzpicture}}
=
0
\end{align}

{\em \(A\)-split/merge relations.} For all \(i,j \in I, x,y \in \Z_{\geq 0}, f \in jA^{(x+y)}i\):
\begin{align}\label{TAaSMRel}
\hackcenter{}
\hackcenter{
\begin{tikzpicture}[scale=.8]
  \draw[ultra thick, red] (0,0.5)--(0,1) .. controls ++(0,0.35) and ++(0,-0.35) .. (-0.4,1.7)--(-0.4,1.8);
  \draw[ultra thick, red] (0,0.5)--(0,1) .. controls ++(0,0.35) and ++(0,-0.35) .. (0.4,1.7)--(0.4,1.8);
  \draw[ultra thick, blue] (0,0)--(0,0.5);
     \draw[thick, fill=yellow]  (0,0.5) circle (7pt);
    \node at (0,0.5) {$ \scriptstyle f$};
         \node[below] at (0,0) {$ \scriptstyle i^{ \scriptstyle (x+y)}$};
      \node[above] at (-0.4,1.8) {$ \scriptstyle j^{ \scriptstyle (x)}$};
            \node[above] at (0.4,1.8) {$ \scriptstyle j^{ \scriptstyle (y)}$};
\end{tikzpicture}}
\;
=
\;
\hackcenter{
\begin{tikzpicture}[scale=.8]
  \draw[ultra thick, blue] (0,0)--(0,0.2) .. controls ++(0,0.35) and ++(0,-0.35) .. (-0.4,0.9)--(-0.4,1)--(-0.4,1.3);
  \draw[ultra thick, blue] (0,0)--(0,0.2) .. controls ++(0,0.35) and ++(0,-0.35) .. (0.4,0.9)--(0.4,1)--(0.4,1.3);
  \draw[ultra thick, red] (-0.4,1.3)--(-0.4,1.8);
    \draw[ultra thick, red] (0.4,1.3)--(0.4,1.8);
     \draw[thick, fill=yellow]  (-0.4,1.3) circle (7pt);
     \draw[thick, fill=yellow]  (0.4,1.3) circle (7pt);
    \node at (-0.4,1.3) {$ \scriptstyle f$};
      \node at (0.4,1.3) {$ \scriptstyle f$};
         \node[below] at (0,0) {$ \scriptstyle i^{ \scriptstyle (x+y)}$};
      \node[above] at (-0.4,1.8) {$ \scriptstyle j^{ \scriptstyle (x)}$};
            \node[above] at (0.4,1.8) {$ \scriptstyle j^{ \scriptstyle (y)}$};
\end{tikzpicture}}
\qquad
\qquad
\hackcenter{
\begin{tikzpicture}[scale=.8]
  \draw[ultra thick, blue] (0,-0.5)--(0,-1) .. controls ++(0,-0.35) and ++(0,0.35) .. (-0.4,-1.7)--(-0.4,-1.8);
  \draw[ultra thick, blue] (0,-0.5)--(0,-1) .. controls ++(0,-0.35) and ++(0,0.35) .. (0.4,-1.7)--(0.4,-1.8);
  \draw[ultra thick, red] (0,0)--(0,-0.5);
     \draw[thick, fill=yellow]  (0,-0.5) circle (7pt);
    \node at (0,-0.5) {$ \scriptstyle f$};
         \node[above] at (0,0) {$ \scriptstyle  j^{ \scriptstyle (x+y)}$};
      \node[below] at (-0.4,-1.8) {$ \scriptstyle i^{ \scriptstyle (x)}$};
            \node[below] at (0.4,-1.8) {$ \scriptstyle i^{ \scriptstyle (y)}$};
\end{tikzpicture}}
\;
=
\;
\hackcenter{
\begin{tikzpicture}[scale=.8]
  \draw[ultra thick, red] (0,0)--(0,-0.2) .. controls ++(0,-0.35) and ++(0,0.35) .. (-0.4,-0.9)--(-0.4,-1)--(-0.4,-1.3);
  \draw[ultra thick, red] (0,0)--(0,-0.2) .. controls ++(0,-0.35) and ++(0,0.35) .. (0.4,-0.9)--(0.4,-1)--(0.4,-1.3);
  \draw[ultra thick, blue] (-0.4,-1.3)--(-0.4,-1.8);
    \draw[ultra thick, blue] (0.4,-1.3)--(0.4,-1.8);
     \draw[thick, fill=yellow]  (-0.4,-1.3) circle (7pt);
     \draw[thick, fill=yellow]  (0.4,-1.3) circle (7pt);
    \node at (-0.4,-1.3) {$ \scriptstyle f$};
      \node at (0.4,-1.3) {$ \scriptstyle f$};
         \node[above] at (0,0) {$ \scriptstyle j^{\scriptstyle (x+y)}$};
      \node[below] at (-0.4,-1.8) {$ \scriptstyle i^{ \scriptstyle (x)}$};
            \node[below] at (0.4,-1.8) {$ \scriptstyle i^{ \scriptstyle (y)}$};
\end{tikzpicture}}
\end{align}

{\em \(A\)-intertwining relations.} For all \(i,j,k \in I, x,y \in \Z_{\geq 0}, f \in jA^{(x)}i\):
\begin{align}\label{AaIntertwine}
\hackcenter{}
\hackcenter{
\begin{tikzpicture}[scale=.8]
  \draw[ultra thick, blue] (0,0)--(0,0.5);
  \draw[ultra thick, red] (0,0.5)--(0,1) .. controls ++(0,0.35) and ++(0,-0.35)  .. (0.8,1.8)--(0.8,2); 
    \draw[ultra thick, green] (0.8,0)--(0.8,1) .. controls ++(0,0.35) and ++(0,-0.35)  .. (0,1.8)--(0,2); 
        \draw[thick, fill=yellow]  (0,0.5) circle (7pt);
    \node at (0,0.5) {$ \scriptstyle f$};
     \node[below] at (0,0) {$ \scriptstyle i^{ \scriptstyle (x)}$};
     \node[below] at (0.8,0) {$  \scriptstyle  k^{ \scriptstyle (y)}$};
     \node[above] at (0,2) {$  \scriptstyle  k^{ \scriptstyle (y)}$};
      \node[above] at (0.8,2) {$ \scriptstyle j^{ \scriptstyle (x)}$};
\end{tikzpicture}}
\;
=
\;
\hackcenter{
\begin{tikzpicture}[scale=.8]
  \draw[ultra thick, blue] (0,0)--(0,0.2) .. controls ++(0,0.35) and ++(0,-0.35)  .. (0.8,1)--(0.8,1.5); 
   \draw[ultra thick, red] (0.8,1.5)--(0.8,2); 
    \draw[ultra thick, green] (0.8,0)--(0.8,0.2) .. controls ++(0,0.35) and ++(0,-0.35)  .. (0,1)--(0,2); 
        \draw[thick, fill=yellow]  (0.8,1.5) circle (7pt);
    \node at (0.8,1.5) {$ \scriptstyle f$};
     \node[below] at (0,0) {$ \scriptstyle i^{ \scriptstyle (x)}$};
     \node[below] at (0.8,0) {$  \scriptstyle  k^{ \scriptstyle (y)}$};
     \node[above] at (0,2) {$  \scriptstyle  k^{ \scriptstyle (y)}$};
      \node[above] at (0.8,2) {$ \scriptstyle j^{ \scriptstyle (x)}$};
\end{tikzpicture}}
\qquad
\qquad
\hackcenter{
\begin{tikzpicture}[scale=.8]
  \draw[ultra thick, green] (0,0)--(0,1) .. controls ++(0,0.35) and ++(0,-0.35)  .. (0.8,1.8)--(0.8,2); 
    \draw[ultra thick, red] (0.8,0.5)--(0.8,1) .. controls ++(0,0.35) and ++(0,-0.35)  .. (0,1.8)--(0,2); 
       \draw[ultra thick, blue] (0.8,0)--(0.8,0.5); 
        \draw[thick, fill=yellow]  (0.8,0.5) circle (7pt);
    \node at (0.8,0.5) {$ \scriptstyle f$};
     \node[below] at (0,0) {$ \scriptstyle  k^{ \scriptstyle (y)}$};
     \node[below] at (0.8,0) {$ \scriptstyle i^{ \scriptstyle (x)}$};
     \node[above] at (0,2) {$ \scriptstyle j^{ \scriptstyle (x)}$};
      \node[above] at (0.8,2) {$  \scriptstyle k^{ \scriptstyle (y)}$};
\end{tikzpicture}}
\;
=
\;
\hackcenter{
\begin{tikzpicture}[scale=.8]
  \draw[ultra thick, green] (0,0)--(0,0.2) .. controls ++(0,0.35) and ++(0,-0.35)  .. (0.8,1)--(0.8,2); 
     \draw[ultra thick, blue] (0.8,0)--(0.8,0.2) .. controls ++(0,0.35) and ++(0,-0.35)  .. (0,1)--(0,1.5); 
    \draw[ultra thick, red] (0,1.5)--(0,2); 
       \draw[thick, fill=yellow]  (0,1.5) circle (7pt);
    \node at (0,1.5) {$ \scriptstyle f$};
     \node[below] at (0,0) {$  \scriptstyle k^{ \scriptstyle (y)}$};
     \node[below] at (0.8,0) {$ \scriptstyle i^{ \scriptstyle (x)}$};
     \node[above] at (0,2) {$ \scriptstyle j^{ \scriptstyle (x)}$};
      \node[above] at (0.8,2) {$  \scriptstyle k^{ \scriptstyle (y)}$};
\end{tikzpicture}}
\end{align}
\end{definition}


\begin{remark}\label{monorels}
The relations \cref{Cox,SplitIntertwineRel,MergeIntertwineRel} are imposed in \(\WebAaI\) only for distinct combinations of colors in order to allow for a more minimal presentation. It will be established in \cref{arbcol} that these relations in fact hold in \(\WebAaI\) for arbitrary combinations of colors. 
\end{remark}

\begin{remark}
In the common situation where \(I= \{1\}\) (and hence \(A\) is a unital algebra as in the introduction), the relations need not be extensive as they appear. All strands in \(\WebAaOne\) will be monotonically colored and there is no contribution from relations \cref{Cox,SplitIntertwineRel,MergeIntertwineRel} which stipulate distinctly colored strands. However, monotonic versions of these relations nonetheless hold in this category, as noted in \cref{monorels}.
\end{remark}

\subsection{An alternate presentation of \texorpdfstring{$\WebAaI$}{WebAa}}\label{alterpres}
It will follow from \cref{imprels} that there exists a slightly modified presentation for \(\WebAaI\), as follows. We may define \(\WebAaI\) as in \cref{defwebaa}, save for replacing the relations \cref{MSrel,KnotholeRel} with the following two relations:

{\em Rung swap.} For all \(i \in I, x,y,r,s \in \Z_{\geq 0}\):
\begin{align}\label{DiagSwitchRel}
\hackcenter{}
\hackcenter{
\begin{tikzpicture}[scale=.8]
\draw[ultra thick, blue] (0,0)--(0,2);
\draw[ultra thick, blue] (1.6,0)--(1.6,2);
  \draw[ultra thick, blue] (0,0)--(0,0.2) .. controls ++(0,0.35) and ++(0,-0.35)  .. (1.6,0.9)--(1.6,2); 
    \draw[ultra thick, blue] (1.6,0)--(1.6,1.1) .. controls ++(0,0.35) and ++(0,-0.35)  .. (0,1.8)--(0,2); 
     \node[below] at (0,0) {$ \scriptstyle i^{ \scriptstyle (x)}$};
     \node[below] at (1.6,0) {$ \scriptstyle i^{ \scriptstyle (y)}$};
     \node[above] at (0,2) {$\scriptstyle i^{ \scriptstyle (x-s+r)} $};
      \node[above] at (2,2) {$\scriptstyle i^{ \scriptstyle (y+s-r)}$};
      \node[left] at (0,1) {$ \scriptstyle i^{ \scriptstyle (x-s)}$};
      \node[right] at (1.6,1) {$ \scriptstyle i^{ \scriptstyle (y+s)}$};
      \node[above] at (0.8,1.5) {$ \scriptstyle i^{ \scriptstyle (r)}$};
      \node[below] at (0.8,0.5) {$ \scriptstyle i^{ \scriptstyle (s)}$};
\end{tikzpicture}}
\;
=
\;
\sum_{t \in \Z_{\geq 0}}
{ x-y+r-s
\choose
t}
\hackcenter{
\begin{tikzpicture}[scale=.8]
\draw[ultra thick, blue] (0,0)--(0,2);
\draw[ultra thick, blue] (1.6,0)--(1.6,2);
  \draw[ultra thick, blue] (1.6,0)--(1.6,0.2) .. controls ++(0,0.35) and ++(0,-0.35)  .. (0,0.8)--(0,2); 
    \draw[ultra thick, blue] (0,0)--(0,1.2) .. controls ++(0,0.35) and ++(0,-0.35)  .. (1.6,1.8)--(1.6,2); 
      \node[below] at (0,0) {$ \scriptstyle i^{ \scriptstyle (x)}$};
     \node[below] at (1.6,0) {$ \scriptstyle i^{ \scriptstyle (y)}$};
     \node[above] at (0,2) {$ \scriptstyle i^{\scriptstyle (x-s+r)}$};
      \node[above] at (1.6,2) {$ \scriptstyle i^{ \scriptstyle (y+s-r)}$};
      \node[left] at (0,1) {$ \scriptstyle i^{ \scriptstyle (x+r-t)}$};
      \node[right] at (1.6,1) {$ \scriptstyle i^{ \scriptstyle (y-r+t)}$};
      \node[above] at (0.8,1.5) {$ \scriptstyle i^{ \scriptstyle (s-t)}$};
      \node[below] at (0.8,0.5) {$ \scriptstyle i^{ \scriptstyle (r-t)}$};
\end{tikzpicture}}
\end{align}

{\em Monotone crossing relation.} For all \(i \in I, x,y \in \Z_{\geq 0}\):
\begin{align}\label{CrossingWebRel}
\hackcenter{}
\hackcenter{
\begin{tikzpicture}[scale=.8]
  \draw[ultra thick, blue] (0,0)--(0,0.2) .. controls ++(0,0.6) and ++(0,-0.6)  .. (1.4,1.8)--(1.4,2); 
    \draw[ultra thick, blue] (1.4,0)--(1.4,0.2) .. controls ++(0,0.6) and ++(0,-0.6)  .. (0,1.8)--(0,2); 
     \node[below] at (0,0) {$ \scriptstyle i^{ \scriptstyle (x)}$};
     \node[below] at (1.4,0) {$ \scriptstyle i^{ \scriptstyle (y)}$};
     \node[above] at (0,2) {$ \scriptstyle i^{ \scriptstyle (y)}$};
      \node[above] at (1.4,2) {$ \scriptstyle i^{ \scriptstyle (x)}$};
\end{tikzpicture}}
\;
=
\sum_{t \in \Z_{\geq 0}}
(-1)^{t}
\hackcenter{
\begin{tikzpicture}[scale=.8]
\draw[ultra thick, blue] (0,0)--(0,2);
\draw[ultra thick, blue] (1.4,0)--(1.4,2);
  \draw[ultra thick, blue] (0,0)--(0,0.2) .. controls ++(0,0.35) and ++(0,-0.35)  .. (1.4,0.9)--(1.4,2); 
    \draw[ultra thick, blue] (1.4,0)--(1.4,1.1) .. controls ++(0,0.35) and ++(0,-0.35)  .. (0,1.8)--(0,2); 
     \node[below] at (0,0) {$ \scriptstyle i^{ \scriptstyle (x)}$};
     \node[below] at (1.4,0) {$ \scriptstyle i^{ \scriptstyle (y)}$};
     \node[above] at (0,2) {$ \scriptstyle i^{\scriptstyle (y)}$};
      \node[above] at (1.4,2) {$ \scriptstyle i^{\scriptstyle (x)}$};
      \node[left] at (0,1) {$ \scriptstyle i^{\scriptstyle (t)}$};
      \node[right] at (1.4,1) {$ \scriptstyle i^{\scriptstyle (y+x-t)}$};
      \node[above] at (0.7,1.52) {$ \scriptstyle i^{\scriptstyle (y-t)}$};
      \node[below] at (0.7,0.5) {$ \scriptstyle i^{\scriptstyle (x-t)}$};
\end{tikzpicture}}
\end{align}
In certain contexts, one or the other of these presentations may prove to be more useful.

\begin{remark}
Again, in the situation when \(I = \{1\}\) we may use \cref{CrossingWebRel} to {\em define} the sole monotone crossing morphism in \(\WebAaOne\) in terms of split and crossing morphisms, and hence do not require any crossing generators in our presentation.
\end{remark}

\section{Implied relations in \texorpdfstring{$\WebAaI$}{WebAa}}\label{imprels}
In this section we establish a number of useful relations in \(\WebAaI\), as well as the assertion from \cref{alterpres} that \(\WebAaI\) has an alternate presentation.    

\subsection{Equivalence of presentations.}
For clarity, throughout this section we will let \(\WebAaI\) designate the category defined in \cref{defwebaa}, and \(\WebAaIAlt\) designate the category defined in \cref{alterpres}.  It will be established in \cref{SAME} that these categories are isomorphic.
\begin{lemma}\label{Same1}
The relations \cref{DiagSwitchRel,CrossingWebRel} hold in \(\WebAaI\).
\end{lemma}
\begin{proof}
We have:
\begin{align}\label{simcomps}
\hackcenter{}
\hackcenter{
\begin{tikzpicture}[scale=.8]
\draw[ultra thick, blue] (0,0)--(0,2);
\draw[ultra thick, blue] (1.6,0)--(1.6,2);
  \draw[ultra thick, blue] (0,0)--(0,0.2) .. controls ++(0,0.35) and ++(0,-0.35)  .. (1.6,0.9)--(1.6,2); 
    \draw[ultra thick, blue] (1.6,0)--(1.6,1.1) .. controls ++(0,0.35) and ++(0,-0.35)  .. (0,1.8)--(0,2); 
     \node[below] at (0,0) {$ \scriptstyle i^{ \scriptstyle (x)}$};
     \node[below] at (1.6,0) {$ \scriptstyle i^{ \scriptstyle (y)}$};
     \node[above] at (0,2) {$\scriptstyle i^{ \scriptstyle (x-s+r)} $};
      \node[above] at (2,2) {$\scriptstyle i^{ \scriptstyle (y+s-r)}$};
      \node[left] at (0,1) {$ \scriptstyle i^{ \scriptstyle (x-s)}$};
      \node[right] at (1.6,1) {$ \scriptstyle i^{ \scriptstyle (y+s)}$};
      \node[above] at (0.8,1.5) {$ \scriptstyle i^{ \scriptstyle (r)}$};
      \node[below] at (0.8,0.5) {$ \scriptstyle i^{ \scriptstyle (s)}$};
\end{tikzpicture}}
\substack{\cref{MSrel}\\
=\\{}}
\sum_{u \in \Z_{\geq 0}}
\hackcenter{
\begin{tikzpicture}[scale=.8]
\draw[ultra thick, blue] (0,0)--(0,2);
\draw[ultra thick, blue] (1.6,0)--(1.6,2);
  \draw[ultra thick, blue] (0,0)--(0,0.2) .. controls ++(0,0.35) and ++(0,-0.35)  .. (0.4,0.6)--(0.4,1); 
    \draw[ultra thick, blue] (0.4,1)--(0.4,1.4) .. controls ++(0,0.35) and ++(0,-0.35)  .. (0,1.8)--(0,2); 
      \draw[ultra thick, blue] (0.4,0.6) .. controls ++(0,0.35) and ++(0,-0.35)  .. (1.6,1.7); 
        \draw[ultra thick, blue] (1.6,0.3) .. controls ++(0,0.35) and ++(0,-0.35)  .. (0.4,1.4); 
     \node[below] at (0,0) {$ \scriptstyle i^{ \scriptstyle (x)}$};
     \node[below] at (1.6,0) {$ \scriptstyle i^{ \scriptstyle (y)}$};
     \node[above] at (0,2) {$\scriptstyle i^{ \scriptstyle (x-s +r)}$};
      \node[above] at (2,2) {$\scriptstyle i^{ \scriptstyle (y+s-r)}$};
      \node[left] at (0,1) {$ \scriptstyle i^{ \scriptstyle (x-s)}$};
      \node[right] at (1.6,1) {$ \scriptstyle i^{ \scriptstyle (y+s)}$};
      \node[above] at (0.55,1.55) {$ \scriptstyle i^{ \scriptstyle (r)}$};
      \node[below] at (0.6,0.6) {$ \scriptstyle i^{ \scriptstyle (s)}$};
       \node[] at (1.05,1.55) {$ \scriptstyle i^{ \scriptstyle (u)}$};
\end{tikzpicture}}
\substack{\cref{AssocRel},\cref{KnotholeRel}\\
=\\{}}
\sum_{u \in \Z_{\geq 0}}
{ x- u \choose x-s}
\hspace{-5mm}
\hackcenter{
\begin{tikzpicture}[scale=.8]
\draw[ultra thick, blue] (0,0)--(0,2);
\draw[ultra thick, blue] (1.6,0)--(1.6,2);
      \draw[ultra thick, blue] (0,0.3) .. controls ++(0,0.35) and ++(0,-0.35)  .. (1.6,1.7); 
        \draw[ultra thick, blue] (1.6,0.3) .. controls ++(0,0.35) and ++(0,-0.35)  .. (0,1.7); 
     \node[below] at (0,0) {$ \scriptstyle i^{ \scriptstyle (x)}$};
     \node[below] at (1.6,0) {$ \scriptstyle i^{ \scriptstyle (y)}$};
     \node[above] at (0,2) {$\scriptstyle i^{ \scriptstyle (x-s+r)} $};
      \node[above] at (2,2) {$\scriptstyle i^{ \scriptstyle (y+s-r)}$};
       \node[] at (1.1,1.6) {$ \scriptstyle i^{ \scriptstyle (u)}$};
\end{tikzpicture}}.
\end{align}
On the other hand, by similar computations to \cref{simcomps}, we have that:
\begin{align*}
&
\sum_{t \in \Z_{\geq 0}}
{ x-y+r-s
\choose
t}
\hspace{-4mm}
\hackcenter{
\begin{tikzpicture}[scale=.8]
\draw[ultra thick, blue] (0,0)--(0,2);
\draw[ultra thick, blue] (1.6,0)--(1.6,2);
  \draw[ultra thick, blue] (1.6,0)--(1.6,0.2) .. controls ++(0,0.35) and ++(0,-0.35)  .. (0,0.8)--(0,2); 
    \draw[ultra thick, blue] (0,0)--(0,1.2) .. controls ++(0,0.35) and ++(0,-0.35)  .. (1.6,1.8)--(1.6,2); 
      \node[below] at (0,0) {$ \scriptstyle i^{ \scriptstyle (x)}$};
     \node[below] at (1.6,0) {$ \scriptstyle i^{ \scriptstyle (y)}$};
     \node[above] at (0,2) {$ \scriptstyle i^{\scriptstyle (x-s+r)}$};
      \node[above] at (1.6,2) {$ \scriptstyle i^{ \scriptstyle (y+s-r)}$};
      \node[above] at (0.8,1.5) {$ \scriptstyle i^{ \scriptstyle (s-t)}$};
      \node[below] at (0.8,0.5) {$ \scriptstyle i^{ \scriptstyle (r-t)}$};
\end{tikzpicture}}
=
\sum_{t, u \in \Z_{\geq 0}}
{ x-y+r-s
\choose
t}
{y -r + s - u \choose s- t - u}
\hspace{-4mm}
\hackcenter{
\begin{tikzpicture}[scale=.8]
\draw[ultra thick, blue] (0,0)--(0,2);
\draw[ultra thick, blue] (1.6,0)--(1.6,2);
      \draw[ultra thick, blue] (0,0.3) .. controls ++(0,0.35) and ++(0,-0.35)  .. (1.6,1.7); 
        \draw[ultra thick, blue] (1.6,0.3) .. controls ++(0,0.35) and ++(0,-0.35)  .. (0,1.7); 
     \node[below] at (0,0) {$ \scriptstyle i^{ \scriptstyle (x)}$};
     \node[below] at (1.6,0) {$ \scriptstyle i^{ \scriptstyle (y)}$};
     \node[above] at (0,2) {$\scriptstyle i^{ \scriptstyle (x-s+r)} $};
      \node[above] at (2,2) {$\scriptstyle i^{ \scriptstyle (y+s-r)}$};
       \node[] at (1.1,1.6) {$ \scriptstyle i^{ \scriptstyle (u)}$};
\end{tikzpicture}}.
\end{align*}
But, by the Chu-Vandermonde identity, we have that
\begin{align*}
{x-u \choose x-s} = \sum_{u \in \Z_{\geq 0}}{ x-y+r-s
\choose
t}
{y -r + s - u \choose s- t - u},
\end{align*}
so \cref{DiagSwitchRel} holds. Furthermore, we have that
\begin{align*}
\sum_{t \in \Z_{\geq 0}}
(-1)^{t}
\hackcenter{
\begin{tikzpicture}[scale=.8]
\draw[ultra thick, blue] (0,0)--(0,2);
\draw[ultra thick, blue] (1.4,0)--(1.4,2);
  \draw[ultra thick, blue] (0,0)--(0,0.2) .. controls ++(0,0.35) and ++(0,-0.35)  .. (1.4,0.9)--(1.4,2); 
    \draw[ultra thick, blue] (1.4,0)--(1.4,1.1) .. controls ++(0,0.35) and ++(0,-0.35)  .. (0,1.8)--(0,2); 
     \node[below] at (0,0) {$ \scriptstyle i^{ \scriptstyle (x)}$};
     \node[below] at (1.4,0) {$ \scriptstyle i^{ \scriptstyle (y)}$};
     \node[above] at (0,2) {$ \scriptstyle i^{\scriptstyle (y)}$};
      \node[above] at (1.4,2) {$ \scriptstyle i^{\scriptstyle (x)}$};
      \node[left] at (0,1) {$ \scriptstyle i^{\scriptstyle (t)}$};
      \node[right] at (1.4,1) {$ \scriptstyle i^{\scriptstyle (y+x-t)}$};
      \node[above] at (0.7,1.52) {$ \scriptstyle i^{\scriptstyle (y-t)}$};
      \node[below] at (0.7,0.5) {$ \scriptstyle i^{\scriptstyle (x-t)}$};
\end{tikzpicture}}
=
\sum_{t,u \in \Z_{\geq 0}}
(-1)^t
{ x- u \choose t}
\hackcenter{
\begin{tikzpicture}[scale=.8]
\draw[ultra thick, blue] (0,0)--(0,2);
\draw[ultra thick, blue] (1.6,0)--(1.6,2);
      \draw[ultra thick, blue] (0,0.3) .. controls ++(0,0.35) and ++(0,-0.35)  .. (1.6,1.7); 
        \draw[ultra thick, blue] (1.6,0.3) .. controls ++(0,0.35) and ++(0,-0.35)  .. (0,1.7); 
     \node[below] at (0,0) {$ \scriptstyle i^{ \scriptstyle (x)}$};
     \node[below] at (1.6,0) {$ \scriptstyle i^{ \scriptstyle (y)}$};
     \node[above] at (0,2) {$\scriptstyle i^{ \scriptstyle (y)} $};
      \node[above] at (1.6,2) {$\scriptstyle i^{ \scriptstyle (x)}$};
       \node[] at (1.1,1.6) {$ \scriptstyle i^{ \scriptstyle (u)}$};
\end{tikzpicture}}
=
\hackcenter{
\begin{tikzpicture}[scale=.8]
  \draw[ultra thick, blue] (0,0)--(0,0.2) .. controls ++(0,0.6) and ++(0,-0.6)  .. (1.4,1.8)--(1.4,2); 
    \draw[ultra thick, blue] (1.4,0)--(1.4,0.2) .. controls ++(0,0.6) and ++(0,-0.6)  .. (0,1.8)--(0,2); 
     \node[below] at (0,0) {$ \scriptstyle i^{ \scriptstyle (x)}$};
     \node[below] at (1.4,0) {$ \scriptstyle i^{ \scriptstyle (y)}$};
     \node[above] at (0,2) {$ \scriptstyle i^{ \scriptstyle (y)}$};
      \node[above] at (1.4,2) {$ \scriptstyle i^{ \scriptstyle (x)}$};
\end{tikzpicture}},
\end{align*}
where the first equality follows by similar computations to \cref{simcomps}, and the second equality follows from the binomial theorem, since \(\sum_{t \in \Z_{\geq 0}} (-1)^t {x - u \choose t} = \delta_{x,u}\). Thus \cref{CrossingWebRel} holds as well.
\end{proof}

\begin{lemma}\label{Same2}
The relations \cref{MSrel,KnotholeRel} hold in \(\WebAaIAlt\).
\end{lemma}
\begin{proof}
First, we note that \cref{KnotholeRel} can be seen as the special case \(y = 0\) in \cref{DiagSwitchRel}, so we now focus on establishing \cref{MSrel}.
We go by induction on \(n:=x+y\), with the base case \(n=0\) being trivial. Fix \(n\) and assume the claim holds for all \(x'+y'<n\). We first consider the case \(y=z\). We have
\begin{align*}
\hackcenter{}
\hackcenter{
\begin{tikzpicture}[scale=0.8]
  \draw[ultra thick, blue] (0,0)--(0,0.2) .. controls ++(0,0.35) and ++(0,-0.35)  .. (0.7,0.8)--(0.7,1.2)
  .. controls ++(0,0.35) and ++(0,-0.35)  .. (0,1.8)--(0,2); 
    \draw[ultra thick, blue] (1.4,0)--(1.4,0.2) .. controls ++(0,0.35) and ++(0,-0.35)  .. (0.7,0.8)--(0.7,1.2)
  .. controls ++(0,0.35) and ++(0,-0.35)  .. (1.4,1.8)--(1.4,2); 
     \node[below] at (0,0) {$\scriptstyle i^{\scriptstyle (x)}$};
     \node[below] at (1.4,0) {$\scriptstyle i^{\scriptstyle (y)}$};
      \node[above] at (0,2) {$\scriptstyle i^{\scriptstyle (y)}$};
      \node[above] at (1.4,2) {$\scriptstyle i^{\scriptstyle (x)}$};
\end{tikzpicture}}
-
\hackcenter{
\begin{tikzpicture}[scale=.8]
  \draw[ultra thick, blue] (0,0)--(0,0.2) .. controls ++(0,0.6) and ++(0,-0.6)  .. (1.4,1.8)--(1.4,2); 
    \draw[ultra thick, blue] (1.4,0)--(1.4,0.2) .. controls ++(0,0.6) and ++(0,-0.6)  .. (0,1.8)--(0,2); 
     \node[below] at (0,0) {$ \scriptstyle i^{ \scriptstyle (x)}$};
     \node[below] at (1.4,0) {$ \scriptstyle i^{ \scriptstyle (y)}$};
     \node[above] at (0,2) {$ \scriptstyle i^{ \scriptstyle (y)}$};
      \node[above] at (1.4,2) {$ \scriptstyle i^{ \scriptstyle (x)}$};
\end{tikzpicture}}
&
\substack{
\cref{CrossingWebRel}\\
=\\
{}
}
\sum_{t \in \Z_{\geq 1}}
(-1)^{t+1}
\hackcenter{
\begin{tikzpicture}[scale=.8]
\draw[ultra thick, blue] (0,0)--(0,2);
\draw[ultra thick, blue] (1.4,0)--(1.4,2);
  \draw[ultra thick, blue] (0,0)--(0,0.2) .. controls ++(0,0.35) and ++(0,-0.35)  .. (1.4,0.9)--(1.4,2); 
    \draw[ultra thick, blue] (1.4,0)--(1.4,1.1) .. controls ++(0,0.35) and ++(0,-0.35)  .. (0,1.8)--(0,2); 
     \node[below] at (0,0) {$ \scriptstyle i^{ \scriptstyle (x)}$};
     \node[below] at (1.4,0) {$ \scriptstyle i^{ \scriptstyle (y)}$};
     \node[above] at (0,2) {$ \scriptstyle i^{\scriptstyle (y)}$};
      \node[above] at (1.4,2) {$ \scriptstyle i^{\scriptstyle (x)}$};
      \node[left] at (0,1) {$ \scriptstyle i^{\scriptstyle (t)}$};
      \node[above] at (0.7,1.52) {$ \scriptstyle i^{\scriptstyle (y-t)}$};
      \node[below] at (0.7,0.5) {$ \scriptstyle i^{\scriptstyle (x-t)}$};
\end{tikzpicture}}
=
\sum_{
\substack{
t \in \Z_{\geq 1}\\
u \in \Z_{\geq 0}
}
}
(-1)^{t+1}
\hackcenter{
\begin{tikzpicture}[scale=.8]
\draw[ultra thick, blue] (0,0)--(0,2);
\draw[ultra thick, blue] (1.6,0)--(1.6,2);
  \draw[ultra thick, blue] (0,0)--(0,0.2) .. controls ++(0,0.35) and ++(0,-0.35)  .. (0.4,0.6)--(0.4,1); 
    \draw[ultra thick, blue] (0.4,1)--(0.4,1.4) .. controls ++(0,0.35) and ++(0,-0.35)  .. (0,1.8)--(0,2); 
      \draw[ultra thick, blue] (0.4,0.6) .. controls ++(0,0.35) and ++(0,-0.35)  .. (1.6,1.7); 
        \draw[ultra thick, blue] (1.6,0.3) .. controls ++(0,0.35) and ++(0,-0.35)  .. (0.4,1.4); 
     \node[below] at (0,0) {$ \scriptstyle i^{ \scriptstyle (x)}$};
     \node[below] at (1.6,0) {$ \scriptstyle i^{ \scriptstyle (y)}$};
     \node[above] at (0,2) {$\scriptstyle i^{ \scriptstyle (y)} $};
      \node[above] at (2,2) {$\scriptstyle i^{ \scriptstyle (x)}$};
      \node[left] at (0,1) {$ \scriptstyle i^{ \scriptstyle (t)}$};
      \node[right] at (1.6,1) {$ \scriptstyle i^{ \scriptstyle (x-u)}$};
      \node[above] at (0.65,1.55) {$ \scriptstyle i^{ \scriptstyle (y-t)}$};
      \node[below] at (0.65,0.55) {$ \scriptstyle i^{ \scriptstyle (x-t)}$};
       \node[] at (1.05,1.55) {$ \scriptstyle i^{ \scriptstyle (u)}$};
\end{tikzpicture}}
\\
&
\substack{
\cref{AssocRel},\cref{KnotholeRel}
\\
=
\\
{}
}
\sum_{
\substack{
t \in \Z_{\geq 1}\\
u \in \Z_{\geq 0}
}
}
(-1)^{t+1}
{x-u \choose t}
\hackcenter{
\begin{tikzpicture}[scale=.8]
\draw[ultra thick, blue] (0,0)--(0,2);t
\draw[ultra thick, blue] (1.6,0)--(1.6,2);
      \draw[ultra thick, blue] (0,0.3) .. controls ++(0,0.35) and ++(0,-0.35)  .. (1.6,1.7); 
        \draw[ultra thick, blue] (1.6,0.3) .. controls ++(0,0.35) and ++(0,-0.35)  .. (0,1.7); 
     \node[below] at (0,0) {$ \scriptstyle i^{ \scriptstyle (x)}$};
     \node[below] at (1.6,0) {$ \scriptstyle i^{ \scriptstyle (y)}$};
     \node[above] at (0,2) {$\scriptstyle i^{ \scriptstyle (y)} $};
      \node[above] at (2,2) {$\scriptstyle i^{ \scriptstyle (x)}$};
       \node[] at (1.1,1.6) {$ \scriptstyle i^{ \scriptstyle (u)}$};
\end{tikzpicture}}
=
\sum_{u =0}^{x-1}
\hackcenter{
\begin{tikzpicture}[scale=.8]
\draw[ultra thick, blue] (0,0)--(0,2);t
\draw[ultra thick, blue] (1.6,0)--(1.6,2);
      \draw[ultra thick, blue] (0,0.3) .. controls ++(0,0.35) and ++(0,-0.35)  .. (1.6,1.7); 
        \draw[ultra thick, blue] (1.6,0.3) .. controls ++(0,0.35) and ++(0,-0.35)  .. (0,1.7); 
     \node[below] at (0,0) {$ \scriptstyle i^{ \scriptstyle (x)}$};
     \node[below] at (1.6,0) {$ \scriptstyle i^{ \scriptstyle (y)}$};
     \node[above] at (0,2) {$\scriptstyle i^{ \scriptstyle (y)} $};
      \node[above] at (2,2) {$\scriptstyle i^{ \scriptstyle (x)}$};
       \node[] at (1.1,1.6) {$ \scriptstyle i^{ \scriptstyle (u)}$};
\end{tikzpicture}},
\end{align*}
where the second equality follows by the induction assumption, and the fourth equality follows from the fact that \(\sum_{t \in \Z_{\geq 1}}(-1)^{t+1}{x - u \choose t} = -\delta_{x,u} + 1\) for all \(u \leq x\) by the binomial theorem. This completes the induction step in the case \(y=z\). 

Now assume that \(y>z\), (the case \(z>y\) may be approached in a similar fashion). We have
\begin{align*}
\hackcenter{}
\hackcenter{
\begin{tikzpicture}[scale=0.8]
  \draw[ultra thick, blue] (0,0)--(0,0.2) .. controls ++(0,0.35) and ++(0,-0.35)  .. (0.7,0.8)--(0.7,1.2)
  .. controls ++(0,0.35) and ++(0,-0.35)  .. (0,1.8)--(0,2); 
    \draw[ultra thick, blue] (1.4,0)--(1.4,0.2) .. controls ++(0,0.35) and ++(0,-0.35)  .. (0.7,0.8)--(0.7,1.2)
  .. controls ++(0,0.35) and ++(0,-0.35)  .. (1.4,1.8)--(1.4,2); 
     \node[below] at (0,0) {$\scriptstyle i^{\scriptstyle (x)}$};
     \node[below] at (1.4,0) {$\scriptstyle i^{\scriptstyle (y)}$};
      \node[above] at (0,2) {$\scriptstyle i^{\scriptstyle (z)}$};
      \node[above] at (1.4,2) {$\scriptstyle i^{\scriptstyle (w)}$};
\end{tikzpicture}}
&
\substack{
\cref{DiagSwitchRel}
\\
=
\\
{}
}
\sum_{t \in \Z_{\geq 0}}
{ z-y
\choose
t}
\hackcenter{
\begin{tikzpicture}[scale=.8]
\draw[ultra thick, blue] (0,0)--(0,2);
\draw[ultra thick, blue] (1.6,0)--(1.6,2);
  \draw[ultra thick, blue] (1.6,0)--(1.6,0.2) .. controls ++(0,0.35) and ++(0,-0.35)  .. (0,0.8)--(0,2); 
    \draw[ultra thick, blue] (0,0)--(0,1.2) .. controls ++(0,0.35) and ++(0,-0.35)  .. (1.6,1.8)--(1.6,2); 
      \node[below] at (0,0) {$ \scriptstyle i^{ \scriptstyle (x)}$};
     \node[below] at (1.6,0) {$ \scriptstyle i^{ \scriptstyle (y)}$};
     \node[above] at (0,2) {$ \scriptstyle i^{\scriptstyle (z)}$};
      \node[above] at (1.6,2) {$ \scriptstyle i^{ \scriptstyle (w)}$};
      \node[left] at (0,1) {$ \scriptstyle i^{ \scriptstyle (x+z-t)}$};
      \node[right] at (1.6,1) {$ \scriptstyle i^{ \scriptstyle (y-z+t)}$};
      \node[above] at (0.8,1.5) {$ \scriptstyle i^{ \scriptstyle (x-t)}$};
      \node[below] at (0.8,0.5) {$ \scriptstyle i^{ \scriptstyle (z-t)}$};
\end{tikzpicture}}
=
\sum_{t,u \in \Z_{\geq 0}}{y+z \choose t}
\hackcenter{
\begin{tikzpicture}[scale=.8]
\draw[ultra thick, blue] (0,0)--(0,2);
\draw[ultra thick, blue] (-1.6,0)--(-1.6,2);
  \draw[ultra thick, blue] (0,0)--(0,0.2) .. controls ++(0,0.35) and ++(0,-0.35)  .. (-0.4,0.6)--(-0.4,1); 
    \draw[ultra thick, blue] (-0.4,1)--(-0.4,1.4) .. controls ++(0,0.35) and ++(0,-0.35)  .. (0,1.8)--(0,2); 
      \draw[ultra thick, blue] (-0.4,0.6) .. controls ++(0,0.35) and ++(0,-0.35)  .. (-1.6,1.7); 
        \draw[ultra thick, blue] (-1.6,0.3) .. controls ++(0,0.35) and ++(0,-0.35)  .. (-0.4,1.4); 
     \node[below] at (0,0) {$ \scriptstyle i^{ \scriptstyle (y)}$};
     \node[below] at (-1.6,0) {$ \scriptstyle i^{ \scriptstyle (x)}$};
     \node[above] at (0,2) {$\scriptstyle i^{ \scriptstyle (w)} $};
      \node[above] at (-1.6,2) {$\scriptstyle i^{ \scriptstyle (z)}$};
      \node[below] at (-1,0.8) {$ \scriptstyle i^{ \scriptstyle (u)}$};
       \node[] at (-0.6,1.8) {$ \scriptstyle i^{ \scriptstyle (x-t)}$};
\end{tikzpicture}}\\
&
\substack{
\cref{AssocRel},\cref{KnotholeRel}
\\
=
{}
\\
}
\sum_{t,u \in \Z_{\geq 0}}{z-y \choose t}
 {w-u \choose  w-x+t}
\hackcenter{
\begin{tikzpicture}[scale=.8]
\draw[ultra thick, blue] (0,0)--(0,2);
\draw[ultra thick, blue] (1.6,0)--(1.6,2);
      \draw[ultra thick, blue] (0,0.3) .. controls ++(0,0.35) and ++(0,-0.35)  .. (1.6,1.7); 
        \draw[ultra thick, blue] (1.6,0.3) .. controls ++(0,0.35) and ++(0,-0.35)  .. (0,1.7); 
     \node[below] at (0,0) {$ \scriptstyle i^{ \scriptstyle (x)}$};
     \node[below] at (1.6,0) {$ \scriptstyle i^{ \scriptstyle (y)}$};
     \node[above] at (0,2) {$\scriptstyle i^{ \scriptstyle (z)} $};
      \node[above] at (2,2) {$\scriptstyle i^{ \scriptstyle (w)}$};
       \node[] at (0.6,0.45) {$ \scriptstyle i^{ \scriptstyle (u)}$};
\end{tikzpicture}}
=
\sum_{u \in \Z_{\geq 0}}
\hackcenter{
\begin{tikzpicture}[scale=.8]
\draw[ultra thick, blue] (0,0)--(0,2);
\draw[ultra thick, blue] (1.6,0)--(1.6,2);
      \draw[ultra thick, blue] (0,0.3) .. controls ++(0,0.35) and ++(0,-0.35)  .. (1.6,1.7); 
        \draw[ultra thick, blue] (1.6,0.3) .. controls ++(0,0.35) and ++(0,-0.35)  .. (0,1.7); 
     \node[below] at (0,0) {$ \scriptstyle i^{ \scriptstyle (x)}$};
     \node[below] at (1.6,0) {$ \scriptstyle i^{ \scriptstyle (y)}$};
     \node[above] at (0,2) {$\scriptstyle i^{ \scriptstyle (z)} $};
      \node[above] at (2,2) {$\scriptstyle i^{ \scriptstyle (w)}$};
       \node[] at (0.6,0.45) {$ \scriptstyle i^{ \scriptstyle (u)}$};
\end{tikzpicture}},
\end{align*}
where the second equality follows from the induction assumption and the fourth equality follows by an application of the index shift formula (see \cite[(6.69)]{G}). This completes the induction step, and the proof.
\end{proof}

\cref{Same1,Same2} together imply

\begin{corollary}\label{SAME}
There is an isomorphism \(\WebAaI \cong \WebAaIAlt\) given by the identity on objects and diagrams.
\end{corollary}

From this point forward we will work with \(\WebAaI\) as defined in \cref{defwebaa} but also freely use relations \cref{DiagSwitchRel,CrossingWebRel} thanks to \cref{Same1}.

\subsection{Monotone relations in \texorpdfstring{$\WebAaI$}{WebAa}} As the statements throughout this subsection deal with monotone diagrams (where all strands are colored by the same \(i\) for some \(i \in I\)), for the sake of space in proofs we will often write \(x\) in place of \(i^{(x)}\), etc.  Because the idempotents and coupons do not play a role here, the relations in this section are analogous to those which have appeared in previous work; e.g., \cite{CKM, QS, ST, TVW} for the quantum setting, and \cite{BEPO} for the non-quantum setting.  In particular, most results which appear here can be extracted from existing work in the literature, but we choose to provide full proofs for the reader's reference.

\begin{lemma}\label{RevDiagSwitchLem} The following equality holds in \(\WebAaI\):
\begin{align*}
\hackcenter{}
\hackcenter{
\begin{tikzpicture}[scale=0.8]
\draw[ ultra thick, color=blue] (0,0)--(0,2);
\draw[ ultra thick, color=blue] (1.4,0)--(1.4,2);
  \draw[ ultra thick, color=blue] (1.4,0)--(1.4,0.2) .. controls ++(0,0.35) and ++(0,-0.35)  .. (0,0.8)--(0,2); 
    \draw[ ultra thick, color=blue] (0,0)--(0,1.2) .. controls ++(0,0.35) and ++(0,-0.35)  .. (1.4,1.8)--(1.4,2); 
     \node[below] at (0,0) {$\scriptstyle i^{\scriptstyle(a)}$};
     \node[below] at (1.4,0) {$\scriptstyle i^{\scriptstyle(b)}$};
     \node[above] at (0,2) {$\scriptstyle i^{\scriptstyle(a-s+r)}$};
      \node[above] at (1.6,2) {$\scriptstyle i^{\scriptstyle(b+s-r)}$};
      \node[left] at (0,1) {$\scriptstyle i^{\scriptstyle(a+r)}$};
      \node[right] at (1.4,1) {$\scriptstyle i^{\scriptstyle(b-r)}$};
      \node[above] at (0.7,1.5) {$\scriptstyle i^{\scriptstyle(s)}$};
      \node[below] at (0.7,0.5) {$\scriptstyle i^{\scriptstyle(r)}$};
\end{tikzpicture}}
\;
=
\;
\sum_{t \in \Z_{\geq 0}}
\binom{ -a+b-r+s}{t}
\hackcenter{
\begin{tikzpicture}[scale=0.8]
\draw[ ultra thick, color=blue] (0,0)--(0,2);
\draw[ ultra thick, color=blue] (1.4,0)--(1.4,2);
  \draw[ ultra thick, color=blue] (0,0)--(0,0.2) .. controls ++(0,0.35) and ++(0,-0.35)  .. (1.4,0.8)--(1.4,2); 
    \draw[ ultra thick, color=blue] (1.4,0)--(1.4,1.2) .. controls ++(0,0.35) and ++(0,-0.35)  .. (0,1.8)--(0,2); 
     \node[below] at (0,0) {$\scriptstyle i^{\scriptstyle(a)}$};
     \node[below] at (1.4,0) {$\scriptstyle i^{\scriptstyle(b)}$};
     \node[above] at (0,2) {$\scriptstyle i^{\scriptstyle(a-s+r)}$};
      \node[above] at (1.6,2) {$\scriptstyle i^{\scriptstyle(b+s-r)}$};
      \node[left] at (0,1) {$\scriptstyle i^{\scriptstyle(a-s+t)}$};
      \node[right] at (1.4,1) {$\scriptstyle i^{\scriptstyle(b+s-t)}$};
      \node[above] at (0.7,1.6) {$\scriptstyle i^{\scriptstyle(r-t)}$};
      \node[below] at (0.7,0.5) {$\scriptstyle i^{\scriptstyle(s-t)}$};
\end{tikzpicture}},
\end{align*}
for all \(i \in I\), $a,b, r, s \in \Z_{\geq 0}$.
\end{lemma}
\begin{proof}
Recall that in diagrams we write \(x\) in place of \(i^{(x)}\), as mentioned in the introduction to this subsection.
For space, we write
\begin{align}\label{RLDef}
L(x,y) :=
\hackcenter{}
\hackcenter{
\begin{tikzpicture}[scale=0.8]
\draw[ ultra thick, color=blue] (0,0)--(0,2);
\draw[ ultra thick, color=blue] (1.4,0)--(1.4,2);
  \draw[ ultra thick, color=blue] (1.4,0)--(1.4,0.2) .. controls ++(0,0.35) and ++(0,-0.35)  .. (0,0.8)--(0,2); 
    \draw[ ultra thick, color=blue] (0,0)--(0,1.2) .. controls ++(0,0.35) and ++(0,-0.35)  .. (1.4,1.8)--(1.4,2); 
     \node[above] at (0,-0.4) {$\scriptstyle a$};
     \node[above] at (1.4,-0.4) {$\scriptstyle b$};
     \node[above] at (0,2) {$\scriptstyle a-y+x$};
      \node[above] at (1.4,2) {$\scriptstyle b+y-x$};
      \node[left] at (0,1) {$\scriptstyle a+x$};
      \node[right] at (1.4,1) {$\scriptstyle b-x$};
      \node[above] at (0.7,1.5) {$\scriptstyle y$};
      \node[below] at (0.7,0.5) {$\scriptstyle x$};
\end{tikzpicture}}
\qquad
\qquad
\textup{and}
\qquad
\qquad
R(x,y):=
\hackcenter{
\begin{tikzpicture}[scale=0.8]
\draw[ ultra thick, color=blue] (0,0)--(0,2);
\draw[ ultra thick, color=blue] (1.4,0)--(1.4,2);
  \draw[ ultra thick, color=blue] (0,0)--(0,0.2) .. controls ++(0,0.35) and ++(0,-0.35)  .. (1.4,0.8)--(1.4,2); 
    \draw[ ultra thick, color=blue] (1.4,0)--(1.4,1.2) .. controls ++(0,0.35) and ++(0,-0.35)  .. (0,1.8)--(0,2); 
     \node[above] at (0,-0.4) {$\scriptstyle a$};
     \node[above] at (1.4,-0.4) {$\scriptstyle b$};
     \node[above] at (0,2) {$\scriptstyle a-y+x$};
      \node[above] at (1.4,2) {$\scriptstyle b+y-x$};
      \node[left] at (0,1) {$\scriptstyle a-y$};
      \node[right] at (1.4,1) {$\scriptstyle b+y$};
      \node[above] at (0.7,1.5) {$\scriptstyle x$};
      \node[below] at (0.7,0.5) {$\scriptstyle y$};
\end{tikzpicture}},
\end{align}
for any \(x,y \in \Z_{\geq 0}\). We prove the lemma statement by induction on \(n=r+s\). The claim is trivial for \(n=0,1\), so let \(n \geq 2\) and assume the claim holds for all \(n' < n\). We have
\begin{align*}
L(r,s) 
\;
&\substack{ \textup{\cref{DiagSwitchRel}} \\ =\\ \,}
\;
 R(r,s) - \sum_{w\in \Z_{>0}} \binom{a-b+r-s}{w} L(r-w,s-w)\\
 &=
 R(r,s)-\sum_{w \in \Z_{> 0}} \binom{a-b+r-s}{w} \sum_{u \in \Z_{\geq 0}} \binom{-a+b - r + s}{u} R(r-w-u,s-w-u)\\
 &=R(r,s)
 -\sum_{t \in \Z_{>0}} \sum_{w=1}^t \binom{a- b+r-s}{w} \binom{-a+b-r+s}{t-w} R(r-t,s-t)\\
 &=
 R(r,s)+\sum_{t \in \Z_{>0}} \binom{-a+b-r+s}{t} R(r-t,s-t),
\end{align*}
where the second equality follows by the induction assumption, the last equality follows from the Chu--Vandermonde identity. This completes the proof.
\end{proof}

\begin{lemma}\label{CoxeterWeb}
The following equalities hold in \(\WebAaI\):
\begin{align*}
\hackcenter{}
\hackcenter{
\begin{tikzpicture}[scale=0.8]
\draw[ ultra thick, color=blue] (0,0)--(0,2);
\draw[ ultra thick, color=blue] (1.6,0)--(1.6,2);
\draw[ ultra thick, color=blue] (3.2,0)--(3.2,2);
\draw[ ultra thick, color=blue] (0,0)--(0,1) .. controls ++(0,0.35) and ++(0,-0.35) .. (1.6,1.6)--(1.6,2);
\draw[ ultra thick, color=blue] (0,0)--(0,0.2) .. controls ++(0,0.35) and ++(0,-0.35) .. (1.6,0.8)--(1.6,2);
\draw[ ultra thick, color=blue] (1.6,0)--(1.6,1) .. controls ++(0,0.35) and ++(0,-0.35) .. (3.2,1.6)--(3.2,2);
      \node[below] at (0.8,0.5) {$\scriptstyle i^{\scriptstyle(s)}$};
      \node[above] at (0.8,1.3) {$\scriptstyle i^{\scriptstyle(s'')}$};
      \node[above] at (2.4,1.3) {$\scriptstyle i^{\scriptstyle(s')}$};
      \node[below] at (0,0) {$\scriptstyle i^{\scriptstyle(a)}$};
      \node[below] at (1.6,0) {$\scriptstyle i^{\scriptstyle(b)}$};
      \node[below] at (3.2,0) {$\scriptstyle i^{\scriptstyle(c)}$};
\end{tikzpicture}}
=
\sum_{t \in \Z_{\geq 0}}
\binom{s-s'+s''}{t}
\hackcenter{
\begin{tikzpicture}[scale=0.8]
\draw[ ultra thick, color=blue] (0,0)--(0,2);
\draw[ ultra thick, color=blue] (1.6,0)--(1.6,2);
\draw[ ultra thick, color=blue] (3.3,0)--(3.3,2);
\draw[ ultra thick, color=blue] (1.6,0)--(1.6,1.2) .. controls ++(0,0.35) and ++(0,-0.35) .. (3.3,1.8)--(3.3,2);
\draw[ ultra thick, color=blue] (1.6,0)--(1.6,0.4) .. controls ++(0,0.35) and ++(0,-0.35) .. (3.3,1)--(3.3,2);
\draw[ ultra thick, color=blue] (0,0)--(0,0.4) .. controls ++(0,0.35) and ++(0,-0.35) .. (1.6,1)--(1.6,2);
      \node[above] at (0.8,0.7) {$\scriptstyle i^{\scriptstyle(s+s'')}$};
      \node[below] at (2.4,0.7) {$\scriptstyle i^{\scriptstyle(s''-t)}$};
      \node[above] at (2.5,1.6) {$\scriptstyle i^{\scriptstyle(s'\hspace{-0.5mm}-\hspace{-0.5mm}s''\hspace{-0.7mm}+t)}$};
      \node[below] at (0,0) {$\scriptstyle i^{\scriptstyle(a)}$};
      \node[below] at (1.6,0) {$\scriptstyle i^{\scriptstyle(b)}$};
      \node[below] at (3.3,0) {$\scriptstyle i^{\scriptstyle(c)}$};
\end{tikzpicture}},
\\
\hackcenter{}
\hackcenter{
\begin{tikzpicture}[scale=0.8]
\draw[ ultra thick, color=blue] (0,0)--(0,-2);
\draw[ ultra thick, color=blue] (1.6,0)--(1.6,-2);
\draw[ ultra thick, color=blue] (3.2,0)--(3.2,-2);
\draw[ ultra thick, color=blue] (0,0)--(0,-1) .. controls ++(0,-0.35) and ++(0,0.35) .. (1.6,-1.6)--(1.6,-2);
\draw[ ultra thick, color=blue] (0,0)--(0,-0.2) .. controls ++(0,-0.35) and ++(0,0.35) .. (1.6,-0.8)--(1.6,-2);
\draw[ ultra thick, color=blue] (1.6,0)--(1.6,-1) .. controls ++(0,-0.35) and ++(0,0.35) .. (3.2,-1.6)--(3.2,-2);
      \node[above] at (0.8,-0.5) {$\scriptstyle i^{\scriptstyle(r'')}$};
      \node[below] at (0.8,-1.3) {$\scriptstyle i^{\scriptstyle(r)}$};
      \node[above] at (2.4,-1.3) {$\scriptstyle i^{\scriptstyle(r')}$};
      \node[below] at (0,-2) {$\scriptstyle i^{\scriptstyle(a)}$};
      \node[below] at (1.6,-2) {$\scriptstyle i^{\scriptstyle(b)}$};
      \node[below] at (3.2,-2) {$\scriptstyle i^{\scriptstyle(c)}$};
\end{tikzpicture}}
=
\sum_{t \in \Z_{\geq 0}}
\binom{r-r'+r''}{t}
\hackcenter{
\begin{tikzpicture}[scale=0.8]
\draw[ ultra thick, color=blue] (0,0)--(0,-2);
\draw[ ultra thick, color=blue] (1.6,0)--(1.6,-2);
\draw[ ultra thick, color=blue] (3.3,0)--(3.3,-2);
\draw[ ultra thick, color=blue] (1.6,0)--(1.6,-1.2) .. controls ++(0,-0.35) and ++(0,0.35) .. (3.3,-1.8)--(3.3,-2);
\draw[ ultra thick, color=blue] (1.6,0)--(1.6,-0.4) .. controls ++(0,-0.35) and ++(0,0.35) .. (3.3,-1)--(3.3,-2);
\draw[ ultra thick, color=blue] (0,0)--(0,-0.4) .. controls ++(0,-0.35) and ++(0,0.35) .. (1.6,-1)--(1.6,-2);
      \node[below] at (0.8,-0.7) {$\scriptstyle i^{(r+r'')}$};
      \node[below] at (2.5,-1.5) {$\scriptstyle i^{(r''-t)}$};
      \node[above] at (2.5,-0.6) {$\scriptstyle i^{(r'\hspace{-0.5mm}-\hspace{-0.5mm}r''\hspace{-0.5mm}+t)}$};
      \node[below] at (0,-2) {$\scriptstyle i^{(a)}$};
      \node[below] at (1.6,-2) {$\scriptstyle i^{(b)}$};
      \node[below] at (3.3,-2) {$\scriptstyle i^{(c)}$};
\end{tikzpicture}},
\end{align*}
for all \(i \in I\),  \(a,b,c,r,r',r'',s,s',s'' \in \Z_{\geq 0}\).
\end{lemma}
\begin{proof}
The proof given in \cite[Lemma 2.9]{E}, in the context of a web category where relations identical to \cref{AssocRel,DiagSwitchRel} hold is directly applicable to our \(\WebAaI\) category.
\end{proof}

\begin{lemma}\label{CrossAbsorb}
For all \(i \in I\), \(a,b \in \Z_{\geq 0}\) we have the following equalities in \(\WebAaI\):
\begin{align*}
\hackcenter{}
\hackcenter{
\begin{tikzpicture}[scale=0.8]
  \draw[ ultra thick, color=blue] (0,0)--(0,0.2) .. controls ++(0,0.35) and ++(0,-0.35)  .. (0.8,0.8) .. controls ++(0,0.35) and ++(0,-0.35)  .. (0.4,1.4)--(0.4,1.6); 
    \draw[ ultra thick, color=blue] (0.8 ,0)--(0.8, 0.2) .. controls ++(0,0.35) and ++(0,-0.35)  .. (0,0.8) .. controls ++(0,0.35) and ++(0,-0.35)  .. (0.4,1.4)--(0.4,1.6);
     \node[below] at (0,0) {$\scriptstyle i^{\scriptstyle(a)}$};
     \node[below] at (0.8,0) {$\scriptstyle i^{\scriptstyle(b)}$};
       \node[above] at (0.4,1.6) {$\scriptstyle i^{\scriptstyle(a+b)}$};
\end{tikzpicture}}
=
\hackcenter{
\begin{tikzpicture}[scale=0.8]
  \draw[ ultra thick, color=blue] (0,0)--(0,0.8) .. controls ++(0,0.35) and ++(0,-0.35)  .. (0.4,1.4)--(0.4,1.6); 
    \draw[ ultra thick, color=blue] (0.8 ,0)--(0.8,0.8) .. controls ++(0,0.35) and ++(0,-0.35)  .. (0.4,1.4)--(0.4,1.6);
     \node[below] at (0,0) {$\scriptstyle i^{\scriptstyle(a)}$};
     \node[below] at (0.8,0) {$\scriptstyle i^{\scriptstyle(b)}$};
       \node[above] at (0.4,1.6) {$\scriptstyle i^{\scriptstyle(a+b)}$};
\end{tikzpicture}}
\qquad
\qquad
\textup{and}
\qquad
\qquad
\hackcenter{
\begin{tikzpicture}[scale=0.8]
  \draw[ ultra thick, color=blue] (0,0)--(0,-0.2) .. controls ++(0,-0.35) and ++(0,0.35)  .. (0.8,-0.8) .. controls ++(0,-0.35) and ++(0,0.35)  .. (0.4,-1.4)--(0.4,-1.6); 
    \draw[ ultra thick, color=blue] (0.8 ,0)--(0.8, -0.2) .. controls ++(0,-0.35) and ++(0,0.35)  .. (0,-0.8) .. controls ++(0,-0.35) and ++(0,0.35)  .. (0.4,-1.4)--(0.4,-1.6);
     \node[above] at (0,0) {$\scriptstyle i^{\scriptstyle(a)}$};
     \node[above] at (0.8,0) {$\scriptstyle i^{\scriptstyle(b)}$};
       \node[below] at (0.4,-1.6) {$\scriptstyle i^{\scriptstyle(a+b)}$};
\end{tikzpicture}}
=
\hackcenter{
\begin{tikzpicture}[scale=0.8]
  \draw[ ultra thick, color=blue] (0,0)--(0,-0.8) .. controls ++(0,-0.35) and ++(0,0.35)  .. (0.4,-1.4)--(0.4,-1.6); 
    \draw[ ultra thick, color=blue] (0.8 ,0)--(0.8,-0.8) .. controls ++(0,-0.35) and ++(0,0.35)  .. (0.4,-1.4)--(0.4,-1.6);
     \node[above] at (0,0) {$\scriptstyle i^{\scriptstyle(a)}$};
     \node[above] at (0.8,0) {$\scriptstyle i^{\scriptstyle(b)}$};
       \node[below] at (0.4,-1.6) {$\scriptstyle i^{\scriptstyle(a+b)}$};
\end{tikzpicture}}.
\end{align*}
\end{lemma}
\begin{proof}
We prove the first equality. The second is similar. We have
\begin{align*}
\hackcenter{
{}
}
\hackcenter{
\begin{tikzpicture}[scale=0.8]
  \draw[ ultra thick, color=blue] (0,0)--(0,0.2) .. controls ++(0,0.35) and ++(0,-0.35)  .. (0.8,1)--(0.8,1.1) .. controls ++(0,0.35) and ++(0,-0.35)  .. (0.4,1.9)--(0.4,2.2); 
    \draw[ ultra thick, color=blue] (0.8 ,0)--(0.8, 0.2) .. controls ++(0,0.35) and ++(0,-0.35)  .. (0,1)--(0,1.1) .. controls ++(0,0.35) and ++(0,-0.35)  .. (0.4,1.9)--(0.4,2.2);
     \node[below] at (0,0) {$\scriptstyle a$};
     \node[below] at (0.8,0) {$\scriptstyle b$};
       \node[above] at (0.4,2.2) {$\scriptstyle a+b$};
\end{tikzpicture}}
&=
\sum_{
s-r=a-b
}
(-1)^{a-s}
\hackcenter{
\begin{tikzpicture}[scale=0.8]
\draw[ ultra thick, color=blue] (0,0)--(0,2.2);
 \draw[ ultra thick, color=blue] (0.8,0)--(0.8,1.5) .. controls ++(0,0.35) and ++(0,-0.35) .. (0,2.1)--(0,2.2);
  \draw[ ultra thick, color=blue] (0,0)--(0,0.2) .. controls ++(0,0.35) and ++(0,-0.35) .. (0.8,0.8);
  \draw[ ultra thick, color=blue] (0.8,0)--(0.8,0.8) .. controls ++(0,0.35) and ++(0,-0.35) .. (0,1.4);      \node[below] at (0,0) {$\scriptstyle a$};
      \node[below] at (0.8,0) {$\scriptstyle b$};
           \node[below] at (0.4,0.5) {$\scriptstyle s$};
           \node[above] at (0.4,1.1) {$\scriptstyle r$};
             \node[above] at (0,2.2) {$\scriptstyle a+b$};
\end{tikzpicture}}
\;
\substack{ \textup{\cref{DiagSwitchRel}} \\ =}
\;
\sum_{
s-r=a-b
}
(-1)^{a-s}
\binom{b+s}{r}
\hackcenter{
\begin{tikzpicture}[scale=0.8]
\draw[ ultra thick, color=blue] (0,0)--(0,2.2);
  \draw[ ultra thick, color=blue] (0,0)--(0,0.2) .. controls ++(0,0.35) and ++(0,-0.35) .. (0.8,0.8);
  \draw[ ultra thick, color=blue] (0.8,0)--(0.8,0.8) .. controls ++(0,0.35) and ++(0,-0.35) .. (0,1.4);
      \node[below] at (0,0) {$\scriptstyle a$};
      \node[below] at (0.8,0) {$\scriptstyle b$};
           \node[below] at (0.4,0.5) {$\scriptstyle s$};
           \node[above] at (0.5,1.1) {$\scriptstyle b+s$};
             \node[above] at (0,2.2) {$\scriptstyle a+b$};
\end{tikzpicture}}
\\
&
\;
\substack{ \textup{\cref{AssocRel}} \\ =}
\;
\sum_{
s-r=a-b
}
(-1)^{a-s}
\binom{b+s}{r}
\hackcenter{
\begin{tikzpicture}[scale=0.8]
\draw[ ultra thick, color=blue] (0,0)--(0,2.2);
  \draw[ ultra thick, color=blue] (0,0)--(0,0.2) .. controls ++(0,0.35) and ++(0,-0.35) .. (0.4,0.7)
     .. controls ++(0,0.35) and ++(0,-0.35) .. (0,1.2);
  \draw[ ultra thick, color=blue] (0.8,0)--(0.8,1.2) .. controls ++(0,0.35) and ++(0,-0.35) .. (0,1.8);
      \node[below] at (0,0) {$\scriptstyle a$};
      \node[below] at (0.8,0) {$\scriptstyle b$};
           \node[below] at (0.4,0.5) {$\scriptstyle s$};
             \node[above] at (0,2.2) {$\scriptstyle a+b$};
\end{tikzpicture}}
\;
\substack{ \textup{\cref{DiagSwitchRel}} \\ =}
\;
\sum_{
s-r=a-b
}
(-1)^{a-s}
\binom{b+s}{r}
\binom{a}{s}
\hackcenter{
\begin{tikzpicture}[scale=0.8]
\draw[ ultra thick, color=blue] (0,0)--(0,2.2);
  \draw[ ultra thick, color=blue] (0.8,0)--(0.8,1.2) .. controls ++(0,0.35) and ++(0,-0.35) .. (0,1.8);
      \node[below] at (0,0) {$\scriptstyle a$};
      \node[below] at (0.8,0) {$\scriptstyle b$};
             \node[above] at (0,2.2) {$\scriptstyle a+b$};
\end{tikzpicture}}.
\end{align*}
Considering the coefficient in the last term and using the substitution \(t:=a-s\), we may write:
\begin{align*}
\sum_{
s-r=a-b
}
(-1)^{a-s}
\binom{b+s}{r}
\binom{a}{s}
=
\sum_{t=0}^a
(-1)^t
\binom{b+a-t}{a}
\binom{a}{t}=1.
\end{align*}
The last equality follows from an application of Euler's finite difference theorem (see \cite[(10.13)]{G}). This completes the proof.
\end{proof}

\begin{lemma}\label{DoubleCross}
For all \(i \in I\), \(a,b \in \Z_{\geq 0}\), the following equality holds in \(\WebAaI\):
\begin{align*}
\hackcenter{}
\hackcenter{
\begin{tikzpicture}[scale=0.8]
  \draw[ ultra thick, color=blue] (0,0)--(0,0.2) .. controls ++(0,0.35) and ++(0,-0.35)  .. (0.8,0.8)--(0.8,1) .. controls ++(0,0.35) and ++(0,-0.35)  .. (0,1.6)--(0,1.8); 
    \draw[ ultra thick, color=blue] (0.8 ,0)--(0.8, 0.2) .. controls ++(0,0.35) and ++(0,-0.35)  .. (0,0.8)--(0,1) .. controls ++(0,0.35) and ++(0,-0.35)  .. (0.8,1.6)--(0.8,1.8);
     \node[below] at (0,0) {$\scriptstyle i^{\scriptstyle(a)}$};
     \node[below] at (0.8,0) {$\scriptstyle i^{\scriptstyle(b)}$};
  \end{tikzpicture}}
\;
=
\;
\hackcenter{
\begin{tikzpicture}[scale=0.8]
  \draw[ ultra thick, color=blue] (0,0)--(0,1.8); 
    \draw[ ultra thick, color=blue] (0.8 ,0)--(0.8,1.8);
       \node[below] at (0,0) {$\scriptstyle i^{\scriptstyle(a)}$};
     \node[below] at (0.8,0) {$\scriptstyle i^{\scriptstyle(b)}$};
\end{tikzpicture}}.
\end{align*}
\end{lemma}
\begin{proof}
We have
\begin{align*}
\hackcenter{}
\hackcenter{
\begin{tikzpicture}[scale=0.8]
  \draw[ ultra thick, color=blue] (0,0)--(0,0.2) .. controls ++(0,0.35) and ++(0,-0.35)  .. (0.8,1.4)--(0.8,1.6) .. controls ++(0,0.35) and ++(0,-0.35)  .. (0,2.8)--(0,3); 
    \draw[ ultra thick, color=blue] (0.8 ,0)--(0.8, 0.2) .. controls ++(0,0.35) and ++(0,-0.35)  .. (0,1.4)--(0,1.6) .. controls ++(0,0.35) and ++(0,-0.35)  .. (0.8,2.8)--(0.8,3);
     \node[above] at (0,-0.4) {$\scriptstyle a$};
     \node[above] at (0.8,-0.4) {$\scriptstyle b$};
\end{tikzpicture}}
&=
\hackcenter{}
\sum_{
\substack{
s-r=a-b\\
s' - r' = b-a}
}
(-1)^{a+b-s-s'}
\hackcenter{
\begin{tikzpicture}[scale=0.8]
\draw[ ultra thick, color=blue] (0,0)--(0,3);
\draw[ ultra thick, color=blue] (1.4,0)--(1.4,3);
  \draw[ ultra thick, color=blue] (0,0)--(0,0.2) .. controls ++(0,0.35) and ++(0,-0.35)  .. (1.4,0.7)--(1.4,0.9)
      .. controls ++(0,0.35) and ++(0,-0.35)  .. (0,1.4)--(0,1.6)
       .. controls ++(0,0.35) and ++(0,-0.35)  .. (1.4,2.1)--(1.4,2.3)
       .. controls ++(0,0.35) and ++(0,-0.35)  .. (0,2.8)--(0,3);
     \node[above] at (0,-0.4) {$\scriptstyle a$};
     \node[above] at (1.4,-0.4) {$\scriptstyle b$};
         \node[above] at (0.7,2.55) {$\scriptstyle r'$};
      \node[above] at (0.7,1.85) {$\scriptstyle s'$};
      \node[below] at (0.7,1.15) {$\scriptstyle r$};
      \node[below] at (0.7,0.45) {$\scriptstyle s$};
\end{tikzpicture}}
\;
\substack{ \textup{\cref{RevDiagSwitchLem}} \\ =\\ \,}
\;
\sum_{\
\substack{
s-r=a-b\\
s' - r' = b-a\\
t \in \Z_{\geq 0}
}}
(-1)^{a+b-s-s'}
\binom{s+s'}{t}
\hackcenter{
\begin{tikzpicture}[scale=0.8]
\draw[ ultra thick, color=blue] (0,0)--(0,3);
\draw[ ultra thick, color=blue] (1.4,0)--(1.4,3);
  \draw[ ultra thick, color=blue] (0,0)--(0,0.2) .. controls ++(0,0.35) and ++(0,-0.35)  .. (1.4,0.7)--(1.4,2.3)
       .. controls ++(0,0.35) and ++(0,-0.35)  .. (0,2.8)--(0,3);
    \draw[ ultra thick, color=blue] (0,0)--(0,0.9)
      .. controls ++(0,0.35) and ++(0,-0.35)  .. (1.4,1.4)--(1.4,1.6)
       .. controls ++(0,0.35) and ++(0,-0.35)  .. (0,2.1)--(0, 3);
     \node[above] at (0,-0.4) {$\scriptstyle a$};
     \node[above] at (1.4,-0.4) {$\scriptstyle b$};
         \node[above] at (0.7,2.55) {$\scriptstyle r'$};
      \node[above] at (0.7,1.85) {$\scriptstyle r-t$};
      \node[below] at (0.7,1.15) {$\scriptstyle s'-t$};
      \node[below] at (0.7,0.45) {$\scriptstyle s$};
\end{tikzpicture}}
\\
&
\;
\substack{ \textup{ \cref{KnotholeRel}} \\ =\\ \,}
\;
\sum_{\
\substack{
s-r=a-b\\
s' - r' = b-a\\
t \in \Z_{\geq 0}
}}
(-1)^{a+b-s-s'}
\binom{s+s' - t}{s'+a-b}
\binom{s+s' - t}{s}
\binom{s+s'}{t}
R(s+s'-t,s+s'-t),
\end{align*}
where we use the fact that \(s+s' = r+r'\) and \(r' = s'-a+b\) in the last equality, and the notation of \cref{RLDef}. Writing \(m:=s+s' - t\), \(k:=s+s'+a-b\), \(c:=b-a\), we may rewrite this as
\begin{align}\label{binsum1}
\sum_{ k,m \in \Z_{\geq 0}}
(-1)^{k}
\binom{c+k}{m}
\sum_{s=0}^k
\binom{m}{k-s}
\binom{m}{s}
R(m,m).
\end{align}
Now we may use the binomial identities
\begin{align*}
\binom{c +k}{m} = \sum_{u=0}^m \binom{c}{m-u} \binom{k}{u}
\qquad 
\qquad
\textup{and}
\qquad
\qquad
\sum_{s=0}^k \binom{m}{k-s} \binom{m}{s} = \binom{2m}{k}
\end{align*}
to rewrite \cref{binsum1} as
\begin{align*}
&\sum_{m \in \Z_{\geq 0}} \sum_{u=0}^m
 \binom{c}{m-u} \sum_{k=0}^{2m} (-1)^k \binom{2m}{k} \binom{k}{u} R(m,m)\\
 &\hspace{20mm}=
 \sum_{m \in \Z_{\geq 0}} \sum_{u=0}^m
 \binom{c}{m-u} \sum_{k=0}^{2m} (-1)^k \binom{2m}{u} \binom{2m-u}{k-u} R(m,m)\\
  &\hspace{20mm}=
 \sum_{m \in \Z_{\geq 0}} \sum_{u=0}^m \binom{2m}{u}
 \binom{c}{m-u} \sum_{\ell=0}^{2m-u} (-1)^{(\ell+u)}  \binom{2m-u}{\ell} R(m,m)\\
   &\hspace{20mm}=R(0,0),
\end{align*}
since the binomial theorem implies
\(
\sum_{\ell=0}^{2m-u} (-1)^\ell  \binom{2m-u}{\ell} = 0
\)
whenever \(m>0\).
\end{proof}

\begin{lemma}\label{SplitMergeCross} 
For all \(i \in I\), \(a,b,c \in \Z_{\geq 0}\), the following equalities hold in \(\WebAaI\):
\begin{align*}
\hackcenter{
{}
}
\hackcenter{
\begin{tikzpicture}[scale=0.8]
\draw[ ultra thick, color=blue] (1.6,0)--(1.6,0.2) .. controls ++(0,0.35) and ++(0,-0.35) .. (0.8,1.2)--(0.8,1.4)
        .. controls ++(0,0.35) and ++(0,-0.35) .. (0.4,2)--(0.4,2.2);
 \draw[ ultra thick, color=blue] (0.8,0)--(0.8,0.2) .. controls ++(0,0.35) and ++(0,-0.35) .. (0,1.2)--(0,1.4)
        .. controls ++(0,0.35) and ++(0,-0.35) .. (0.4,2)--(0.4,2.2);       
  \draw[ ultra thick, color=blue] (0,0)--(0,0.2) .. controls ++(0,0.35) and ++(0,-0.35) .. (1.4,1.2)--(1.4,2.2);        
       \node[below] at (0,0) {$\scriptstyle i^{\scriptstyle(c)}$};
      \node[below] at (0.8,0) {$\scriptstyle i^{\scriptstyle(a)}$};
      \node[below] at (1.6,0) {$\scriptstyle i^{\scriptstyle(b)}$};
       \node[above] at (0.4,2.2) {$\scriptstyle i^{\scriptstyle(a+b)}$};
        \node[above] at (1.4,2.2) {$\scriptstyle i^{\scriptstyle(c)}$};
\end{tikzpicture}}
=
\hackcenter{
\begin{tikzpicture}[scale=0.8]
\draw[ ultra thick, color=blue] (0,0)--(0,0.6) .. controls ++(0,0.35) and ++(0,-0.35) .. (1.4,2)--(1.4,2.2);
\draw[ ultra thick, color=blue] (0.8,0)--(0.8,0.2) .. controls ++(0,0.35) and ++(0,-0.35) .. (1.2,0.8)--(1.2,1)
        .. controls ++(0,0.35) and ++(0,-0.35) .. (0.4,2)--(0.4,2.2);
        \draw[ ultra thick, color=blue] (1.6,0)--(1.6,0.2) .. controls ++(0,0.35) and ++(0,-0.35) .. (1.2,0.8)--(1.2,1)
        .. controls ++(0,0.35) and ++(0,-0.35) .. (0.4,2)--(0.4,2.2);
       \node[below] at (0,0) {$\scriptstyle i^{\scriptstyle(c)}$};
      \node[below] at (0.8,0) {$\scriptstyle i^{\scriptstyle(a)}$};
      \node[below] at (1.6,0) {$\scriptstyle i^{\scriptstyle(b)}$};
       \node[above] at (0.4,2.2) {$\scriptstyle i^{\scriptstyle(a+b)}$};
        \node[above] at (1.4,2.2) {$\scriptstyle i^{\scriptstyle(c)}$};
\end{tikzpicture}},
\qquad
\qquad
\hackcenter{
\begin{tikzpicture}[scale=0.8]
\draw[ ultra thick, color=blue] (-1.6,0)--(-1.6,0.2) .. controls ++(0,0.35) and ++(0,-0.35) .. (-0.8,1.2)--(-0.8,1.4)
        .. controls ++(0,0.35) and ++(0,-0.35) .. (-0.4,2)--(-0.4,2.2);
 \draw[ ultra thick, color=blue] (-0.8,0)--(-0.8,0.2) .. controls ++(0,0.35) and ++(0,-0.35) .. (0,1.2)--(0,1.4)
        .. controls ++(0,0.35) and ++(0,-0.35) .. (-0.4,2)--(-0.4,2.2);       
  \draw[ ultra thick, color=blue] (0,0)--(0,0.2) .. controls ++(0,0.35) and ++(0,-0.35) .. (-1.4,1.2)--(-1.4,2.2);        
      \node[below] at (0,0) {$\scriptstyle i^{\scriptstyle(c)}$};
      \node[below] at (-0.8,0) {$\scriptstyle i^{\scriptstyle(b)}$};
      \node[below] at (-1.6,0) {$\scriptstyle i^{\scriptstyle(a)}$};
       \node[above] at (-0.4,2.2) {$\scriptstyle i^{\scriptstyle(a+b)}$};
        \node[above] at (-1.4,2.2) {$\scriptstyle i^{\scriptstyle(c)}$};
\end{tikzpicture}}
=
\hackcenter{
\begin{tikzpicture}[scale=0.8]
\draw[ ultra thick, color=blue] (0,0)--(0,0.6) .. controls ++(0,0.35) and ++(0,-0.35) .. (-1.4,2)--(-1.4,2.2);
\draw[ ultra thick, color=blue] (-0.8,0)--(-0.8,0.2) .. controls ++(0,0.35) and ++(0,-0.35) .. (-1.2,0.8)--(-1.2,1)
        .. controls ++(0,0.35) and ++(0,-0.35) .. (-0.4,2)--(-0.4,2.2);
        \draw[ ultra thick, color=blue] (-1.6,0)--(-1.6,0.2) .. controls ++(0,0.35) and ++(0,-0.35) .. (-1.2,0.8)--(-1.2,1)
        .. controls ++(0,0.35) and ++(0,-0.35) .. (-0.4,2)--(-0.4,2.2);
      \node[below] at (0,0) {$\scriptstyle i^{\scriptstyle(c)}$};
      \node[below] at (-0.8,0) {$\scriptstyle i^{\scriptstyle(b)}$};
      \node[below] at (-1.6,0) {$\scriptstyle i^{\scriptstyle(a)}$};
       \node[above] at (-0.4,2.2) {$\scriptstyle i^{\scriptstyle(a+b)}$};
        \node[above] at (-1.4,2.2) {$\scriptstyle i^{\scriptstyle(c)}$};
\end{tikzpicture}},
\\
\hackcenter{
{}
}
\hackcenter{
\begin{tikzpicture}[scale=0.8]
\draw[ ultra thick, color=blue] (1.6,0)--(1.6,-0.2) .. controls ++(0,-0.35) and ++(0,0.35) .. (0.8,-1.2)--(0.8,-1.4)
        .. controls ++(0,-0.35) and ++(0,0.35) .. (0.4,-2)--(0.4,-2.2);
 \draw[ ultra thick, color=blue] (0.8,0)--(0.8,-0.2) .. controls ++(0,-0.35) and ++(0,0.35) .. (0,-1.2)--(0,-1.4)
        .. controls ++(0,-0.35) and ++(0,0.35) .. (0.4,-2)--(0.4,-2.2);       
  \draw[ ultra thick, color=blue] (0,0)--(0,-0.2) .. controls ++(0,-0.35) and ++(0,0.35) .. (1.4,-1.2)--(1.4,-2.2);        
        \node[above] at (0,0) {$\scriptstyle i^{\scriptstyle(c)}$};
      \node[above] at (0.8,0) {$\scriptstyle i^{\scriptstyle(a)}$};
      \node[above] at (1.6,0) {$\scriptstyle i^{\scriptstyle(b)}$};
       \node[below] at (0.4,-2.2) {$\scriptstyle i^{\scriptstyle(a+b)}$};
        \node[below] at (1.4,-2.2) {$\scriptstyle i^{\scriptstyle(c)}$};
\end{tikzpicture}}
=
\hackcenter{
\begin{tikzpicture}[scale=0.8]
\draw[ ultra thick, color=blue] (0,0)--(0,-0.6) .. controls ++(0,-0.35) and ++(0,0.35) .. (1.4,-2)--(1.4,-2.2);
\draw[ ultra thick, color=blue] (0.8,0)--(0.8,-0.2) .. controls ++(0,-0.35) and ++(0,0.35) .. (1.2,-0.8)--(1.2,-1)
        .. controls ++(0,-0.35) and ++(0,0.35) .. (0.4,-2)--(0.4,-2.2);
        \draw[ ultra thick, color=blue] (1.6,0)--(1.6,-0.2) .. controls ++(0,-0.35) and ++(0,0.35) .. (1.2,-0.8)--(1.2,-1)
        .. controls ++(0,-0.35) and ++(0,0.35) .. (0.4,-2)--(0.4,-2.2);
        \node[above] at (0,0) {$\scriptstyle i^{\scriptstyle(c)}$};
      \node[above] at (0.8,0) {$\scriptstyle i^{\scriptstyle(a)}$};
      \node[above] at (1.6,0) {$\scriptstyle i^{\scriptstyle(b)}$};
       \node[below] at (0.4,-2.2) {$\scriptstyle i^{\scriptstyle(a+b)}$};
        \node[below] at (1.4,-2.2) {$\scriptstyle i^{\scriptstyle(c)}$};
\end{tikzpicture}},
\qquad
\qquad
\hackcenter{
\begin{tikzpicture}[scale=0.8]
\draw[ ultra thick, color=blue] (-1.6,0)--(-1.6,-0.2) .. controls ++(0,-0.35) and ++(0,0.35) .. (-0.8,-1.2)--(-0.8,-1.4)
        .. controls ++(0,-0.35) and ++(0,0.35) .. (-0.4,-2)--(-0.4,-2.2);
 \draw[ ultra thick, color=blue] (-0.8,0)--(-0.8,-0.2) .. controls ++(0,-0.35) and ++(0,0.35) .. (0,-1.2)--(0,-1.4)
        .. controls ++(0,-0.35) and ++(0,0.35) .. (-0.4,-2)--(-0.4,-2.2);       
  \draw[ ultra thick, color=blue] (0,0)--(0,-0.2) .. controls ++(0,-0.35) and ++(0,0.35) .. (-1.4,-1.2)--(-1.4,-2.2);        
       \node[above] at (0,0) {$\scriptstyle i^{\scriptstyle(c)}$};
      \node[above] at (-0.8,0) {$\scriptstyle i^{\scriptstyle(b)}$};
      \node[above] at (-1.6,0) {$\scriptstyle i^{\scriptstyle(a)}$};
       \node[below] at (-0.4,-2.2) {$\scriptstyle i^{\scriptstyle(a+b)}$};
        \node[below] at (-1.4,-2.2) {$\scriptstyle i^{\scriptstyle(c)}$};
\end{tikzpicture}}
=
\hackcenter{
\begin{tikzpicture}[scale=0.8]
\draw[ ultra thick, color=blue] (0,0)--(0,-0.6) .. controls ++(0,-0.35) and ++(0,0.35) .. (-1.4,-2)--(-1.4,-2.2);
\draw[ ultra thick, color=blue] (-0.8,0)--(-0.8,-0.2) .. controls ++(0,-0.35) and ++(0,0.35) .. (-1.2,-0.8)--(-1.2,-1)
        .. controls ++(0,-0.35) and ++(0,0.35) .. (-0.4,-2)--(-0.4,-2.2);
        \draw[ ultra thick, color=blue] (-1.6,0)--(-1.6,-0.2) .. controls ++(0,-0.35) and ++(0,0.35) .. (-1.2,-0.8)--(-1.2,-1)
        .. controls ++(0,-0.35) and ++(0,0.35) .. (-0.4,-2)--(-0.4,-2.2);
       \node[above] at (0,0) {$\scriptstyle i^{\scriptstyle(c)}$};
      \node[above] at (-0.8,0) {$\scriptstyle i^{\scriptstyle(b)}$};
      \node[above] at (-1.6,0) {$\scriptstyle i^{\scriptstyle(a)}$};
       \node[below] at (-0.4,-2.2) {$\scriptstyle i^{\scriptstyle(a+b)}$};
        \node[below] at (-1.4,-2.2) {$\scriptstyle i^{\scriptstyle(c)}$};
\end{tikzpicture}},
\end{align*} 
for all \(a,b, c \in \Z_{\geq 0}\).
\end{lemma}

\begin{proof}
We prove the first equality. The others are proved via analogous arguments. We have
\begin{align*}
\hackcenter{
{}
}
\hackcenter{
\begin{tikzpicture}[scale=0.8]
\draw[ ultra thick, color=blue] (1.6,0)--(1.6,0.2) .. controls ++(0,0.35) and ++(0,-0.35) .. (0.8,1.7)--(0.8,1.9)
        .. controls ++(0,0.35) and ++(0,-0.35) .. (0.4,2.5)--(0.4,3);
 \draw[ ultra thick, color=blue] (0.8,0)--(0.8,0.2) .. controls ++(0,0.35) and ++(0,-0.35) .. (0,1.7)--(0,1.9)
        .. controls ++(0,0.35) and ++(0,-0.35) .. (0.4,2.5)--(0.4,3);       
  \draw[ ultra thick, color=blue] (0,0)--(0,0.2) .. controls ++(0,0.35) and ++(0,-0.35) .. (1.6,1.7)--(1.6,3);        
      \node[below] at (0,0) {$\scriptstyle c$};
      \node[below] at (0.8,0) {$\scriptstyle a$};
      \node[below] at (1.6,0) {$\scriptstyle b$};
       \node[above] at (0.4,3) {$\scriptstyle a+b$};
        \node[above] at (1.6,3) {$\scriptstyle c$};
\end{tikzpicture}}
&=
\sum_{
\substack{
s-r=c-a\\
s'-r'=c-b
}}
(-1)^{s+s'}
\hackcenter{
\begin{tikzpicture}[scale=0.8]
\draw[ ultra thick, color=blue] (0,0)--(0,3);
\draw[ ultra thick, color=blue] (1.6,0)--(1.6,3);
 \draw[ ultra thick, color=blue] (0.8,0)--(0.8,2.3) .. controls ++(0,0.35) and ++(0,-0.35) .. (0,2.9)--(0,3);
  \draw[ ultra thick, color=blue] (0,0)--(0,0.2) .. controls ++(0,0.35) and ++(0,-0.35) .. (0.8,0.8);
  \draw[ ultra thick, color=blue] (0.8,0)--(0.8,0.8) .. controls ++(0,0.35) and ++(0,-0.35) .. (0,1.4);
  \draw[ ultra thick, color=blue] (0.8,0)--(0.8,1.1) .. controls ++(0,0.35) and ++(0,-0.35) .. (1.6,1.7);
  \draw[ ultra thick, color=blue] (1.6,0)--(1.6,1.7) .. controls ++(0,0.35) and ++(0,-0.35) .. (0.8,2.3);
      \node[below] at (0,0) {$\scriptstyle c$};
      \node[below] at (0.8,0) {$\scriptstyle a$};
      \node[below] at (1.6,0) {$\scriptstyle b$};
       \node[above] at (0,3) {$\scriptstyle a+b$};
          \node[above] at (1.6,3) {$\scriptstyle c$};
           \node[below] at (0.4,0.6) {$\scriptstyle s$};
           \node[above] at (0.4,1.1) {$\scriptstyle r$};
             \node[above] at (0.4,2.6) {$\scriptstyle b$};
           \node[below] at (1.2,1.5) {$\scriptstyle s'$};
           \node[above] at (1.2,2) {$\scriptstyle r'$};
\end{tikzpicture}}
\;
\substack{ \textup{\cref{AssocRel}} \\ =\\ \,}
\;
\sum_{
\substack{
s-r=c-a\\
s'-r'=c-b
}}
(-1)^{s+s'}
\hackcenter{
\begin{tikzpicture}[scale=0.8]
\draw[ ultra thick, color=blue] (0,0)--(0,3);
\draw[ ultra thick, color=blue] (1.6,0)--(1.6,3);
 \draw[ ultra thick, color=blue] (0.8,0)--(0.8,2.3) .. controls ++(0,0.35) and ++(0,-0.35) .. (0,2.9)--(0,3);
  \draw[ ultra thick, color=blue] (0,0)--(0,0.2) .. controls ++(0,0.35) and ++(0,-0.35) .. (0.8,0.8);
  \draw[ ultra thick, color=blue] (0.8,0)--(0.8,1.5) .. controls ++(0,0.35) and ++(0,-0.35) .. (0,2.1);
  \draw[ ultra thick, color=blue] (0.8,0)--(0.8,0.9) .. controls ++(0,0.35) and ++(0,-0.35) .. (1.6,1.5);
  \draw[ ultra thick, color=blue] (1.6,0)--(1.6,1.7) .. controls ++(0,0.35) and ++(0,-0.35) .. (0.8,2.3);
      \node[below] at (0,0) {$\scriptstyle c$};
      \node[below] at (0.8,0) {$\scriptstyle a$};
      \node[below] at (1.6,0) {$\scriptstyle b$};
       \node[above] at (0,3) {$\scriptstyle a+b$};
          \node[above] at (1.6,3) {$\scriptstyle c$};
           \node[below] at (0.4,0.6) {$\scriptstyle s$};
           \node[above] at (0.4,1.3) {$\scriptstyle r$};
             \node[above] at (0.4,2.6) {$\scriptstyle b$};
           \node[below] at (1.2,1.3) {$\scriptstyle s'$};
           \node[above] at (1.2,2) {$\scriptstyle r'$};
\end{tikzpicture}}
\\
\hackcenter{}
&
\;
\substack{ \textup{\cref{CoxeterWeb}} \\ =\\ \,}
\;
\sum_{
\substack{
s-r=c-a\\
s'-r'=c-b
}}
(-1)^{s+s'}
\binom{a+s-s'}{r}
\hackcenter{
\begin{tikzpicture}[scale=0.8]
\draw[ ultra thick, color=blue] (0,0)--(0,3);
\draw[ ultra thick, color=blue] (1.6,0)--(1.6,3);
 \draw[ ultra thick, color=blue] (0.8,0)--(0.8,2.1) .. controls ++(0,0.35) and ++(0,-0.35) .. (0,2.7)--(0,3);
  \draw[ ultra thick, color=blue] (0,0)--(0,0.2) .. controls ++(0,0.35) and ++(0,-0.35) .. (0.8,0.8);
  \draw[ ultra thick, color=blue] (0.8,0)--(0.8,0.9) .. controls ++(0,0.35) and ++(0,-0.35) .. (1.6,1.5);
  \draw[ ultra thick, color=blue] (1.6,0)--(1.6,1.5) .. controls ++(0,0.35) and ++(0,-0.35) .. (0.8,2.1);
      \node[below] at (0,0) {$\scriptstyle c$};
      \node[below] at (0.8,0) {$\scriptstyle a$};
      \node[below] at (1.6,0) {$\scriptstyle b$};
       \node[above] at (0,3) {$\scriptstyle a+b$};
          \node[above] at (1.6,3) {$\scriptstyle c$};
           \node[below] at (0.4,0.6) {$\scriptstyle s$};
             \node[above] at (0.4,2.4) {$\scriptstyle b+r$};
           \node[below] at (1.2,1.3) {$\scriptstyle s'$};
           \node[above] at (1.2,1.8) {$\scriptstyle r'$};
\end{tikzpicture}}
\\
\hackcenter{}
&
\;
\substack{ \textup{\cref{DiagSwitchRel}} \\ =\\ \,}
\;
\sum_{
\substack{
s-r=c-a\\
s'-r'=c-b\\
t \in \Z_{\geq 0}
}}
(-1)^{s+s'}
\binom{a+s-s'}{r}
\binom{r}{t}
\hackcenter{
\begin{tikzpicture}[scale=0.8]
\draw[ ultra thick, color=blue] (0,0)--(0,3);
\draw[ ultra thick, color=blue] (1.6,0)--(1.6,3);
 \draw[ ultra thick, color=blue] (0.8,0)--(0.8,2.1) .. controls ++(0,0.35) and ++(0,-0.35) .. (0,2.7)--(0,3);
  \draw[ ultra thick, color=blue] (0,0)--(0,0.2) .. controls ++(0,0.35) and ++(0,-0.35) .. (0.8,0.8);
  \draw[ ultra thick, color=blue] (1.6,0)--(1.6,0.9) .. controls ++(0,0.35) and ++(0,-0.35) .. (0.8,1.5);
  \draw[ ultra thick, color=blue] (0.8,0)--(0.8,1.5) .. controls ++(0,0.35) and ++(0,-0.35) .. (1.6,2.1);
      \node[below] at (0,0) {$\scriptstyle c$};
      \node[below] at (0.8,0) {$\scriptstyle a$};
      \node[below] at (1.6,0) {$\scriptstyle b$};
       \node[above] at (0,3) {$\scriptstyle a+b$};
          \node[above] at (1.6,3) {$\scriptstyle c$};
           \node[below] at (0.4,0.6) {$\scriptstyle s$};
             \node[above] at (0.4,2.4) {$\scriptstyle b+r$};
           \node[below] at (1.2,1.3) {$\scriptstyle r'\hspace{-0.5mm}-t$};
           \node[above] at (1.2,1.8) {$\scriptstyle s'\hspace{-0.5mm}-t$};
\end{tikzpicture}}
\\
&
\;
\substack{ \textup{\cref{AssocRel}} \\ =\\ \,}
\;
\sum_{
\substack{
s-r=c-a\\
s'-r'=c-b\\
t \in \Z_{\geq 0}
}}
(-1)^{s+s'}
\binom{a+s-s'}{r}
\binom{r}{t}
\hackcenter{
\begin{tikzpicture}[scale=0.8]
\draw[ ultra thick, color=blue] (0.8,0)--(0.8,2.8);
\draw[ ultra thick, color=blue] (1.6,0)--(1.6,2.8);
\draw[ ultra thick, color=blue] (0,0)--(0,1.8) .. controls ++(0,0.35) and ++(0,-0.35) .. (0.8,2.4)--(0.8,2.8);
  \draw[ ultra thick, color=blue] (0,0)--(0,0.7) .. controls ++(0,0.35) and ++(0,-0.35) .. (0.8,1.3);
  \draw[ ultra thick, color=blue] (1.6,0)--(1.6,0.2) .. controls ++(0,0.35) and ++(0,-0.35) .. (0.8,0.8);
  \draw[ ultra thick, color=blue] (0.8,0)--(0.8,1.5) .. controls ++(0,0.35) and ++(0,-0.35) .. (1.6,2.1);
      \node[below] at (0,0) {$\scriptstyle c$};
      \node[below] at (0.8,0) {$\scriptstyle a$};
      \node[below] at (1.6,0) {$\scriptstyle b$};
       \node[above] at (0.8,2.8) {$\scriptstyle a+b$};
          \node[above] at (1.6,2.8) {$\scriptstyle c$};
           \node[below] at (0.4,1) {$\scriptstyle s$};
             \node[above] at (0.4,2.1) {$\scriptstyle c-s\;$};
           \node[below] at (1.2,1.15) {$\scriptstyle r'\hspace{-0.5mm}-t$};
           \node[above] at (1.2,1.8) {$\scriptstyle s'\hspace{-0.5mm}-t$};
\end{tikzpicture}}
\\
&
\;
\substack{ \textup{\cref{CoxeterWeb}} \\ =\\ \,}
\;
\sum_{
\substack{
s-r=c-a\\
s'-r'=c-b\\
t \in \Z_{\geq 0}\\
u \in \Z_{\geq 0}
}}
(-1)^{s+s'}
\binom{a+s-s'}{r}
\binom{r}{t}
\binom{c-s'+t}{u}
\hackcenter{
\begin{tikzpicture}[scale=0.8]
\draw[ ultra thick, color=blue] (0.8,0)--(0.8,2.8);
\draw[ ultra thick, color=blue] (2.2,0)--(2.2,2.8);
  \draw[ ultra thick, color=blue] (0,0)--(0,1) .. controls ++(0,0.35) and ++(0,-0.35) .. (0.8,1.6);
  \draw[ ultra thick, color=blue] (2.2,0)--(2.2,0.2) .. controls ++(0,0.35) and ++(0,-0.35) .. (0.8,0.8);
  \draw[ ultra thick, color=blue] (0.8,0)--(0.8,0.8) .. controls ++(0,0.35) and ++(0,-0.35) .. (2.2,1.4);
  \draw[ ultra thick, color=blue] (0.8,0)--(0.8,1.7) .. controls ++(0,0.35) and ++(0,-0.35) .. (2.2,2.3);
      \node[below] at (0,0) {$\scriptstyle c$};
      \node[below] at (0.8,0) {$\scriptstyle a$};
      \node[below] at (2.2,0) {$\scriptstyle b$};
       \node[above] at (0.8,2.8) {$\scriptstyle a+b$};
          \node[above] at (2.2,2.8) {$\scriptstyle c$};
           \node[below] at (1.5,0.6) {$\scriptstyle r'\hspace{-0.5mm}-t$};
           \node[above] at (1.5,1.15) {$\scriptstyle c\hspace{-0.5mm}-s-u$};
\end{tikzpicture}}
\\
&
=
\sum_{
\substack{
s-r=c-a\\
s'-r'=c-b\\
t \in \Z_{\geq 0}\\
u \in \Z_{\geq 0}
}}
(-1)^{s+s'}
\binom{a+s-s'}{r}
\binom{r}{t}
\binom{c-s'+t}{u}
\hackcenter{
\begin{tikzpicture}[scale=0.8]
\draw[ ultra thick, color=blue] (0,0)--(0,2.8);
\draw[ ultra thick, color=blue] (1.4,0)--(1.4,2.8);
\draw[ ultra thick, color=blue] (0,0)--(0,1.8) .. controls ++(0,0.35) and ++(0,-0.35) .. (1.4,2.4)--(1.4,2.8);
\draw[ ultra thick, color=blue] (1.4,0)--(1.4,1.1) .. controls ++(0,0.35) and ++(0,-0.35) .. (0,1.7);
\draw[ ultra thick, color=blue] (2.8,0)--(2.8,1.2) .. controls ++(0,0.35) and ++(0,-0.35) .. (1.4,1.8);
\draw[ ultra thick, color=blue] (2.8,0)--(2.8,0.3) .. controls ++(0,0.35) and ++(0,-0.35) .. (1.4,0.9);
       \node[below] at (0,0) {$\scriptstyle c$};
      \node[below] at (1.4,0) {$\scriptstyle a$};
      \node[below] at (2.8,0) {$\scriptstyle b$};
       \node[above] at (0,2.8) {$\scriptstyle a+b$};
          \node[above] at (1.4,2.8) {$\scriptstyle c$};
           \node[below] at (2.1,0.7) {$\scriptstyle r'\hspace{-0.5mm}-t$};
           \node[below] at (0.7,1.4) {$\scriptstyle r+r'\;\;$};
            \node[below] at (0.7,1.1) {$\scriptstyle -t+u$};
\end{tikzpicture}}.
\end{align*}
Now, fixing \(L,R \in\Z_{\geq 0}\), we consider the coefficient \(\kappa(L,R)\) of
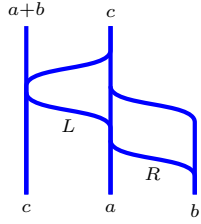
\begin{align*}
\hackcenter{}
\hackcenter{
\begin{tikzpicture}[scale=0.8]
\draw[ ultra thick, color=blue] (0,0)--(0,2.8);
\draw[ ultra thick, color=blue] (1.4,0)--(1.4,2.8);
\draw[ ultra thick, color=blue] (0,0)--(0,1.8) .. controls ++(0,0.35) and ++(0,-0.35) .. (1.4,2.4)--(1.4,2.8);
\draw[ ultra thick, color=blue] (1.4,0)--(1.4,1.1) .. controls ++(0,0.35) and ++(0,-0.35) .. (0,1.7);
\draw[ ultra thick, color=blue] (2.8,0)--(2.8,1.2) .. controls ++(0,0.35) and ++(0,-0.35) .. (1.4,1.8);
\draw[ ultra thick, color=blue] (2.8,0)--(2.8,0.3) .. controls ++(0,0.35) and ++(0,-0.35) .. (1.4,0.9);
        \node[below] at (0,0) {$\scriptstyle c$};
      \node[below] at (1.4,0) {$\scriptstyle a$};
      \node[below] at (2.8,0) {$\scriptstyle b$};
       \node[above] at (0,2.8) {$\scriptstyle a+b$};
          \node[above] at (1.4,2.8) {$\scriptstyle c$};
           \node[below] at (2.1,0.6) {$\scriptstyle R$};
           \node[below] at (0.7,1.4) {$\scriptstyle L$};
           \end{tikzpicture}}
\end{align*}
in the sum above. Using the substitutions
\(
A=L-R-a+c-s, \,B=c-s'
\), we have
\begin{align*}
\kappa(L,R)
&=
(-1)^{L+R+a}
\sum_{B=0}^{ b-R}
\sum_{A=0}^{b-R}
(-1)^{A+B}
\binom{L-R-A+B}{B}
\binom{L-R-A}{b-R-B}
\binom{b-R}{A}
\\
&=
\sum_{B=0}^{ b-R}
\sum_{A=0}^{b-R}
(-1)^{a+A+B}
\binom{b-R}{B}
\binom{L-R-A+B}{b-R}
\binom{b-R}{A}
\\
&=
(-1)^{L+R+a}
\sum_{B=0}^{ b-R}
(-1)^{B}
\binom{b-R}{B}
\sum_{A=0}^{b-R}
(-1)^{A}
\binom{L-R-A+B}{b-R}
\binom{b-R}{A}.
\end{align*}
An application of Euler's finite difference theorem (see \cite[(10.13)]{G}) shows that 
\begin{align*}
\sum_{A=0}^{b-R}
(-1)^{A}
\binom{L-R-A+B}{b-R}
\binom{b-R}{A} =1,
\end{align*}
so 
\begin{align*}
\kappa(L,R)
&=
(-1)^{L+R+a}
\sum_{B=0}^{ b-R}
(-1)^{B}
\binom{b-R}{B}
=(-1)^{L+a+b} \delta_{b,R},
\end{align*}
by the binomial theorem. Therefore, we have
\begin{align*}
\hackcenter{
{}
}
\hackcenter{
\begin{tikzpicture}[scale=0.8]
\draw[ ultra thick, color=blue] (1.6,0)--(1.6,0.2) .. controls ++(0,0.35) and ++(0,-0.35) .. (0.8,1.7)--(0.8,1.9)
        .. controls ++(0,0.35) and ++(0,-0.35) .. (0.4,2.5)--(0.4,2.8);
 \draw[ ultra thick, color=blue] (0.8,0)--(0.8,0.2) .. controls ++(0,0.35) and ++(0,-0.35) .. (0,1.7)--(0,1.9)
        .. controls ++(0,0.35) and ++(0,-0.35) .. (0.4,2.5)--(0.4,2.8);       
  \draw[ ultra thick, color=blue] (0,0)--(0,0.2) .. controls ++(0,0.35) and ++(0,-0.35) .. (1.6,1.7)--(1.6,2.8);        
       \node[below] at (0,0) {$\scriptstyle c$};
      \node[below] at (0.8,0) {$\scriptstyle a$};
      \node[below] at (1.6,0) {$\scriptstyle b$};
       \node[above] at (0.4,2.8) {$\scriptstyle a+b$};
        \node[above] at (1.6,2.8) {$\scriptstyle c$};
\end{tikzpicture}}
=
\sum_{M-L = c-(a+b)}
(-1)^{(a+b)-L}
\hackcenter{
\begin{tikzpicture}[scale=0.8]
\draw[ ultra thick, color=blue] (0,0)--(0,2.8);
\draw[ ultra thick, color=blue] (1.4,0)--(1.4,2.8);
\draw[ ultra thick, color=blue] (0,0)--(0,1.8) .. controls ++(0,0.35) and ++(0,-0.35) .. (1.4,2.4)--(1.4,2.8);
\draw[ ultra thick, color=blue] (1.4,0)--(1.4,1.1) .. controls ++(0,0.35) and ++(0,-0.35) .. (0,1.7);
\draw[ ultra thick, color=blue] (2.8,0)--(2.8,0.3) .. controls ++(0,0.35) and ++(0,-0.35) .. (1.4,0.9);
        \node[below] at (0,0) {$\scriptstyle c$};
      \node[below] at (1.4,0) {$\scriptstyle a$};
      \node[below] at (2.8,0) {$\scriptstyle b$};
       \node[above] at (0,2.8) {$\scriptstyle a+b$};
          \node[above] at (1.4,2.8) {$\scriptstyle c$};
           \node[below] at (0.7,1.4) {$\scriptstyle L$};
            \node[above] at (0.7,2.1) {$\scriptstyle M$};
           \end{tikzpicture}}
=
\hackcenter{
\begin{tikzpicture}[scale=0.8]
\draw[ ultra thick, color=blue] (0,0)--(0,0.6) .. controls ++(0,0.35) and ++(0,-0.35) .. (1.4,2)--(1.4,2.8);
\draw[ ultra thick, color=blue] (0.8,0)--(0.8,0.2) .. controls ++(0,0.35) and ++(0,-0.35) .. (1.2,0.8)--(1.2,1)
        .. controls ++(0,0.35) and ++(0,-0.35) .. (0.4,2)--(0.4,2.8);
        \draw[ ultra thick, color=blue] (1.6,0)--(1.6,0.2) .. controls ++(0,0.35) and ++(0,-0.35) .. (1.2,0.8)--(1.2,1)
        .. controls ++(0,0.35) and ++(0,-0.35) .. (0.4,2)--(0.4,2.8);
       \node[below] at (0,0) {$\scriptstyle c$};
      \node[below] at (0.8,0) {$\scriptstyle a$};
      \node[below] at (1.6,0) {$\scriptstyle b$};
       \node[above] at (0.4,2.8) {$\scriptstyle a+b$};
        \node[above] at (1.4,2.8) {$\scriptstyle c$};
\end{tikzpicture}},
\end{align*}
completing the proof.
\end{proof}

\begin{lemma}\label{BraidCox}
Monotone crossings in \(\WebAaI\) satisfy the braid relation:
\begin{align*}
\hackcenter{}
\hackcenter{
\begin{tikzpicture}[scale=0.8]
  \draw[ ultra thick, color=blue] (0.2,0)--(0.2,0.1) .. controls ++(0,0.35) and ++(0,-0.35) .. (-0.4,0.75)
  .. controls ++(0,0.35) and ++(0,-0.35) .. (0.2,1.4)--(0.2,1.5);
  \draw[ ultra thick, color=blue] (-0.4,0)--(-0.4,0.1) .. controls ++(0,0.35) and ++(0,-0.35) .. (0.8,1.4)--(0.8,1.5);
  \draw[ ultra thick, color=blue] (0.8,0)--(0.8,0.1) .. controls ++(0,0.35) and ++(0,-0.35) .. (-0.4,1.4)--(-0.4,1.5);
   \node[below] at (0.2,0) {$\scriptstyle i^{(b)}$};
        \node[below] at (-0.4,0) {$\scriptstyle i^{(a)}$};
         \node[below] at (0.8,0) {$\scriptstyle i^{(c)}$};
          \node[above] at (0.2,1.5) {$\scriptstyle i^{(b)}$};
        \node[above] at (-0.4,1.5) {$\scriptstyle i^{(c)}$};
         \node[above] at (0.8,1.5) {$\scriptstyle i^{(a)}$};
\end{tikzpicture}}
=
\hackcenter{
\begin{tikzpicture}[scale=0.8]
  \draw[ ultra thick, color=blue] (0.2,0)--(0.2,0.1) .. controls ++(0,0.35) and ++(0,-0.35) .. (0.8,0.75)
  .. controls ++(0,0.35) and ++(0,-0.35) .. (0.2,1.4)--(0.2,1.5);
  \draw[ ultra thick, color=blue] (-0.4,0)--(-0.4,0.1) .. controls ++(0,0.35) and ++(0,-0.35) .. (0.8,1.4)--(0.8,1.5);
  \draw[ ultra thick, color=blue] (0.8,0)--(0.8,0.1) .. controls ++(0,0.35) and ++(0,-0.35) .. (-0.4,1.4)--(-0.4,1.5);
   \node[below] at (0.2,0) {$\scriptstyle i^{(b)}$};
        \node[below] at (-0.4,0) {$\scriptstyle i^{(a)}$};
         \node[below] at (0.8,0) {$\scriptstyle i^{(c)}$};
          \node[above] at (0.2,1.5) {$\scriptstyle i^{(b)}$};
        \node[above] at (-0.4,1.5) {$\scriptstyle i^{(c)}$};
         \node[above] at (0.8,1.5) {$\scriptstyle i^{(a)}$};
\end{tikzpicture}},
\end{align*}
for all \(a,b,c \in \Z_{\geq 0}\).
\end{lemma}
\begin{proof}
We have
\begin{align*}
\hackcenter{}
\hackcenter{
\begin{tikzpicture}[scale=0.8]
  \draw[ ultra thick, color=blue] (0,0)--(0,0.2)  .. controls ++(0,0.35) and ++(0,-0.35)  .. (1.6, 3.4)--(1.6,3.6);  
   \draw[ ultra thick, color=blue] (1.6,0)--(1.6,0.2)  .. controls ++(0,0.35) and ++(0,-0.35)  .. (0, 3.4)--(0,3.6);  
   \draw[ ultra thick, color=blue] (0.8,0)--(0.8,0.2)  .. controls ++(0,0.35) and ++(0,-0.35) .. (0, 1.2)--(0,2.4)
         .. controls ++(0,0.35) and ++(0,-0.35) .. (0.8, 3.4)--(0.8,3.6);  
      \node[below] at (0,0) {$\scriptstyle a$};
        \node[below] at (0.8,0) {$\scriptstyle b$};
         \node[below] at (1.6,0) {$\scriptstyle c$};
\end{tikzpicture}}
\hspace{-2mm}
&=
\sum_{s-r = a-c}
(-1)^{a-s}
\hackcenter{
\begin{tikzpicture}[scale=0.8]
  \draw[ ultra thick, color=blue] (0,0)--(0,0.2)  .. controls ++(0,0.35) and ++(0,-0.35) .. (0.8, 1.2)--(0.8,2.4)
         .. controls ++(0,0.35) and ++(0,-0.35) .. (0, 3.4)--(0,3.6);  
   \draw[ ultra thick, color=blue] (0.8,0)--(0.8,0.2)  .. controls ++(0,0.35) and ++(0,-0.35) .. (0, 1.2)--(0,2.4)
         .. controls ++(0,0.35) and ++(0,-0.35) .. (0.8, 3.4)--(0.8,3.6);  
  \draw[ ultra thick, color=blue] (0.8,1.2)  .. controls ++(0,0.35) and ++(0,-0.35) .. (1.6, 1.8);
    \draw[ ultra thick, color=blue] (1.6,1.8)  .. controls ++(0,0.35) and ++(0,-0.35) .. (0.8, 2.4);
      \draw[ ultra thick, color=blue] (1.6,0)--(1.6,3.6);        
      \node[below] at (0,0) {$\scriptstyle a$};
        \node[below] at (0.8,0) {$\scriptstyle b$};
         \node[below] at (1.6,0) {$\scriptstyle c$};
          \node[above] at (1.2,2.1) {$\scriptstyle r$};
          \node[below] at (1.2,1.5) {$\scriptstyle s$};
\end{tikzpicture}}
\;
\substack{ \textup{\cref{SplitMergeCross}} \\ =\\ \,}
\;
\sum_{s-r = a-c}
(-1)^{a-s}
\hackcenter{
\begin{tikzpicture}[scale=0.8]
  \draw[ ultra thick, color=blue] (0,0)--(0,0.2)  .. controls ++(0,0.35) and ++(0,-0.35) .. (0.8, 1.8)
         .. controls ++(0,0.35) and ++(0,-0.35) .. (0, 3.4)--(0,3.6);  
    \draw[ ultra thick, color=blue] (0.8,0)--(0.8,0.2)  .. controls ++(0,0.35) and ++(0,-0.35) .. (0, 1.8)
         .. controls ++(0,0.35) and ++(0,-0.35) .. (0.8, 3.4)--(0.8,3.6);     
    \draw[ ultra thick, color=blue] (0,0.2)  .. controls ++(0,0.35) and ++(0,-0.35) .. (1.6, 1.4);       
      \draw[ ultra thick, color=blue] (1.6,2.2)  .. controls ++(0,0.35) and ++(0,-0.35) .. (0, 3.4);  
      \draw[ ultra thick, color=blue] (1.6,0)--(1.6,3.6);        
      \node[below] at (0,0) {$\scriptstyle a$};
        \node[below] at (0.8,0) {$\scriptstyle b$};
         \node[below] at (1.6,0) {$\scriptstyle c$};
           \node[above] at (1.2,2.6) {$\scriptstyle r$};
          \node[below] at (1.2,1) {$\scriptstyle s$};
\end{tikzpicture}}
\\
\;
&
\substack{ \textup{\cref{DoubleCross}} \\ =\\ \,}
\;
\sum_{s-r = a-c}
(-1)^{a-s}
\hackcenter{
\begin{tikzpicture}[scale=0.8]
  \draw[ ultra thick, color=blue] (0,0)--(0,3.6);  
    \draw[ ultra thick, color=blue] (0.8,0)--(0.8,0.2)  .. controls ++(0,0.35) and ++(0,-0.35) .. (1.6, 0.8)
        .. controls ++(0,0.35) and ++(0,-0.35) .. (0.8, 1.8)
         .. controls ++(0,0.35) and ++(0,-0.35) .. (1.6, 2.8)
        .. controls ++(0,0.35) and ++(0,-0.35) .. (0.8, 3.4)--(0.8,3.6);
     \draw[ ultra thick, color=blue] (1.6,0)--(1.6,0.2)  .. controls ++(0,0.35) and ++(0,-0.35) .. (0.8, 0.8)
        .. controls ++(0,0.35) and ++(0,-0.35) .. (1.6, 1.4)
        .. controls ++(0,0.35) and ++(0,-0.35) .. (1.6, 2.2)
         .. controls ++(0,0.35) and ++(0,-0.35) .. (0.8, 2.8)
        .. controls ++(0,0.35) and ++(0,-0.35) .. (1.6, 3.4)--(1.6,3.6);     
    \draw[ ultra thick, color=blue] (0,0.8)  .. controls ++(0,0.35) and ++(0,-0.35) .. (1.6, 1.7);       
      \draw[ ultra thick, color=blue] (1.6,1.9)  .. controls ++(0,0.35) and ++(0,-0.35) .. (0, 2.8);       
      \node[below] at (0,0) {$\scriptstyle a$};
        \node[below] at (0.8,0) {$\scriptstyle b$};
         \node[below] at (1.6,0) {$\scriptstyle c$};
               \node[above] at (0.4,2.5) {$\scriptstyle r$};
          \node[below] at (0.4,1.1) {$\scriptstyle s$};
\end{tikzpicture}}
\;
\substack{ \textup{\cref{SplitMergeCross}} \\ =\\ \,}
\;
\sum_{s-r = a-c}
(-1)^{a-s}
\hackcenter{
\begin{tikzpicture}[scale=0.8]
  \draw[ ultra thick, color=blue] (0,0)--(0,3.6);  
    \draw[ ultra thick, color=blue] (0.8,0)--(0.8,0.2)  .. controls ++(0,0.35) and ++(0,-0.35) .. (1.6, 0.8)--(1.6,1.2)
       .. controls ++(0,0.35) and ++(0,-0.35) .. (0.8, 1.8)
        .. controls ++(0,0.35) and ++(0,-0.35) .. (1.6, 2.4)--(1.6,2.8)
        .. controls ++(0,0.35) and ++(0,-0.35) .. (0.8, 3.4)--(0.8,3.6);
 \draw[ ultra thick, color=blue] (1.6,0)--(1.6,0.2)  .. controls ++(0,0.35) and ++(0,-0.35) .. (0.8, 0.8)--(0.8,1.2)
       .. controls ++(0,0.35) and ++(0,-0.35) .. (1.6, 1.8)
        .. controls ++(0,0.35) and ++(0,-0.35) .. (0.8, 2.4)--(0.8,2.8)
        .. controls ++(0,0.35) and ++(0,-0.35) .. (1.6, 3.4)--(1.6,3.6);
    \draw[ ultra thick, color=blue] (0,0.4)  .. controls ++(0,0.35) and ++(0,-0.35) .. (0.8, 1);       
      \draw[ ultra thick, color=blue] (0.8,2.6)  .. controls ++(0,0.35) and ++(0,-0.35) .. (0, 3.2);       
      \node[below] at (0,0) {$\scriptstyle a$};
        \node[below] at (0.8,0) {$\scriptstyle b$};
         \node[below] at (1.6,0) {$\scriptstyle c$};
           \node[above] at (0.4,2.9) {$\scriptstyle r$};
          \node[below] at (0.4,0.7) {$\scriptstyle s$};
\end{tikzpicture}}
\\
&
\;
\substack{ \textup{\cref{DoubleCross}} \\ =\\ \,}
\;
\sum_{s-r = a-c}
(-1)^{a-s}
\hackcenter{
\begin{tikzpicture}[scale=0.8]
  \draw[ ultra thick, color=blue] (0,0)--(0,3.6);  
    \draw[ ultra thick, color=blue] (0.8,0)--(0.8,0.2)  .. controls ++(0,0.35) and ++(0,-0.35) .. (1.6, 0.8)--(1.6,2.8)
        .. controls ++(0,0.35) and ++(0,-0.35) .. (0.8, 3.4)--(0.8,3.6);
 \draw[ ultra thick, color=blue] (1.6,0)--(1.6,0.2)  .. controls ++(0,0.35) and ++(0,-0.35) .. (0.8, 0.8)--(0.8,2.8)
        .. controls ++(0,0.35) and ++(0,-0.35) .. (1.6, 3.4)--(1.6,3.6);
    \draw[ ultra thick, color=blue] (0,0.4)  .. controls ++(0,0.35) and ++(0,-0.35) .. (0.8, 1);       
      \draw[ ultra thick, color=blue] (0.8,2.6)  .. controls ++(0,0.35) and ++(0,-0.35) .. (0, 3.2);       
      \node[below] at (0,0) {$\scriptstyle a$};
        \node[below] at (0.8,0) {$\scriptstyle b$};
         \node[below] at (1.6,0) {$\scriptstyle c$};
         \node[above] at (0.4,2.9) {$\scriptstyle r$};
          \node[below] at (0.4,0.7) {$\scriptstyle s$};
\end{tikzpicture}}
=
\hackcenter{
\begin{tikzpicture}[scale=0.8]
  \draw[ ultra thick, color=blue] (0,0)--(0,0.2)  .. controls ++(0,0.35) and ++(0,-0.35)  .. (1.6, 3.4)--(1.6,3.6);  
   \draw[ ultra thick, color=blue] (1.6,0)--(1.6,0.2)  .. controls ++(0,0.35) and ++(0,-0.35)  .. (0, 3.4)--(0,3.6);  
   \draw[ ultra thick, color=blue] (0.8,0)--(0.8,0.2)  .. controls ++(0,0.35) and ++(0,-0.35) .. (1.6, 1.2)--(1.6,2.4)
         .. controls ++(0,0.35) and ++(0,-0.35) .. (0.8, 3.4)--(0.8,3.6);  
      \node[below] at (0,0) {$\scriptstyle a$};
        \node[below] at (0.8,0) {$\scriptstyle b$};
         \node[below] at (1.6,0) {$\scriptstyle c$};
\end{tikzpicture}},
\end{align*}
as desired.
\end{proof}

We may encapsulate \cref{DoubleCross,SplitMergeCross,BraidCox} with the following corollary:

\begin{corollary}\label{arbcol}
The relations \cref{Cox,SplitIntertwineRel,MergeIntertwineRel} hold for arbitrary \(i,j,k,\ell,m \in I\), \(x,y,z \in \Z_{\geq 0}\).
\end{corollary}
\begin{proof}
That the claim holds for \cref{SplitIntertwineRel}, \cref{SplitMergeCross}, and the left side of \cref{Cox} is immediate from \cref{DoubleCross,SplitMergeCross,BraidCox}. The right side of \cref{Cox} holds when \(i = j = k\) by \cref{BraidCox}. To establish the right side of \cref{Cox} when \(|\{i,j,k\}| = 2\), we note that any same-color crossings may be rewritten as combinations of splits and merges via \cref{CrossingWebRel}, and then \cref{SplitIntertwineRel,MergeIntertwineRel} may be used to pull these splits and merges across the differently-colored strand, and then rewritten as same-colored crossings again to achieve the result.
\end{proof}

\section{The defining representation}\label{S:WebAaRep}

\subsection{The twist map} \label{twistsec}
Let \(M,N \in \sModk\).  Recall that for short we write $M \otimes N$ for $M \otimes_{\k} N$.  We have the \emph{twist map}, 
\begin{align}\label{ExpTwist}
\tau_{M,N}: M\otimes N \to N \otimes M,
\end{align}
given by 
\begin{align*}
\tau_{M,N}(m \otimes n) = (-1)^{\overline m \cdot \overline n}n \otimes m,
\end{align*}
for homogeneous \(m \in M\), \(n \in N\). The tensor operation and the twist make \(\sModk\) into a symmetric monoidal supercategory. 

We have a right action of the symmetric group \(\mathfrak{S}_d\) on \(M^{\otimes d}\) given by signed permutations:
\begin{align*}
( m_1 \otimes \cdots \otimes m_d).\tau =  (-1)^{\langle \tau; (m_1, \ldots, m_d) \rangle} m_{\tau (1)} \otimes \cdots \otimes m_{\tau (d)}
\end{align*}
where
\begin{align*}
\langle \tau; (m_1, \ldots, m_d) \rangle := \sum_{
\substack{
1 \leq r < s \leq d\\
\tau(r) > \tau(s)
}
}
\overline{m}_r \overline{m}_s.
\end{align*}

\subsection{Symmetric powers}\label{sympowsec} Let \(M \in \sModk\). The {\em tensor superalgebra} \(M^\bullet = \bigoplus_{x= 0}^\infty M^{\otimes x}\) is a \(\k\)-superalgebra under the product \(\otimes\). 
Let \(J\subseteq M^\bullet\) be the ideal generated by elements of the form
\begin{align*}
m_1 \otimes m_2 - (-1)^{\overline{m}_1 \overline{m}_2}m_2 \otimes m_1,
\qquad
\textup{and}
\qquad
m_3 \otimes m_3, 
\end{align*}
ranging over all homogeneous \(m_1, m_2 \in M\), and all \(m_3 \in M_{\bar 1}\). The {\em symmetric superalgebra} is \(S^\bullet M := M^{\bullet}/J\).
Let \(\sigma: M^\bullet \to S^\bullet M\) denote the canonical quotient map of \(\k\)-supermodules.
For \(x \in \Z_{\geq 0}\), denote the \(\k\)-supermodule
\(S^xM := \sigma(M^{\otimes x})\), so that \(S^\bullet M = \bigoplus_{x =0}^\infty S^xM\).
Note in particular that \(S^1M \cong M\) and \(S^0M \cong  \k\). 

The symmetric superbialgebra \(S^\bullet M\) has associative product \(\nabla: S^\bullet M\otimes S^\bullet M \to S^\bullet M\),  coassociative coproduct \(\Delta: S^\bullet M \to S^\bullet M \otimes S^\bullet M\), and projections \(p_x:S^\bullet M \to S^xM\) and inclusions \(\iota_x: S^xM \to S^\bullet M\) for all \(x \in \Z_{\geq 0}\).

\subsubsection{The split map}
We define the \emph{split map} \({}_M\textup{spl}^{x,y}_{x+y}: S^{x+y}M \to S^xM \otimes S^yM\) by  
\begin{align*}
{}_M\textup{spl}^{x,y}_{x+y} = (p_x \otimes p_y) \circ \Delta \circ \iota_{x+y}.
\end{align*}
Explicitly,
\begin{align}\label{ExpSpl}
{}_M\textup{spl}^{x,y}_{x+y}(m_1 \cdots m_{x+y})= \sum_{\substack{T = \{t_1< \cdots< t_x\} \\ U = \{u_1< \cdots<u_y\} \\ T \cup U = \{1,\ldots, x+y\} }} (-1)^{\varepsilon(T,U)}m_{t_1} \cdots m_{t_x} \otimes m_{u_1} \cdots m_{u_y},
\end{align}
for all homogeneous \(m_1, \ldots, m_{x+y} \in M\), where \(\varepsilon(T,U) \in \Z_2\) is defined by
\begin{align*}
\varepsilon(T,U)=\#\{(t,u) \in T \times U \mid t>u, \,\overline{m}_t = \overline{m}_u = \bar 1\}.
\end{align*}

\subsubsection{The merge map}
We may similarly define the \emph{merge map}
\(
{}_M\textup{mer}_{x,y}^{x+y}: S^xM \otimes S^xM \to S^{x+y}M
\)
by  \[{}_M\textup{mer}_{x,y}^{x+y}:= p_{x+y} \circ \nabla \circ (\iota_x \otimes \iota_y).\] 
Explicitly, 
\begin{align}\label{ExpMer}
{}_M\textup{mer}_{x,y}^{x+y}(m_1 \cdots m_x \otimes m_1' \cdots m_y') = m_1 \cdots m_x m_1' \cdots m_y',
\end{align}
for all \(m_1, \ldots, m_x, m_1', \ldots, m_y' \in M\).

\subsection{The superalgebra  \texorpdfstring{$M_n(A)$}{Mn(A)}  and the Lie superalgebra \texorpdfstring{$\gl_n(A)$}{gln(A)}} 
Fix \(n \in \Z_{\geq 0}\). Let \(A\) be a locally unital \(\k\)-superalgebra  with distinguished idempotent set \(I\). The matrix algebra \(M_n(A) = M_n(\k) \otimes  A\) is also naturally a \(\k\)-superalgebra. For \(r,s \in [1,n]\), \(f \in A\) let \(E_{r,s}^f := E_{r,s} \otimes f\), the matrix where the $(r,s)$-entry is $f$ and all others are zero. For homogeneous \(f\in A\) we have \(\overline{E^f_{r,s}} = \overline f\). We write \(\mathfrak{gl}_n(A)\) for  \(M_n(A)\) viewed as a Lie superalgebra using the supercommutator bracket.  That is,  \(\gl_n(A) = M_n(A)\) as \(\k\)-supermodules and
\begin{align*}
[E_{r,s}^f, E_{r',s'}^{f'}] = \delta_{s, r'} E_{r,s'}^{ff'} - (-1)^{\bar f \bar f'} \delta_{s', r} E_{r',s}^{f'f}
\end{align*}
for all \(r,s, r', s' \in [1,n]\) and homogeneous \(f, f' \in A\).
We write \(U(\gl_n(A))\) for the universal enveloping superalgebra of \(\gl_n(A)\).

When $I$ is a finite set $A$ is unital.  In this case, let $\fh = \fh_{n}$ be the $\k$-submodule of $\gl_{n}(A)$ generated by $\left\{E_{r,r}^{1} \mid 1 \leq r \leq n  \right\}$.  Let $X(T_{n}) = \bigoplus_{r=1}^{n} \Z \varepsilon_{r}$ be the free abelian group on generators $\varepsilon_{1}, \dotsc , \varepsilon_{n}$ and let
\[
\Lambda_{n} = \left\{\lambda = \sum_{i=1}^{n}\lambda_{i}\varepsilon_{i} \in X(T_{n}) \mid \lambda_{i} \geq 0 \text{ for $i=1, \dotsc , n$} \right\}.
\]  Given a left $\gl_{n}(A)$-module $M$, a vector $u \in M$ has \emph{weight} $\lambda = \sum_{k=1}^{n} \lambda_{i}\varepsilon_{i} \in X(T_{n})$ if $E_{i,i}^{1}u = \lambda_{i} u$ for all $1 \leq i \leq  n$.  The set of all vectors in $M$ of weight $\lambda$ is the $\lambda$ \emph{weight space} of $M$.  There are obvious right module versions of these definitions, as well.

\subsection{Colored symmetric power modules for \texorpdfstring{$\gl_{n}(A)$}{gln(A)}} \label{leftmultact}
For \(i \in I\), write \({}_i V = \bigoplus_{j \in I} iAj \subseteq A\), and \({}_i V_n = ({}_i V)^{\oplus n}\).
For \(f \in {}_iV, r \in [1,n]\), we will write \(v_r^f\) to denote \(f\) in the \(r\)th component of \({}_iV_n\). 

There is a right action of \(\gl_n(A)\) on \({}_iV_n\) by right matrix multiplication on row vectors, so that
\begin{align*}
v_r^f \cdot E_{s,t}^{g} = \delta_{r,s} v_t^{fg}
\end{align*}
for all \(r,s,t \in [1,n]\), \(f \in {}_iV\), \(g \in A\).
The coproduct on \(U(\gl_n(A))\) given by \(\Delta(u) = 1 \otimes u + u \otimes 1\) for $u \in \gl_{n}(A)$ makes the category \(\textup{mod-}\gl_n(A)\) of right \(\gl_n(A)\)-supermodules into a monoidal supercategory.  Namely, for all \(M,N \in \textup{mod-}\gl_n(A)\) the action of $\gl_n(A)$ on $M\otimes N$ is given on homogeneous elements by
\begin{align*}
(m \otimes n) \cdot u = (-1)^{\overline u \,\overline n}mu \otimes n + m \otimes nu.
\end{align*} 

Let \(x \in \Z_{>0}\). We will denote the \(x\)th symmetric power \(S^x({}_iV_n)\) by \(S^x_iV_n\).
Then, via the coproduct of \(U(\gl_n(A))\), \(S^x_i V_n \in \textup{mod-}\gl_n(A)\) with the action given on homogeneous elements by
\begin{align*}
v_{r_1}^{f_1} \cdots v_{r_x}^{f_x} \cdot E_{s,t}^{g}
=
\sum_{u =1}^x 
(-1)^{\bar g ( \bar f_{u+1} + \cdots + \bar f_x)} \delta_{r_u s}v_{r_1}^{f_1} \cdots v_{r_{u-1}}^{f_{u-1}} v_{t}^{f_u g} v_{u+1}^{f_{u+1}} \cdots v_x^{f_x}.
\end{align*}
For all \(f \in jA^{(x)}i\), we have a \(\k\)-linear map
\(
L_x^f: S^x_iV_n \to S^x_jV_n 
\)
given by
\begin{align}\label{ExpL}
L_x^f(v_{r_1}^{g_1} \cdots v_{r_x}^{g_x}) = v_{r_1}^{fg_1} \cdots v_{r_x}^{fg_x}
\end{align}
for all \(r_t \in [1,n]\), \(g_t \in {}_iV\), \(t \in [1,x]\). Note that when \(x \geq 2\), \(\bar{f} = \bar 0\), so there is no need to worry about signs. 

We write \(\modglnAS\) for the full monoidal subcategory of right \(\gl_n(A)\)-modules generated by the objects \(S^{x}_{i}V_n\), for all \(x \in \Z_{\geq 0}\), \(i \in I\), so that
\begin{align*}
\Ob(\modglnAS) = \left\{ S^{\bx}_{\bi}V_n := S^{x_1}_{i_1}V_n \otimes \cdots \otimes S^{x_t}_{i_t}V_n \mid t \in \Z_{\geq 0}, \bx \in  \Z_{\geq 0}^t, \bi \in I^t\right\}.
\end{align*}
It is straightforward to check that the maps  \(\tau_{S_i^xV_n, S_j^y V_n}\), \({}_{{}_iV_n}\textup{spl}_{x+y}^{x,y}\), \({}_{{}_iV_n}\textup{mer}_{x+y}^{x,y}\), \(L_x^f\) as defined in \cref{ExpTwist,ExpSpl,ExpMer,ExpL} are even homomorphisms of right \(\gl_n(A)\)-supermodules.

\subsection{The defining representation of \texorpdfstring{$\WebAaI$}{WebAaI}}

\begin{theorem}\label{Gthm}
There is a monoidal superfunctor
\begin{align*}
G_n: \WebAaI \to \modglnAS
\end{align*}
given by \(G_ni^{(x)} = S^x_iV_n\) on generating objects, and
\begin{align*}
\hackcenter{
{}
}
\hackcenter{
\begin{tikzpicture}[scale=.8]
  \draw[ultra thick,blue] (0,0)--(0,0.2) .. controls ++(0,0.35) and ++(0,-0.35) .. (-0.4,0.9)--(-0.4,1);
  \draw[ultra thick,blue] (0,0)--(0,0.2) .. controls ++(0,0.35) and ++(0,-0.35) .. (0.4,0.9)--(0.4,1);
      \node[above] at (-0.4,1) {$ \scriptstyle i^{\scriptstyle (x)}$};
      \node[above] at (0.4,1) {$ \scriptstyle i^{\scriptstyle (y)}$};
      \node[below] at (0,0) {$ \scriptstyle i^{\scriptstyle (x+y)} $};
\end{tikzpicture}}
\mapsto
{}_{{}_iV_n} \textup{spl}_{x+y}^{x,y}
\qquad
\hackcenter{
\begin{tikzpicture}[scale=.8]
  \draw[ultra thick,blue ] (-0.4,0)--(-0.4,0.1) .. controls ++(0,0.35) and ++(0,-0.35) .. (0,0.8)--(0,1);
\draw[ultra thick, blue] (0.4,0)--(0.4,0.1) .. controls ++(0,0.35) and ++(0,-0.35) .. (0,0.8)--(0,1);
      \node[below] at (-0.4,0) {$ \scriptstyle i^{ \scriptstyle (x)}$};
      \node[below] at (0.4,0) {$ \scriptstyle i^{ \scriptstyle (y)}$};
      \node[above] at (0,1) {$ \scriptstyle i^{ \scriptstyle (x+y)}$};
\end{tikzpicture}}
\mapsto
{}_{{}_iV_n} \textup{mer}_{x,y}^{x+y}
\qquad
\hackcenter{
\begin{tikzpicture}[scale=.8]
  \draw[ultra thick,red] (0.4,0)--(0.4,0.1) .. controls ++(0,0.35) and ++(0,-0.35) .. (-0.4,0.9)--(-0.4,1);
  \draw[ultra thick,blue] (-0.4,0)--(-0.4,0.1) .. controls ++(0,0.35) and ++(0,-0.35) .. (0.4,0.9)--(0.4,1);
      \node[above] at (-0.4,1) {$ \scriptstyle j^{ \scriptstyle (y)}$};
      \node[above] at (0.4,1) {$ \scriptstyle i^{ \scriptstyle (x)}$};
       \node[below] at (-0.4,0) {$ \scriptstyle i^{ \scriptstyle (x)}$};
      \node[below] at (0.4,0) {$ \scriptstyle j^{ \scriptstyle (y)}$};
\end{tikzpicture}}
\mapsto
\tau_{S^x_iV_n, S^y_jV_n}
\qquad
\hackcenter{
\begin{tikzpicture}[scale=.8]
  \draw[ultra thick, blue] (0,0)--(0,0.5);
   \draw[ultra thick, red] (0,0.5)--(0,1);
   \draw[thick, fill=yellow]  (0,0.5) circle (7pt);
    \node at (0,0.5) {$ \scriptstyle f$};
     \node[below] at (0,0) {$ \scriptstyle i^{ \scriptstyle (x)}$};
      \node[above] at (0,1) {$ \scriptstyle j^{ \scriptstyle (x)}$};
\end{tikzpicture}}
\mapsto L_x^f
\end{align*}
on generating morphisms, for all \(x,y \in \Z_{\geq 0}\), \(i,j \in I\), and \(f \in jA^{(x)}i\).
\end{theorem}

\begin{proof}
We simply check that images of relations \cref{AssocRel}--\cref{AaIntertwine} are preserved by the functor \(G_n\). This is routine. In particular, relations \cref{AssocRel,DiagSwitchRel,CrossingWebRel} can be checked as in \cite[Theorem 6.1.1]{DKM} (where full details are included in the arXiv version of that paper). Relations \cref{Cox}--\cref{MergeIntertwineRel} are straightforward, as are \cref{TAaSMRel,AaIntertwine}, noting that \(\bar{f} = \bar 0\) for \(f\)-coupons on thick strands.

The left side of \cref{AaRel1} is clear. For the right side, we have
\begin{align*}
L_x^{\alpha f} (v_{r_1}^{g_1} \cdots v_{r_x}^{g_x}) &= v_{r_1}^{\alpha fg_1} \cdots v_{r_x}^{\alpha fg_x} = \alpha^x v_{r_1}^{fg_1} \cdots v_{r_x}^{fg_x} = \alpha^x L_x^f(v_{r_1}^{g_1} \cdots v_{r_x}^{g_x}),
\end{align*}
for all \(\alpha \in \k\), \(f\in jA{i}\), \(r_t \in [1,n]\), \(g_t \in {}_iV\), \(t \in [1,x]\), as desired.

The left side of \cref{AaRel2} is clear. The right side is clear for \(x=1\), so assume \(x>1\), and therefore \(\bar{f}=\overline{g} = \bar 0\). Then 
for all 
\(r_t \in [1,n]\), \(h_t \in {}_iV\), \(t \in [1,x]\), we have:
\begin{align*}
L^{f+g}_x
(v_{r_1}^{h_1} \cdots v_{r_x}^{h_x}) &= v_{r_1}^{(f+g)h_1} \cdots v_{r_x}^{(f+g)h_x} \\
&= \sum_{t=0}^x
\sum_{
\substack{
B = \{b_1 < \cdots < b_t\}
\\
C= \{c_1 < \cdots < c_{x-t}\}
\\
B \cup C = [1,x]}
}
v_{r_1}^{(\delta_{1 \in B}f+\delta_{1 \in C}g)h_1} \cdots v_{r_x}^{ (\delta_{x \in B}f+\delta_{x \in C}g)h_x} 
\\
&= \sum_{t=0}^x
\sum_{
\substack{
B = \{b_1 < \cdots < b_t\}
\\
C= \{c_1 < \cdots < c_{x-t}\}
\\
B \cup C = [1,x]}
}
(-1)^{\varepsilon(B,C)}v_{r_{b_1}}^{fh_{b_1}} \cdots v_{r_{b_t}}^{fh_{b_t}} v_{r_{c_1}}^{gh_{c_1}} \cdots v_{r_{c_{x-t}}}^{gh_{c_{x-t}}},
\end{align*}
where
\begin{align*}
\varepsilon(B,C) = \#\{
(b,c) \in B \times C \mid b >c, \bar h_b = \bar h_c= \bar 1
\}.
\end{align*}
It is easily seen, using the formulas \cref{ExpSpl,ExpMer}, that this last line is equal to
\begin{align*}
 \sum_{t=0}^x
{}_{{}_jV_n} \textup{mer}_{t,x-t}^{x} \circ (L_t^f \otimes L_{x-t}^g) \circ {}_{{}_iV_n}\textup{spl}^{t,x-t}_{x}
(v_{r_1}^{h_1} \cdots v_{r_x}^{h_x}),
\end{align*}
as desired.

Finally, for \cref{OddKnotholeRel}, let \(i,j \in I\), \(f \in jA_{\bar 1}i\). Then for \(r_1, r_2 \in [1,n]\), \(h_1, h_2 \in {}_iV\), we have
\begin{align*}
{}_{{}_jV_n} \textup{mer}_{1,1}^{2} \circ (L_1^f \otimes L_1^f) \circ {}_{{}_iV_n}\textup{spl}^{1,1}_{2}(v_{r_1}^{h_1}v_{r_2}^{h_2})
&=
(-1)^{\bar h_1}v_{r_1}^{fh_1}v_{r_2}^{fh_2}
+
(-1)^{\bar h_2(\bar h_1+1)}v_{r_2}^{fh_2}v_{r_1}^{fh_1}\\
&=
(-1)^{\bar h_1}v_{r_1}^{fh_1}v_{r_2}^{fh_2}
+
(-1)^{\bar h_1+1}v_{r_1}^{fh_1}v_{r_2}^{fh_2} = 0,
\end{align*}
completing the proof.
\end{proof}

\section{Special morphisms and bases for \texorpdfstring{$\WebAaI$}{WebAaI}}

\subsection{Explosion and contraction}\label{ExplSec}

Define morphisms 
\begin{align}\label{E:multisplitmerge}
\hackcenter{}
y_i^{(x_1, \ldots, x_n)}:=
\hackcenter{
\begin{tikzpicture}[scale=0.8]
 \draw[ultra thick, color=blue] (0,-0.2)--(0,0);
 \draw[ultra thick, color=blue] (0,0) .. controls ++(0,0.5) and ++(0,-0.5) .. (-1.2-.6,1);
  \draw[ultra thick, color=blue] (0,0) .. controls ++(0,0.5) and ++(0,-0.5) .. (-0.6-.3,1);
    \draw[ultra thick, color=blue] (0,0) .. controls ++(0,0.5) and ++(0,-0.5) .. (0.6+.3,1);
       \draw[ultra thick, color=blue] (0,0) .. controls ++(0,0.5) and ++(0,-0.5) .. (1.2+.6,1);
          \node at (0,1) {$\scriptstyle \cdots $};
            \node[below] at (0.1,-0.2) {$\scriptstyle i^{(\scriptstyle x_1 + \cdots + x_{n})} $};
             \node[above] at (-1.2-.6,1) {$\scriptstyle i^{(\scriptstyle x_1)} $};
              \node[above] at (-0.6-.3,1) {$\scriptstyle i^{(\scriptstyle x_2)} $};
               \node[above] at (0.6+.3,1) {$\scriptstyle i^{(\scriptstyle x_{n-1})} $};
                \node[above] at (1.2+.8,0.97) {$\scriptstyle i^{(\scriptstyle x_{n})} $};
\end{tikzpicture}}
;
\qquad
\qquad
z_i^{(x_1, \ldots, x_n)}:=
\hackcenter{
\begin{tikzpicture}[scale=0.8]
 \draw[ultra thick, color=blue] (0,0.2)--(0,0);
 \draw[ultra thick, color=blue] (0,0) .. controls ++(0,-0.5) and ++(0,0.5) .. (-1.2-.6,-1);
  \draw[ultra thick, color=blue] (0,0) .. controls ++(0,-0.5) and ++(0,0.5) .. (-0.6-.3,-1);
    \draw[ultra thick, color=blue] (0,0) .. controls ++(0,-0.5) and ++(0,0.5) .. (0.6+.3,-1);
       \draw[ultra thick, color=blue] (0,0) .. controls ++(0,-0.5) and ++(0,0.5) .. (1.2+.6,-1);
          \node at (0,-1) {$\scriptstyle \cdots $};
            \node[above] at (0.1,0.2) {$\scriptstyle i^{(\scriptstyle x_1 + \cdots + x_{n})} $};
             \node[below] at (-1.2-.6,-1) {$\scriptstyle i^{(\scriptstyle x_1)} $};
              \node[below] at (-0.6-.3,-1) {$\scriptstyle i^{(\scriptstyle x_2)} $};
               \node[below] at (0.6+.3,-1) {$\scriptstyle i^{(\scriptstyle x_{n-1})} $};
                \node[below] at (1.2+.8,-1) {$\scriptstyle i^{(\scriptstyle x_{n})} $};
\end{tikzpicture}},
\end{align}
in \(\WebAaI(i^{(x_1 + \cdots + x_n)}, i^{(x_1)} \cdots  i^{(x_n)})\)
and
 \(\WebAaI(i^{(x_1)} \cdots i^{(x_n)}, i^{(x_1 + \cdots + x_n)})\), respectively.  The diagrams should be interpreted as \(n-1\) vertically composed splits (or merges). By \cref{AssocRel} the resulting morphism is independent of the split (or merge) order. 

For \(\bi^{(\bx)} = i_1^{(x_1)} \cdots i_t^{(x_t)} \in \textup{Ob}(\WebAaI)\) define \(\bi^\bx = (i_1^{(1)})^{ x_1}  \cdots (i_t^{(1)})^{ x_t} \in \textup{Ob}(\WebAaI)\), and 
\begin{align*}
Y_{\bi^{(\bx)}}&:= y_{i_1}^{(1, \ldots, 1)} \otimes \cdots \otimes y_{i_t}^{(1, \ldots, 1)} \in \WebAaI(\bi^{(\bx)}, \bi^\bx)\\
Z_{\bi^{(\bx)}}&:= z_{i_1}^{(1, \ldots, 1)} \otimes \cdots \otimes z_{i_t}^{(1, \ldots, 1)} \in \WebAaI(\bi^\bx,\bi^{(\bx)})
\end{align*}
Then for \(\bi^{(\bx)}, \bj^{(\by)} \in \textup{Ob}(\WebAaI)\), we have linear maps: 
\begin{align*}
\exp_{\bi^{(\bx)}, \bj^{(\by)}}: \WebAaI(\bi^{(\bx)}, \bj^{(\by)}) \to \WebAaI(\bi^\bx, \bj^\by),
\qquad
f \mapsto Y_{\bj^{(\by)}}\circ f\circ Z_{\bi^{(\bx)}}\\
\con_{\bi^{(\bx)}, \bj^{(\by)}}: \WebAaI(\bi^{\bx}, \bj^{\by}) \to \WebAaI(\bi^{(\bx)}, \bj^{(\by)}),
\qquad
f \mapsto Z_{\bj^{(\by)}}\circ f\circ Y_{\bi^{(\bx)}}
\end{align*}
We refer to these maps a \textit{explosion} and \textit{contraction}, respectively. 
The next lemma follows by repeated application of \cref{KnotholeRel}.

\begin{lemma}\label{L:ExplCon} For all \(f \in \WebAaI(\bi^{(\bx)}, \bj^{(\by)})\), we have
\begin{align*}
\con_{\bi^{(\bx)}, \bj^{(\by)}} \circ \exp_{\bi^{(\bx)}, \bj^{(\by)}} (f)= \bx! \,\by! \, f.
\end{align*}
\end{lemma}

\subsection{Additional coupons}

Let \(x \in \Z_{\geq 0}\), \(i,j \in I\), \(f \in jAi\). We first define some shorthand coupons in \(\WebAaI(i^{(x)}, j^{(x)})\).
Set:
\begin{align}\label{greendotdef}
{}
\hackcenter{
\begin{tikzpicture}[scale=0.8]
  \draw[ultra thick, blue] (0,0)--(0,0.9);
  \draw[ultra thick, red] (0,0.9)--(0,1.8);
  \draw[thick, fill=yellow]  (0,0.9) circle (9pt);
    \node at (0,0.9) {$\scriptstyle f^{\scriptstyle \diamond}$};
     \node[below] at (0,0) {$\scriptstyle i^{\scriptstyle(x)}$};
      \node[above] at (0,1.8) {$\scriptstyle j^{\scriptstyle(x)}$};
\end{tikzpicture}}
=
\hackcenter{
\begin{tikzpicture}[scale=0.8]
  \draw[ultra thick, blue] (0,0)--(0,0.1) .. controls ++(0,0.35) and ++(0,-0.35) .. (-0.6,0.6)--(-0.6,0.9);
    \draw[ultra thick, red]  (-0.6,0.9)--(-0.6,1.2) 
  .. controls ++(0,0.35) and ++(0,-0.35) .. (0,1.7)--(0,1.8);
  \draw[ultra thick, blue] (0,0)--(0,0.1) .. controls ++(0,0.35) and ++(0,-0.35) .. (0.6,0.6)--(0.6,0.9);
    \draw[ultra thick, red]  (0.6,0.9)--(0.6,1.2) 
  .. controls ++(0,0.35) and ++(0,-0.35) .. (0,1.7)--(0,1.8);
  .. controls ++(0,0.35) and ++(0,-0.35) .. (0,1.7)--(0,1.8);
   \draw[thick, fill=yellow] (-0.6,0.9) circle (7pt);
     \draw[thick, fill=yellow] (0.6,0.9) circle (7pt);
       \node[above] at (0,1.8) {$\scriptstyle j^{\scriptstyle(x)}$};
         \node[below] at (0,0) {$\scriptstyle i^{\scriptstyle(x)}$};
           \node at (-0.6,0.9) {$\scriptstyle f$};
             \node at (0.6,0.9) {$\scriptstyle f$};
             \node[right] at (0.5,0.3) {$\scriptstyle i^{\scriptstyle(1)}$};
               \node[left] at (-0.5,0.3) {$\scriptstyle i^{\scriptstyle(1)}$};
                \node at (0,0.9) {$\scriptstyle \cdots $};
\end{tikzpicture}}
\qquad
\text{ and }
\qquad
\hackcenter{
\begin{tikzpicture}[scale=0.8]
  \draw[ultra thick, blue] (0,0)--(0,0.9);
  \draw[ultra thick, red] (0,0.9)--(0,1.8);
  \draw[thick, fill=yellow]  (0,0.9) circle (9pt);
    \node at (0,0.9) {$\scriptstyle f^{\scriptstyle a}$};
     \node[below] at (0,0) {$\scriptstyle i^{\scriptstyle(x)}$};
      \node[above] at (0,1.8) {$\scriptstyle j^{\scriptstyle(x)}$};
\end{tikzpicture}}
\;
&=
\begin{cases}
\;\;\;\;
\hackcenter{
\begin{tikzpicture}[scale=0.8]
  \draw[ultra thick, blue] (0,0)--(0,0.9);
  \draw[ultra thick, red] (0,0.9)--(0,1.8);
     \draw[thick, fill=yellow] (0,0.9) circle (8pt);
    \node at (0,0.9) {$\scriptstyle f$};
     \node[below] at (0,0) {$\scriptstyle i^{\scriptstyle(x)}$};
      \node[above] at (0,1.8) {$\scriptstyle j^{\scriptstyle(x)}$};
\end{tikzpicture}}
&
\textup{if }f \in jai, \textup{ or } f \in jA_{\bar 1}i, \,x=1;
\\
\;\;\;\;
\hackcenter{
\begin{tikzpicture}[scale=0.8]
  \draw[ultra thick, blue] (0,0)--(0,0.9);
  \draw[ultra thick, red] (0,0.9)--(0,1.8);
     \draw[thick, fill=yellow] (0,0.9) circle (9pt);
    \node at (0,0.9) {$\scriptstyle f^{\scriptstyle \diamond}$};
     \node[below] at (0,0) {$\scriptstyle i^{\scriptstyle(x)}$};
      \node[above] at (0,1.8) {$\scriptstyle j^{\scriptstyle(x)}$};
\end{tikzpicture}}
&
\textup{if }
f \in jA_{\bar 0}i \backslash jai.
\end{cases}
\end{align}

We will say a finitely supported function \(\mu: {}_j\BasisB_i \to \Z_{\geq 0}\) is a {\em restricted \( {}_j\BasisB_i\)-composition} provided that \(\mu(b) \leq 1\) whenever \(\overline{b} = \bar 1\). We write \(|\mu| = \sum_{b \in  {}_j\BasisB_i} \mu(b) \in \Z_{\geq 0}\).
Then for such \(\mu\), enumerating \(\textup{supp}(\mu) = \{b_1 <  \cdots < b_d\}\), we define the following morphisms in \(\WebAaI(i^{(|\mu|)}, j^{(|\mu|)})\):
\begin{align}\label{mudef}
{}
\hackcenter{
\begin{tikzpicture}[scale=0.8]
  \draw[ultra thick, blue] (0,0)--(0,0.9);
  \draw[ultra thick, red] (0,0.9)--(0,1.8);
  \draw[thick, fill=yellow]  (0,0.9) circle (9pt);
    \node at (0,0.9) {$\scriptstyle \mu^{\scriptstyle \diamond}$};
     \node[below] at (0,0) {$\scriptstyle i^{\scriptstyle(|\mu|)}$};
      \node[above] at (0,1.8) {$\scriptstyle j^{\scriptstyle(|\mu|)}$};
\end{tikzpicture}}
=
\hackcenter{
\begin{tikzpicture}[scale=0.8]
  \draw[ultra thick, blue] (0,0)--(0,0.1) .. controls ++(0,0.35) and ++(0,-0.35) .. (-0.6,0.6)--(-0.6,0.9);
    \draw[ultra thick, red]  (-0.6,0.9)--(-0.6,1.2) 
  .. controls ++(0,0.35) and ++(0,-0.35) .. (0,1.7)--(0,1.8);
  \draw[ultra thick, blue] (0,0)--(0,0.1) .. controls ++(0,0.35) and ++(0,-0.35) .. (0.6,0.6)--(0.6,0.9);
    \draw[ultra thick, red]  (0.6,0.9)--(0.6,1.2) 
  .. controls ++(0,0.35) and ++(0,-0.35) .. (0,1.7)--(0,1.8);
  .. controls ++(0,0.35) and ++(0,-0.35) .. (0,1.7)--(0,1.8);
   \draw[thick, fill=yellow] (-0.6,0.9) circle (9pt);
     \draw[thick, fill=yellow] (0.6,0.9) circle (9pt);
       \node[above] at (0,1.8) {$\scriptstyle j^{\scriptstyle(|\mu|)}$};
         \node[below] at (0,0) {$\scriptstyle i^{\scriptstyle(|\mu|)}$};
           \node at (-0.6,0.9) {$\scriptstyle b_1^{\scriptstyle \diamond}$};
             \node at (0.6,0.9) {$\scriptstyle b_d^{\scriptstyle \diamond}$};
             \node[right] at (0.5,0.3) {$\scriptstyle i^{\scriptstyle(\mu(b_d))}$};
               \node[left] at (-0.5,0.3) {$\scriptstyle i^{\scriptstyle(\mu(b_1))}$};
                \node at (0,0.9) {$\scriptstyle \cdots $};
\end{tikzpicture}}
\qquad
\text{ and }
\qquad
\hackcenter{
\begin{tikzpicture}[scale=0.8]
  \draw[ultra thick, blue] (0,0)--(0,0.9);
  \draw[ultra thick, red] (0,0.9)--(0,1.8);
  \draw[thick, fill=yellow]  (0,0.9) circle (9pt);
    \node at (0,0.9) {$\scriptstyle \mu^{\scriptstyle a}$};
     \node[below] at (0,0) {$\scriptstyle i^{\scriptstyle(|\mu|)}$};
      \node[above] at (0,1.8) {$\scriptstyle j^{\scriptstyle(|\mu|)}$};
\end{tikzpicture}}
=
\hackcenter{
\begin{tikzpicture}[scale=0.8]
  \draw[ultra thick, blue] (0,0)--(0,0.1) .. controls ++(0,0.35) and ++(0,-0.35) .. (-0.6,0.6)--(-0.6,0.9);
    \draw[ultra thick, red]  (-0.6,0.9)--(-0.6,1.2) 
  .. controls ++(0,0.35) and ++(0,-0.35) .. (0,1.7)--(0,1.8);
  \draw[ultra thick, blue] (0,0)--(0,0.1) .. controls ++(0,0.35) and ++(0,-0.35) .. (0.6,0.6)--(0.6,0.9);
    \draw[ultra thick, red]  (0.6,0.9)--(0.6,1.2) 
  .. controls ++(0,0.35) and ++(0,-0.35) .. (0,1.7)--(0,1.8);
  .. controls ++(0,0.35) and ++(0,-0.35) .. (0,1.7)--(0,1.8);
   \draw[thick, fill=yellow] (-0.6,0.9) circle (9pt);
     \draw[thick, fill=yellow] (0.6,0.9) circle (9pt);
       \node[above] at (0,1.8) {$\scriptstyle j^{\scriptstyle(|\mu|)}$};
         \node[below] at (0,0) {$\scriptstyle i^{\scriptstyle(|\mu|)}$};
           \node at (-0.6,0.9) {$\scriptstyle b_1^{\scriptstyle a}$};
             \node at (0.6,0.9) {$\scriptstyle b_d^{\scriptstyle a}$};
             \node[right] at (0.5,0.3) {$\scriptstyle i^{\scriptstyle(\mu(b_d))}$};
               \node[left] at (-0.5,0.3) {$\scriptstyle i^{\scriptstyle(\mu(b_1))}$};
                \node at (0,0.9) {$\scriptstyle \cdots $};
\end{tikzpicture}}.
\end{align}

\subsection{A spanning set for morphisms}
Let \(\bi^{(\bx)} = i_1^{(x_1)} \cdots i_t^{(x_t)}, \,\bj^{(\by)} = j_1^{(y_1)} \cdots j_u^{(y_u)} \in \textup{Ob}(\WebAaI)\). Let \(\mathcal{M}(\bi^{(\bx)}, \bj^{(\by)})\) be the collection of tuples \(\bmu = ( \mu_{r,s})_{r \in [1,t], s \in [1,u]}\), where each \(\mu_{r,s}\) is a restricted \({}_{j_s}\BasisB_{i_r}\)-composition, and we have
\begin{align*}
\sum_{s \in [1,u]} |\mu_{r,s}| = x_r \;\;\;(\forall r \in[1,t])
\qquad
\text{and}
\qquad
\sum_{r \in [1,t]} |\mu_{r,s}| = y_s \;\;\;(\forall s \in[1,u]).
\end{align*}
For \(\bmu \in \mathcal{M}(\bi^{(\bx)},\bj^{(\by)})\), we define an associated diagram \(\eta_{\bmu}^a \in \WebAaI(\bi^{(\bx)}, \bi^{(\by)})\) by setting
\begin{align}\label{basiseldef}
\hackcenter{}
{}
\eta_{\bmu}^a&:=
\hackcenter{
\begin{tikzpicture}[scale=1.1]
  \draw[ultra thick, blue] (0,1)--(0,1.1) .. controls ++(0,0.35) and ++(0,-0.35) .. (-0.6,1.6)--(-0.6,1.9);
  \draw[ultra thick, blue] (0,1)--(0,1.1) .. controls ++(0,0.35) and ++(0,-0.35) .. (0.6,1.6)--(0.6,1.9);
         \node[below] at (0,1) {$\scriptstyle i_1^{\scriptstyle(x_1)}$};
             \node[right] at (0.4,1.2) {$\scriptstyle i_1^{\scriptstyle(|\mu_{1,u}|)}$};
               \node[left] at (-0.4,1.2) {$\scriptstyle i_1^{\scriptstyle(|\mu_{1,1})|}$};
                \node at (0,1.5) {$\scriptstyle \cdots $};
  \draw[ultra thick, red] (0+3,1)--(0+3,1.1) .. controls ++(0,0.35) and ++(0,-0.35) .. (-0.6+3,1.6)--(-0.6+3,1.9);
  \draw[ultra thick, red] (0+3,1)--(0+3,1.1) .. controls ++(0,0.35) and ++(0,-0.35) .. (0.6+3,1.6)--(0.6+3,1.9);
         \node[below] at (0+3,1) {$\scriptstyle i_t^{\scriptstyle(x_t)}$};
             \node[right] at (0.4+3,1.2) {$\scriptstyle i_t^{\scriptstyle(|\mu_{t,u}|)}$};
               \node[left] at (-0.4+3,1.2) {$\scriptstyle i_t^{\scriptstyle(|\mu_{t,1}|)}$};
                \node at (0+3,1.5) {$\scriptstyle \cdots $};    
    \draw[ultra thick, blue] (-0.6,1.6)--(-0.6,2.7); 
      \draw[ultra thick, orange] (-0.6,2.1)--(-0.6,2.7); 
               \draw[thick, fill=yellow] (-0.6,2.1) circle (9pt);
    \node at (-0.6,2.1) {$\scriptstyle \mu_{1,1}^{\scriptstyle a}$};
        \draw[ultra thick, blue] (0.6,1.6)--(0.6,2.7); 
        \draw[ultra thick, green] (0.6,2.1)--(0.6,2.7); 
               \draw[thick, fill=yellow] (0.6,2.1) circle (9pt);
    \node at (0.6,2.1) {$\scriptstyle \mu_{1,u}^{\scriptstyle a}$};
        \draw[ultra thick, red] (2.4,1.6)--(2.4,2.7); 
        \draw[ultra thick, orange] (2.4,2.1)--(2.4,2.7); 
               \draw[thick, fill=yellow] (2.4,2.1) circle (9pt);
    \node at (2.4,2.1) {$\scriptstyle \mu_{t,1}^{\scriptstyle a}$};
        \draw[ultra thick, red] (3.6,1.6)--(3.6,2.7); 
         \draw[ultra thick, green] (3.6,2.1)--(3.6,2.7); 
               \draw[thick, fill=yellow] (3.6,2.1) circle (9pt);
    \node at (3.6,2.1) {$\scriptstyle \mu_{t,u}^{\scriptstyle a}$};
      \draw[ultra thick, orange] (0,4.2-0)--(0,4.2-0.1) .. controls ++(0,-0.35) and ++(0,0.35) .. (-0.6,4.2-0.6)--(-0.6,4.2-0.9);
  \draw[ultra thick, orange] (0,4.2-0)--(0,4.2-0.1) .. controls ++(0,-0.35) and ++(0,0.35) .. (0.6,4.2-0.6)--(0.6,4.2-0.9);
         \node[above] at (0,4.2-0) {$\scriptstyle j_1^{\scriptstyle(y_1)}$};
             \node[right] at (0.5,4.3-0.3) {$\scriptstyle j_1^{\scriptstyle(|\mu_{t,1}|)}$};
               \node[left] at (-0.3,4.3-0.3) {$\scriptstyle j_1^{\scriptstyle(|\mu_{1,1}|)}$};
                \node at (0,4.2-0.7) {$\scriptstyle \cdots $};
  \draw[ultra thick, green] (0+3,4.2-0)--(0+3,4.2-0.1) .. controls ++(0,-0.35) and ++(0,0.35) .. (-0.6+3,4.2-0.6)--(-0.6+3,4.2-0.9);
  \draw[ultra thick, green] (0+3,4.2-0)--(0+3,4.2-0.1) .. controls ++(0,-0.35) and ++(0,0.35) .. (0.6+3,4.2-0.6)--(0.6+3,4.2-0.9);
         \node[above] at (0+3,4.2-0) {$\scriptstyle j_u^{\scriptstyle(y_u)}$};
             \node[right] at (0.5+3,4.3-0.3) {$\scriptstyle j_u^{\scriptstyle(|\mu_{t,u}|)}$};
               \node[left] at (-0.3+3,4.3-0.3) {$\scriptstyle j_u^{\scriptstyle(|\mu_{1,u}|)}$};
                \node at (0+3,4.2-0.7) {$\scriptstyle \cdots $};   
       \node at (1.5,2.1) {$\scriptstyle \cdots $};          
        \node at (0,2.1) {$\scriptstyle \cdots $};     
         \node at (3,2.1) {$\scriptstyle \cdots $};     
              \draw[thick, fill=gray!50!white] (-.8,2.7)--(3.8,2.7)--(3.8,3.3)--(-.8,3.3)--(-.8,2.7);   
     \node at (1.5, 3) {$\scriptstyle Z$};  
\end{tikzpicture}},
\end{align}
where \(Z\) is any diagram composed only of crossings, in which the the strand labeled with thickness \(|\mu_{r,s}|\) at the bottom of \(Z\) is carried to the strand labeled with thickness \(|\mu_{r,s}|\) at the top of \(Z\). We similarly define \(\eta_{\bmu}^\diamond \in \WebAaI(\bi^{(\bx)}, \bi^{(\by)})\) by replacing all \(a\) superscripts in the diagram above with \(\diamond\) superscripts.

We also associate to \(\bmu\) and any \(\mathcal{m} \subseteq \BasisB\) the following scalar:
\begin{align*}
[\bmu]_{\mathcal{m}}^! = \prod_{\substack{ r \in [1,t], \, s \in [1,u] \\ b \in \mathcal{m}}} 
\mu_{r,s}(b)!
\qquad
\in
\k.
\end{align*}

\begin{lemma}\label{startogather}
 For \(\bi^{(\bx)}, \bj^{(\by)} \in \textup{Ob}(\WebAaI)\), \(\bmu \in \mathcal{M}(\bi^{(\bx)}, \bj^{(\by)})\), we have:
\begin{align*}
\eta_{\bmu}^\diamond = [\bmu]^!_{\basisb} \eta^a_{\bmu}.
\end{align*}
\end{lemma}
\begin{proof}
Follows from \cref{mudef} and repeated applications of the relations \cref{TAaSMRel,KnotholeRel}.
\end{proof}

\subsection{Action on colored symmetric power modules}

\begin{lemma}\label{explicitmaplemma}
Let \(\bi^{(\bx)} = i_1^{(x_1)} \cdots i_t^{(x_t)}, \bj^{(\by)}=j_1^{(y_1)} \cdots j_u^{(y_u)} \in \textup{Ob}(\WebAaI)\), and \(\bmu \in \mathcal{M}(\bi^{(\bx)}, \bj^{(\by)})\). Then the map
\begin{align*}
G_n(\eta^a_{\bmu}): S^{\bx}_{\bi}V \to S^{\by}_{\bj}V
\end{align*}
is given by 
\begin{align}\label{Getaformula}
G_n(\eta^a_{\bmu}):
\bigotimes_{\alpha = 1}^t \prod_{\gamma = 1}^{x_\alpha} v_{\ell_\gamma^{(\alpha)}}^{f_\gamma^{(\alpha)}}
&
\mapsto
[\bmu]^!_{\BasisB\backslash \basisb}
\sum_T
(-1)^{\sigma_1(T) + \sigma_2(T) + \sigma_3(T)}
\bigotimes_{\beta = 1}^u
\prod_{\alpha = 1}^t
\prod_{c \in \BasisB}^{<}
\prod^<_{\gamma \in T_{\alpha,\beta}^c} v_{\ell_\gamma^{(\alpha)}}^{cf_\gamma^{(\alpha)}}
\end{align}
for all \(\ell_\gamma^{(\alpha)} \in [1,n]\), \(f_\gamma^{(\alpha)} \in {}_{i_\alpha}V\), where the sum ranges over all tuples of sets \(T = (T_{r,s}^b)_{b \in \BasisB, r \in [1,t], s \in [1,u]}\) such that \(T_{r,s}^b \subseteq [1, x_r]\), \(|T_{r,s}^b| = \mu_{r,s}(b)\) and \(\bigsqcup_{s,b} T_{r,s}^b = [1,x_r]\), and 
\begin{align*}
\sigma_1(T) = 
\sum_{\substack{
b \in \BasisB\\
r'<r \in [1,t]\\
s \in [1,u]
}}
\sum_{
\gamma \in [1,x_{r'}]
}
\overline{\mu}_{r,s}(b) \bar{b}\bar{f}^{(r')}_{\gamma} 
+
\sum_{\substack{
b, b' \in \BasisB \\
r \in [1,t] \\
s'< s \in [1,u]
}}
\sum_{\gamma \in T^{b'}_{r,s'}} \overline{\mu}_{r,s}(b)\bar{b} \bar{f}_{\gamma}^{(r)}
+
\sum_{\substack{
b' < b \in \BasisB\\
r \in [1,t]\\
s \in [1,u]
}}
\sum_{\gamma \in T^{b'}_{r,s}}
\overline{\mu}_{r,s}(b)\bar{b} \bar{f}_{\gamma}^{(r)};
\end{align*}
\begin{align*}
\sigma_2(T) = 
\sum_{
\substack{
r \in [1,t] \\
 s' < s \in [1,u] \\
\gamma < \gamma' \in [1,x_r]
}}
\sum_{\substack{
\gamma \in T^b_{r,s}\\
\gamma' \in T^{b'}_{r,s'}
}}
\bar{f}_\gamma^{(r)} \bar{f}_{\gamma'}^{(r)};
\qquad
\qquad
\sigma_3(T) = 
\sum_{\substack{ b,b' \in \BasisB \\ r <r' \in [1,t] \\ s'<s \in[1,u]  }} 
\sum_{\substack{ \gamma \in T^b_{r,s} \\ \gamma' \in T^{b'}_{r',s'}}} (\bar{b} + \bar{f}_\gamma^{(r)})(\bar{b}' + \bar{f}_{\gamma'}^{(r')}).
\end{align*}
\end{lemma}
\begin{proof}
The formula follows directly from \cref{Gthm} and the formulas \cref{ExpTwist,ExpSpl,ExpMer,ExpL}. The sign \((-1)^{\sigma_1(T)}\) arises via the coupon morphisms in \(\eta_{\bmu}^{a}\), the sign \((-1)^{\sigma_2(T)}\) arises via the split morphisms, and the sign \((-1)^{\sigma_3(T)}\) arises via the \(Z\) morphism. The scalar \([\bmu]^!_{\BasisB \backslash \basisb}\) arises from consideration of \cref{greendotdef} and the following calculation:
\begin{align*}
{}
G_n\left(
\hackcenter{
\begin{tikzpicture}[scale=0.8]
  \draw[ultra thick, blue] (0,0.3)--(0,0.9);
  \draw[ultra thick, red] (0,0.9)--(0,1.5);
  \draw[thick, fill=yellow]  (0,0.9) circle (9pt);
    \node at (0,0.9) {$\scriptstyle b^{\scriptstyle a}$};
     \node[below] at (0,0.3) {$\scriptstyle i^{\scriptstyle(m)}$};
      \node[above] at (0,1.5) {$\scriptstyle j^{\scriptstyle(m)}$};
\end{tikzpicture}}
\right)
(v_{z_1}^{g_1} \cdots v_{z_m}^{g_m})
\;
&=
\begin{cases}
v_{z_1}^{bg_1} \cdots v_{z_m}^{bg_m}
&
\textup{if } b\in {}_j\basisb_i, \textup{ or } b \in  ({}_j\BasisB_i)_{\bar 1}, \,x=1;
\\
m!\,v_{z_1}^{bg_1} \cdots v_{z_m}^{bg_m}
&
\textup{if }
b \in ({}_j\BasisB_i)_{\bar 0} \backslash {}_j\basisb_i,
\end{cases}
\end{align*}
for all \(m \in \Z_{\geq 0}\), \(z_1, \ldots, z_m \in [1,n]\), \(g_1, \ldots, g_m \in {}_iV\).
\end{proof}

\subsection{Spanning results for \texorpdfstring{$\WebAaI$}{WebAaI}}
\begin{proposition}\label{SpanLem}
The morphisms \(\{\eta^a_{\bmu} \mid \bmu \in \mathcal{M}(\bi^{(\bx)}, \bi^{(\by)})\}\) comprise a \(\k\)-spanning set for \( \WebAaI \left( \bi^{(\bx)}, \bi^{(\by)}\right) \).
\end{proposition}
\begin{proof}
Let \(F\) be a diagram in \( \WebAaI(\bi^{(\bx)}, \bi^{(\by)})\). Then we may rewrite \(F\) as a \(\k\)-linear combination of diagrams composed of the generating splits, merges, crossings and coupons, where:
\begin{itemize}
\item no merge occurs below any split (via \cref{MSrel});
\item no crossing occurs below any split (via \cref{SplitIntertwineRel,MergeIntertwineRel});
\item no crossing occurs below any coupon (via  \cref{AaIntertwine});
\item no coupon occurs above any merge or below any split (via \cref{TAaSMRel});
\item there is at most one coupon on any given strand (via  \cref{AaRel2});
\item all coupons are elements of \(\BasisB\) (via  \cref{AaRel1,AaRel2});
\end{itemize}
Applying relations in this manner, and together with \cref{AssocRel} and \cref{CrossAbsorb}, we may express \(F\) as a \(\k\)-linear combination of diagrams of the form
\begin{align}\label{midwaydiag}
\hackcenter{}
{}
\hackcenter{
\begin{tikzpicture}[scale=1.1]
  \draw[ultra thick, blue] (0,1)--(0,1.1) .. controls ++(0,0.35) and ++(0,-0.35) .. (-0.6,1.6)--(-0.6,1.9);
  \draw[ultra thick, blue] (0,1)--(0,1.1) .. controls ++(0,0.35) and ++(0,-0.35) .. (0.6,1.6)--(0.6,1.9);
         \node[below] at (0,1) {$\scriptstyle i_1^{\scriptstyle(x_1)}$};
             \node[right] at (0.4,1.2) {$\scriptstyle i_1^{\scriptstyle(m_{1,u})}$};
               \node[left] at (-0.4,1.2) {$\scriptstyle i_1^{\scriptstyle(m_{1,1})}$};
                \node at (0,1.5) {$\scriptstyle \cdots $};
  \draw[ultra thick, red] (0+3,1)--(0+3,1.1) .. controls ++(0,0.35) and ++(0,-0.35) .. (-0.6+3,1.6)--(-0.6+3,1.9);
  \draw[ultra thick, red] (0+3,1)--(0+3,1.1) .. controls ++(0,0.35) and ++(0,-0.35) .. (0.6+3,1.6)--(0.6+3,1.9);
         \node[below] at (0+3,1) {$\scriptstyle i_t^{\scriptstyle(x_t)}$};
             \node[right] at (0.4+3,1.2) {$\scriptstyle i_t^{\scriptstyle(m_{t,u})}$};
               \node[left] at (-0.4+3,1.2) {$\scriptstyle i_t^{\scriptstyle(m_{t,1})}$};
                \node at (0+3,1.5) {$\scriptstyle \cdots $};    
    \draw[ultra thick, blue] (-0.6,1.6)--(-0.6,2.7); 
      \draw[ultra thick, orange] (-0.6,2.1)--(-0.6,2.7); 
       \draw[thick, fill=gray!50!white] (3.3-4.2,1.8)--(3.9-4.2,1.8)--(3.9-4.2,2.4)--(3.3-4.2,2.4)--(3.3-4.2,1.8);
    \node at (-0.6,2.1) {$\scriptstyle h_{1,1}$};
        \draw[ultra thick, blue] (0.6,1.6)--(0.6,2.7); 
        \draw[ultra thick, green] (0.6,2.1)--(0.6,2.7); 
               \draw[thick, fill=gray!50!white] (3.3-3,1.8)--(3.9-3,1.8)--(3.9-3,2.4)--(3.3-3,2.4)--(3.3-3,1.8);
    \node at (0.6,2.1) {$\scriptstyle h_{1,u}$};
        \draw[ultra thick, red] (2.4,1.6)--(2.4,2.7); 
        \draw[ultra thick, orange] (2.4,2.1)--(2.4,2.7); 
                  \draw[thick, fill=gray!50!white] (3.3-1.2,1.8)--(3.9-1.2,1.8)--(3.9-1.2,2.4)--(3.3-1.2,2.4)--(3.3-1.2,1.8);
    \node at (2.4,2.1) {$\scriptstyle h_{t,1}$};
        \draw[ultra thick, red] (3.6,1.6)--(3.6,2.7); 
         \draw[ultra thick, green] (3.6,2.1)--(3.6,2.7); 
          \draw[thick, fill=gray!50!white] (3.3,1.8)--(3.9,1.8)--(3.9,2.4)--(3.3,2.4)--(3.3,1.8);
    \node at (3.6,2.1) {$\scriptstyle h_{t,u}$};
      \draw[ultra thick, orange] (0,4.2-0)--(0,4.2-0.1) .. controls ++(0,-0.35) and ++(0,0.35) .. (-0.6,4.2-0.6)--(-0.6,4.2-0.9);
  \draw[ultra thick, orange] (0,4.2-0)--(0,4.2-0.1) .. controls ++(0,-0.35) and ++(0,0.35) .. (0.6,4.2-0.6)--(0.6,4.2-0.9);
         \node[above] at (0,4.2-0) {$\scriptstyle j_1^{\scriptstyle(y_1)}$};
             \node[right] at (0.5,4.3-0.3) {$\scriptstyle j_1^{\scriptstyle(m_{t,1})}$};
               \node[left] at (-0.3,4.3-0.3) {$\scriptstyle j_1^{\scriptstyle(m_{1,1})}$};
                \node at (0,4.2-0.7) {$\scriptstyle \cdots $};
  \draw[ultra thick, green] (0+3,4.2-0)--(0+3,4.2-0.1) .. controls ++(0,-0.35) and ++(0,0.35) .. (-0.6+3,4.2-0.6)--(-0.6+3,4.2-0.9);
  \draw[ultra thick, green] (0+3,4.2-0)--(0+3,4.2-0.1) .. controls ++(0,-0.35) and ++(0,0.35) .. (0.6+3,4.2-0.6)--(0.6+3,4.2-0.9);
         \node[above] at (0+3,4.2-0) {$\scriptstyle j_u^{\scriptstyle(y_u)}$};
             \node[right] at (0.5+3,4.3-0.3) {$\scriptstyle j_u^{\scriptstyle(m_{t,u})}$};
               \node[left] at (-0.3+3,4.3-0.3) {$\scriptstyle j_u^{\scriptstyle(m_{1,u})}$};
                \node at (0+3,4.2-0.7) {$\scriptstyle \cdots $};   
       \node at (1.5,2.1) {$\scriptstyle \cdots $};          
        \node at (0,2.1) {$\scriptstyle \cdots $};     
         \node at (3,2.1) {$\scriptstyle \cdots $};     
              \draw[thick, fill=gray!50!white] (-.8,2.7)--(3.8,2.7)--(3.8,3.3)--(-.8,3.3)--(-.8,2.7);   
     \node at (1.5, 3) {$\scriptstyle Z$};  
\end{tikzpicture}},
\end{align}
where \(Z\) is as in \cref{basiseldef}, \(\sum_{s=1}^{u} m_{r,s} = x_r\) for all \(r \in [1,t]\), \(\sum_{r=1}^{t} m_{r,s} = y_s\) for all \(s \in [1,u]\), and \(h_{r,s}\) is some diagram in \(\WebAaI( i_r^{(m_{r,s})}, j_s^{(m_{r,s})})\).

Comparing \cref{midwaydiag} with \cref{basiseldef}, it remains to prove that for any \(i,j \in I\), \(x \in \Z_{\geq 0}\), the morphism space \(\WebAaI(i^{(x)}, j^{(x)})\) is spanned by the diagrams
\begin{align}\label{midway3}
\hackcenter{
\begin{tikzpicture}[scale=0.8]
  \draw[ultra thick, blue] (0,0.3)--(0,0.9);
  \draw[ultra thick, red] (0,0.9)--(0,1.5);
  \draw[thick, fill=yellow]  (0,0.9) circle (9pt);
    \node at (0,0.9) {$\scriptstyle \mu^{\scriptstyle a}$};
     \node[below] at (0,0.3) {$\scriptstyle i^{\scriptstyle(x)}$};
      \node[above] at (0,1.5) {$\scriptstyle j^{\scriptstyle(x)}$};
\end{tikzpicture}}
\end{align}
as \(\mu\) ranges over \({}_j\BasisB_i\)-compositions of \(x\). Indeed, applying the argument that led us to \cref{midwaydiag} in the special case \(t=u=1\), we have that  \(\WebAaI(i^{(x)}, j^{(x)})\) is spanned by diagrams of the form
\begin{align}\label{midway2}
\hackcenter{
\begin{tikzpicture}[scale=0.8]
  \draw[ultra thick, blue] (0,0)--(0,0.1) .. controls ++(0,0.35) and ++(0,-0.35) .. (-0.6,0.6)--(-0.6,0.9);
    \draw[ultra thick, red]  (-0.6,0.9)--(-0.6,1.2) 
  .. controls ++(0,0.35) and ++(0,-0.35) .. (0,1.7)--(0,1.8);
  \draw[ultra thick, blue] (0,0)--(0,0.1) .. controls ++(0,0.35) and ++(0,-0.35) .. (0.6,0.6)--(0.6,0.9);
    \draw[ultra thick, red]  (0.6,0.9)--(0.6,1.2) 
  .. controls ++(0,0.35) and ++(0,-0.35) .. (0,1.7)--(0,1.8);
  .. controls ++(0,0.35) and ++(0,-0.35) .. (0,1.7)--(0,1.8);
   \draw[thick, fill=yellow] (-0.6,0.9) circle (9pt);
     \draw[thick, fill=yellow] (0.6,0.9) circle (9pt);
       \node[above] at (0,1.8) {$\scriptstyle j^{\scriptstyle(x)}$};
         \node[below] at (0,0) {$\scriptstyle i^{\scriptstyle(x)}$};
           \node at (-0.6,0.9) {$\scriptstyle b_{\scriptstyle 1}$};
             \node at (0.6,0.9) {$\scriptstyle b_{\scriptstyle d}$};
             \node[right] at (0.5,0.3) {$\scriptstyle i^{\scriptstyle(\ell_d)}$};
               \node[left] at (-0.5,0.3) {$\scriptstyle i^{\scriptstyle(\ell_1)}$};
                \node at (0,0.9) {$\scriptstyle \cdots $};
\end{tikzpicture}}
\end{align}
where \(\ell_1 + \cdots + \ell_d = x\), \(b_1\leq  \cdots\leq b_d \in {}_j\BasisB_i\). Assume \(z\) is such that \(b_z =b_{z+1} = \cdots = b_{z+k}\) for some \(z, k >0\). Then, at the expense of perhaps introducing a scalar, we may group these into a single strand with the shorthand coupon \(
\hackcenter{
\begin{tikzpicture}[scale=0.8]
  \draw[thick, fill=yellow]  (0,0) circle (8pt);
    \node at (0,0) {$\scriptstyle b_{\scriptstyle z}^{\scriptstyle a}$};
\end{tikzpicture}}
\) as follows.

\begin{itemize} 
\item If \(b_z \in {}_j\basisb_i\), then repeated application of \cref{KnotholeRel} yields a scalar multiple of the coupon 
\(
\hackcenter{
\begin{tikzpicture}[scale=0.8]
  \draw[thick, fill=yellow]  (0,0) circle (8pt);
    \node at (0,0) {$\scriptstyle b_{\scriptstyle z}$};
\end{tikzpicture}}
=
\hackcenter{
\begin{tikzpicture}[scale=0.8]
  \draw[thick, fill=yellow]  (0,0) circle (8pt);
    \node at (0,0) {$\scriptstyle b_{\scriptstyle z}^{\scriptstyle a}$};
\end{tikzpicture}}
\) on a strand of thickness \(\ell_z + \cdots + \ell_{z+k}\).

\item If \(b_z \in ({}_j\BasisB_i)_{\overline 0} \backslash {}_j\basisb_i\), then \(\ell_z = \cdots = \ell_{z+k} = 1\), which yields the shorthand coupon \(
\hackcenter{
\begin{tikzpicture}[scale=0.8]
  \draw[thick, fill=yellow]  (0,0) circle (8pt);
    \node at (0,0) {$\scriptstyle b_{\scriptstyle z}^{\scriptstyle \diamond}$};
\end{tikzpicture}}
=
\hackcenter{
\begin{tikzpicture}[scale=0.8]
  \draw[thick, fill=yellow]  (0,0) circle (8pt);
    \node at (0,0) {$\scriptstyle b_{\scriptstyle z}^{\scriptstyle a}$};
\end{tikzpicture}}
\)
on a strand of thickness \(k+1\).
\item If \(b_z \in ({}_j\BasisB_i)_{\bar 1}\), then the diagram is zero by \cref{OddKnotholeRel}.
\end{itemize}
Once identical basis elements in \cref{midway2} are collected onto single strands as above, we  arrive at (a scalar multiple of) a morphism of the form \cref{midway3} as desired, completing the proof.
\end{proof}


\subsection{Basis results for \texorpdfstring{$\WebAaI$}{WebAaI}} 

\begin{proposition}\label{injhom}
Let \(\bi^{(\bx)} = i_1^{(x_1)} \cdots i_t^{(x_t)}, \bj^{(\by)}=j_1^{(y_1)} \cdots j_u^{(y_u)} \in \textup{Ob}(\WebAaI)\), and assume \(n \geq \lVert \bx\rVert = \lVert \by \rVert\). Then the morphisms \(\{\eta_{\bmu}^{a} \mid \bmu \in \mathcal{M}(\bi^{(\bx)}, \bj^{(\by)})\}\) have linearly independent images under the functor \(G_n\) and, moreover, \(G_n\) defines an injective map
\begin{align*}
G_n: \WebAaI(\bi^{(\bx)}, \bj^{(\by)}) \to \Hom_{\gl_n(A)}(S^{\bx}_{\bi}V, S^{\by}_{\bj}V).
\end{align*}
\end{proposition}
\begin{proof}
Choose some enumeration of tuples
\begin{align*}
\{(b_1, r_1, s_1), \ldots, (b_z, r_z, s_z) \} = \{(b,r,s) \mid r \in [1,t], s \in [1,u], b \in {}_{j_s}\BasisB_{i_r}, 
\mu_{r,s}(b)>0
\}.
\end{align*}
Note that by definition of \(\bmu\) we have \(n \geq z\), so we may define associated elements
\begin{align*}
v^{\bmu} = \bigotimes_{\ell = 1}^t \prod_{m=1}^z (v_m^{i_\ell})^{\delta_{r_m, \ell}
\mu_{\ell, s_m}(b_m)
}
\in 
S^{\bx}_{\bi}V,
\end{align*}
and
\begin{align*}
w^{\bmu} = \bigotimes_{\ell = 1}^u \prod_{m=1}^z (v_m^{b_m})^{\delta_{s_m, \ell}
\mu_{r_m, \ell}(b_m)
}
\in 
S^{\by}_{\bj}V.
\end{align*}

For all \(y \in \Z_{\geq 0}\), \(j \in I\), \(S^y_jV\) has a \(\k\)-basis consisting of restricted monomials in \(v_k^b\), ranging over \(k \in [1,n]\), \(i \in I\), and \(b \in {}_j\BasisB_i\). Here, {\em restricted} means that  the multiplicity of \(v_k^b\) in the monomial is less than or equal to one when \(\bar b = \bar 1\). Fixing such a basis for each \(S^y_jV\), this gives a basis \(W\) of tensor products of monomials for \(S^\by_\bj V\). Up to sign, \(w^{\bmu}\) is an element of \(W\), so we have a map \(p_{\bmu}: S^\by_\bj V \to \k w^{\bmu}\) given by linear projection along \(W\). Now it follows from \cref{explicitmaplemma} that
\begin{align*}
p_{\bmu} \circ G(\eta_{\bnu}^\a)(v^{\bmu}) = 
\pm \delta_{\bmu, \bnu} [\bmu]^!_{\BasisB\backslash \basisb}w^{\bmu},
\end{align*}
for all \(\bmu, \bnu \in \mathcal{M}(\bi^{(\bx)}, \bj^{(\by)})\).
Thus the images under \(G_n\) of morphisms \(\{\eta_{\bmu}^{a} \mid \bmu \in \mathcal{M}(\bi^{(\bx)}, \bi^{(\by)})\}\) are linearly independent.
\end{proof}

The next result follows immediately.
\begin{corollary}\label{AsymFaith}
The functors \(\{G_n \mid n \in \Z_{\geq 0}\}\) are asymptotically locally faithful on \( \WebAaI\). That is, for any fixed morphism space in \(\WebAaI\), the functor \(G_n\) defines an injective map for \(n \gg 0\).
\end{corollary}

\cref{SpanLem,injhom} also yield the following result.
\begin{corollary}\label{BasisThm}
The morphisms \(\{\eta_{\bmu}^a \mid \bmu \in \mathcal{M}(\bi^{(\bx)}, \bj^{(\by)})\}\) comprise a \(\k\)-basis for \( \WebAaI(\bi^{(\bx)}, \bj^{(\by)})\).
\end{corollary}

\subsection{A remark on arbitrary commutative rings}\label{SS:remarkonArbRings}
Recalling that \(\k\) is a characteristic zero domain, let \(\LL\) be an arbitrary commutative ring of characteristic \(p \geq 0\) equipped with ring homomorphism \(\k\to \LL\). Letting \(A_\LL := \LL \otimes_\k A, a_\LL := \LL \otimes_\k a\), we may follow \cref{defwebaa} to construct an \(\LL\)-linear monoidal supercategory \(\Web^{A_\LL, a_\LL}_I\). As all relations are defined integrally, this category is isomorphic to \(\LL \otimes_\k \WebAaI\), and thus the morphism set of \cref{BasisThm} comprises an \(\LL\)-basis for morphism spaces in \(\Web^{A_\LL, a_\LL}_I\).

More generally if \((A', a')_I\) is a good pair of \(\LL\)-superalgebras which are not necessarily a scalar extension of \(\k\)-superalgebras, we may still construct the \(\LL\)-linear monoidal supercategory \(\Web^{A', a'}_I\) as in \cref{defwebaa}. In this setting \cref{SpanLem} holds, but there is now no guarantee of a basis result analogous to \cref{BasisThm}.

\subsection{Good subpairs}

In this subsection we assume we have a good subpair \((A, a')_I \subseteq (A, a)\). There is then an inclusion functor \(\iota_{a',a}: \Web^{A,a'}_I \to \WebAaI\) given by the identity on objects and diagrams.

\begin{lemma}\label{imunder}
For all \(\bi^{(\bx)}, \bj^{(\by)} \in \textup{Ob}(\WebAaI)\) and \(\bmu \in \mathcal{M}( \bi^{(\bx)}, \bj^{(\by)})\), we have
\begin{align*}
\iota_{a', a}(\eta^{a'}_{\bmu}) = [\bmu]^!_{\basisb\backslash \basisb'} \,\eta^a_{\bmu}.
\end{align*}
\end{lemma}
\begin{proof}
Let \(\eta^a_{\bmu} \in \WebAaI(\bi^{(\bx)}, \bj^{(\by)})\), \(\eta^{a'}_{\bmu} \in \Web^{A,a'}_I(\bi^{(\bx)}, \bj^{(\by)})\), be the basis diagrams associated to \(\bmu\) as defined in their respective categories. These diagrams are identical, save for the shorthand coupons on each being of the form \(
\hackcenter{
\begin{tikzpicture}[scale=0.7]
  \draw[thick, fill=yellow]  (0,0) circle (11pt);
    \node at (0,0) {$\scriptstyle b^{\scriptstyle a}$};
\end{tikzpicture}}
\)
and
\(
\hackcenter{
\begin{tikzpicture}[scale=0.7]
  \draw[thick, fill=yellow]  (0,0) circle (11pt);
    \node at (0,0) {$\scriptstyle b^{\scriptstyle a'}$};
\end{tikzpicture}}
\)
respectively.
When \(b \notin \basisb\) or \(b \in \basisb'\), the image under \(\iota_{a',a}\) of the coupon \(
\hackcenter{
\begin{tikzpicture}[scale=0.7]
  \draw[thick, fill=yellow]  (0,0) circle (11pt);
    \node at (0,0) {$\scriptstyle b^{\scriptstyle a'}$};
\end{tikzpicture}}
\)
is
\(
\hackcenter{
\begin{tikzpicture}[scale=0.7]
  \draw[thick, fill=yellow]  (0,0) circle (11pt);
    \node at (0,0) {$\scriptstyle b^{\scriptstyle a}$};
\end{tikzpicture}}
\). When \(b \in \basisb \backslash \basisb'\), the image under \(\iota_{a',a}\) of the coupon \(
\hackcenter{
\begin{tikzpicture}[scale=0.7]
  \draw[thick, fill=yellow]  (0,0) circle (11pt);
    \node at (0,0) {$\scriptstyle b^{\scriptstyle a'}$};
\end{tikzpicture}}
\)
on a strand of thickness \(m\)
is
\(
\hackcenter{
\begin{tikzpicture}[scale=0.7]
  \draw[thick, fill=yellow]  (0,0) circle (11pt);
    \node at (0,0) {$\scriptstyle b^{\scriptstyle \diamond}$};
\end{tikzpicture}}
=
m!
\;
\hackcenter{
\begin{tikzpicture}[scale=0.7]
  \draw[thick, fill=yellow]  (0,0) circle (11pt);
    \node at (0,0) {$\scriptstyle b^{\scriptstyle a}$};
\end{tikzpicture}}
\)
by 
\cref{KnotholeRel}, which implies the result.
\end{proof}

\cref{imunder} and \cref{BasisThm} imply the following result.

\begin{corollary}\label{aprimeembedding}
The functor \(\iota_{a',a}: \Web^{A,a'}_I \to \WebAaI\) is faithful, and for all \(\bi^{(\bx)}, \bj^{(\by)} \in \textup{Ob}(\WebAaI)\), the image of \(\Web^{A,a'}_I(\bi^{(\bx)}, \bj^{(\by)})\) under \(\iota_{a',a}\) is the full-rank \(\k\)-submodule
\begin{equation*}
\k\{[\bmu]_{\basisb \backslash \basisb'}^!\eta_{\bmu}^a \mid \bmu \in \mathcal{M}( \bi^{(\bx)}, \bj^{(\by)})\} \subseteq \WebAaI(\bi^{(\bx)}, \bj^{(\by)}).
\end{equation*}
That is, \(\Web^{A,a'}_I\) may be identified as a \(\k\)-form in \(\WebAaI\).
\end{corollary}

\begin{remark}
In particular, one may use \cref{aprimeembedding} to identify \(\WebAaI\) with the \(\k\)-form in \(\Web^{A,A_{\bar 0}}_I\) spanned by morphisms of the form \([\bmu]_{\BasisB_{\bar 0} \backslash \basisb}^! \eta^{A_{\bar 0}}_{\bmu}\). One may compare this feature with the definition of the Schurification of $(A,a)$ described in \cref{SS:SchurificationofaGoodPair}.
\end{remark}

\subsection{Truncation} Let \((A, a)_I\) be a good pair and assume \(I\) is finite. As \(A, a\) are both unital with identity \(1 = \sum_{i \in I} i\), we may consider alternatively the good pair \((A, a)_{\{1\}}\). Then the category \(\WebAaI\) is an idempotent truncation of \(\Web^{A,a}_{\{1\}}\), in a sense we make clear in \cref{trunc1}.

For \(i \in I, x \in \N\), set
\begin{align*}
\varepsilon_{i^{(x)}} :=
\hackcenter{
\begin{tikzpicture}[scale=.8]
  \draw[ultra thick, black ] (0,0)--(0,0.5);
   \draw[ultra thick, black] (0,0.5)--(0,1);
   \draw[thick, fill=yellow]  (0,0.5) circle (7pt);
    \node at (0,0.5) {$ \scriptstyle i$};
     \node[below] at (0,0) {$ \scriptstyle 1^{ \scriptstyle (x)}$};
      \node[above] at (0,1) {$ \scriptstyle 1^{ \scriptstyle (x)}$};
\end{tikzpicture}}
\in \Web^{A, a}_{\{1\}}(1^{(x)}, 1^{(x)}).
\end{align*}
Now, we define an `\(I\)-truncated' version of \(\Web^{A, a}_{\{1\}}\) as follows. Let \({}_I\Web^{A, a}_{\{1\}}\) be the monoidal supercategory with \(\Ob({}_I\Web^{A, a}_{\{1\}}) = \Lambda_I\), and for \(\bi^{(\bx)} = i_1^{(x_1)}  \cdots i_r^{(x_r)}, \bj^{(\by)} = j_1^{(y_1)} \cdots  j_s^{(y_s)} \in \Ob({}_I\Web^{A, a}_{\{1\}})\), we set:
\begin{align}\label{ItruncHom}
{}_I\Web^{A, a}_{\{1\}}(\bi^{(\bx)}, \bj^{(\by)}) =(\varepsilon_{j_1^{(y_1)}} \otimes \cdots \otimes \varepsilon_{j_s^{(y_s)}}) \circ
\Web^{A, a}_{\{1\}}(1^{(\bx)}, 1^{(\by)})
\circ
 (\varepsilon_{i_1^{(x_1)}} \otimes \cdots \otimes \varepsilon_{i_r^{(x_r)}}),
\end{align}
with composition of morphisms inherited from \(\Web^{A, a}_{\{1\}}\). 

\begin{theorem}\label{trunc1}
There is an isomorphism of monoidal supercategories 
\begin{align*}
\iota_I:  \WebAaI \to {}_I\Web^{A, a}_{\{1\}}
\end{align*}
given by \(i^{(x)} \mapsto i^{(x)}\) on generating objects, and 
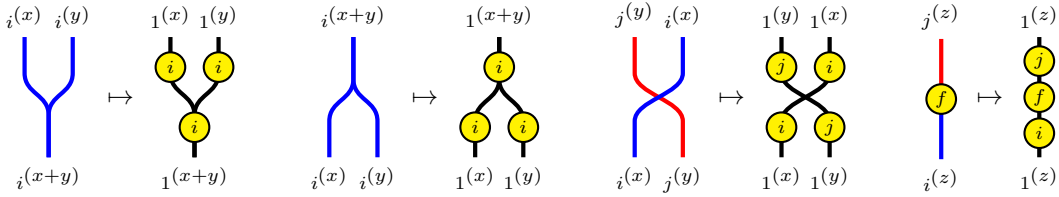
\begin{align*}
\hackcenter{
{}
}
\hackcenter{
\begin{tikzpicture}[scale=.8]
  \draw[ultra thick,blue] (0,-0.5)--(0,0.2) .. controls ++(0,0.35) and ++(0,-0.35) .. (-0.4,0.9)--(-0.4,1.5);
  \draw[ultra thick,blue] (0,-0.5)--(0,0.2) .. controls ++(0,0.35) and ++(0,-0.35) .. (0.4,0.9)--(0.4,1.5);
      \node[above] at (-0.4,1.5) {$ \scriptstyle i^{\scriptstyle (x)}$};
      \node[above] at (0.4,1.5) {$ \scriptstyle i^{\scriptstyle (y)}$};
      \node[below] at (0,-0.5) {$ \scriptstyle i^{\scriptstyle (x+y)} $};
\end{tikzpicture}}
\mapsto
\hackcenter{
\begin{tikzpicture}[scale=.8]
  \draw[ultra thick,black] (0,-0.5)--(0,0.2) .. controls ++(0,0.35) and ++(0,-0.35) .. (-0.4,0.9)--(-0.4,1.5);
  \draw[ultra thick,black] (0,-0.5)--(0,0.2) .. controls ++(0,0.35) and ++(0,-0.35) .. (0.4,0.9)--(0.4,1.5);
      \node[above] at (-0.4,1.5) {$ \scriptstyle 1^{\scriptstyle (x)}$};
      \node[above] at (0.4,1.5) {$ \scriptstyle 1^{\scriptstyle (y)}$};
      \node[below] at (0,-0.5) {$ \scriptstyle 1^{\scriptstyle (x+y)} $};
       \draw[thick, fill=yellow]  (0.4,1) circle (7pt);
        \draw[thick, fill=yellow]  (-0.4,1) circle (7pt);
         \draw[thick, fill=yellow]  (0,0) circle (7pt);
             \node at (0,0) {$ \scriptstyle i$};
                 \node at (0.4,1) {$ \scriptstyle i$};
                     \node at (-0.4,1) {$ \scriptstyle i$};
\end{tikzpicture}}
\qquad
\hackcenter{
\begin{tikzpicture}[scale=.8]
  \draw[ultra thick,blue] (0,0.5)--(0,-0.2) .. controls ++(0,-0.35) and ++(0,0.35) .. (-0.4,-0.9)--(-0.4,-1.5);
  \draw[ultra thick,blue] (0,0.5)--(0,-0.2) .. controls ++(0,-0.35) and ++(0,0.35) .. (0.4,-0.9)--(0.4,-1.5);
      \node[below] at (-0.4,-1.5) {$ \scriptstyle i^{\scriptstyle (x)}$};
      \node[below] at (0.4,-1.5) {$ \scriptstyle i^{\scriptstyle (y)}$};
      \node[above] at (0,0.5) {$ \scriptstyle i^{\scriptstyle (x+y)} $};
\end{tikzpicture}}
\mapsto
\hackcenter{
\begin{tikzpicture}[scale=.8]
  \draw[ultra thick,black] (0,0.5)--(0,-0.2) .. controls ++(0,-0.35) and ++(0,0.35) .. (-0.4,-0.9)--(-0.4,-1.5);
  \draw[ultra thick,black] (0,0.5)--(0,-0.2) .. controls ++(0,-0.35) and ++(0,0.35) .. (0.4,-0.9)--(0.4,-1.5);
      \node[below] at (-0.4,-1.5) {$ \scriptstyle 1^{\scriptstyle (x)}$};
      \node[below] at (0.4,-1.5) {$ \scriptstyle 1^{\scriptstyle (y)}$};
      \node[above] at (0,0.5) {$ \scriptstyle 1^{\scriptstyle (x+y)} $};
       \draw[thick, fill=yellow]  (0.4,-1) circle (7pt);
        \draw[thick, fill=yellow]  (-0.4,-1) circle (7pt);
         \draw[thick, fill=yellow]  (0,0) circle (7pt);
             \node at (0,0) {$ \scriptstyle i$};
                 \node at (0.4,-1) {$ \scriptstyle i$};
                     \node at (-0.4,-1) {$ \scriptstyle i$};
\end{tikzpicture}}
\qquad
\hackcenter{
\begin{tikzpicture}[scale=.8]
  \draw[ultra thick,red] (0.4,-0.5)--(0.4,0.1) .. controls ++(0,0.35) and ++(0,-0.35) .. (-0.4,0.9)--(-0.4,1.5);
  \draw[ultra thick,blue] (-0.4,-0.50)--(-0.4,0.1) .. controls ++(0,0.35) and ++(0,-0.35) .. (0.4,0.9)--(0.4,1.5);
      \node[above] at (-0.4,1.5) {$ \scriptstyle j^{ \scriptstyle (y)}$};
      \node[above] at (0.4,1.5) {$ \scriptstyle i^{ \scriptstyle (x)}$};
       \node[below] at (-0.4,-0.5) {$ \scriptstyle i^{ \scriptstyle (x)}$};
      \node[below] at (0.4,-0.5) {$ \scriptstyle j^{ \scriptstyle (y)}$};
\end{tikzpicture}}
\mapsto
\hackcenter{
\begin{tikzpicture}[scale=.8]
  \draw[ultra thick,black] (0.4,-0.5)--(0.4,0.1) .. controls ++(0,0.35) and ++(0,-0.35) .. (-0.4,0.9)--(-0.4,1.5);
  \draw[ultra thick,black] (-0.4,-0.50)--(-0.4,0.1) .. controls ++(0,0.35) and ++(0,-0.35) .. (0.4,0.9)--(0.4,1.5);
      \node[above] at (-0.4,1.5) {$ \scriptstyle 1^{ \scriptstyle (y)}$};
      \node[above] at (0.4,1.5) {$ \scriptstyle 1^{ \scriptstyle (x)}$};
       \node[below] at (-0.4,-0.5) {$ \scriptstyle 1^{ \scriptstyle (x)}$};
      \node[below] at (0.4,-0.5) {$ \scriptstyle 1^{ \scriptstyle (y)}$};
             \draw[thick, fill=yellow]  (0.4,1) circle (7pt);
        \draw[thick, fill=yellow]  (-0.4,1) circle (7pt);
                     \draw[thick, fill=yellow]  (0.4,0) circle (7pt);
        \draw[thick, fill=yellow]  (-0.4,0) circle (7pt);
                   \node at (0.4,0) {$ \scriptstyle j$};
                     \node at (-0.4,0) {$ \scriptstyle i$};
                                \node at (0.4,1) {$ \scriptstyle i$};
                     \node at (-0.4,1) {$ \scriptstyle j$};
\end{tikzpicture}}
\qquad
\hackcenter{
\begin{tikzpicture}[scale=.8]
  \draw[ultra thick, blue] (0,-0.5)--(0,0.5);
   \draw[ultra thick, red] (0,0.5)--(0,1.5);
   \draw[thick, fill=yellow]  (0,0.5) circle (7pt);
    \node at (0,0.5) {$ \scriptstyle f$};
     \node[below] at (0,-0.5) {$ \scriptstyle i^{ \scriptstyle (z)}$};
      \node[above] at (0,1.5) {$ \scriptstyle j^{ \scriptstyle (z)}$};
\end{tikzpicture}}
\mapsto
\hackcenter{
\begin{tikzpicture}[scale=.8]
  \draw[ultra thick, black] (0,-0.5)--(0,0.5);
   \draw[ultra thick, black] (0,0.5)--(0,1.5);
   \draw[thick, fill=yellow]  (0,0.5) circle (7pt);
    \node at (0,0.5) {$ \scriptstyle f$};
       \draw[thick, fill=yellow]  (0,1.1) circle (7pt);
    \node at (0,1.1) {$ \scriptstyle j$};
       \draw[thick, fill=yellow]  (0,-0.1) circle (7pt);
    \node at (0,-0.1) {$ \scriptstyle i$};
     \node[below] at (0,-0.5) {$ \scriptstyle 1^{ \scriptstyle (z)}$};
      \node[above] at (0,1.5) {$ \scriptstyle 1^{ \scriptstyle (z)}$};
\end{tikzpicture}}
\end{align*}
on generating morphisms.
\end{theorem}
\begin{proof}
It is straightforward to check that \(\iota_I\) is a well-defined functor. In consideration of \cref{ItruncHom} and \cref{BasisThm} as applied to \(\WebAaI\) and \(\Web_{\{1\}}^{A,a}\), it is also straightforward to check that \(\iota_I\) induces a bijection on \(\k\)-bases for morphism spaces, completing the proof.
\end{proof}

\section{Connection to Schurification}\label{S:webificationofSchurification}

\subsection{The algebra \texorpdfstring{$\WAand$}{WAand}}\label{SS:CategoricalSchurificationofaGoodPair}
We have come to one of our primary motivations: the connection between \(\WebAaI\) and the generalized Schur algebra (or, `Schurification') $T^{A}_{a}(n,d)$ defined in \cite{KM}. Throughout this section, we assume that \((A,a)_I\) is a good pair, and \(I\) is a finite set. We write \(1 = \sum_{i \in I}i\) for the unit in \(A\) and \(a\). 
For \(\bx = (x_1, \ldots, x_t) \in \Z_{\geq 0}^t\) we write \(1^{(\bx)}:=1^{(x_1)} \cdots 1^{(x_t)} \in \Omega_{\{1\}}\).

For \(n,d \in \Z_{\geq0}, \bx, \by \in \Omega(n,d)\), we set \(\WAand(\bx, \by) =  \WebAaOne(1^{(\bx)}, 1^{(\by)})\). Then we set 
\begin{align}\label{WAandDef}
\WAand := \bigoplus_{\bx, \by \in \Omega(n,d)} \WAand(\bx, \by)
\end{align}
and view \(\WAand\) as a unital superalgebra, as follows. For \(\bx, \bx', \by, \by' \in \Omega(n,d)\), \(f \in \WAand(\bx, \by), f' \in \WAand(\bx', \by')\), we set \(ff' = \delta_{\bx, \by'} f \circ f' \in \WAand(\bx', \by)\), where the composition is taken in \(\WebAaOne\). 
The unit in \(\WAand\) is 
\(
\mathbbm{1} = \sum_{\bx \in \Omega(n,d)} \textup{Id}_{1^{(\bx)}}.
\)
By \cref{BasisThm} \(\WAand(\bx,\by)\) has \(\k\)-basis \(\{\eta^a_{\bmu} \mid \bmu \in \mathcal{M}(1^{(\bx)}, 1^{(\by)})\}\).

\subsection{The Schurification of  \texorpdfstring{$(A, a)$}{(A,a)}}\label{SS:SchurificationofaGoodPair}
In this section we follow \cite{KM} to define the generalized Schur algebra \(T^{A}_a(n,d)\) and refer the interested reader to that paper for more detail. The tensor product \(\k\)-superalgebra \(M_n(A)^{\otimes d}\) has a signed permutation action of the symmetric group \(\mathfrak{S}_d\) as described in \cref{twistsec}. We write
\begin{align*}
S^A(n,d) := (M_n(A)^{\otimes d})^{\mathfrak{S}_d}
\end{align*}
for the sub-superalgebra of \(\mathfrak{S}_d\)-invariants. \\

Let \(d_1, d_2 \in \Z_{\geq 0}\), with \(\xi_k \in S^A(n,d_k)\) for \(k =1,2\). Let \(\mathscr{D}^{d_1,d_2}\) be a set of right coset representatives for \((\mathfrak{S}_{d_1} \times \mathfrak{S}_{d_2}) \backslash \mathfrak{S}_{d_1 + d_2}\). Then we define the associative {\em star product}:
\begin{align*}
\xi_1 * \xi_2 := \sum_{\sigma \in \mathscr{D}^{d_1,d_2}} (\xi_1 \otimes \xi_2)^\sigma \in S^A(n,d_1 + d_2).
\end{align*}
If \(U = \{u_1<  \ldots< u_t\}\) is a finite indexing set equipped with total order \(<\), we denote 
\begin{align*}
\scalebox{2}{\textasteriskcentered}^<_{u \in U} \xi_u  = \xi_{u_1} * \cdots * \xi_{u_t}.
\end{align*}
We have `left' and `right'-sided injective algebra homomorphisms \({\tt L}_n^{n+m}, {\tt R}_n^{n+m}: M_n(A) \to M_{n+m}(A)\) given by \({\tt L}_n^{n+m}(E^b_{r,s}) = E^b_{r,s}\) and \({\tt R}_n^{n+m}(E^b_{r,s}) = E^b_{r+m,s+m}\), which induce algebra homomorphisms  \({\tt L}_n^{n+m}, {\tt R}_n^{n+m}: S^A(n,d) \to S^A(n+m,d)\). For all \(\xi_1 \in S^A(n_1, d_1)\), \(\xi_2 \in S^A(n_2, d_2)\), we define the {\em shifted} star product:
\begin{align*}
\xi_1 \bar{*} \,\xi_2 := {\tt L}_{n_1}^{n_1+n_2}(\xi_1) * {\tt R}_{n_2}^{n_1 + n_2}(\xi_2) \in S^A(n_1 + n_2, d_1 + d_2).
\end{align*}

For \(\bx, \by \in \Omega(n,d)\),  \(\bmu \in \mathcal{M}(1^{(\bx)}, 1^{(\bx)})\) we set
\begin{align*}
\tilde{\eta}_{\bmu}^{a} :=  [\bmu]^!_{\BasisB \backslash \basisb} \left(
\scalebox{2}{\textasteriskcentered}_{r=1}^n
\scalebox{2}{\textasteriskcentered}_{s=1}^n
\scalebox{2}{\textasteriskcentered}_{b \in \BasisB}^<
(E_{s,r}^b)^{\otimes \mu_{r,s}(b)}
  \right)
  \in S^A(n,d),
  \end{align*}
viewing $(E_{s,r}^b)^{\otimes \mu_{r,s}(b)}\in S^A(n,\mu_{r,s}(b))$.  While our notation in this paper is slightly different, it is straightforward to check that if \(\bb = (b_1, \ldots, b_d) \in \BasisB^d, \br = (r_1, \ldots, r_d), \bs = (s_1, \ldots, s_d) \in [1,n]^d\) are such that 
 \begin{align*}
 \#\{ t \in [1,d] \mid b_t = b, r_t = r, s_t = s\} = \mu_{r,s}(b), \qquad \textup{ for all } r,s \in [1,n], b \in \BasisB,
 \end{align*}
 then we have that \(\tilde{\eta}^a_{\bmu}\) is equal to \(\pm\eta_{\bb, \br, \bs}\), as defined in \cite[\S3B]{KM}. Then we define
  \begin{align*}
T^A_a(n,d) := \k \{ \tilde{\eta}_{\bmu}^{a} \mid \bx, \by \in \Omega(n,d), \bmu \in \mathcal{M}(1^{(\bx)}, 1^{(\by)})\} \subseteq S^A(n,d).
  \end{align*}

\begin{lemma}\cite[Lemma 3.10, Proposition 3.12]{KM}\label{KMbasis}
 \(T^A_a(n,d)\) is a unital sub-superalgebra of \(S^A(n,d)\), with \(\k\)-basis 
\(
 \{ \tilde{\eta}_{\bmu}^{a} \mid \bx, \by \in \Omega(n,d), \bmu \in \mathcal{M}(1^{(\bx)}, 1^{(\by)})\} \).
\end{lemma}

\subsection{Connecting algebras}
For fixed $n \geq 1$, set \(V_n=A^{\oplus n}\).  Viewing this as column vectors of height $n$ with entries in $A$, we consider this as a left \(M_n(A)\)-supermodule via matrix multiplication.

\begin{lemma}\label{WebVdaction}
There is a well-defined \(\k\)-superalgebra map \(\tau:\WAand \to \End_\k\left(V_n^{\otimes d}\right)\) given by 
\begin{align}\label{Vdaction}
\tau(\eta_{\bmu}^{a}) :\bigotimes_{\gamma = 1}^d v_{r_\gamma}^{f_\gamma} \mapsto
[\bmu]^!_{\BasisB \backslash \basisb}
\sum_{T} (-1)^{\sigma(T)} \bigotimes_{\gamma = 1}^d v_{s_\gamma^T}^{b_\gamma^T f_\gamma}.
\end{align}
for all \(\bx, \by \in \Omega(n,d)\), \(\bmu \in \mathcal{M}(1^{(\bx)}, 1^{(\by)})\). 
Here the sum ranges over all \(T = (T_{r,s}^b)\) such that \(T_{r,s}^b \subseteq \{\gamma \in [1,d] \mid r_\gamma = r\}\), $\bigsqcup_{b,s}T^b_{r,s}=\{\gamma\in [1,d]\:|\:r_{\gamma}=r \}$, and \(|T_{r,s}^b| = \mu_{r,s}(b)\). For \(\gamma \in [1,d]\), we let \(s_\gamma^T \in [1,n]\), \(b_\gamma^T \in \BasisB\) be the unique elements such that \(\gamma \in T_{r_\gamma, s_\gamma^T}^{b_\gamma^T}\). Finally, the sign is given by
\begin{align*}
\sigma(T) = \sum_{1 \leq \gamma < \gamma' \leq d} \bar b_{\gamma'}^T \bar f_\gamma
+
 \sum_{\substack{ 1 \leq \gamma < \gamma' \leq d \\  r_{\gamma '} < r_{\gamma}   }}
 \bar b_\gamma^T \bar b_{\gamma'}^T
 +
  \sum_{\substack{ 1 \leq \gamma < \gamma' \leq d \\  r_{\gamma '} = r_{\gamma}  \\ s_{\gamma '}^T< s_{\gamma}^T}}
   \bar b_\gamma^T \bar b_{\gamma'}^T
   +
     \sum_{\substack{ 1 \leq \gamma < \gamma' \leq d \\  r_{\gamma '} = r_{\gamma}  \\ s_{\gamma '}^T= s_{\gamma}^T \\ b_{\gamma '}^T < b_\gamma^T}}
   \bar b_\gamma^T \bar b_{\gamma'}^T.
 \end{align*}

\end{lemma}

\begin{proof}
There is an injective \(\k\)-module homomorphism:
\begin{align}\label{Vdsymisom}
\chi: 
V_n^{\otimes d} \to \bigoplus_{\bx  \in \Omega(n,d)} S^{\bx}V_n, \qquad
\bigotimes_{\gamma = 1}^d v_{r_\gamma}^{f_\gamma} \mapsto
(-1)^{\langle \br, \bbbf\rangle}
 \bigotimes_{\alpha=1}^n \prod_{\substack{\gamma \in [1,d] \\ r_\gamma = \alpha }}^< v_\gamma^{f_\gamma} \in S^{\by}V_n,
\end{align}
where \(\br = (r_1, \ldots, r_d) \in [1,n]^d\), \(\bbbf = (f_1, \ldots, f_d) \in A^d\), \(y_{\alpha} = \#\{\gamma \mid r_\gamma = \alpha\}\), and
\begin{align*}
\langle \br, \bbbf \rangle := \sum_{\substack{1 \leq \gamma < \gamma' \leq d \\ r_{\gamma '}< r_{\gamma}}} \bar f_\gamma \bar f_{\gamma'}.
\end{align*}
Here and below the symbol \(\prod^<_{u \in U}\) indicates that the product should be taken in increasing order in accordance with the order \(<\) on the set \(U\). 

Then we have 
\begin{align}\label{uglyeqs1}
G_n(\eta_{\bmu}^a) \circ &\chi \left( \bigotimes_{\gamma = 1}^d v_{r_\gamma}^{f_\gamma}\right)
=
(-1)^{\langle \br, \bbbf\rangle} G_n( \eta_{\bmu}^a)
\left( \bigotimes_{\alpha=1}^n \prod_{\substack{\gamma \in [1,d] \\ r_\gamma = \alpha }}^< v_\gamma^{f_\gamma}
\right)\\
&= [\bmu]^!_{\BasisB \backslash \basisb}(-1)^{\langle \br, \bbbf\rangle} \sum_T 
(-1)^{\sigma_1'(T) + \sigma_2'(T) + \sigma_3'(T)}
\bigotimes_{s=1}^n \prod_{r=1}^n \prod_{b \in \BasisB}^< \prod_{\gamma \in T_{r,s}^b}^< v_\gamma^{b^T_\gamma f_\gamma}
\label{uglyeqs2}
\\
&=
[\bmu]^!_{\BasisB \backslash \basisb}
(-1)^{\langle \br, \bbbf\rangle} \sum_T 
(-1)^{\sigma_1'(T) + \sigma_2'(T) + \sigma_3'(T)+ \sigma_4'(T)}
\bigotimes_{s=1}^n \prod_{\substack{\gamma \in [1,d] \\ s_\gamma^T = s}}^<
v_\gamma^{b^T_\gamma f_\gamma}
\label{uglyeqs3}
\\
&=
[\bmu]^!_{\BasisB \backslash \basisb}
(-1)^{\langle \br, \bbbf\rangle}
\sum_T
(-1)^{\sigma_1'(T) + \sigma_2'(T) + \sigma_3'(T)+ \sigma_4'(T) + \langle (s_1^T, \ldots, s_d^T),(b_1^T f_1, \ldots, b_d^Tf_d) \rangle }
\chi\left(
 \bigotimes_{\gamma =1}^d v_{s_\gamma^T}^{b_\gamma^T f_\gamma}
\right).\label{uglyeqs4}
\end{align}
In the above expressions,
\begin{align*}
\sigma_1'(T) &= \sum_{\substack{
1 \leq \gamma < \gamma' \leq d \\
r_\gamma = r_{\gamma'}\\
s_{\gamma'} < s_\gamma
}}
\bar f_\gamma \bar f_{\gamma '}
+
\sum_{\substack{
1 \leq \gamma < \gamma' \leq d \\
r_\gamma = r_{\gamma'}\\
s_{\gamma}^T = s_{\gamma'}^T\\
b_{\gamma '}^T < b_\gamma^T
}}
\bar f_\gamma \bar f_{\gamma '},
\\
\sigma_2'(T) &=
\sum_{\substack{
\gamma, \gamma' \in [1,d] \\
r_\gamma < r_{\gamma'}
}}
\bar b_{\gamma '}^T \bar f_\gamma
+
\sum_{\substack{
\gamma, \gamma' \in [1,d]\\
r_\gamma = r_{\gamma'}\\
s_\gamma^T < s_{\gamma '}^T
}}
\bar b_{\gamma '}^T \bar f_\gamma
+
\sum_{\substack{
\gamma, \gamma' \in [1,d]\\
r_\gamma = r_{\gamma '}\\
s_\gamma^T = s_{\gamma '}^T\\
b_\gamma^T < b_{\gamma '}^T
}}
\bar b_{\gamma '}^T \bar f_\gamma,\\
\sigma_3'(T) &=
\sum_{\substack{
\gamma, \gamma' \in [1,d]\\
r_\gamma < r_{\gamma '}\\
s_{\gamma '}^T < s_\gamma^T
}}
( \bar b_\gamma^T + \bar f_\gamma)(\bar b_{\gamma '}^T + \bar f_{\gamma'}),\\
\sigma_4'(T) &= 
\sum_{\substack{
1 \leq \gamma < \gamma' \leq d \\
s_\gamma^T = s_{\gamma '}^T\\
r_{\gamma '} < r_\gamma
}}
( \bar b_\gamma^T + \bar f_\gamma)(\bar b_{\gamma '}^T+ \bar f_{\gamma'}).
\end{align*}
The equality \cref{uglyeqs2} follows from \cref{Getaformula}, where \(\sigma_1'(T), \sigma_2'(T), \sigma_3'(T)\) are the values \(\sigma_1(T), \sigma_2(T), \sigma_3(T)\) reindexed in terms of \(\gamma \in [1,d]\). The equality \cref{uglyeqs3} follows from reordering the monomials in each tensor factor so that subscript indices are in increasing order, which incurs the sign associated with \(\sigma_4'(T)\). The equality \cref{uglyeqs4} is a direct application of \cref{Vdsymisom}.
It is straightforward, if tedious, to verify that for all \(T\) we have
\begin{align*}
\langle \br, \bbbf \rangle + \sigma_1'(T) + \sigma_2'(T) + \sigma_3'(T) + \sigma_4'(T) +
\langle (s_1^T, \ldots, s_d^T),(b_1^T f_1, \ldots, b_d^Tf_d) \rangle 
+ \sigma(T) 
=
\sum_{\substack{
1 \leq \gamma < \gamma ' \leq d\\
r_\gamma = r_{\gamma '}\\
s_\gamma^T = s_{\gamma'}^T\\
b_\gamma^T = b_{\gamma '}^T
}}
b_{\gamma '}^T \bar f_\gamma.
\end{align*}
But by the definition of \(\bmu\), if \(r_\gamma = r_\gamma'\), \(s_\gamma^T = s_{\gamma'}^T\), \(b_\gamma^T = b_{\gamma '}^T\), then we must have \(\bar b_{\gamma '}^T = \bar 0\). Therefore
\begin{align*}
\langle \br, \bbbf \rangle + \sigma_1'(T) + \sigma_2'(T) + \sigma_3'(T) + \sigma_4'(T) +
\langle (s_1^T, \ldots, s_d^T),(b_1^T f_1, \ldots, b_d^Tf_d) \rangle 
=\sigma(T).
\end{align*}
Writing \(\chi^{-1}\) for the inverse map \(\chi(V_n^{\otimes d}) \to V_n^{\otimes d}\), it follows then that taking \(\tau(f):=  \chi^{-1} \circ G_n(f) \circ \chi\) for any nonzero \(f \in \WAand\) satisfies the claim.
\end{proof}

\begin{lemma}\label{TVdactionlem}
The left multiplication action of \(T^A_a(n,d)\) on \(V_n^{\otimes d}\) gives a \(\k\)-superalgebra map \(\tilde{\tau}: T^A_a(n,d) \to \End_\k(V_n^{\otimes d})\), with
\begin{align}\label{TVdaction}
\tilde{\tau}(\tilde{\eta}_{\bmu}^{a}) :\bigotimes_{\gamma = 1}^d v_{r_\gamma}^{f_\gamma} \mapsto
[\bmu]^!_{\BasisB \backslash \basisb}
\sum_{T} (-1)^{\sigma(T)} \bigotimes_{\gamma = 1}^d v_{s_\gamma^T}^{b_\gamma^T f_\gamma}.
\end{align}
for all \(\bx, \by \in \Omega(n,d)\), \(\bmu \in \mathcal{M}(1^{(\bx)}, 1^{(\by)})\), where \(T\) and \(\sigma(T)\) are as in \cref{WebVdaction}.
\end{lemma}
\begin{proof}
We have
\begin{align}
\tilde{\tau}(\tilde{\eta}_{\bmu}^{a})\left(\bigotimes_{\gamma = 1}^d v_{r_\gamma}^{f_\gamma}  \right)
&= [\bmu]^!_{\BasisB \backslash \basisb} \left(
\scalebox{2}{\textasteriskcentered}_{r=1}^n
\scalebox{2}{\textasteriskcentered}_{s=1}^n
\scalebox{2}{\textasteriskcentered}_{b \in \BasisB}^<
(E_{s,r}^b)^{\otimes \mu_{r,s}(b)}
  \right)
\cdot
\bigotimes_{\gamma = 1}^d v_{r_\gamma}^{f_\gamma}
\\
  &=
  [\bmu]^!_{\BasisB \backslash \basisb} \left(
  \sum_T
  (-1)^{\tilde{\sigma}_1(T)}
  \bigotimes_{\gamma=1}^d
  E_{s_\gamma^T, r_\gamma}^{b_\gamma^T}
  \right)
  \cdot 
  \bigotimes_{\gamma = 1}^d v_{r_\gamma}^{f_\gamma} \label{uglyeq5}\\
  &= 
    [\bmu]^!_{\BasisB \backslash \basisb} 
  \sum_T
  (-1)^{\tilde{\sigma}_1(T) + \tilde{\sigma}_2(T)} v_{s_\gamma^T}^{b_\gamma^T f_\gamma},\label{uglyeq6}
\end{align}
where
\begin{align*}
\tilde{\sigma}_1(T) &=  \sum_{\substack{ 1 \leq \gamma < \gamma' \leq d \\  r_{\gamma '} < r_{\gamma}   }}
 \bar b_\gamma^T \bar b_{\gamma'}^T
 +
  \sum_{\substack{ 1 \leq \gamma < \gamma' \leq d \\  r_{\gamma '} = r_{\gamma}  \\ s_{\gamma '}^T< s_{\gamma}^T}}
   \bar b_\gamma^T \bar b_{\gamma'}^T
   +
     \sum_{\substack{ 1 \leq \gamma < \gamma' \leq d \\  r_{\gamma '} = r_{\gamma}  \\ s_{\gamma '}^T= s_{\gamma}^T \\. b_{\gamma '}^T < b_\gamma^T}}
   \bar b_\gamma^T \bar b_{\gamma'}^T\\
   \tilde{\sigma}_2(T) &=
   \sum_{1 \leq \gamma < \gamma' \leq d} \bar b_{\gamma'}^T \bar f_\gamma.
\end{align*}
The equality \cref{uglyeq5} follows by deleting all terms in
\(\scalebox{2}{\textasteriskcentered}_{r=1}^n
\scalebox{2}{\textasteriskcentered}_{s=1}^n
\scalebox{2}{\textasteriskcentered}_{b \in \BasisB}^<
(E_{s,r}^b)^{\otimes \mu_{r,s}(b)}\) which act as zero on \( \bigotimes_{\gamma = 1}^d v_{r_\gamma}^{f_\gamma} \). Noting that \(\tilde{\sigma}_1(T) + \tilde{\sigma}_2(T) = \sigma(T)\), the result follows.
\end{proof}

\begin{theorem}\label{sameSchur}
We have an isomorphism of superalgebras
\begin{align*}
w:\WAand \xrightarrow{\sim} T^A_a(n,d),
\qquad
\eta_{\bmu}^a \mapsto \tilde{\eta}_{\bmu}^a,
\end{align*}
for all \(\bx, \by \in \Omega(n,d)\), \(\bmu \in \mathcal{M}(1^{(\bx)}, 1^{(\by)})\).
\end{theorem}
\begin{proof}
By \cref{WebVdaction,TVdactionlem}, we have \(\tau(\eta_{\bmu}^\a) = \tilde{\tau}(\tilde{\eta}_{\bmu}^a)\) for all \(\bmu\), so it follows from \cref{BasisThm} and \cref{KMbasis} that \(\tau(\WAand ) = \tilde{\tau}(T^A_a(n,d))\).
We have an inclusion \(\iota: T^A_a(n,d) \hookrightarrow S^A(n,d)\), and a map \(\tilde{\tau}': S^A(n,d) \to \End_\k(V_n^{\otimes d})\) induced by the left action of \(S^{A}(n,d)\) on \(V_n^{\otimes d}\). By \cite[Lemma 5.7]{EK}, \(\tilde{\tau}'\) is injective. Thus \(\tilde{\tau} = \tilde{\tau}' \circ \iota\) is injective, so it follows that \(\tilde{\tau}( T^A_a(n,d)) \cong T^A_a(n,d)\). Writing \(\tilde{\tau}^{-1}\) for the inverse superalgebra map \(\tilde{\tau}( T^A_a(n,d)) \xrightarrow{\sim} T^A_a(n,d)\), we thus have a superalgebra map \(\tilde{\tau}^{-1} \circ \tau: \WAand  \to T_a^A(n,d)\) which gives a bijection \(\eta_{\bmu}^a \mapsto \tilde{\eta}_{\bmu}^a\) on bases, proving the result.
\end{proof}

There is a natural right action of the wreath product superalgebra $\SS_{d} \wr A$  on \(V_n^{\otimes d}\), where \(A\) acts on \(V_n = A^{\oplus n}\) diagonally and \(\mathfrak{S}_d\) acts by permutation of tensor factors.

\begin{corollary}
We have
\begin{align*}
W^{A, A_{\bar 0}}_{n,d} \cong S^A(n,d) \cong \End_{ \SS_{d} \wr A }(V_n^{\otimes d}).
\end{align*}
\end{corollary}
\begin{proof}
The left isomorphism follows by \cref{sameSchur}, since \(T^A_{A_{\bar 0}}(n,d) =S^A(n,d)\), as explained in \cite[Remark 3.17]{KM}. The right isomorphism is due to \cite[Lemma 5.7]{EK}.
\end{proof}

\subsection{Monoidal equivalence}
Recall that \((A,a)_I\) is a good pair, with \(I\) finite and \(1 = \sum_{i \in I}i\). For any \(\bi^{(\bx)} = i_1^{(x_1)} \cdots i_n^{(x_n)} \in \Omega_I(n,d)\), define
\begin{align}\label{ixidems}
\xi_{\bi^{(\bx)}} = (E^{i_1}_{1,1})^{\otimes x_1} * \cdots *(E^{i_n}_{n,n})^{\otimes x_n} \in T^A_a(n,d).
\end{align}
We remark that, if \(\bmu \in \mathcal{M}(\bi^{(\bx)}, \bi^{(\bx)})\) is defined by \(\mu_{r,s}(b) = \sum_{k=1}^n \delta_{r,k}\delta_{s,k} \delta_{b,i_k} x_k\) for all \(r,s,b\), then it follows that \(\xi_{\bi^{(\bx)}} = \tilde{\eta}_{\bmu}^a\). The set \(\{ \xi_{\bi^{(\bx)}} \mid \bi^{(\bx)} \in \Omega_I(n,d)\}\) is a collection of orthogonal idempotents in \(T^A_a(n,d)\), and we have \(1_{T^A_a(n,d)} = \sum_{\bi^{(\bx)} \in \Omega_{I}(n,d)} \xi_{\bi^{(\bx)}}\)  if \(I = \{1\}\).

\subsubsection{The category \texorpdfstring{$\TTAaI$}{TTAaI}} We define a monoidal supercategory \(\TTAaI\) as follows.
Set \(\Ob(\TTAa) = \Omega_I\), with the monoidal structure on objects inherited from \(\Omega_I\). For \(\bi^{(\bx)}, \bj^{( \by)} \in \Omega_I\), define morphism spaces by 
\begin{align*}
\TTAaI(\bi^{(\bx)},\bj^{(\by)}) = \xi_{\bj^{(\by)}} T^A_a(n,d) \xi_{\bi^{(\bx)}} \subseteq T^A_a(n,d)
\end{align*}
provided \(\bi^{(\bx)}, \bj^{( \by)} \in \Omega_I(n,d)\) for some \(n,d\), and set \(\TTAaI(\bi^{(\bx)},\bj^{(\by)}) = 0\) otherwise.
Composition of morphisms is given by multiplication in \(T^A_a(n,d)\), and the monoidal structure on morphisms is given by the shifted star product \(\bar{*}\). It is straightforward to check (see for instance \cite[\S5C]{KM}) that this yields a well-defined monoidal supercategory.

\begin{remark}
If \(I =\{1\}\), then, considered as algebras, we have 
\begin{align*}
\tilde{\TT}^{A,a}_{\{1\}} = \bigoplus_{1^{(\bx)}, 1^{(\by)} \in \Omega_{\{1\}}(n,d)} \xi_{1^{(\by)}} T^A_a(n,d) \xi_{1^{(\bx)}} = \bigoplus_{n,d \geq 0} T^A_a(n,d),
\end{align*}
so \(\TTAaOne\) serves as an umbrella for all of the Schurifications of \((A,a)\). More generally, \(\tilde{\TT}^{A,a}_{I}\) is an idempotent truncation of \(\bigoplus_{n,d \geq 0} T^A_a(n,d)\).
\end{remark}

\begin{theorem}\label{moneq1}
We have an equivalence of monoidal supercategories
\begin{align*}
\chi:\TTAaI \to \WebAaI
\end{align*}
given by the identity on objects, and \(\tilde{\eta}_{\bmu} \mapsto \eta_{\bmu}\) on morphisms.
\end{theorem}
\begin{proof}
The isomorphisms \(w\) of \cref{sameSchur} and \(\iota_I\) of \cref{trunc1} yield an isomorphism
\begin{align*}
\TTAaOne(\bi^{(\bx)}, \bj^{(\by)}) = \tilde{\eta}_{\bj^{(\by)}} T^A_a(n,d)\tilde{\eta}_{\bi^{(\bx)}} \xrightarrow{w^{-1}} 
\eta_{\bj^{(\by)}} W^{A,a}_{n,d}\eta_{\bi^{(\bx)}} =  {}_I\WebAaOne(\bi^{(\bx)}, \bj^{(\by)})
\xrightarrow{\iota_I^{-1}} \WebAaI(\bi^{(\bx)}, \bj^{(\by)})
\end{align*}
for every \(\bi^{(\bx)}, \bj^{(\by)} \in \Omega_I\) which sends \(\tilde{\eta}_{\bmu} \mapsto \eta_{\bmu}\), so the functor \(\chi\) is well-defined and fully faithful. That \(\chi\) respects monoidal structure follows from the following observation. Letting \(\bmu \in \mathcal{M}(\bi^{(\bx)}, \bj^{(\by)}), \bnu \in \mathcal{M}(\bk^{(\bz)}, \bm^{(\bw)})\), for some \(\bi^{(\bx)}, \bj^{(\by)}, \in \Omega_I(n_1,d_1)\), \(\bk^{(\bz)}, \bm^{(\bw)} \in \Omega_I(n_2,d_2)\), define \(\bom \in \mathcal{M}(\bi^{(\bx)} \bk^{(\bz)}, \bj^{(\by)} \bm^{(\bw)})\) by setting \(\omega_{r,s}(b) = \mu_{r,s}(b)\) for \(r,s \in [1,n_1]\), and \(\omega_{r+n_1,s+n_1}(b) = \nu_{r,s}(b)\) for \(r,s \in [1,n_2]\). Then it is straightforward to check that 
 \(\tilde{\eta}_{\bmu} \bar{*} \, \tilde{\eta}_{\bnu} = \tilde{\eta}_{\bom}\) and \(\eta_{\bmu} \otimes \eta_{\bnu} = \eta_{\bom}\), yielding the result.
\end{proof}

\section{The wreath category}\label{S:wreathcategory}

\subsection{The wreath category, I}\label{SS:wreathcategory}  Now, we return to the assumption that \((A,a)_I\) is a good pair, where \(I\) is not necessarily finite.
\begin{definition}\label{defwr}
The {\em wreath category} \(\WrAI\) is the \(\k\)-linear strict monoidal supercategory defined as follows.  The objects are monoidally generated by objects \(i\) for \(i \in I\). Then we have
 \begin{align*}
 \textup{Ob}(\WrAI) = \left\{\bi:= i_1 \otimes \cdots \otimes i_t \mid t \in \Z_{\geq 0}, \bi =( i_1, \ldots, i_t) \in I^t \right\}.
 \end{align*}
The generating morphisms of \(\WrA\) are given by the diagrams:
\begin{align}\label{WrAGens}
\hackcenter{
}
\hackcenter{
\begin{tikzpicture}[scale=.8]
  \draw[ultra thick,red] (0.4,0)--(0.4,0.1) .. controls ++(0,0.35) and ++(0,-0.35) .. (-0.4,0.9)--(-0.4,1);
  \draw[ultra thick,blue] (-0.4,0)--(-0.4,0.1) .. controls ++(0,0.35) and ++(0,-0.35) .. (0.4,0.9)--(0.4,1);
      \node[above] at (-0.4,1) {$ \scriptstyle j$};
      \node[above] at (0.4,1) {$ \scriptstyle i$};
       \node[below] at (-0.4,0) {$ \scriptstyle i$};
      \node[below] at (0.4,0) {$ \scriptstyle j$};
\end{tikzpicture}}
\qquad
\qquad
\hackcenter{
\begin{tikzpicture}[scale=.8]
  \draw[ultra thick, blue] (0,0)--(0,0.5);
   \draw[ultra thick, red] (0,0.5)--(0,1);
   \draw[thick, fill=yellow]  (0,0.5) circle (7pt);
    \node at (0,0.5) {$ \scriptstyle f$};
     \node[below] at (0,0) {$ \scriptstyle i$};
      \node[above] at (0,1) {$ \scriptstyle j$};
\end{tikzpicture}}
\end{align}
for \(i,j \in I\), \(f \in jAi\). Crossings have parity \(\bar 0\), and the parity of the \(f\) coupon is \(\bar{f}\). Morphisms satisfy the following relations:

{\em Coxeter relations.} For all \(i,j,k \in I\):
\begin{align}\label{CoxThin}
\hackcenter{}
\hackcenter{
\begin{tikzpicture}[scale=.8]
  \draw[ultra thick, blue] (0,0)--(0,0.1) .. controls ++(0,0.35) and ++(0,-0.35)  .. (0.8,0.8)--(0.8,1) .. controls ++(0,0.35) and ++(0,-0.35)  .. (0,1.7)--(0,1.8); 
    \draw[ultra thick, red] (0.8 ,0)--(0.8, 0.1) .. controls ++(0,0.35) and ++(0,-0.35)  .. (0,0.8)--(0,1) .. controls ++(0,0.35) and ++(0,-0.35)  .. (0.8,1.7)--(0.8,1.8);
     \node[below] at (0,0) {$ \scriptstyle i$};
     \node[below] at (0.8,0) {$ \scriptstyle j$};
       \node[above] at (0,1.8) {$ \scriptstyle i$};
     \node[above] at (0.8,1.8) {$ \scriptstyle j$};
\end{tikzpicture}}
\;
=
\;
\hackcenter{
\begin{tikzpicture}[scale=.8]
  \draw[ultra thick, blue] (0,0)--(0,1.8); 
    \draw[ultra thick, red] (0.8 ,0)--(0.8,1.8);
     \node[below] at (0,0) {$ \scriptstyle i$};
     \node[below] at (0.8,0) {$ \scriptstyle j$};
       \node[above] at (0,1.8) {$ \scriptstyle i$};
     \node[above] at (0.8,1.8) {$ \scriptstyle j$};
\end{tikzpicture}}
\qquad
\qquad
\hackcenter{
\begin{tikzpicture}[scale=.8]
  \draw[ultra thick, red] (0.2,0)--(0.2,0.1) .. controls ++(0,0.35) and ++(0,-0.35) .. (-0.4,0.9)
  .. controls ++(0,0.35) and ++(0,-0.35) .. (0.2,1.7)--(0.2,1.8);
  \draw[ultra thick, blue] (-0.6,0)--(-0.6,0.1) .. controls ++(0,0.35) and ++(0,-0.35) .. (1,1.7)--(1,1.8);
  \draw[ultra thick, green] (1,0)--(1,0.1) .. controls ++(0,0.35) and ++(0,-0.35) .. (-0.6,1.7)--(-0.6,1.8);
   \node[below] at (0.2,0) {$ \scriptstyle j$};
        \node[below] at (-0.6,0) {$ \scriptstyle i$};
         \node[below] at (1,0) {$ \scriptstyle k$};
          \node[above] at (0.2,1.8) {$ \scriptstyle j$};
        \node[above] at (-0.6,1.8) {$ \scriptstyle k$};
         \node[above] at (1,1.8) {$ \scriptstyle i$};
\end{tikzpicture}}
=
\hackcenter{
\begin{tikzpicture}[scale=.8]
  \draw[ultra thick, red] (0.2,0)--(0.2,0.1) .. controls ++(0,0.35) and ++(0,-0.35) .. (0.8,0.9)
  .. controls ++(0,0.35) and ++(0,-0.35) .. (0.2,1.7)--(0.2,1.8);
  \draw[ultra thick, blue] (-0.6,0)--(-0.6,0.1) .. controls ++(0,0.35) and ++(0,-0.35) .. (1,1.7)--(1,1.8);
  \draw[ultra thick, green] (1,0)--(1,0.1) .. controls ++(0,0.35) and ++(0,-0.35) .. (-0.6,1.7)--(-0.6,1.8);
   \node[below] at (0.2,0) {$ \scriptstyle j$};
        \node[below] at (-0.6,0) {$ \scriptstyle i$};
         \node[below] at (1,0) {$ \scriptstyle k$};
          \node[above] at (0.2,1.8) {$ \scriptstyle j$};
        \node[above] at (-0.6,1.8) {$ \scriptstyle k$};
         \node[above] at (1,1.8) {$ \scriptstyle i$};
\end{tikzpicture}}
\end{align}

{\em Coupon relations.} For all \(i,j,k \in I\), \(f,g \in jAi\), \(h \in kAj\), \(\alpha \in \k\):
\begin{align}\label{AaRel1Thin}
\hackcenter{}
\hackcenter{
\begin{tikzpicture}[scale=.8]
  \draw[ultra thick, blue] (0,0)--(0,1.8);
   \draw[thick, fill=yellow]  (0,0.9) circle (8pt);
    \node at (0,0.9) {$ \scriptstyle i$};
     \node[below] at (0,0) {$ \scriptstyle i$};
           \node[above] at (0,1.8) {$ \scriptstyle i$};
\end{tikzpicture}}
\;
=
\hackcenter{
\begin{tikzpicture}[scale=.8]
  \draw[ultra thick, blue] (0,0)--(0,1.8);
     \node[below] at (0,0) {$ \scriptstyle i$};
           \node[above] at (0,1.8) {$ \scriptstyle i$};
\end{tikzpicture}}
\qquad
\qquad
\hackcenter{
\begin{tikzpicture}[scale=.8]
  \draw[ultra thick, blue] (0,0)--(0,0.9);
    \draw[ultra thick, red] (0,0.9)--(0,1.8);
    \draw[thick, fill=yellow]  (0,0.9) circle (9pt);
    \node at (0,0.9) {$ \scriptstyle \alpha f$};
     \node[below] at (0,0) {$ \scriptstyle i $};
       \node[above] at (0,1.8) {$ \scriptstyle j $};
\end{tikzpicture}}
\;
=
\;
\alpha\;
\hackcenter{
\begin{tikzpicture}[scale=.8]
  \draw[ultra thick, blue] (0,0)--(0,0.9);
    \draw[ultra thick, red] (0,0.9)--(0,1.8);
 \draw[thick, fill=yellow]  (0,0.9) circle (7pt);
    \node at (0,0.9) {$ \scriptstyle f$};
     \node[below] at (0,0) {$ \scriptstyle i $};
      \node[above] at (0,1.8) {$ \scriptstyle j $};
\end{tikzpicture}}
\qquad
\qquad
\hackcenter{
\begin{tikzpicture}[scale=.8]
  \draw[ultra thick, blue] (0,0)--(0,0.5);
    \draw[ultra thick, red] (0,0.5)--(0,1.3);
      \draw[ultra thick, green] (0,1.3)--(0,1.8);
     \draw[thick, fill=yellow]  (0,0.5) circle (7pt);
    \node at (0,0.5) {$ \scriptstyle f$};
   \draw[thick, fill=yellow]  (0,1.3) circle (7pt);
    \node at (0,1.3) {$\scriptstyle h$};
     \node[below] at (0,0) {$ \scriptstyle i $};
      \node[above] at (0,1.8) {$ \scriptstyle k $};
         \node[left] at (-0.1,0.9) {$ \scriptstyle j $};
\end{tikzpicture}}
\;
=
\;
\hackcenter{
\begin{tikzpicture}[scale=.8]
  \draw[ultra thick, blue] (0,0)--(0,0.9);
    \draw[ultra thick, green] (0,0.9)--(0,1.8);
      \draw[thick, fill=yellow]  (0,0.9) circle (9pt);
    \node at (0,0.9) {$\scriptstyle hf$};
     \node[below] at (0,0) {$ \scriptstyle i $};
      \node[above] at (0,1.8) {$  \scriptstyle k $};
\end{tikzpicture}}
\qquad
\qquad
\hackcenter{
\begin{tikzpicture}[scale=.8]
  \draw[ultra thick, blue] (0,0)--(0,0.9);
    \draw[ultra thick, red] (0,0.9)--(0,1.8);
     \draw[thick, fill=yellow]  (0,0.9) circle (11pt);
    \node at (0,0.9) {$ \scriptstyle f\hspace{-0.3mm}+\hspace{-0.2mm}g$};
     \node[below] at (0,0) {$ \scriptstyle i $};
      \node[above] at (0,1.8) {$ \scriptstyle j$};
\end{tikzpicture}}
\;
=
\hackcenter{
\begin{tikzpicture}[scale=.8]
  \draw[ultra thick, blue] (0,0)--(0,0.9);
    \draw[ultra thick, red] (0,0.9)--(0,1.8);
      \draw[thick, fill=yellow]  (0,0.9) circle (9pt);
    \node at (0,0.9) {$\scriptstyle f$};
     \node[below] at (0,0) {$ \scriptstyle i $};
      \node[above] at (0,1.8) {$  \scriptstyle j $};
\end{tikzpicture}}
+
\hackcenter{
\begin{tikzpicture}[scale=.8]
  \draw[ultra thick, blue] (0,0)--(0,0.9);
    \draw[ultra thick, red] (0,0.9)--(0,1.8);
      \draw[thick, fill=yellow]  (0,0.9) circle (9pt);
    \node at (0,0.9) {$\scriptstyle g$};
     \node[below] at (0,0) {$ \scriptstyle i $};
      \node[above] at (0,1.8) {$  \scriptstyle j $};
\end{tikzpicture}}
\end{align}

{\em \(A\)-intertwining relations.} For all \(i,j,k \in I,  f \in jAi\):
\begin{align}\label{AaIntertwineThin}
\hackcenter{}
\hackcenter{
\begin{tikzpicture}[scale=.8]
  \draw[ultra thick, blue] (0,0)--(0,0.5);
  \draw[ultra thick, red] (0,0.5)--(0,1) .. controls ++(0,0.35) and ++(0,-0.35)  .. (0.8,1.8)--(0.8,2); 
    \draw[ultra thick, green] (0.8,0)--(0.8,1) .. controls ++(0,0.35) and ++(0,-0.35)  .. (0,1.8)--(0,2); 
        \draw[thick, fill=yellow]  (0,0.5) circle (7pt);
    \node at (0,0.5) {$ \scriptstyle f$};
     \node[below] at (0,0) {$ \scriptstyle i $};
     \node[below] at (0.8,0) {$  \scriptstyle  k $};
     \node[above] at (0,2) {$  \scriptstyle  k $};
      \node[above] at (0.8,2) {$ \scriptstyle j $};
\end{tikzpicture}}
\;
=
\;
\hackcenter{
\begin{tikzpicture}[scale=.8]
  \draw[ultra thick, blue] (0,0)--(0,0.2) .. controls ++(0,0.35) and ++(0,-0.35)  .. (0.8,1)--(0.8,1.5); 
   \draw[ultra thick, red] (0.8,1.5)--(0.8,2); 
    \draw[ultra thick, green] (0.8,0)--(0.8,0.2) .. controls ++(0,0.35) and ++(0,-0.35)  .. (0,1)--(0,2); 
        \draw[thick, fill=yellow]  (0.8,1.5) circle (7pt);
    \node at (0.8,1.5) {$ \scriptstyle f$};
     \node[below] at (0,0) {$ \scriptstyle i$};
     \node[below] at (0.8,0) {$  \scriptstyle  k $};
     \node[above] at (0,2) {$  \scriptstyle  k $};
      \node[above] at (0.8,2) {$ \scriptstyle j $};
\end{tikzpicture}}
\qquad
\qquad
\hackcenter{
\begin{tikzpicture}[scale=.8]
  \draw[ultra thick, green] (0,0)--(0,1) .. controls ++(0,0.35) and ++(0,-0.35)  .. (0.8,1.8)--(0.8,2); 
    \draw[ultra thick, red] (0.8,0.5)--(0.8,1) .. controls ++(0,0.35) and ++(0,-0.35)  .. (0,1.8)--(0,2); 
       \draw[ultra thick, blue] (0.8,0)--(0.8,0.5); 
        \draw[thick, fill=yellow]  (0.8,0.5) circle (7pt);
    \node at (0.8,0.5) {$ \scriptstyle f$};
     \node[below] at (0,0) {$ \scriptstyle  k $};
     \node[below] at (0.8,0) {$ \scriptstyle i $};
     \node[above] at (0,2) {$ \scriptstyle j $};
      \node[above] at (0.8,2) {$  \scriptstyle k $};
\end{tikzpicture}}
\;
=
\;
\hackcenter{
\begin{tikzpicture}[scale=.8]
  \draw[ultra thick, green] (0,0)--(0,0.2) .. controls ++(0,0.35) and ++(0,-0.35)  .. (0.8,1)--(0.8,2); 
     \draw[ultra thick, blue] (0.8,0)--(0.8,0.2) .. controls ++(0,0.35) and ++(0,-0.35)  .. (0,1)--(0,1.5); 
    \draw[ultra thick, red] (0,1.5)--(0,2); 
       \draw[thick, fill=yellow]  (0,1.5) circle (7pt);
    \node at (0,1.5) {$ \scriptstyle f$};
     \node[below] at (0,0) {$  \scriptstyle k $};
     \node[below] at (0.8,0) {$ \scriptstyle i $};
     \node[above] at (0,2) {$ \scriptstyle j $};
      \node[above] at (0.8,2) {$  \scriptstyle k $};
\end{tikzpicture}}
\end{align}
\end{definition}

Let \(\WebAaIthin\) be the full monoidal subcategory of \(\WebAaI\) consisting of the `thin strands'; that is, the full subcategory generated by the objects \(\{i^{(1)} \mid i \in I\}\), so that
\begin{align*}
\Ob(\WebAaIthin) = \left\{ i_1^{(1)} \otimes \cdots \otimes i_t^{(1)} \mid t \in \Z_{\geq 0}, i_1, \ldots, i_t \in I\right\}.
\end{align*}

\begin{proposition}\label{P:WreathtoThin}
There is an isomorphism of monoidal supercategories,
\[
\iota_{\textup{Wr}}: \WrAI \to \WebAaIthin,
\] given by \(i \mapsto i^{(1)}\) on generating objects, and 
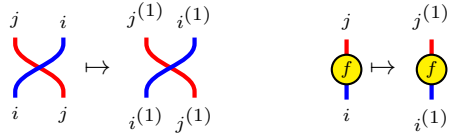
\begin{align*}
\hackcenter{
}
\hackcenter{
\begin{tikzpicture}[scale=.8]
  \draw[ultra thick,red] (0.4,0)--(0.4,0.1) .. controls ++(0,0.35) and ++(0,-0.35) .. (-0.4,0.9)--(-0.4,1);
  \draw[ultra thick,blue] (-0.4,0)--(-0.4,0.1) .. controls ++(0,0.35) and ++(0,-0.35) .. (0.4,0.9)--(0.4,1);
      \node[above] at (-0.4,1) {$ \scriptstyle j$};
      \node[above] at (0.4,1) {$ \scriptstyle i$};
       \node[below] at (-0.4,0) {$ \scriptstyle i$};
      \node[below] at (0.4,0) {$ \scriptstyle j$};
\end{tikzpicture}}
\mapsto
\hackcenter{
\begin{tikzpicture}[scale=.8]
  \draw[ultra thick,red] (0.4,0)--(0.4,0.1) .. controls ++(0,0.35) and ++(0,-0.35) .. (-0.4,0.9)--(-0.4,1);
  \draw[ultra thick,blue] (-0.4,0)--(-0.4,0.1) .. controls ++(0,0.35) and ++(0,-0.35) .. (0.4,0.9)--(0.4,1);
      \node[above] at (-0.4,1) {$ \scriptstyle j^{\scriptstyle (1)}$};
      \node[above] at (0.4,1) {$ \scriptstyle i^{\scriptstyle (1)}$};
       \node[below] at (-0.4,0) {$ \scriptstyle i^{\scriptstyle (1)}$};
      \node[below] at (0.4,0) {$ \scriptstyle j^{\scriptstyle (1)}$};
\end{tikzpicture}}
\qquad
\qquad
\hackcenter{
\begin{tikzpicture}[scale=.8]
  \draw[ultra thick, blue] (0,0)--(0,0.5);
   \draw[ultra thick, red] (0,0.5)--(0,1);
   \draw[thick, fill=yellow]  (0,0.5) circle (7pt);
    \node at (0,0.5) {$ \scriptstyle f$};
     \node[below] at (0,0) {$ \scriptstyle i$};
      \node[above] at (0,1) {$ \scriptstyle j$};
\end{tikzpicture}}
\mapsto
\hackcenter{
\begin{tikzpicture}[scale=.8]
  \draw[ultra thick, blue] (0,0)--(0,0.5);
   \draw[ultra thick, red] (0,0.5)--(0,1);
   \draw[thick, fill=yellow]  (0,0.5) circle (7pt);
    \node at (0,0.5) {$ \scriptstyle f$};
     \node[below] at (0,0) {$ \scriptstyle i^{\scriptstyle (1)}$};
      \node[above] at (0,1) {$ \scriptstyle j^{\scriptstyle (1)}$};
\end{tikzpicture}}
\end{align*}
on generating morphisms,
for all \(i,j \in I\), \(f \in jAi\). 
\end{proposition}
\begin{proof}
That \(\iota_{\textup{Wr}}\) is well-defined follows from \cref{AaRel1,AaRel2,arbcol}. By the defining relations in \(\WrAI\), it is straightforward to see that for \(\bi = i_1 \otimes \cdots \otimes i_r, \bj = j_1 \otimes \cdots \otimes j_r \in \Ob(\WrAI)\) the morphism space \(\WrAI(\bi, \bj)\) is spanned by diagrams of the form:
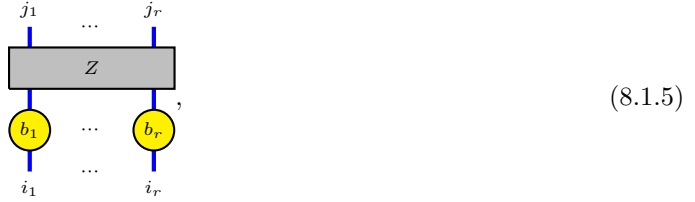
\begin{align}
\hackcenter{}
{}
\hackcenter{
\begin{tikzpicture}[scale=1.1]
      \draw[ultra thick, blue] (0.25,-1)--(0.25,0.75); 
       \draw[ultra thick, blue] (1.75,-1)--(1.75,0.75); 
          \draw[thick, fill=yellow]  (0.25,-0.5) circle (7pt);
             \draw[thick, fill=yellow]  (1.75,-0.5) circle (7pt);
    \node at (0.25,-0.5) {$ \scriptstyle b_{1}$};
       \node at (1.75,-0.5) {$ \scriptstyle b_r$};
              \draw[thick, fill=gray!50!white] (0,0)--(2,0)--(2,0.5)--(0,0.5)--(0,0); 
     \node at (1,0.25) {$\scriptstyle Z$};  
     \node[above] at (0.25,0.75) {$\scriptstyle j_1$};
       \node[above] at (1.75,0.75) {$\scriptstyle j_r$};
     \node[below] at (0.25,-1) {$\scriptstyle i_1$};
       \node[below] at (1.75,-1) {$\scriptstyle i_r$};
              \node[] at (1,-0.5) {$\scriptstyle \cdots $};
               \node[] at (1,-1) {$\scriptstyle \cdots $};
                \node[] at (1,0.75) {$\scriptstyle \cdots $};
\end{tikzpicture}},
\end{align}
for some \(\omega \in \mathfrak{S}_r\), \(b_k \in {}_{j_{\omega (k)}}\BasisB_{i_k}\), and where for all \(k \in [1,r]\), \(Z\) consists only of crossings and carries the \(k\)th strand below to the \(\omega (k)\)th strand above.
 But, by \cref{BasisThm}, the images of such morphisms comprise a \(\k\)-basis for \(\WebAaIthin(\iota_{\textup{Wr}} \bi, \iota_{\textup{Wr}}\bj)\), so \(\iota_{\textup{Wr}}\) is an isomorphism.
\end{proof}

\subsection{The wreath category, II}\label{SS:ThinStrandsandWreathProducts} We give another formulation of the wreath product category.  Fix $d \geq 0$.  Let $\boldsymbol{\mathcal{W}}^{A }_I(d)$ be the $\k$-linear supercategory defined as follows.  The objects are the elements of the set
\[
\left\{\bi = 1_{\SS_{d}}\otimes i_{1} \otimes \dotsb \otimes i_{d}  \mid i_{1}, \dotsc , i_{d} \in I \right\}.
\] For objects $\bi$ and $\bj$, let
\[
\boldsymbol{\mathcal{W}}^{A}_I(d) (\bi, \bj ) = \bigoplus_{\sigma \in \SS_{d}} \sigma \otimes j_{\sigma 1} A i_1 \otimes \dotsb \otimes j_{\sigma d} A i_d .
\]  Composition is given on homogeneous morphisms by the rule 
\[
(\sigma \otimes f_{1} \otimes \dotsb \otimes f_{d}) \circ (\tau \otimes g_{1} \otimes \dotsb \otimes g_{d}) = (-1)^{\langle \tau; (f_{1}, \dotsc , f_{d}) \rangle} \sigma\tau \otimes (f_{\tau(1)} g_{1}) \otimes \dotsb \otimes (f_{\tau(d)}  g_{d}),
\] where the sign is according to the signed place permutation action as in \cref{twistsec}.
Let 
\[
\boldsymbol{\mathcal{W}}^{A}_I = \bigoplus_{d \geq 0} \boldsymbol{\mathcal{W}}^{A}_I(d).
\]

This is a monoidal supercategory via concatenation.  Namely, given $d_{1}, d_{2} \geq 0$ let $\iota = \iota_{d_{1}, d_{2}}: \SS_{d_{1}} \times \SS_{d_{2}} \hookrightarrow \SS_{d_{1}+d_{2}}$ be the embedding which identifies $\SS_{d_{1}} \times \SS_{d_{2}}$ as a Young subgroup of $\SS_{d_{1}+d_{2}}$ in the obvious way. Then the tensor product of objects $\bi = 1_{\SS_{d_{1}}} \otimes i_{1} \otimes \dotsb \otimes i_{d_{1}}$ and $\bj = 1_{\SS_{d_{2}}} \otimes j_{1} \otimes \dotsb \otimes j_{d_{2}}$ is given by
\[
\bi \otimes \bj =  1_{\SS_{d_{1}+d_{2}}} \otimes i_{1} \otimes \dotsb \otimes i_{d_{1}} \otimes j_{1} \otimes \dotsb \otimes j_{d_{2}}.
\]  The tensor product of morphisms is given by 
\[
(\sigma \otimes f_{1} \otimes \dotsb \otimes f_{d_{1}}) \otimes (\tau \otimes g_{1} \otimes \dotsb \otimes g_{d_{2}}) = \iota(\sigma, \tau) \otimes f_{1} \otimes \dotsb \otimes f_{d_{1}} \otimes g_{1} \otimes \dotsb \otimes g_{d_{2}}.
\]

\begin{theorem}\label{T:anotherwreathproductcategory}  There is an equivalence of monoidal $k$-linear supercategories
\[
\WrAI \cong \boldsymbol{\mathcal{W}}^{A}_I.
\]

\end{theorem}

\begin{proof} Let $s_{1}\in \SS_{2}$ be the nontrivial element.  A check of the relations given in \cref{defwr} shows there is a $\k$-linear monoidal functor $H: \WrAI \to \boldsymbol{\mathcal{W}}^{A}_I$ defined as follows.  On generating objects it is defined by $H(i) =  1_{\SS_{1}} \otimes i$ for all objects $i$ of $A$.  On morphisms, if $f \in jAi$, then on the generating morphisms of $\WrAI$ the functor is given by
\[
H\left(\hackcenter{\begin{tikzpicture}[scale=.8]
  \draw[ultra thick,red] (0.4,0)--(0.4,0.1) .. controls ++(0,0.35) and ++(0,-0.35) .. (-0.4,0.9)--(-0.4,1);
  \draw[ultra thick,blue] (-0.4,0)--(-0.4,0.1) .. controls ++(0,0.35) and ++(0,-0.35) .. (0.4,0.9)--(0.4,1);
      \node[above] at (-0.4,1) {$ \scriptstyle j$};
      \node[above] at (0.4,1) {$ \scriptstyle i$};
       \node[below] at (-0.4,0) {$ \scriptstyle i$};
      \node[below] at (0.4,0) {$ \scriptstyle j$};
\end{tikzpicture}}\right) = s_{1} \otimes (1_{i} \otimes 1_{j}) \in \Hom_{\mathbf{\mathcal{W}}^{A}_I}\left( 1_{\SS_{d}}\otimes i \otimes j, 1_{\SS_{d}}\otimes  j \otimes i \right),
\]
and
\[
H \left(
\hackcenter{\begin{tikzpicture}[scale=.8]
  \draw[ultra thick, blue] (0,0)--(0,0.5);
   \draw[ultra thick, red] (0,0.5)--(0,1);
   \draw[thick, fill=yellow]  (0,0.5) circle (7pt);
    \node at (0,0.5) {$ \scriptstyle f$};
     \node[below] at (0,0) {$ \scriptstyle i$};
      \node[above] at (0,1) {$ \scriptstyle j$};
\end{tikzpicture}} \right) = 1_{\SS_{0}} \otimes f \in \Hom_{\mathbf{\mathcal{W}}^{A}_I}\left( 1_{\SS_{d}}\otimes i,  1_{\SS_{d}}\otimes  j \right).
\]

It is evident that this functor is essentially surjective and full.  Combining \cref{P:WreathtoThin,BasisThm} yields bases for the morphism spaces of $\WrA$. Applying $H$ to these bases yields bases for the morphism spaces of $\boldsymbol{\mathcal{W}}^{A}_I$.  Thus $H$ is an equivalence of categories. 
\end{proof}

For each $d \geq 0$ consider the full subcategory $\WrAI (d)$ of $\WrAI$ consisting of all objects 
\[
I^{\otimes d}:=\left\{\bi = i_{1} \otimes \dotsb \otimes i_{d} \mid i_{1}, \dotsc , i_{d} \in I \right\}.
\]
  Since each of the morphisms in $\WrAI$ preserve the number of objects in a tensor product we have the decomposition
\[
\WrAI = \bigoplus_{d \geq 0} \WrAI (d).
\]  Correspondingly, for each $d \geq 0$ there is a superalgebra 
\[
\operatorname{Wr}^{A}_I (d) = \bigoplus_{\bi, \bj \in \Ob(\WrAI (d))} \Hom_{\WrAI}\left(\bi, \bj \right).
\]

 Let $\SS_{d} \wr A$ denote the wreath product superalgebra.  Since $A$ has a distingished set of idempotents $I$, there is a corresponding set of distinguished idempotents for $\SS_{d} \wr A$:
\[
1_{\SS_{d}} \otimes I^{\otimes d}=\left\{ 1_{\SS_{d}} \otimes i_{1} \otimes \dotsb \otimes i_{d} \mid i_{1}, \dotsc , i_{d} \in I \right\}.
\]  The equivalence in the previous theorem restricts to define an equivalence $\WrAI (d) \cong \boldsymbol{\mathcal{W}}^{A}_I(d)$ for every $d \geq 0$ and yields the following result.

\begin{corollary}\label{C:WreathsAreIsomorphic} For every $d \geq 0$, there is an even isomorphism of idempotented superalgebras,
\[
\psi_{A, d} :\operatorname{Wr}^{A}_I (d)\to \SS_{d} \wr A.
\]
\end{corollary}

\section{Fullness of representations}\label{S:FunctorFullness}

\subsection{Defining representations}\label{SS:DefiningReps}

Write $\modglnAT$ for the full monoidal subcategory of right $U(\gl_n(A))$-supermodules of the form ${}_{i_{1}}V_{n} \otimes \dotsb \otimes {}_{i_{d}}V_{n} $ for some $d \geq 0$ and $i_{1}, \dotsc , i_{d} \in I$.  For a fixed $d$, let  $\modglnAT(d)$ denote the full subcategory consisting of objects in the aforementioned form with precisely $d$ terms in the tensor product.

We consider the following functors to be the defining representation of the diagrammatic category in question:
\begin{gather*}
G_{n}:  \WebAaI \to \modglnAS, \\
F_{n} := G_{n} \circ \iota_{\textup{Wr}}:  \WrAI\to  \modglnAT, \\
F_{n}^{d}:   \WrAI(d) \to  \modglnAT(d),
\end{gather*}
where $G_{n}$ is from \cref{Gthm} and $\iota_{\textup{Wr}}$ is from \cref{P:WreathtoThin}.
It is evident that the functors $G_{n}$, $F_{n}$, and $F_{n}^{d}$ are essentially surjective.

\subsection{Fullness for the web category}\label{SS:ReductionToThinStrands} 
The following result allows a reduction to thin strands when working over a field of characteristic zero.
\begin{proposition}\label{P:ReductionToThinStrands}  Assume $\k$ is a field of characteristic zero. For a given $n \geq0$, the functor, 
\[
G_{n}: \WebAaI \to \modglnAS,
\]
is full if and only if the functor, 
\[
F_{n}:  \WrAI \to \modglnAT,
\]
is full.
\end{proposition}

\begin{proof}   For the proof, let $\widehat{G}_{n}:\WebAaIthin \to \modglnAT$ be the restriction of $G_{n}$.  Since $\iota_{Wr}$ is an equivalence, $F_{n}$ will be full if and only if $\widehat{G}_{n}$ is full.  If $G_{n}$ is full then clearly the restriction $\widehat{G}_{n}$ to a full subcategory is still full.  On the other hand, assume $\widehat{G}_{n}$ is full.    A standard argument using the explosion and contraction maps of \cref{ExplSec} shows $G_{n}$ must also be full. Namely, the explosion and contraction maps along with \cref{L:ExplCon} establishes a commutative diagram as in \cite[Lemma 6.7.1]{DKM} and the argument used in the proof of \cite[Theorem 6.7.2]{DKM} then readily adapts to this setting.  
\end{proof}

\subsection{Fullness for the wreath category}\label{SS:fullnessforwreathcategory}

Consider
\begin{equation}\label{E:tensorspace}
\bigoplus_{\bi \in \WrAI (d)} F^{d}_{n}(\bi ) = \bigoplus_{\bi = i_{1}\otimes \dotsb \otimes i_{d} \in I^{\otimes d}} i_{1}V_{n} \otimes \dotsb \otimes i_{d}V_{n} \cong V_{n}^{\otimes d}.
\end{equation}
This is a left $\operatorname{Wr}^{A}_{I}(d)$-supermodule and, via $\psi_{A , d}$ from \cref{C:WreathsAreIsomorphic}, a left $\SS_{d} \wr A$-supermodule.  This action coincides with the action on \(V_{n}^{\otimes d}\), where \(A^{\otimes d}\) acts diagonally by left multiplication and \(\mathfrak{S}_d\) acts by signed permutation of tensor factors. Furthermore, this action commutes with the right action by $\gl_{n}(A)$ given by right matrix multiplication (where we write $V_{n} = A^{\oplus n}$ as row vectors).

Let 
\[
\rho^{A}_{n, d}: \SS_{d} \wr A  \to \End_{\k}\left(V_{ n}^{\otimes d} \right)
\] be the corresponding representation.
We call a $\k$-linear map $f:V_{ n}^{\otimes d} \to V_{ n}^{\otimes d}$ \emph{finite} with respect to $I$ if there is a finite set $J \subseteq I^{\otimes d}$ such that $f( F^{d}_{n}(\bi ) ) = 0$ whenever $\bi \not\in J$ and the image of $f$ lies in $\bigoplus_{\bj \in J} F^{d}_{n}(\bj )$.  We write $ \End_{\k}\left(V_{ n}^{\otimes d} \right)^{\fin}$ for the $\k$-subsuperalgebra of all finite maps and note that it contains the image of $\rho_{n,d}^A$.

\begin{proposition}\label{P:fullnesseqaulssurjective} For any fixed $n \geq 0$, the functor 
\[
F_{n}: \WrAI \to  \modglnAT
\] 
is full if and only if the superalgebra map 
\[
\rho^{A}_{n,d}:  \SS_{d} \wr A \to \End_{U(\gl_n(A))}\left(V_{n}^{\otimes d} \right)^{\textup{\fin}}
\]
is surjective for all $d \geq 0$.  
\end{proposition}

\begin{proof}  We first claim that if there is a nonzero morphism from an object of $\modglnAT (d_{1})$ to an object of $\modglnAT (d_{2})$, then $d_{1}= d_{2}$.  For any finite set of elements in $A$ there is an $e \in A$ given by a sum of distinguished idempotents which acts as the identity on those elements when right multiplied.  In turn, given a fixed vector in an object of $\modglnAT (d)$ there is an $e \in A$ such that $\sum_{r=1}^{n} E_{r,r}^{e} \in \gl_{n}(A )$ acts on that vector by the scalar $d \in \k$.  Now, say there is an object $X$ of $\modglnAT (d_{1})$ and an object $Y$ of $\modglnAT (d_{2})$ with a nonzero $\gl_{n}(A )$-supermodule homomorphism $f:X \to Y$.  Choose $x \in X$ so that $f(x) \neq 0$.  Then there exists $e \in A$ such that  $\sum_{r=1}^{n} E_{r,r}^{e} $ scales $x$ by $d_{1}$ and scales $f(x)$ by $d_{2}$.  It follows that $d_{1}=d_{2}$, as claimed. 

From the previous paragraph it follows there is a decomposition
\begin{equation*}
\modglnAT = \bigoplus_{d \geq 0} \modglnAT (d).
\end{equation*}
Thus the functor $F_{n}$ is full if and only if the functor $F^{d}_{n}$ is full for every $d \geq 0$. 
However, from the discussion at the beginning of this section, this happens if and only if the induced superalgebra map
\begin{equation*}
\rho_{n,d}^{A}:  \SS_{d} \wr A \to \End_{U(\gl_{n}(A))}\left(V_{n}^{\otimes d} \right)^{\fin}
\end{equation*}
is surjective.
\end{proof}

Combining \cref{L:SchurWeylDuality,P:ReductionToThinStrands,P:fullnesseqaulssurjective} gives the following situation where the defining representations $G_{n}$ and $F_{n}$ are known to be full.

\begin{proposition}\label{C:Fullness} If $\k$ is a field of characteristic zero and $A$ is a finite-dimensional semisimple superalgebra, then the functors 
\begin{align*}
G_{n} &: \WebAaI \to \modglnAS, \\
F_{n} &:  \WrAI \to  \modglnAT,
\end{align*}
are full.
\end{proposition}

\section{The category \texorpdfstring{$\UglnA$}{U(gln(A))}}

\subsection{The category \texorpdfstring{$\UglnA$}{U(gln(A))}}\label{SS:UglnACategoryDef} In this section we assume $(A,a)_I$ is a good pair, \(I\) is finite, and \(1 = \sum_{i \in I} i\).
Let \(\Lambda_n\) be the free \(\Z_{\geq 0}\)-module of rank \(n\) on basis \(\{\varepsilon_r \mid r \in [1,n]\}\). For \(\lambda \in \Lambda_n\) we write \(\lambda_r\) for the coefficient of \(\varepsilon_r\) in \(\lambda\).
\begin{definition} \label{UglnAdef}
Let \(\UglnA\) be the \(\k\)-linear supercategory consisting of objects \(\Lambda_n \), and of morphisms generated by symbols
\begin{align*}
\{\mu E_{r,s}^f\lambda: \lambda \to \mu \mid r,s \in [1,n],\; \lambda, \mu \in \Lambda, \;\mu + \varepsilon_s = \lambda  + \varepsilon_r, \;f \in A\}.
\end{align*}
We will often omit the symbols \(\lambda\) and \(\mu\), writing \(E_{r,s}^f\) when the domain/codomain of the morphism is general or clear by context. 
The parities of morphisms are given by \(\overline{E_{r,s}^f} = \bar f\). 
Morphisms in \(\UglnA\) are subject to the following relations, for all admissible \(\lambda \in \Lambda_n\), \(r,s,p,q \in [1,n]\), and homogeneous \(f,g \in A\):
\begin{enumerate}
\item \(E^{kf}_{r,s}= kE^f_{r,s}\) for all \(k \in \k\).
\item \(E^{f+g}_{r,s}= E^f_{r,s}  + E^g_{r,s}\).
\item \( E^{1}_{r,r}\lambda = \lambda_r \textup{id}_\lambda\).
\item \(
E_{p,q}^fE_{r,s}^g  - (-1)^{\bar f \bar g}E_{r,s}^gE_{p,q}^f = \delta_{q,r} E_{p,s}^{fg} - (-1)^{\bar f \bar g}\delta_{p,s} E_{r,q}^{gf}.
\)
\item \((E^f_{r,s})^2 = \delta_{r,s}E^{f^2}_{r,s}\) if \(\bar f = \bar 1\).
\end{enumerate}
\end{definition}

\begin{remark}
If \(\k\) is a ring where 2 is invertible, relation (5) is redundant, as \(2(E^f_{r,s})^2 = 2\delta_{r,s}E^{f^2}_{r,s}\) is implied by (4) when \(\bar f = \bar 1\).
\end{remark}

\subsection{A functor from \texorpdfstring{$\UglnA$}{U(gln(A))}  to  \texorpdfstring{$\WebAaOnen$}{WebAA0n}}
Let \(\WebAaOnen\) be the full subcategory of \(\WebAaOne\) consisting of objects \(1^{(\lambda_1)} \cdots 1^{(\lambda_n)}\), for \(\lambda = \sum \lambda_{i}\varepsilon_{i} \in \Lambda_n\). In this section, as all strands are monotonically colored by \(1\), for sake of space we will often write \(x\) for the object \(1^{(x)}\in \WebAaOne\), and \(\lambda\) for the object \(1^{(\lambda_1)}  \cdots  1^{(\lambda_n)}\) in \(\WebAaOne\).

Let \(t \in \Z_{> 0}\), \(r,s \in [1,n]\), \(\lambda \in \Lambda_n\), and \(f \in A^{(t)}\). We define the following associated morphisms in \(\WebAaOnen\):

\begin{align*}
e^{(f,t)}_{[r,s],\lambda}
&:=
\begin{cases}
\hackcenter{}
\hackcenter{
\begin{tikzpicture}[scale=0.9]
\draw[ultra thick, blue] (0,0)--(0,1);
\node at (0.5, 0.5) {$\scriptstyle \cdots $};
\draw[ultra thick, blue] (1,0)--(1,1);
\draw[ultra thick, blue] (1.8,0)--(1.8,1);
\node at (2.3, 0.95) {$\scriptstyle \cdots $};
\node at (2.3, 0.05) {$\scriptstyle \cdots $};
\node at (3.2, 0.65) {$\scriptstyle t $};
\draw[ultra thick, blue] (2.8,0)--(2.8,1);
\draw[ultra thick, blue] (3.6,0)--(3.6,1);
\node at (4.1, 0.5) {$\scriptstyle \cdots $};
\draw[ultra thick, blue] (4.6,0)--(4.6,1);
  \draw[ultra thick, blue] (3.6,0.1)--(3.6,0.2) .. controls ++(0,0.35) and ++(0,-0.35) .. (1,0.8)--(1,0.9);
  \node[below] at (0, 0) {$\scriptstyle \lambda_1 $};
    \node[below] at (1, 0) {$\scriptstyle \lambda_r $};
        \node[below] at (1.8, 0) {$\scriptstyle \lambda_{r+1} $};
            \node[below] at (2.8, 0) {$\scriptstyle \lambda_{s-1} $};
               \node[below] at (3.6, 0) {$\scriptstyle \lambda_s $};
               \node[below] at (4.6, 0) {$\scriptstyle \lambda_n $};
                 \node[above] at (0, 1) {$\scriptstyle \lambda_1 $};
    \node[above] at (1,1) {$\scriptstyle \lambda_r +t$\;\;};
        \node[above] at (1.8, 1) {$\scriptstyle \lambda_{r+1} $};
            \node[above] at (2.8, 1) {$\scriptstyle \lambda_{s-1} $};
               \node[above] at (3.6, 1) {\;\;$\scriptstyle \lambda_s -t$};
               \node[above] at (4.6, 1) {$\scriptstyle \lambda_n $};
                \draw[thick, fill=yellow] (2.3, 0.5) circle (6pt);
       \node[] at (2.3, 0.5) {$\scriptstyle f$};
\end{tikzpicture}}
&
\textup{if }r < s;
\\
\\
\hackcenter{
\begin{tikzpicture}[scale=0.9]
\draw[ultra thick, blue] (0,0)--(0,1);
\node at (0.5, 0.5) {$\scriptstyle \cdots $};
\draw[ultra thick, blue] (1,0)--(1,1);
\draw[ultra thick, blue] (1.8,0)--(1.8,1);
\node at (2.3, 0.95) {$\scriptstyle \cdots $};
\node at (2.3, 0.05) {$\scriptstyle \cdots $};
\node at (1.4, 0.65) {$\scriptstyle t $};
\draw[ultra thick, blue] (2.8,0)--(2.8,1);
\draw[ultra thick, blue] (3.6,0)--(3.6,1);
\node at (4.1, 0.5) {$\scriptstyle \cdots $};
\draw[ultra thick, blue] (4.6,0)--(4.6,1);
  \draw[ultra thick, blue] (1,0.1)--(1,0.2) .. controls ++(0,0.35) and ++(0,-0.35) .. (3.6,0.8)--(3.6,0.9);
  \node[below] at (0, 0) {$\scriptstyle \lambda_1 $};
    \node[below] at (1, 0) {$\scriptstyle \lambda_s $};
        \node[below] at (1.8, 0) {$\scriptstyle \lambda_{s+1} $};
            \node[below] at (2.8, 0) {$\scriptstyle \lambda_{r-1} $};
               \node[below] at (3.6, 0) {$\scriptstyle \lambda_r $};
               \node[below] at (4.6, 0) {$\scriptstyle \lambda_n $};
                 \node[above] at (0, 1) {$\scriptstyle \lambda_1 $};
    \node[above] at (1,1) {$\scriptstyle \lambda_s -t$\;\;};
        \node[above] at (1.8, 1) {$\scriptstyle \lambda_{s+1} $};
            \node[above] at (2.8, 1) {$\scriptstyle \lambda_{r-1} $};
               \node[above] at (3.6, 1) {\;\;$\scriptstyle \lambda_r +t$};
               \node[above] at (4.6, 1) {$\scriptstyle \lambda_n $};
                               \draw[thick, fill=yellow] (2.3, 0.5) circle (6pt);
       \node[] at (2.3, 0.5) {$\scriptstyle f$};
\end{tikzpicture}}
&
\textup{if }r>s;
\\
\\
\hackcenter{
\begin{tikzpicture}[scale=0.9]
\draw[ultra thick, blue] (0,0)--(0,1);
\node at (0.5, 0.5) {$\scriptstyle \cdots $};
\draw[ultra thick, blue] (1,0)--(1,1);
 \draw[ultra thick, blue] (2.3,0)--(2.3,0.1) .. controls ++(0,0.1) and ++(0,-0.1) .. (2.8,0.4)--(2.8,0.6)  .. controls ++(0,0.15) and ++(0,-0.15) .. (2.3,0.9)--(2.3,1) ;
  \draw[ultra thick, blue] (2.3,0)--(2.3,0.1) .. controls ++(0,0.1) and ++(0,-0.1) .. (1.8,0.4)--(1.8,0.6)  .. controls ++(0,0.15) and ++(0,-0.15) .. (2.3,0.9)--(2.3,1) ;
\node at (2.6, 0.1) {$\scriptstyle t $};
\node at (1.65, 0.1) {$\scriptstyle \lambda_r \hspace{-0.2mm}-\hspace{-0.2mm} t $};
\draw[ultra thick, blue] (3.6,0)--(3.6,1);
\node at (4.1, 0.5) {$\scriptstyle \cdots $};
\draw[ultra thick, blue] (4.6,0)--(4.6,1);
   \node[below] at (0, 0) {$\scriptstyle \lambda_1 $};
      \node[below] at (2.3, 0) {$\scriptstyle \lambda_r $};
       \node[above] at (2.3, 1) {$\scriptstyle \lambda_r $};
    \node[below] at (1, 0) {$\scriptstyle \lambda_{r-1} $};
               \node[below] at (3.6, 0) {$\scriptstyle \lambda_{r+1} $};
               \node[below] at (4.6, 0) {$\scriptstyle \lambda_n $};
                 \node[above] at (0, 1) {$\scriptstyle \lambda_1 $};
    \node[above] at (1,1) {$\scriptstyle \lambda_{r -1}$};
               \node[above] at (3.6, 1) {$\scriptstyle \lambda_{r +1}$};
               \node[above] at (4.6, 1) {$\scriptstyle \lambda_n $};
                               \draw[thick, fill=yellow] (2.8, 0.5) circle (6pt);
       \node[] at (2.8, 0.5) {$\scriptstyle f$};
\end{tikzpicture}}
&
\textup{if }r=s.
\end{cases}
\end{align*}
We occasionally omit the \(\lambda\) subscript and write \(e_{[r,s]}^{(f,t)}\) when the domain is general or clear from context.

\begin{theorem}\label{T:UdottoWebs}  For each $n \geq 0$ there is a functor 
\[
W_{n}: \UglnA \to \WebAaOnen 
\] given on objects by 
\[
W_{n}\left( \sum_{i=1}^{n}\lambda_{i}\varepsilon_{i}\right) =1^{(\lambda_{1})} \dotsb  1^{(\lambda_{i})}
\] and on morphisms by 
\[
W_{n}\left( E_{r,s}^{f}\lambda\right) = e_{[r,s], \lambda}^{(f,1)}.
\]
\end{theorem}

\begin{proof}
We check that the defining relations in \cref{UglnAdef} are preserved by \(W_n\). Indeed, (1) follows from \cref{AaRel1}, (2) follows from \cref{AaRel2}, (3) follows from \cref{KnotholeRel}, (4) follows from \cref{AssocRel,MSrel,AaRel2}, and (5) follows from \cref{MSrel,OddKnotholeRel}.
\end{proof}

\begin{lemma}\label{thickgen}
Let \(k \in \Z_{>0} \cup \{\infty\}\), and set
\begin{align}\label{Ysets}
Y^n_k :=
\left\{ \left.  e_{[r,s],\lambda}^{(f, t)} \; \right| \lambda \in \Lambda_n, r,s \in [1,n], 1 \leq t \leq k, f \in  A^{(t)} \right\}.
\end{align}
Let \(S^n_k\) denote the \(\k\)-linear subcategory of \(\WebAaOnen\) consisting of all objects in \(\WebAaOnen\) and with morphisms generated by \(Y^n_k\). For all \(\lambda, \nu \in \Lambda_n\), \(r,s \in [1,n]\), \(1 \leq t \leq k\),  \(f \in A^{(t)}\), and \(y' \in S^n_k(\lambda - t\varepsilon_s, \nu- t\varepsilon_r)\), the morphism 
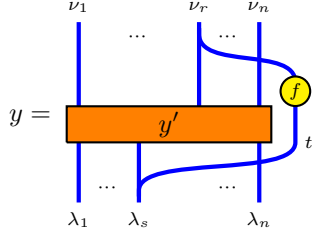
\begin{align*}
y=
\hackcenter{}
\hackcenter{
\begin{tikzpicture}[scale=0.8]
 \draw[ultra thick, color=blue] (0,0)--(0,1);
  \node at (0.5, 0.25) {$\scriptstyle \cdots $};
 \draw[ultra thick, color=blue] (1,0)--(1,1);
 \node at (2.5, 0.25) {$\scriptstyle \cdots $};
  \draw[ultra thick, color=blue] (3,0)--(3,1);
   \draw[ultra thick, color=blue] (0,1.5)--(0,3);
  \node at (1, 2.75) {$\scriptstyle \cdots $};
 \draw[ultra thick, color=blue] (2,1.5)--(2,3);
 \node at (2.5, 2.75) {$\scriptstyle \cdots $};
  \draw[ultra thick, color=blue] (3,1.5)--(3,3);
   \draw[thick, fill=orange] (-0.2,1)--(3.2,1)--(3.2,1.6)--(-0.2,1.6)--(-0.2,1);
    \node at (1.5,1.3) {$ y'$};
    \draw[ultra thick, color=blue] (1,0.1)--(1,0.2) .. controls ++(0,0.6) and ++(0,-0.6) .. (3.6,1)--(3.6,2)
    .. controls ++(0,0.6) and ++(0,-0.6) .. (2,2.8)--(2,2.9);
   \node[below] at (0, 0) {$\scriptstyle \lambda_1 $};
    \node[below] at (1, 0) {$\scriptstyle \lambda_s $};
    \node[below] at (3, 0) {$\scriptstyle \lambda_n $};
    \node[above] at (0, 3) {$\scriptstyle \nu_1 $};
    \node[above] at (2, 3) {$\scriptstyle \nu_r $};
    \node[above] at (3, 3) {$\scriptstyle \nu_n $}; 
     \node[right] at (3.6, 1) {$\scriptstyle t $};  
                                  \draw[thick, fill=yellow] (3.6, 1.85) circle (7pt);
       \node[] at (3.6, 1.85) {$\scriptstyle f$};
      \end{tikzpicture}}
\end{align*}
is in \(S^n_k(\lambda, \nu)\).
\end{lemma}
\begin{proof}
We may assume that \(y'\) is itself a composition of \(m\) many morphisms in \(Y_1^n\). 
First, note that if \(t=0\), the claim holds trivially, and if \(m=0\), then \(y = e_{[r,s], \lambda}^{(f,t)} \in Y^n_k\), so we now make an inductive assumption that \(m>0, t>0\) and the claim holds when \(m'<m\) or \(t'<t\).

As \(m>0\), we may write \(y'=y''e_{[v,u]}^{(g,\ell)}\) for some \(v,u \in [1,n]\), \(1 \leq \ell \leq k\), \(g \in A^{(\ell)}\), where \(y''\) is a composition of \(m-1\) morphisms from \(Y_k^n\). If \(v \neq s\), then \(e_{[v,u]}^{(g,\ell)}\) moves freely down below the split on the \(s\)th strand in \(y\). Then the induction assumption on \(m\) may be used to complete the proof of the claim. So we now assume that \(v = s\), so \(y'=y''e_{[s,u]}^{(g,\ell)}\), and we may write, using \cref{AssocRel,TAaSMRel,AaIntertwine}:
\begin{align*}
y=
\hackcenter{}
\hackcenter{
\begin{tikzpicture}[scale=0.8]
 \draw[ultra thick, color=blue] (0,-1)--(0,1);
  \node at (0.27, 0) {$\scriptstyle \cdots $};
 \draw[ultra thick, color=blue] (0.5,-1)--(0.5,1);
 \node at (2.77, 0) {$\scriptstyle \cdots $};
  \draw[ultra thick, color=blue] (2.5,-1)--(2.5,1);
  \draw[ultra thick, color=blue] (3,-1)--(3,1);
   \draw[ultra thick, color=blue] (0,1.5)--(0,3);
  \node at (1, 2.75) {$\scriptstyle \cdots $};
 \draw[ultra thick, color=blue] (2,1.5)--(2,3);
 \node at (2.5, 2.75) {$\scriptstyle \cdots $};
  \draw[ultra thick, color=blue] (3,1.5)--(3,3);
   \draw[thick, fill=orange] (-0.2,1)--(3.2,1)--(3.2,1.6)--(-0.2,1.6)--(-0.2,1);
    \node at (1.5,1.3) {$ y''$};
    \draw[ultra thick, color=blue] (0.5,-0.3)--(0.5,-0.2) .. controls ++(0,0.6) and ++(0,-0.6) .. (3.6,1)--(3.6,2)
    .. controls ++(0,0.6) and ++(0,-0.6) .. (2,2.8)--(2,2.9);
        \draw[ultra thick, color=blue] (0.5,0.7)--(0.5,0.8) .. controls ++(0,-0.6) and ++(0,0.4) .. (2.5,-0.9);
   \node[below] at (0, -1) {$\scriptstyle \lambda_1 $};
    \node[below] at (0.5, -1) {$\scriptstyle \lambda_s $};
        \node[below] at (2.5, -1) {$\scriptstyle \lambda_u $};
    \node[below] at (3, -1) {$\scriptstyle \lambda_n $};
    \node[above] at (0, 3) {$\scriptstyle \nu_1 $};
    \node[above] at (2, 3) {$\scriptstyle \nu_r $};
    \node[above] at (3, 3) {$\scriptstyle \nu_n $}; 
     \node[right] at (3.6, 1) {$\scriptstyle t $};  
                                  \draw[thick, fill=yellow] (3.6, 1.85) circle (7pt);
       \node[] at (3.6, 1.85) {$\scriptstyle f$};
                                         \draw[thick, fill=yellow] (2,-0.5) circle (7pt);
       \node[] at (2,-0.5) {$\scriptstyle g$};
          \node[] at (1.4,-0.4) {$\scriptstyle \ell$};
      \end{tikzpicture}}.
\end{align*}
(Note that the above equality holds even when \(u=s\)). Now, applying \cref{CrossingWebRel}, \(y\) may be rewritten as a \(\k\)-linear combination of diagrams of the form

\begin{align*}
\hackcenter{}
\hackcenter{
\begin{tikzpicture}[scale=0.8]
 \draw[ultra thick, color=blue] (0,-1)--(0,1);
  \node at (0.27, 0) {$\scriptstyle \cdots $};
 \draw[ultra thick, color=blue] (0.5,-1)--(0.5,1);
 \node at (2.77, 0) {$\scriptstyle \cdots $};
  \draw[ultra thick, color=blue] (2.5,-1)--(2.5,1);
  \draw[ultra thick, color=blue] (3,-1)--(3,1);
   \draw[ultra thick, color=blue] (0,1.5)--(0,3);
  \node at (1, 2.75) {$\scriptstyle \cdots $};
 \draw[ultra thick, color=blue] (2,1.5)--(2,3);
 \node at (2.5, 2.75) {$\scriptstyle \cdots $};
  \draw[ultra thick, color=blue] (3,1.5)--(3,3);
   \draw[thick, fill=orange] (-0.2,1)--(3.2,1)--(3.2,1.6)--(-0.2,1.6)--(-0.2,1);
    \node at (1.5,1.3) {$ y''$};
    \draw[ultra thick, color=blue] (2,0.4) .. controls ++(0,0.6) and ++(0,-0.6) .. (3.6,1.1)--(3.6,2)
    .. controls ++(0,0.6) and ++(0,-0.6) .. (2,2.8)--(2,2.9);
        \draw[ultra thick, color=blue] (2,0.4)--(2,-0.3) .. controls ++(0,-0.6) and ++(0,0.4) .. (2.5,-0.9);
          \draw[ultra thick, color=blue] (0.5,-0.9)  .. controls ++(0,0.4) and ++(0,-0.4) .. (1,-0.4);
            \draw[ultra thick, color=blue] (1,-0.4)--(1,0.4)  .. controls ++(0,0.4) and ++(0,-0.4) .. (0.5,0.9);
               \draw[ultra thick, color=blue] (1,-0.5)  .. controls ++(0,0.4) and ++(0,-0.4) .. (2,0);
                \draw[ultra thick, color=blue] (1,0.5)  .. controls ++(0,-0.4) and ++(0,0.4) .. (2,0);
   \node[below] at (0, -1) {$\scriptstyle \lambda_1 $};
    \node[below] at (0.5, -1) {$\scriptstyle \lambda_s $};
        \node[below] at (2.5, -1) {$\scriptstyle \lambda_u $};
    \node[below] at (3, -1) {$\scriptstyle \lambda_n $};
    \node[above] at (0, 3) {$\scriptstyle \nu_1 $};
    \node[above] at (2, 3) {$\scriptstyle \nu_r $};
    \node[above] at (3, 3) {$\scriptstyle \nu_n $}; 
     \node[right] at (3.6, 1) {$\scriptstyle t $};  
                                  \draw[thick, fill=yellow] (3.6, 1.85) circle (7pt);
       \node[] at (3.6, 1.85) {$\scriptstyle f$};
                                         \draw[thick, fill=yellow]   (2,-0.55) circle (7pt);
       \node[] at (2,-0.55) {$\scriptstyle g$};
          \node[] at (2.25,-0.88) {$\scriptstyle \ell$};
      \end{tikzpicture}},
\end{align*}
which may then, using \cref{AssocRel,KnotholeRel} be written as \(\k\)-linear combinations of diagrams of the form 
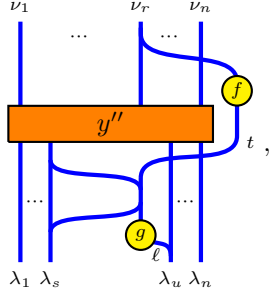
\begin{align*}
\hackcenter{}
\hackcenter{
\begin{tikzpicture}[scale=0.8]
 \draw[ultra thick, color=blue] (0,-1)--(0,1);
  \node at (0.27, 0) {$\scriptstyle \cdots $};
 \draw[ultra thick, color=blue] (0.5,-1)--(0.5,1);
 \node at (2.77, 0) {$\scriptstyle \cdots $};
  \draw[ultra thick, color=blue] (2.5,-1)--(2.5,1);
  \draw[ultra thick, color=blue] (3,-1)--(3,1);
   \draw[ultra thick, color=blue] (0,1.5)--(0,3);
  \node at (1, 2.75) {$\scriptstyle \cdots $};
 \draw[ultra thick, color=blue] (2,1.5)--(2,3);
 \node at (2.5, 2.75) {$\scriptstyle \cdots $};
  \draw[ultra thick, color=blue] (3,1.5)--(3,3);
   \draw[thick, fill=orange] (-0.2,1)--(3.2,1)--(3.2,1.6)--(-0.2,1.6)--(-0.2,1);
    \node at (1.5,1.3) {$ y''$};
    \draw[ultra thick, color=blue] (2,0.4) .. controls ++(0,0.6) and ++(0,-0.6) .. (3.6,1.1)--(3.6,2)
    .. controls ++(0,0.6) and ++(0,-0.6) .. (2,2.8)--(2,2.9);
        \draw[ultra thick, color=blue] (2,0.4)--(2,-0.3) .. controls ++(0,-0.6) and ++(0,0.4) .. (2.5,-0.9);
               \draw[ultra thick, color=blue] (0.5,-0.5)  .. controls ++(0,0.4) and ++(0,-0.4) .. (2,0);
                \draw[ultra thick, color=blue] (0.5,0.7)  .. controls ++(0,-0.4) and ++(0,0.4) .. (2,0.2);
   \node[below] at (0, -1) {$\scriptstyle \lambda_1 $};
    \node[below] at (0.5, -1) {$\scriptstyle \lambda_s $};
        \node[below] at (2.5, -1) {$\scriptstyle \lambda_u $};
    \node[below] at (3, -1) {$\scriptstyle \lambda_n $};
    \node[above] at (0, 3) {$\scriptstyle \nu_1 $};
    \node[above] at (2, 3) {$\scriptstyle \nu_r $};
    \node[above] at (3, 3) {$\scriptstyle \nu_n $}; 
     \node[right] at (3.6, 1) {$\scriptstyle t $};  
                                  \draw[thick, fill=yellow] (3.6, 1.85) circle (7pt);
       \node[] at (3.6, 1.85) {$\scriptstyle f$};
                                         \draw[thick, fill=yellow]   (2,-0.55) circle (7pt);
       \node[] at (2,-0.55) {$\scriptstyle g$};
          \node[] at (2.25,-0.88) {$\scriptstyle \ell$};
      \end{tikzpicture}},
\end{align*}
which in turn, using \cref{DiagSwitchRel}, may be written as a \(\k\)-linear combinations of diagrams of the form
\begin{align*}
\hackcenter{}
\hackcenter{
\begin{tikzpicture}[scale=0.8]
 \draw[ultra thick, color=blue] (0,-1)--(0,1);
  \node at (0.27, 0) {$\scriptstyle \cdots $};
 \draw[ultra thick, color=blue] (0.5,-1)--(0.5,1);
 \node at (2.77, 0) {$\scriptstyle \cdots $};
  \draw[ultra thick, color=blue] (2.5,-1)--(2.5,1);
  \draw[ultra thick, color=blue] (3,-1)--(3,1);
   \draw[ultra thick, color=blue] (0,1.5)--(0,3);
  \node at (1, 2.75) {$\scriptstyle \cdots $};
 \draw[ultra thick, color=blue] (2,1.5)--(2,3);
 \node at (2.5, 2.75) {$\scriptstyle \cdots $};
  \draw[ultra thick, color=blue] (3,1.5)--(3,3);
   \draw[thick, fill=orange] (-0.2,1)--(3.2,1)--(3.2,1.6)--(-0.2,1.6)--(-0.2,1);
    \node at (1.5,1.3) {$ y''$};
    \draw[ultra thick, color=blue] (2,0.5) .. controls ++(0,0.6) and ++(0,-0.6) .. (3.6,1.2)--(3.6,2)
    .. controls ++(0,0.6) and ++(0,-0.6) .. (2,2.8)--(2,2.9);
        \draw[ultra thick, color=blue] (2,0.5)--(2,-0.3) .. controls ++(0,-0.6) and ++(0,0.4) .. (2.5,-0.9);
               \draw[ultra thick, color=blue] (2,-0.3)  .. controls ++(0,0.4) and ++(0,-0.4) .. (0.5,0.1);
                \draw[ultra thick, color=blue] (2,0.5)  .. controls ++(0,-0.4) and ++(0,0.4) .. (0.5,0.1);
   \node[below] at (0, -1) {$\scriptstyle \lambda_1 $};
    \node[below] at (0.5, -1) {$\scriptstyle \lambda_s $};
        \node[below] at (2.5, -1) {$\scriptstyle \lambda_u $};
    \node[below] at (3, -1) {$\scriptstyle \lambda_n $};
    \node[above] at (0, 3) {$\scriptstyle \nu_1 $};
    \node[above] at (2, 3) {$\scriptstyle \nu_r $};
    \node[above] at (3, 3) {$\scriptstyle \nu_n $}; 
     \node[right] at (3.6, 1) {$\scriptstyle t $};  
                                  \draw[thick, fill=yellow] (3.6, 1.85) circle (7pt);
       \node[] at (3.6, 1.85) {$\scriptstyle f$};
                                         \draw[thick, fill=yellow]   (2,-0.55) circle (7pt);
       \node[] at (2,-0.55) {$\scriptstyle g$};
          \node[] at (2.25,-0.88) {$\scriptstyle \ell$};
      \end{tikzpicture}}.
\end{align*}
Finally, applications of \cref{AssocRel,TAaSMRel} allow us to write such diagrams as \(\k\)-linear combinations of diagrams of the form
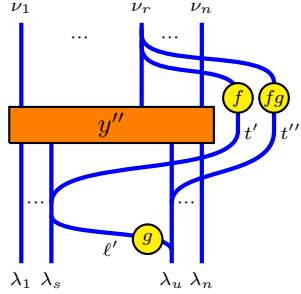
\begin{align*}
\hackcenter{}
\hackcenter{
\begin{tikzpicture}[scale=0.8]
 \draw[ultra thick, color=blue] (0,-1)--(0,1);
  \node at (0.27, 0) {$\scriptstyle \cdots $};
 \draw[ultra thick, color=blue] (0.5,-1)--(0.5,1);
 \node at (2.77, 0) {$\scriptstyle \cdots $};
  \draw[ultra thick, color=blue] (2.5,-1)--(2.5,1);
  \draw[ultra thick, color=blue] (3,-1)--(3,1);
   \draw[ultra thick, color=blue] (0,1.5)--(0,3);
  \node at (1, 2.75) {$\scriptstyle \cdots $};
 \draw[ultra thick, color=blue] (2,1.5)--(2,3);
 \node at (2.5, 2.75) {$\scriptstyle \cdots $};
  \draw[ultra thick, color=blue] (3,1.5)--(3,3);
   \draw[thick, fill=orange] (-0.2,1)--(3.2,1)--(3.2,1.6)--(-0.2,1.6)--(-0.2,1);
    \node at (1.5,1.3) {$ y''$};
    \draw[ultra thick, color=blue] (0.5,0) .. controls ++(0,0.8) and ++(0,-0.8) .. (3.6,1.2)--(3.6,1.8)
    .. controls ++(0,0.6) and ++(0,-0.6) .. (2,2.6)--(2,2.7);
       \draw[ultra thick, color=blue] (2.5,0) .. controls ++(0,0.6) and ++(0,-0.8) .. (4.2,1.2)--(4.2,2)
    .. controls ++(0,0.6) and ++(0,-0.6) .. (2,2.8)--(2,2.9);
               \draw[ultra thick, color=blue] (2.5,-0.8)  .. controls ++(0,0.4) and ++(0,-0.4) .. (0.5,-0.2);
   \node[below] at (0, -1) {$\scriptstyle \lambda_1 $};
    \node[below] at (0.5, -1) {$\scriptstyle \lambda_s $};
        \node[below] at (2.5, -1) {$\scriptstyle \lambda_u $};
    \node[below] at (3, -1) {$\scriptstyle \lambda_n $};
    \node[above] at (0, 3) {$\scriptstyle \nu_1 $};
    \node[above] at (2, 3) {$\scriptstyle \nu_r $};
    \node[above] at (3, 3) {$\scriptstyle \nu_n $}; 
     \node[right] at (3.55, 1.2) {$\scriptstyle t' $};  
       \node[right] at (4.15, 1.2) {$\scriptstyle t'' $};  
             \node[right] at (1.2, -0.75) {$\scriptstyle \ell' $};  
                                  \draw[thick, fill=yellow] (3.6, 1.75) circle (7pt);
       \node[] at (3.6, 1.75) {$\scriptstyle f$};
              \draw[thick, fill=yellow] (4.2, 1.75) circle (7pt);
       \node[] at (4.2, 1.75) {$\scriptstyle f\hspace{-0.2mm}g$};
                                         \draw[thick, fill=yellow]  (2.1,-0.6)circle (7pt);
       \node[] at (2.1,-0.6) {$\scriptstyle g$};
      \end{tikzpicture}}
\end{align*}
where \(t' + t'' = t\) and \(\ell' + t'' = \ell\). If \(t' = 0\) or \(t''=0\), then such diagrams are in \(S^n_k(\lambda, \nu)\) by the induction assumption on \(m\). If \(t',t''>0\), then applying the induction assumption on \(t\) to the \(t'\)-thick strand, and subsequently to the \(t''\)-thick strand shows that such diagrams are in \(S^n_k(\lambda, \nu)\). This completes the induction step and the proof of the lemma.
\end{proof}

\begin{theorem}
Morphisms in \(\WebAaOnen\) are generated under composition by the set \(Y^n_\infty\) in \cref{Ysets}. If \(\k\) is a field of characteristic zero, then morphisms in \(\WebAaOnen\) are generated under composition by the set
\begin{align*}
\tilde{Y}_1^n := \left\{ e_{[r,s],\lambda}^{(1,1)} \mid \lambda \in \Lambda_n, r,s \in [1,n]\right\} \cup \left\{e_{[r,r], \lambda}^{(f,1)} \mid \lambda \in \Lambda_n, r \in [1,n], f \in A\right\} \subseteq Y^n_1.
\end{align*}
\end{theorem}
\begin{proof}
For the first claim, we have that morphism spaces in \(\WebAaOnen\) are spanned by elements of the form \(\eta^{a}_{\bmu}\) by \cref{SpanLem}.
We show that such spanning elements belong to the set \(S^n_\infty\) by induction on \(d=|\lambda|\).
Let \( \eta^{a}_{\bmu}\)  for some \(\bmu \in \mathcal{M}(\lambda, \nu)\), with \(|\lambda| = d\). If \(d=0\), then \(\eta^a_{\bmu}\) is the identity morphism. Thus we may assume \(d>0\), and that the claim holds for all \(d'<d\). Then there exists some \(r,s \in [1,n]\), \(b \in \mathcal{B}\) such that \(\mu_{s,r}(b) >0\), and so, pulling the strand with the coupon \(b\) and thickness \(\mu_{s,r}(b)\) off to the right in \(\eta^{a}_{\bmu}\), we may write
\begin{align*}
\eta^{\a}_{\bmu}=\pm
\hackcenter{}
\hackcenter{
\begin{tikzpicture}[scale=0.8]
 \draw[ultra thick, color=blue] (0,0)--(0,1);
  \node at (0.5, 0.25) {$\scriptstyle \cdots $};
 \draw[ultra thick, color=blue] (1,0)--(1,1);
 \node at (2.5, 0.25) {$\scriptstyle \cdots $};
  \draw[ultra thick, color=blue] (3,0)--(3,1);
   \draw[ultra thick, color=blue] (0,1.5)--(0,3);
  \node at (1, 2.75) {$\scriptstyle \cdots $};
 \draw[ultra thick, color=blue] (2,1.5)--(2,3);
 \node at (2.5, 2.75) {$\scriptstyle \cdots $};
  \draw[ultra thick, color=blue] (3,1.5)--(3,3);
   \draw[thick, fill=orange] (-0.2,1)--(3.2,1)--(3.2,1.6)--(-0.2,1.6)--(-0.2,1);
    \node at (1.5,1.3) {$ \eta'$};
    \draw[ultra thick, color=blue] (1,0.1)--(1,0.2) .. controls ++(0,0.6) and ++(0,-0.6) .. (3.6,1)--(3.6,2)
    .. controls ++(0,0.6) and ++(0,-0.6) .. (2,2.8)--(2,2.9);
   \node[below] at (0, 0) {$\scriptstyle \lambda_1 $};
    \node[below] at (1, 0) {$\scriptstyle \lambda_s $};
    \node[below] at (3, 0) {$\scriptstyle \lambda_n $};
    \node[above] at (0, 3) {$\scriptstyle \nu_1 $};
    \node[above] at (2, 3) {$\scriptstyle \nu_r $};
    \node[above] at (3, 3) {$\scriptstyle \nu_n $}; 
     \node[right] at (3.6, 1) {$\scriptstyle \mu_{s,r}(b) $};  
                                  \draw[thick, fill=yellow] (3.6, 1.85) circle (7pt);
       \node[] at (3.6, 1.85) {$\scriptstyle b$};
      \end{tikzpicture}},
\end{align*}
where \(\eta' = \eta^a_{\bmu'}\) for some \(\bmu' \in \mathcal{M}(\lambda - \mu_{s,r}(b)\varepsilon_s, \nu - \mu_{s,r}(b)\varepsilon_r)\). By the induction assumption on \(d\), we have that \(\eta' \in S_\infty^n\), and hence \(\eta^{a}_{\bmu} \in S_\infty^n\) by \cref{thickgen}. This completes the induction step and the proof of the first claim above.

For the second claim, assume that \(\k\) is a field of characteristic zero. We first show that the morphisms in \(Y_\infty^n\) are morphisms in \(S^n_1\). Note that, when \(r \neq s\), we have \(e_{[r,s]}^{(f,t)} = \frac{1}{t!} (e_{[r,s]}^{(f,1)})^t \) by repeated applications of \cref{KnotholeRel}, so \(e_{[r,s]}^{(f,t)} \in S^n_1\). Next, we show by induction on \(t\) that \(e_{[r,r]}^{(f,t)} \in S_1^n\). Indeed, for \(t >1\), we have:
\begin{align*}
\hackcenter{}
e_{[r,r],\lambda}^{(f,t)}=
\hackcenter{
\begin{tikzpicture}[scale=0.9]
\draw[ultra thick, blue] (1,-0.5)--(1,1.5);
\node[] at (1.5,0.5) {$\scriptstyle \cdots$};
\node[] at (3.2,0.5) {$\scriptstyle \cdots$};
 \draw[ultra thick, blue] (2.3,-0.5)--(2.3,0.15) .. controls ++(0,0.1) and ++(0,-0.1) .. (2.6,0.4)--(2.6,0.6)  .. controls ++(0,0.15) and ++(0,-0.15) .. (2.3,0.85)--(2.3,1.5) ;
  \draw[ultra thick, blue] (2.3,-0.5)--(2.3,0.15) .. controls ++(0,0.1) and ++(0,-0.1) .. (2,0.4)--(2,0.6)  .. controls ++(0,0.15) and ++(0,-0.15) .. (2.3,0.85)--(2.3,1.25) ;
\node[right] at (2.3, 0.13) {$\scriptstyle t $};
\node[left] at (2.15, 0.13) {$\scriptstyle \lambda_r - t $};
\draw[ultra thick, blue] (3.6,-0.5)--(3.6,1.5);
      \node[below] at (2.3, -0.5) {$\scriptstyle \lambda_r $};
       \node[above] at (2.3, 1.5) {$\scriptstyle \lambda_r $};
    \node[below] at (1, -0.5) {$\scriptstyle \lambda_{1} $};
               \node[below] at (3.6, -0.5) {$\scriptstyle \lambda_{n} $};
    \node[above] at (1,1.5) {$\scriptstyle \lambda_{1}$};
               \node[above] at (3.6, 1.5) {$\scriptstyle \lambda_{n}$};
                               \draw[thick, fill=yellow] (2.6, 0.5) circle (6pt);
       \node[] at (2.6, 0.5) {$\scriptstyle f$};
\end{tikzpicture}}
\substack{ (\textup{\ref{KnotholeRel}}) \\ =}
\left( \frac{1}{t} \right)
\hackcenter{
\begin{tikzpicture}[scale=0.9]
\draw[ultra thick, blue] (1,-0.5)--(1,1.5);
\node[] at (1.5,0.5) {$\scriptstyle \cdots$};
\node[] at (3.2,0.5) {$\scriptstyle \cdots$};
 \draw[ultra thick, blue] (2.3,-0.5)--(2.3,0.15) .. controls ++(0,0.1) and ++(0,-0.1) .. (2.6,0.4)--(2.6,0.6)  .. controls ++(0,0.15) and ++(0,-0.15) .. (2.3,0.85)--(2.3,1.5) ;
  \draw[ultra thick, blue] (2.3,-0.5)--(2.3,0.15) .. controls ++(0,0.1) and ++(0,-0.1) .. (2,0.4)--(2,0.6)  .. controls ++(0,0.15) and ++(0,-0.15) .. (2.3,0.85)--(2.3,1.25) ;
   \draw[ultra thick, blue] (2.3,-0.5) .. controls ++(0,0.6) and ++(0,-0.6) .. (4,0.3)--(4,0.7)  .. controls ++(0,0.6) and ++(0,-0.6) .. (2.3,1.5);
\node[right] at (2.35, 0.12) {$\scriptstyle t-1 $};
\node[left] at (2.15, 0.12) {$\scriptstyle \lambda_r - t $};
\node[right] at (2.6, -0.3) {$\scriptstyle 1 $};
\draw[ultra thick, blue] (3.6,-0.5)--(3.6,1.5);
      \node[below] at (2.3, -0.5) {$\scriptstyle \lambda_r $};
       \node[above] at (2.3, 1.5) {$\scriptstyle \lambda_r $};
    \node[below] at (1, -0.6) {$\scriptstyle \lambda_{1} $};
               \node[below] at (3.6, -0.5) {$\scriptstyle \lambda_{n} $};
    \node[above] at (1.1,1.45) {$\scriptstyle \lambda_{1}$};
               \node[above] at (3.6, 1.5) {$\scriptstyle \lambda_{n}$};
                               \draw[thick, fill=yellow] (2.6, 0.5) circle (6pt);
       \node[] at (2.6, 0.5) {$\scriptstyle f$};
                                    \draw[thick, fill=yellow] (4, 0.5) circle (6pt);
       \node[] at (4, 0.5) {$\scriptstyle f$};
\end{tikzpicture}}.
\end{align*}
Note that the central part of this last diagram is \(e_{[r,r], \lambda- \varepsilon_r}^{(f,t-1)}\), which is in \(S_1^n\) by the induction assumption. But then \(e_{[r,r],\lambda}^{(f,t)} \in S_1^n\) by \cref{thickgen}.

Thus we have that \(\WebAaOnen\) is generated under composition by \(Y_1^n\) when \(\k\) is a field of characteristic zero. To show that it is generated by the smaller set \(\tilde{Y}_1^n\), we note that for all \(f \in A\), \(r,s \in [1,n]\), we have
\begin{align*}
e_{[r,s]}^{(f,1)} = e_{[r,s]}^{(1,1)} e_{[s,s]}^{(f,1)} - e_{[s,s]}^{(f,1)} e_{[r,s]}^{(1,1)},
\end{align*}
by an application of \cref{MSrel}.
\end{proof}

The following is an immediate corollary of the previous result.
\begin{corollary}\label{C:FullnessofWebFunctor}  Let $\k$ be a field of characteristic zero.  Then the functor 
\[
W_{n}: \UglnA \to \WebAaOnen
\] is full for all $n \geq 0$
\end{corollary}

\section{Polynomial representations and Schur--Weyl duality}\label{S:SWDualityandPolyReps}

\subsection{Polynomial representations}\label{SS:polyreps} In this section we assume $(A,a)_I$ is a good pair, \(I\) is finite, and \(1 = \sum_{i \in I} i\).  Let $\KK$ be a field equipped with a unital ring homomorphism $\k \to \KK$.

Recall from \cref{SS:fullnessforwreathcategory} that $V_{n}^{\otimes d}$ is a $(\SS_{d}\wr A, U(\gl_{n}(A))$-bisupermodule. Thus, after base change, there is a superalgebra homomorphism
\[
\zeta^{d}_{n, \KK}: U(\gl_{n}(A_{\KK} )) \to S^{A}(n,d)_{\KK}.
\] 

\begin{lemma}\label{L:EnvelopingAlgebraisontoSchurAlgebra}  If $\KK$ is characteristic zero, then the homomorphism $\zeta^{d}_{n, \KK}$ is surjective.
\end{lemma}

\begin{proof}  Consider the good pair $(A, \k 1)_{\{1\}}$. First, observe that $S^{\k 1}(n,d)_{\KK}=S(n,d)_{\KK}$ is the ordinary Schur algebra. It is well known that the corresponding map $U(\gl_{n}(\KK )) \to S(n,d)_{\KK}$ is surjective and thus the restriction of $\zeta^{d}_{n, \KK}$ to the subalgebra $ U(\gl_{n}(\KK 1 )) \subseteq  U(\gl_{n}(A ))$,
\[
\zeta^{d}_{n, \KK}: U(\gl_{n}(\KK 1 )) \to S^{\k  1}(n,d)_{\KK},
\] is surjective.

From the definition of $T^{A}_{\k 1}(n,d)$ given in \cref{SS:SchurificationofaGoodPair} and the fact that $\KK$ has characteristic zero it follows that $S^{A}(n,d)_{\KK} = T^{A}_{\k 1} (n,d)_{\KK}$.  By \cite[Theorem 2]{KM}, $ T^{A}_{\k 1} (n,d)_{\KK}$ is generated as a superalgebra by $\left\{\zeta^{d}_{n}(x) \mid x \in \gl_{n}(A_{\KK}) \right\}$ along with the subsuperalgebra $S^{\k 1}(n,d)_{\KK}$. Since both of these sets are in the image of $\zeta^{d}_{n}$, the result follows.
\end{proof} 

By inflation through $\zeta^{d}_{n, \KK}$ we can view the category of right $S^{A}(n,d)_{\KK}$-supermodules as a subcategory of right $U(\gl_{n}(A_{\KK}))$-supermodules which we call the \emph{polynomial representations of degree $d$}.  When $\KK$ is characteristic zero the previous result implies this is a full subcategory.   We will freely identify these categories without comment in what follows.

\subsection{Schur--Weyl duality}\label{SS:SchurWeylDuality}

Recall from \cref{SS:fullnessforwreathcategory} that the image of $\rho^{A}_{n,d}$ lies in the finite part of the centralizer of $S^{A}(n,d)$:
\begin{equation}\label{E:SWmap2}
\rho^{A}_{n,d}: \SS_{d} \wr A  \to  \End_{S^{A}(n,d)}\left(V_{n}^{\otimes d} \right)^{\fin}.
\end{equation}

Let $\KK$ be an algebraically closed field of characteristic different from two equipped with a unital ring homomorphism $\k \to \KK$.  Upon base change to $\KK$ there is a superalgebra map
\begin{equation}\label{E:SWmap}
\rho^{A_{\KK}}_{n,d}: \SS_{d} \wr A_{\KK}  \to  \End_{S^{A}(n,d)_{\KK}}\left(V_{n, \KK}^{\otimes d} \right)^{\fin} = \End_{T^{A}_{a}(n,d)_{\KK}}\left(V_{n, \KK}^{\otimes d} \right)^{\fin}.
\end{equation}

We remark that $\k$ is an integral domain, $V_{n}^{\otimes d}$ is a free $\k$-supermodule, and $T^{A}_{a}(n,d)$ is a $\k$-lattice of $S^{A}(n,d)$.  From this it follows that $\End_{S^{A}(n,d)}(V_{n}^{\otimes d}) = \End_{T^{A}_{a}(n,d)}(V_{n}^{\otimes d})$ and, hence, after base change, the equality of endomorphism algebras given in \cref{E:SWmap} holds.  We say \emph{Schur--Weyl duality holds for $A$ over $\KK$} if $\rho^{A_{\KK}}_{n,d}$ is surjective for all $n, d \geq 0$.  By \cref{P:fullnesseqaulssurjective} this is equivalent to the functor $F_{n}$ being full for all $n \geq 0$. 

\begin{remark}\label{R:SchurWeylDuality}
The results of \cite[Section 5.3]{EK} show $\rho^{A_{\k }}_{n,d}$ is surjective whenever $n \geq d$ and $|I| < \infty$.
\end{remark}

 Determining when Schur--Weyl duality holds is an interesting question which is beyond the scope of this paper.  The following result verifies it in a situation which cover many cases of interest.

\begin{theorem}\label{L:SchurWeylDuality}  Let $\KK$ be an algebraically closed field of characteristic zero equipped with a unital ring homomorphism $\k \to \KK$. If $A_{\KK}$ is a finite-dimensional semisimple $\KK$-superalgebra, then Schur--Weyl duality holds for $A$ over $\KK$.
\end{theorem}

\begin{proof}  Since $A_{\KK}$ is semisimple it follows from \cref{R:ArbitraryA} that $\mathfrak{S}_d \wr A_{\KK}$ is also semisimple.  Since  $ S^{A}(n,d)_{\KK} \cong \End_{ \SS_d \wr A_{\KK}}\left(V_{n, \KK}^{\otimes d} \right)$ by \cite[Lemma 5.7]{EK},  Schur--Weyl duality holds by the double centralizer theorem for finite-dimensional semisimple superalgebras (e.g., see \cite[Proposition 3.5]{CW}).
\end{proof}

\subsection{Representations of finite-dimensional semisimple superalgebras}\label{SS:RepsofSemisimpleSuperalgebras}
 For the remainder of this section we assume $\k$ is an algebraically closed field of characteristic zero and that $A$ is finite-dimensional and semisimple.  We remind the reader that we allow all (not just parity preserving) supermodule homomorphisms and that our convention is to write maps on the side opposite to the action.

We give a brief overview of the representation theory of finite-dimensional superalgebras.  Details can be found in, for example, \cite{BK, CW, KleshBook}.  We will discuss what happens for left supermodules and the reader can make the obvious translations to right supermodules and bisupermodules.  Given an $A$-supermodule $M$, write $\Pi M$ for the parity shift of $M$. For short we sometimes write $M^{p|q} = M^{\oplus p} \oplus (\Pi M)^{\oplus q}$. A simple $A$-supermodule $S$ is \emph{type $\typeM$} if it remains irreducible as an $A$-\emph{module} (ie.,  if $\Z_{2}$-gradings are ignored).  A simple supermodule $S$ is of type $\typeM$ if and only if $\End_{A}(S)$ is $1|0$-dimensional.  Otherwise, $S \cong T \oplus \Pi T$ for some simple $A$-module $T$.  In this case $S$ is said to be of type $\typeQ$.  A simple supermodule $S$ is of type $\typeQ$ if and only if $\End_{A}(S)$ is $1|1$-dimensional, in which case $S$ admits an odd involution.  As another characterization, a simple supermodule $S$ is of type $\typeQ $ if $S$ and $\Pi S$ are isomorphic via an even map and of type $\typeM$ if they are only isomorphic via an odd map. 

Correspondingly, there are two types of finite-dimensional simple superalgebras: the superalgebra of linear endomorphisms of an $m|n$-dimensional superspace $V$, $M(V)$; and the superalgebra of linear endomorphisms of an $n|n$-dimensional superspace $V$ which preserve an odd involution, $Q(V)$.  In both cases $V$ is the unique simple supermodule and is of type $\typeM$ and type $\typeQ$, respectively.  More generally, if $A$ is a semisimple finite-dimensional associative unital superalgebra, then the Artin--Wedderburn theorem for superalgebras states that
\begin{equation}\label{E:SemisimpleDecomposition}
A \cong \left(\bigoplus_{V \text{ of type $\typeM$}} M(V)  \right) \oplus \left(  \bigoplus_{V \text{ of type $\typeQ $}} Q(V) \right)
\end{equation}
as superalgebras where the direct sum is over a complete, irredundant set of simple supermodules for $A$.  If $V$ is of type $\typeM$, then upon choosing a suitably ordered homogeneous basis of $m$ even and $n$ odd vectors we can write 
\begin{equation}\label{E:supermatrix}
M(V) \cong  M_{m|n}(\k ) = \left\{ \left(  \begin{matrix} W & X\\
                                                 Y  & Z
\end{matrix} \right)  \right\}.
\end{equation}
where $W$, $X$, $Y$, and $Z$ are matrices over $\k$ with $W$ being $m \times m$, $X$ being $m \times n$, $Y$ being $n \times m$, and $Z$ being $n \times n$.  Then $V \cong \k^{m|n}$ with $M(V)$ acting by left multiplication.  If $V$ is of type $\typeQ$, then upon choosing a suitably ordered homogenous basis of $n$ even and $n$ odd vectors which are interchanged by the odd involution we have $Q(V) \cong Q_{n}(\k )$ being matrices of the form \cref{E:supermatrix} with the additional condition that $W=Z$ and $X=Y$. Again $V \cong \k^{n|n}$.

If $A$ and $B$ are associative superalgebras and $S$ and $T$ are simple $A$- and $B$-supermodules, respectively, then the tensor product $S \otimes T$ is an $A \otimes B$-supermodule.  If both $S$ and $T$ are of type $\typeM$, then $S \otimes T$ is a simple $A \otimes B$-supermodule of type $\typeM$. If exactly one of $S$ and $T$ are of type $\typeQ $, then $S \otimes T$ is a simple $A \otimes B$-supermodule of type $\typeQ$. If both $S$ and $T$ are of type $\typeQ$, then $S \otimes T \cong (S \star T) \oplus \Pi (S \star T)$ for a simple $A \otimes B$-supermodule $S \star T$ of type $\typeM$ (we make an arbitrary fixed choice of the simple subsupermodule denoted $S \star T$).   We extend notation by declaring $S \star T$ to be the simple $A \otimes B$-supermodule $S \otimes T$ in the other cases. Applying this construction and ranging over all possible simple supermodules $S$ and $T$ yields a complete, irredundant set of simple $A \otimes B$-supermodules.

For a finite-dimensional semisimple superalgebra $A$, let $\Xi_{A}$ be an index set for the simple left $A$-supermodules and write $S^{\alpha}$ for the simple labeled by $\alpha \in \Xi_{A}$.  For definiteness which will simplify notation later, we sometimes fix an order on $\Xi_{A}$ and write $\Xi_{A}=\left\{1,\dotsc ,\ell \right\}$ according to that order.  
Define a function 
\[
\delta: \Xi_{A} \to \{0,1 \}
\]
by $\delta (\alpha)=0$ if $S^{\alpha}$ is of type $\typeM$ and $\delta (\alpha) =1$ if $S^{\alpha}$ is of type $\typeQ$.

Applying the tensor product rule discussed above shows that 
\begin{equation}\label{E:AdSimples}
\left\{ S^{\balpha}:=S^{\alpha_{1}} \star S^{\alpha_{2}} \star \dotsb \star S^{\alpha_{d}} \mid \balpha = (\alpha_{1}, \dotsc , \alpha_{d}) \in \Xi_{A}^{d} \right\}
\end{equation}
is a complete, irredundant set of simple $A^{\otimes d}$-supermodules.  Furthermore, if we define 
\[
\delta (\balpha) = \begin{cases} 0, \text{ if  $\sum_{k=1}^{d} \delta (\alpha_{k})$ is even};\\
                             1, \text{ if  $\sum_{k=1}^{d} \delta (\alpha_{k})$ is odd},
\end{cases}  
\] then $S^{\balpha}$ is of type $\typeM$ if  $\delta (\balpha ) = 0$ and of type $\typeQ $ if $\delta (\balpha ) = 1$.

Given an element $\sigma \in \SS_{d}$ and an $A^{\otimes d}$-supermodule $M$, let ${}^{\sigma}M=M$ as a $\k$-supermodule with $A^{\otimes d}$ action given by $(a_{1}\otimes \dotsb \otimes a_{d}).m = ((a_{1}\otimes \dotsb \otimes a_{d})\sigma)m$, where $\sigma$ acts on $A^{\otimes d}$ by signed place permutation.  There is an even isomorphism of $A^{\otimes d}$-supermodules,
\[
{}^{\sigma}\left(S^{\alpha_{1}} \otimes S^{\alpha_{2}} \otimes \dotsb \otimes S^{\alpha_{d}} \right) \cong S^{\alpha_{\sigma(1)}} \otimes S^{\alpha_{\sigma(2)}} \otimes \dotsb \otimes S^{\alpha_{\sigma(d)}}.
\] Hence, after possibly replacing some of the simple supermodules in \cref{E:AdSimples} with their parity shifts, there is an even isomorphism of $A^{\otimes d}$-supermodules, 
\begin{equation}\label{E:TwistedAtensordSimples}
{}^{\sigma}S^{\balpha}  \cong S^{\sigma\cdot \balpha},
\end{equation} for all $\balpha \in \Xi_{A}^{d}$ and all $\sigma \in \SS_{d}$, where $\sigma$ acts on the left on $\Xi_{A}^{d}$ by place permutation.

\subsection{Representations of wreath product superalgebras}\label{SS:RepsofWreathProductSuperalgebras}
For nonnegative integers $n$ and $d$, let
\[
\Lambda^{0}_{+}(n,d) \qquad \text{and} \qquad \Lambda^{1}_{+}(n,d)
\] be the set of partitions of $d$ with at most $n$ nonzero parts, and the set of strict partitions of $d$ with at most $n$ nonzero parts, respectively.  Recall that a partition is  \emph{strict} if it has no nonzero repeated parts.
For fixed $d \geq 0$, set
\[
\Lambda^{0}_{+}(d) = \bigcup_{n \geq 0} \Lambda^{0}_{+}(n,d) \qquad \text{and} \qquad  \Lambda^{1}_{+}(d) = \bigcup_{n \geq 0} \Lambda^{1}_{+}(n,d).
\]
Define $\delta (\lambda) = 0$ for all $\lambda \in \Lambda^{0}_{+}(d)$.    Given any partition $\lambda$, let $\ell (\lambda)$ be the number of nonzero parts in $\lambda$.  For  $\lambda \in \Lambda^{1}_{+}(d)$, define
\[
\delta (\lambda) = \begin{cases} 0, & \text{ if $\ell(\lambda)$ is even};\\
                                 1, & \text{ if $\ell(\lambda)$ is odd}.
\end{cases}
\]

Recall that we have a fixed total order on $\Xi_{A}$ and use it to identify $\Xi_{A}$ with $\left\{1, \dotsc , \ell \right\}$.  Given a sequence of nonnegative integers $\bn= (n_{1}, \dotsc , n_{\ell})$ indexed by $\Xi_{A}$ and a fixed $d\geq 0$,  let
\[
\Lambda^{A}_{+}(\bn, d) = \left\{\tuplambda = (\lambda^{(1)}, \dotsc , \lambda^{(\ell)}) \; \left| \; \lambda^{(\alpha)} \in \Lambda^{\delta (\alpha)}_{+}(n_{\alpha},d_{\alpha}), \sum_{\alpha=1}^{\ell} d_{\alpha} = d \right. \right\}
\] be the set of $\Xi_{A}$-indexed multi-partitions of $d$ where the partition $\lambda^{(\alpha)}$ has at most $n_{\alpha}$ nonzero parts, and $\lambda^{(\alpha)}$ is strict whenever $S^{\alpha}$ is of type $\typeQ$. Set 
\[
\Lambda^{A}_{+}(d) = \bigcup_{\bn \in \Z_{\geq 0}^{\Xi_{A}}} \Lambda^{A}_{+}(\bn, d).
\]  Define 
\[
\delta: \Lambda^{A}_{+}(d) \to \{0,1 \}
\]
by 
\[
\delta (\lambda^{(1)}, \dotsc , \lambda^{(\ell)}) = \begin{cases} 0, &\text{ if $\sum_{\alpha=1}^{\ell}  \delta (\lambda^{(\alpha)})$ is even};\\
                                                                1, &\text{ if $\sum_{\alpha=1}^{\ell} \delta (\lambda^{(\alpha)})$ is odd}.
\end{cases}
\]

If $A$ is semisimple, then $\SS_{d} \wr A$ is semisimple for all $d \geq 0$.  Thus the representation theory of $\SS_{d} \wr A$ comes down to describing the simple supermodules.  These were given by Rosso--Savage in \cite[Proposition 4.3]{RossoSavage} and shown to be indexed by $\Lambda^{A}_{+}(d)$.  As we will need a description of these simple supermodules which is somewhat different than theirs, we provide an alternate proof.

We will use without comment the following observation.  Let $M$ be a $\SS_{d} \wr A$-supermodule. If, when it is viewed as a $A^{\otimes d}$-supermodule via restriction, $M$ contains a submodule isomorphic to $S^{\balpha}:=S^{\alpha_{1}} \star \dotsb \star S^{\alpha_{d}}$ for some $\balpha = (\alpha_{1}, \dotsc , \alpha_{d}) \in \Xi_{A}^{d}$, then it also contains $\sigma S^{\balpha} \cong {}^{\sigma}S^{\balpha} \cong S^{\sigma \cdot \balpha}$ for any $\sigma \in \SS_{d}$.  A similar statement obviously applies to supermodules for the Young subalgebras of $\SS_{d} \wr A$.

\subsubsection{Simple supermodules for Young subalgebras}\label{SSS:simplesforYoungsubalgebras}  Let $S= S^{\alpha}$ be a simple $A$ supermodule.  Let $e \in A$ be a minimal even idempotent such that $Ae \cong S$.   By the tensor product rule given in \cref{SS:RepsofSemisimpleSuperalgebras}, $S^{\star d}$ is a simple supermodule for $A^{\otimes d}$ and $e^{\otimes d}(S^{\beta_{1}}\star \dotsb \star S^{\beta_{d}}) \neq 0$ if and only if $\beta_{1}=\dotsb = \beta_{d}=\alpha$. 

Let $\mathcal{e}$ denote the even idempotent $1_{\k \SS_{d}} \otimes e^{\otimes d} \in \SS_{d} \wr A$.  Recalling the theory of idempotent truncation (e.g., see \cite[Section 6.2]{Green} or \cite[Section 2]{BK}), there is an exact functor
\[
F_{\mathcal{e}}: \Awreathsmod  \to  \eAewreathsmod  
\] given by $F_{\mathcal{e}}M = \mathcal{e}M$.  This functor defines a type-preserving bijection between the simple $\SS_{d} \wr A$-supermodules not annihilated by $F_{\mathcal{e}}$ and the simple $\mathcal{e}(\SS_{d}\wr A)\mathcal{e} \cong \SS_{d} \wr (eAe)$-supermodules.   From the previous paragraph we see that for a simple $\SS_{d}\wr A$-supermodule $T$, $F_{\mathcal{e}}T \neq 0$ if and only if $T$ contains a simple $A^{\otimes d}$-supermodule which is isomorphic to $S^{\star d}$.

Contemplating the matrix form of $A$ given in \cref{SS:RepsofSemisimpleSuperalgebras} shows that if $S$ is of type $\typeM$, then $eAe \cong \k$ and $\mathcal{e}(\SS_{d} \wr A)\mathcal{e} \cong \SS_{d}\wr (eAe) \cong \k \SS_{d}$, where we view $\k \SS_{d}$ as a superalgebra concentrated in parity $\bar{0}$.  For $\lambda \in \Lambda^{0}_{+}(d)$ let $D^{\lambda^{(\alpha)}}$ be the simple $\SS_{d} \wr A$-supermodule for which $F_{e}D^{\lambda^{(\alpha)}}$ is evenly isomorphic to the simple $\k \SS_{d}$-module $T^{0, \lambda^{(\alpha)}}$ (viewed as a supermodule concentrated in parity $\bar{0}$) labeled by the partition $\lambda$.  Then the set 
\[
\left\{D^{\lambda^{(\alpha)}} \mid \lambda \in \Lambda_{+}^{0}(d) \right\}
\] is a complete, irredundant set of simple $\SS_{d} \wr A$-supermodules which contain an $A^{\otimes d}$-submodule isomorphic to $S^{\star d}$.  Furthermore, they are of type $\typeM$ whenever $\delta (\lambda) =0$ (i.e., always).

On the other hand, if $S$ is of $\typeQ$, then $eAe \cong \Cliff$ and $\mathcal{e}(\SS_{d} \wr A)\mathcal{e} \cong \SS_{d} \wr (eAe) \cong \Ser_{d}$, the Sergeev superalgebra.  For $\lambda \in \Lambda^{1}_{+}(d)$ let $D^{\lambda^{(\alpha)}}$ be the simple $\SS_{d} \wr A$-supermodule for which $F_{e}D^{\lambda^{(\alpha)}}$ is evenly isomorphic to the simple $\Ser_{d}$-supermodule, $T^{1, \lambda^{(\alpha)}}$, labeled by the strict partition $\lambda$.  Then the set 
\[
\left\{D^{\lambda^{(\alpha)}} \mid \lambda \in \Lambda_{+}^{1}(d) \right\}
\] is a complete, irredundant set of simple $\SS_{d} \wr A$-supermodules which contain a submodule isomorphic to $S^{\star d}$ when viewed as a $A^{\otimes d}$-supermodule.  Furthermore, $D^{\lambda^{(\alpha)}}$ is of type $\typeM$ if $\delta(\lambda)=0$ and is of type $\typeQ$ if $\delta (\lambda)=1$.

Now, fix a composition of $d$, $\bd = (d_{1}, \dotsc , d_{\ell})$ with parts labeled by $\Xi_{A}$.  Let $\SS_{\bd}= \SS_{d_{1}} \times \dotsb \times \SS_{d_{\ell}}$ be the corresponding Young subgroup of $\SS_{d}$ and let  $\SS_{\bd} \wr A =  \k \SS_{\bd} \otimes A^{\otimes d}$ be  the corresponding Young subalgebra of $\SS_{d} \wr A$. Given $\tuplambda = (\lambda^{(1)}, \dotsc , \lambda^{(\ell)}) \in  \Lambda^{A}_{+}(d)$ such that $(\lVert \lambda^{(1)} \rVert, \dotsc , \lVert \lambda^{(\ell)} \rVert ) = \bd$, we can apply the tensor product rule from \cref{SS:RepsofSemisimpleSuperalgebras} to $\SS_{\bd} \wr A \cong (\SS_{d_{1}} \wr A) \otimes \dotsb \otimes (\SS_{d_{\ell}} \wr A)$ to see that 
\[
 D^{\lambda^{(1)}} \star \dotsb \star D^{\lambda^{(\ell)}}
\] is a simple $\SS_{\bd} \wr A$-supermodule.  Furthermore, it is of type $\typeM$ if $\delta(\lambda^{(1)}, \dotsc , \lambda^{(\ell)}) = 0$ and of type $\typeQ$ if $\delta(\lambda^{(1)}, \dotsc , \lambda^{(\ell)}) = 1$.

Finally, observe that $D=D^{\lambda^{(1)}} \star \dotsb \star D^{\lambda^{(\ell)}}$ contains a simple $A^{\otimes d}$-supermodule isomorphic to $S^{\bd}:=(S^{1})^{\star d_{1}} \star \dotsb \star (S^{\ell})^{\star d_{\ell}}$.  Indeed, since the sum of all subsupermodules isomorphic to $S^{\bd}$ will be stable under the action of $\SS_{\bd} \wr A$, it follows that \emph{every} simple $A^{\otimes d}$-subsupermodule of $D$ is isomorphic to $S^{\bd}$.  Conversely, any simple $\SS_{\bd} \wr A$-supermodule which contains a simple $A^{\otimes d}$-supermodule isomorphic to $S^{\bd}$ is isomorphic to $ D^{\mu^{(1)}} \star \dotsb \star D^{\mu^{(\ell)}}$ for some $\tupmu \in \Lambda_{+}^{A}(d)$ which satisfies $(\lVert \mu^{(1)} \rVert, \dotsc , \lVert \mu^{(\ell)} \rVert ) = (d_{1}, \dotsc , d_{\ell})$,

\subsubsection{Simple supermodules for wreath product superalgebras}\label{SSS:simplesforwreathproducts}

For $\tupmu \in \Lambda^{A}_{+}(d)$, let $\bd  = (\lVert \mu^{(1)} \rVert, \dotsc , \lVert \mu^{(\ell)} \rVert )$ and let $\SS_{\tupmu} \wr A = \SS_{\bd} \wr A$.  Set 
\begin{equation}\label{E:WreathProductSimpleModules}
\left\{ D^{\tupmu} = \SS_{d} \wr A \otimes_{\SS_{\tupmu} \wr A} (D^{\mu^{(1)}} \star \dotsb \star D^{\mu^{(\ell)}}) \mid \tupmu \in \Lambda^{A}_{+}(d) \right\}.
\end{equation}
We claim this is a complete, irredundant set of simple supermodules for $\SS_{d}\wr A$.

Before proving the claim we make one observation which will be helpful in the proof.  Fix $\tupmu =(\mu^{(1)}, \dotsc , \mu^{(\ell)}) \in \Lambda^{A}_{+}(d)$ and let $\bd = (d_{1}, \dotsc , d_{\ell})= ({\lVert \mu^{(1)}\rVert}, \dotsc , {\lVert \mu^{(\ell)}\rVert})$. As discussed in the previous subsection,   $D^{\mu^{(1)}} \star \dotsb \star D^{\mu^{(\ell)}}$ will then be a direct sum of copies of $S^{\bd} = (S^{1})^{\star d_{1}} \star \dotsb \star (S^{\ell})^{\star d_{\ell}}$ when viewed as an $A^{\otimes d}$-supermodule.   Therefore, since 
\begin{equation}\label{E:DlambdaDecomposition}
D^{\tupmu} = \bigoplus_{\sigma \in \SS_{d}/\SS_{\tupmu}} \sigma \otimes (D^{\mu^{(1)}} \star \dotsb \star D^{\mu^{(\ell)}}) \cong \bigoplus_{\sigma \in \SS_{d}/\SS_{\tupmu}} {}^{\sigma} (D^{\mu^{(1)}} \star \dotsb \star D^{\mu^{(\ell)}})
\end{equation}
is a decomposition as $A^{\otimes d}$-supermodules, it follows that $S^{\beta_{1}} \star\dotsb \star S^{\beta_{d}}$ appears as an $A^{\otimes d}$-summand of $D^{\tupmu}$ if and only if $(\beta_{1}, \dotsc , \beta_{d}) \in \SS_{d}\cdot  (1^{\lVert \mu^{(1)}\rVert}, \dotsc , \ell^{\lVert \mu^{(\ell)}\rVert})$.

\begin{proposition}\label{P:WreathProductAlgebraSimples} Let $A$ be a finite-dimensional semisimple superalgebra.  Then, the set
\[
\left\{D^{\tuplambda} \mid \tuplambda \in \Lambda_{+}^{A}(d) \right\}
\] is a complete irredundant set of simple $\SS_{d} \wr A$-supermodules and $D^{\tuplambda}$ is of type $\typeM$ if $\delta (\tuplambda)=0$ and is of type $\typeQ$ if $\delta (\tuplambda ) =1$.
\end{proposition}

\begin{proof}

First, let $M$ be a simple $\SS_{d} \wr A$-supermodule.  We can choose a composition of $d$  indexed by $\Xi_{A}$, $\bd  = (d_{1}, \dotsc , d_{\ell})$, for which $S^{\bd} = (S^{1})^{\star d_{1}} \star \dotsb \star (S^{\ell})^{\star d_{\ell}}$ appears as an $A^{\otimes d}$-subsupermodule of $M$.   Since $\SS_{\bd} \wr A$ is a semisimple superalgebra, the discussion at the end of \cref{SSS:simplesforYoungsubalgebras} implies there is a $\tuplambda = (\lambda^{(1)}, \dotsc , \lambda^{(\ell)}) \in \Lambda^{A}_{+}(d)$ with $(\lVert \lambda^{(1)}\rVert, \dotsc ,\lVert \lambda^{(\ell)}\rVert)  = \bd$  for which $D^{\lambda^{(1)}} \star \dotsb \star D^{\lambda^{(\ell)}}$  appears as a $\SS_{\tuplambda}\wr A = \SS_{\bd } \wr A$-subsupermodule of $M$.  Frobenius reciprocity along with the simplicity of $M$ implies there is a homogenous surjective $\SS_{d} \wr A$-supermodule homomorphism 
\[
D^{\tuplambda} \twoheadrightarrow M.
\]  In short, every simple $\SS_{d}\wr A$-supermodule is the quotient of $D^{\tuplambda}$ for some $\tuplambda \in \Lambda^{A}_{+}(d)$.

Next, we show \cref{E:WreathProductSimpleModules} forms an irredundant set of simple $\SS_{d}\wr A$-supermodules by studying $\Hom_{\SS_{d}\wr A}(D^{\tuplambda}, D^{\tupmu})$. First, obviously for there to be nontrivial homomorphisms it must be that $D^{\tuplambda}$ and $D^{\tupmu}$  have $A^{\otimes d}$ composition factors in common.  From \cref{E:DlambdaDecomposition} it follows that  $(\lVert \lambda^{(1)} \rVert, \dotsc , \lVert \lambda^{(\ell)} \rVert) = (\lVert \mu^{(1)} \rVert, \dotsc , \lVert \mu^{(\ell)}\rVert)$.  Assume this is the case and let $\bd = (d_{1}, \dotsc , d_{\ell})$ be this composition of $d$.  By Frobenius reciprocity we have 
\begin{equation}\label{E:Hom1}
\Hom_{\SS_{d}\wr A}(D^{\tuplambda}, D^{\tupmu}) \cong \Hom_{\SS_{\bd} \wr A}( D^{\lambda^{(1)}} \star \dotsb \star D^{\lambda^{(\ell)}}, D^{\tupmu}).
\end{equation}  By considering $A^{\otimes d}$ composition factors and \cref{E:DlambdaDecomposition} it follows that the image of a $\SS_{\bd}\wr A$-supermodule homomorphism $D^{\lambda^{(1)}} \star \dotsb \star D^{\lambda^{(\ell)}} \to D^{\tupmu}$ must have its image contained within the subspace $1_{\k\SS_{d}} \otimes D^{\mu^{(1)}} \star\dotsb \star D^{\mu^{(\ell)}}$.  That is, 
\begin{equation}\label{E:Hom2}
\Hom_{\SS_{\bd} \wr A}( D^{\lambda^{(1)}} \star \dotsb \star D^{\lambda^{(\ell)}}, D^{\tupmu}) \cong \Hom_{\SS_{\bd} \wr A}( D^{\lambda^{(1)}} \star \dotsb \star D^{\lambda^{(\ell)}}, D^{\mu^{(1)}}\star \dotsb \star D^{\mu^{(\ell)}}).
\end{equation}

Combining \cref{E:Hom1,E:Hom2} along with the discussion in \cref{SSS:simplesforYoungsubalgebras}  demonstrates that $\Hom_{\SS_{d}\wr A}(D^{\tuplambda}, D^{\tupmu})=0$ whenever $\tuplambda \neq \tupmu$.  When $\tuplambda =\tupmu$ this endomorphism space has graded dimension $1|0$ whenever $\delta (\tuplambda ) =0$ and graded dimension $1|1$ whenever  $\delta (\tuplambda ) =1$.  By Schur's Lemma it follows that the set 
\[
\left\{D^{\tuplambda} \mid \tuplambda \in \Lambda^{A}_{+}(d) \right\}
\] is an irredundant set of simple $\SS_{d}\wr A$-supermodules and the type is as asserted.  The first paragraph implies this is also a complete set of simple supermodules.  This proves the claim.
\end{proof}

\begin{remark}\label{R:ArbitraryA}  While we will not use this fact, it is worth noting that the above construction gives the simple supermodules for $\SS_{d} \wr A$ whenever $A$ is a finite-dimensional superalgebra over an algebraically closed field of characteristic zero.  Write $J(B)$ for the Jacobson radical of a superalgebra $B$. By \cite[Lemma 2.6]{BK} this is a homogeneous ideal and coincides with the Jacobson radical of $B$ as an algebra. The arguments of \cite[Section 3]{CT}  show that
\[
(\SS_{d} \wr A)/J(\SS_{d} \wr A) \cong \SS_{d} \wr \left(A/J(A) \right).
\]  From this we see the semisimplicity condition can be dropped.
\end{remark}

Since dualizing is a contravariant equivalence between left- and right-supermodules which preserves type, the following is immediate.

\begin{corollary}\label{C:WreathProductRightSimples}  The set 
\[
\left\{D^{\tuplambda , *} := (D^{\tuplambda})^{*} \mid \tuplambda \in \Lambda^{A}_{+}(d) \right\}
\] is a complete, irredundant set of simple right $\SS_{d} \wr A$-supermodules.  The supermodule $D^{\tuplambda , *}$ is of type $\typeM$ if $\delta (\tuplambda ) =0$ and of type $\typeQ$ if $\delta(\tuplambda ) =1$.
\end{corollary}

\subsection{Schur--Weyl--Berele--Regev--Sergeev duality}\label{SS:StrongMultiplicityFreeSWduality}  In this section we recall some well-known strong multiplicity-free decompositions.

Fix $M, N \geq 0$ and $d \geq 1$.  Let 
\[
\Lambda_{+}(M|N,d)^{\text{hook}} = \left\{\lambda = (\lambda_{1}, \dotsc , \lambda_{M+N}) \in \Lambda_{+}^{0}(M+N,d) \mid  \lambda_{M+1} \leq N \right\}.
\] That is, these are the partitions $\lambda$ of $d$ with not more than $M+N$ parts that also satisfy $\lambda_{M+1} \leq N$ (the so-called \emph{hook condition}). 

Let $W_{M|N}= \k^{M|N}$. For any $d \geq 1$, let $\SS_{d}$ act on the right on $W_{M|N}^{\otimes d}$ by signed place permutation and let 
\[
S(M|N,d) = \End_{\SS_{d}}\left(W_{M|N}^{\otimes d} \right)
\]
be the type $A$ Schur superalgebra.  For a partition $\lambda \in \Lambda_{+}^{0}(d)$ recall we write $T^{0, \lambda}$ for the simple left $\k \SS_{d}$-supermodule labeled by $\lambda$; it is viewed as a type $\typeM$ supermodule concentrated in parity $\bar{0}$.

Applying the double centralizer theorem for semisimple superalgebras (e.g., see \cite[Theorem 3.10]{CW}), there is a decomposition into simple $(S(M|N,d), \k  \SS_{d})$-bisupermodules:
\begin{equation}\label{E:TypeASchur-WeylDuality}
W_{M|N}^{\otimes d} \cong \bigoplus_{\lambda \in \Lambda_{+}(M|N,d)^{\text{hook}}}   L_{\gl (M|N)}(\lambda) \star (T^{0, \lambda})^{*}.
\end{equation}
The set 
\[
\left\{L_{\gl (M|N)}(\lambda) \mid  \lambda \in \Lambda_{+}(M|N,d)^{\text{hook}} \right\}
\] is a complete, irredundant set of simple $S(M|N,d)$-supermodules.  Every $L_{\gl (M|N)}(\lambda)$ is of type $\typeM$.  Correspondingly, for every $\lambda \in \Lambda_{+}(M|N, d)^{\text{hook}}$ we declare $\delta (\lambda) = 0$.
type

Next, fix $N \geq 0$ and $d \geq 1$.  Set $\Lambda_{+}(N,d)^{\text{strict}}$ to be the set of all strict partitions of $d$ with not more than $N$ parts.   Let $W_{N} = \k^{N|N}$.  There is a right action of $\Ser_{d}$, the Sergeev superalgebra, on $(W_{N})^{\otimes d}$.  Let 
\[
Q(N,d) = \End_{\Ser_{d}}\left(W_{N}^{\otimes d} \right)
\]
be the type $Q$ Schur superalgebra.  For a strict partition $\lambda$ of $d$ recall we write $T^{1, \lambda}$ for the simple left $\Ser_{d}$-supermodule labeled by $\lambda$.  For strict partitions define 
\[
\delta(\lambda) = \begin{cases} 0, & \text{ if $\ell(\lambda)$ is even}; \\
                                 1, & \text{ if $\ell(\lambda)$ is odd}.
\end{cases}
\] The supermodule $T^{1, \lambda}$ is of type $\typeM$ if $\delta (\lambda) = 0$ and of type $\typeQ$ if $\delta (\lambda) =1$.

Applying the double centralizer theorem for semisimple superalgebras (e.g., see \cite[Theorem 3.46]{CW}), there is a decomposition into simple $(Q(N,d), \Ser_{d})$-bisupermodules: 
\begin{equation}\label{E:TypeQSchur-WeylDuality}
W_{N}^{\otimes d} \cong \bigoplus_{\lambda \in \Lambda_{+}(N,d)^{\text{strict}}}   L_{\mathfrak{q}(N)}(\lambda) \star (T^{1, \lambda})^{*}.
\end{equation}
The set 
\[
\left\{L_{\mathfrak{q}(N)}(\lambda) \mid  \lambda \in \Lambda_{+}(N,d)^{\text{strict}} \right\}
\] is a complete, irredundant set of simple $Q(N,d)$-supermodules.   The supermodules $L_{\mathfrak{q}(N)}(\lambda)$ and $T^{1, \lambda}$ have the same type; namely, they are of type $\typeM$ if $\delta (\lambda) = 0$ and of type $\typeQ$ if $\delta(\lambda)=1$.

\subsection{Polynomial representations and generalized Schur--Weyl duality}\label{SS:PolyReps}

We next study how the previous section generalizes when $A$ is an arbitrary finite-dimensional semisimple superalgebra. The double centralizer theorem for semisimple superalgebras (see \cite[Proposition 3.5]{CW}) combined with the classification of simple right $\SS_{d} \wr A$-supermodules given in \cref{C:WreathProductRightSimples}  yields the following strong multiplicity-free result.
\begin{proposition}\label{L:SWMultiplicityFree}  If $A$ be a finite-dimensional semisimple superalgebra, then there is a subset 
\[
\mathbb{X}^{A}(n,d) \subseteq \Lambda^{A}_{+}(d)
\]
such that there is a decomposition
\[
V_{n}^{\otimes d}   \cong \bigoplus_{\tuplambda \in \mathbb{X}^{A}(n,d)} L^{A}_{n}(\tuplambda) \star D^{\tuplambda, *}
\] into simple $(S^{A}(n,d), \SS_{d} \wr A)$-bisupermodules.

The set 
\[
\left\{L(\tuplambda ):= L_{n}^{A}(\tuplambda )\mid \tuplambda \in \mathbb{X}^{A}(n,d)  \right\}
\] is a complete, irredundant set of simple $S^{A}(n,d)$-supermodules.  The supermodule $L(\tuplambda )$ is of type $\typeM$ if $\delta (\tuplambda ) =0$ and of type $\typeQ$ if $\delta (\tuplambda ) =1$.
\end{proposition}

The obvious problem is to determine the set $\mathbb{X}^{A}(n,d)$.  We do this next.  For each $\alpha \in \Xi_{A}$, say $S^{\alpha}$ has graded dimension $m_{\alpha}|n_{\alpha}$.  For a fixed $n \geq 1$, let $M_{\alpha}=nm_{\alpha}$ and $N_{\alpha}=nn_{\alpha}$.  For each $\alpha \in \Xi_{A}$, and fixed $n \geq 1$  and $d \geq 0$, let 
\[
\Lambda^{\alpha}_{+}(n,d) = \begin{cases}   \Lambda_{+}(M_{\alpha}|N_{\alpha},d)^{\text{hook}}, & \text{ if $\delta(\alpha)=0$};\\
                                       \Lambda_{+}(N_{\alpha},d)^{\text{strict}}, & \text{ if $\delta(\alpha) = 1$}.
\end{cases} 
\]
For each $\alpha \in \Xi_{A}$ and fixed $n \geq 1$, let
\[
\Lambda^{\alpha}_{+}(n,\bullet) = \bigcup_{d \geq 0} \Lambda_{+}^{\alpha}(n,d).
\]  In other words, elements of $\Lambda^{\alpha}_{+}(n, \bullet)$ are partitions which have a limited number of nonzero parts where the limit depends on $S^{\alpha}$ and $n$, and which are hook partitions if $S^{\alpha}$ is of type $\typeM$ and are strict partitions if $S^{\alpha}$ is of type $\typeQ$. 

Finally,
for fixed $n \geq 1$ and $d \geq 0$, set 
\[
\Lambda^{A}_{+}(n,d) = \left\{\tuplambda = (\lambda^{(1)}, \dotsc , \lambda^{(\ell)}) \in \prod_{\alpha \in \Xi_{A}} \Lambda^{\alpha}_{+}(n,\bullet) \;  \left| \; \sum_{i=1}^{\ell} \lVert \lambda^{(i)} \rVert = d  \right. \right\}.
\] In other words, $\Lambda^{A}_{+}(n,d)$ is the set of multipartitions of $d$ indexed by the set $\Xi_{A}$ where each $\lambda^{(\alpha)}$  has a limited number of nonzero parts where the limit depends on $S^{\alpha}$ and $n$, and which is a hook partition if $S^{\alpha}$ is of type $\typeM$ and is a strict partition if $S^{\alpha}$ is of type $\typeQ$.  For a fixed composition of $d$ with parts indexed by $\Xi_{A}$, $\bd =(d_{1}, \dotsc , d_{\ell})$, let 
\[
\Lambda^{A}_{+}(n,\bd) = \left\{\tuplambda = (\lambda^{(1)}, \dotsc , \lambda^{(\ell)}) \in \Lambda^{A}_{+}(n,d) \;  \left| \; \lVert \lambda^{(\alpha)} \rVert = d_{\alpha} \text{ for $\alpha \in \Xi_{A}$} \right. \right\}.
\]
Obviously, $\Lambda^{A}_{+}(n,\bd )\subseteq \Lambda^{A}_{+}(n,d) \subseteq \Lambda^{A}_{+}(d)$ and 
\[
\bigcup_{\bd} \Lambda^{A}_{+}(n,\bd ) = \Lambda^{A}_{+}(n,d),
\]
where the union is over all compositions of $d$ indexed by $\Xi_{A}$.

The next result confirms $\Lambda^{A}_{+}(n,d)$ is the set we seek.

\begin{theorem}\label{T:LabellingSimpleSAndmodules}  If $A$ is a finite-dimensional semisimple $\k$-superalgebra, then for all $n \geq 1$ and $d \geq 0$,
\[
\mathbb{X}^{A}(n,d) = \Lambda^{A}_{+}(n,d).
\]

That is, the set 
\[
\left\{ L(\tuplambda):=L_{n}^{A}(\tuplambda) \; \left| \;  \tuplambda \in \Lambda^{A}_{+}(n,d) \right. \right\}
\] is a complete, irredundant set of simple $S^{A}(n,d)$-supermodules.  The supermodule $L_{n}^{A}(\tuplambda)$ is of type $\typeM$ if $\delta (\tuplambda ) =0$ and of type $\typeQ$ if $\delta (\tuplambda )=1$.
\end{theorem}

\begin{proof}  First, since $A \cong A^{*}$ as $(A,A)$-bisupermodules, we have 
\[
\left(V_{n}^{\otimes d}\right)^{*} \cong V_{n}^{\otimes d}
\] as left $\SS_{d} \wr A$-supermodules.  In particular, $D^{\tuplambda, *}$ appears as a summand of $V_{n}^{\otimes d}$ as a right $\SS_{d} \wr A$-supermodule if and only if $D^{\tuplambda}$ appears in $V_{n}^{\otimes d}$ as a left  $\SS_{d} \wr A$-supermodule.

Now let $\tuplambda = (\lambda^{(1)}, \dotsc , \lambda^{(\ell)}) \in  \Lambda^{A}_{+}(d)$ and let $\bd  = (\lVert\lambda^{(1)} \rVert, \dotsc , \lVert\lambda^{(\ell)}\rVert) = (d_{1}, \dotsc , d_{\ell})$ be the corresponding composition of $d$.  In light of \cref{L:SWMultiplicityFree} and the first paragraph, it suffices to compute the dimension of 
\[
\Hom_{\SS_{d} \wr  A}\left(D^{\tuplambda}, V_{n}^{\otimes d} \right)  \cong \Hom_{\SS_{\tuplambda} \wr A} \left(D^{\tuplambda^{(1)}} \star \dotsb \star D^{\lambda^{(\ell)}}, V_{n}^{\otimes d} \right),
\] where the isomorphism is parity preserving and follows from Frobenius reciprocity.

For $\alpha \in \Xi_{A}$, fix a minimal even idempotent $e_{\alpha} \in A$ so that $Ae_{\alpha}  \cong S^{\alpha}$ as left $A$-supermodules.  Set 
\[
\mathcal{e} =  1_{\k \SS_{\bd}} \otimes e_{1}^{\otimes d_{1}} \otimes \dotsb \otimes e_{\ell}^{\otimes d_{\ell}} \in \SS_{\bd} \wr A.
\]  A direct calculation verifies that 
\[
\mathcal{e} (\SS_{\tuplambda} \wr A) \mathcal{e} \cong \left( \SS_{d_{1}} \wr (e_{1}Ae_{1})\right) \otimes \dotsb  \otimes \left( \SS_{d_{\ell}} \wr (e_{\ell}Ae_{\ell}) \right).
\] 
Contemplating the description of $A$ as matrices given after \cref{E:SemisimpleDecomposition} shows that for all $\alpha \in \Xi_{A}$,
\[
e_{\alpha}Ae_{\alpha} \cong \begin{cases} \k, &\text{ if $\delta(\alpha) = 0$};\\
                                \Cliff, &\text{ if $\delta(\alpha) = 1$}.   
\end{cases}
\] For short, let $ \SS^{0}_{d} = \k \SS_{d}$ and $ \SS^{1}_{d} = \Ser_{d}$. Then,
\[
\mathcal{e} (\SS_{\tuplambda} \wr A) \mathcal{e} \cong \SS^{\delta(1)}_{d_{1}} \otimes   \dotsb \otimes  \SS^{\delta(\ell)}_{d_{\ell}}.
\] For $z=0,1$ and $\mu \in \Lambda_{+}^{z}(d)$, let $T^{z, \mu}$ be the simple $\SS^{z}_{d}$-supermodule labeled by $\mu$ using the labeling choice given in \cref{SS:RepsofWreathProductSuperalgebras}.

Using the theory of idempotent trunction (e.g., see \cite[Section 6.2]{Green} or see \cite[Section 2]{BK} for the super version), since $\mathcal{e}D^{\tuplambda} \neq 0$  it must be a simple $\mathcal{e}(\SS_{\tuplambda} \wr A)\mathcal{e}$-supermodule of the same type as $D^{\tuplambda}$ and, moreover, there is an isomorphism 
\[
 \Hom_{\SS_{\tuplambda} \wr A} \left(D^{\tuplambda}, V_{n}^{\otimes d} \right) \cong  \Hom_{\mathcal{e}(\SS_{\tuplambda} \wr A)\mathcal{e}} \left(\mathcal{e}D^{\tuplambda}, \mathcal{e} \left( V_{n}^{\otimes d}\right) \right).
\]  Since $\mathcal{e}$ is even the isomorphism is parity preserving. It follows from the discussion in \cref{SS:RepsofWreathProductSuperalgebras} that 
\[
\mathcal{e}D^{\tuplambda} = e_{1}^{\otimes d_{1}} \otimes \dotsb \otimes e_{\ell}^{\otimes d_{\ell}}\left(D^{\lambda^{(1)}} \star \dotsb \star D^{\lambda^{\ell}} \right) \cong T^{\delta(1), \lambda^{(1)}} \star \dotsb \star T^{\delta(\ell), \lambda^{(\ell)}}.
\] Thus we are reduced to determining the multiplicity of $T^{\delta(1), \lambda^{(1)}} \star \dotsb \star T^{\delta(\ell), \lambda^{(\ell)}}$ in $\mathcal{e}\left( {V}_{n}^{\otimes d} \right)$ as $\mathcal{e} \left(   \SS_{d} \wr A \right)\mathcal{e} \cong \SS^{\delta(1)}_{d_{1}} \otimes   \dotsb \otimes  \SS^{\delta(\ell)}_{d_{\ell}}$-supermodules.

For each $\alpha \in \Xi_{A}$, say $S^{\alpha}$ has graded dimension $m_{\alpha}|n_{\alpha}$. Then,
\begin{align}\label{E:eVinitialdecomposition}
\mathcal{e} \left( V_{n}^{\otimes d} \right) & = \left( e_{1}V_{n}\right)^{\otimes d_{1}} \otimes \dotsb \otimes \left( e_{\ell}V_{n}\right)^{\otimes d_{\ell}} \\
& = \left( (e_{1}A)^{\oplus n}\right)^{\otimes d_{1}} \otimes \dotsb \otimes \left( (e_{\ell}A)^{\oplus n}\right)^{\otimes d_{\ell}} \\
& = \left( (\k^{m_{1}|n_{1}})^{\oplus n}  \right)^{\otimes d_{1}} \otimes \dotsb \otimes \left( (\k^{m_{\ell}|n_{\ell}})^{\oplus n} \right)^{\otimes d_{\ell}}.
\end{align}
Let us analyze this supermodule by first studying the tensor factor corresponding to a fixed $\alpha \in \Xi_{A}$.

The first case is if $\delta (\alpha)=0$, $S^{\alpha}$ is of type $\typeM$, and $S^{\alpha}$ has graded dimension $m_{\alpha}|n_{\alpha}$.  Setting $M_{\alpha}=nm_{\alpha}$ and $N_{\alpha}=nn_{\alpha}$, we have $(\k^{m_{\alpha}|n_{\alpha}})^{\oplus n} \cong W_{M_{\alpha}|N_{\alpha}}$ as superspaces and $\SS^{0}_{d_{\alpha}} = \k \SS_{d_{\alpha}}$ acts on $(W_{M_{\alpha}|N_{\alpha}} )^{\otimes d_{\alpha}}$ by signed place permutation as in \cref{E:TypeASchur-WeylDuality}. By \cref{E:TypeASchur-WeylDuality} and the observation at the start of the proof, $T^{0, \lambda^{(\alpha)}}$ appears in $(W_{M_{\alpha}|N_{\alpha}})^{\otimes d_{\alpha} }$ if and only if $\lambda^{(\alpha)} \in \Lambda_{+}(M_{\alpha}|N_{\alpha},d_{\alpha})^{\text{hook}}$; moreover, its multiplicity equals $2^{-\delta (\lambda^{(\alpha)})} \dim_{\k} (L_{\gl (M_{\alpha}|N_{\alpha})}(\lambda^{(\alpha)}))$.  For the remainder of this proof when we are in this case we sometimes write $L^{0}_{M_{\alpha}|N_{\alpha}}(\lambda^{(\alpha)})$ for $L_{\gl (M_{\alpha}|N_{\alpha})}(\lambda^{(\alpha)})$.

The second case is if $\delta (\alpha) = 1$, $S^{\alpha}$ is of type $\typeQ$, and $S^{\alpha}$ has graded dimension $n_{\alpha}|n_{\alpha}$.  Setting $N_{\alpha}=nn_{\alpha}$, we have $(\k^{n_{\alpha}|n_{\alpha}})^{\oplus n} \cong \k^{N_{\alpha}|N_{\alpha}} = W_{N_{\alpha}}$ as superspaces and $\SS^{1}_{d_{\alpha}} = \Ser_{d_{\alpha}}$ acts on $W_{N_{\alpha}}^{\otimes d_{\alpha}}$ as in \cref{E:TypeQSchur-WeylDuality}.  By \cref{E:TypeQSchur-WeylDuality} and the observation at the start of the proof,  $T^{1, \lambda^{(\alpha)}}$ appears in $(W_{N_{\alpha}})^{d_{\alpha}}$ if and only if $\lambda^{(\alpha)} \in \Lambda_{+}(N_{\alpha},d_{\alpha})^{\text{strict}}$; moreover, its multiplicity equals $2^{-\delta (\lambda^{(\alpha)})} \dim_{\k} (L_{\mathfrak{q}(N_{\alpha})}(\lambda^{(\alpha)}))$.  For the remainder of this proof when we are in this case we declare $M_{\alpha}=N_{\alpha}$ and sometimes write $L^{1}_{M_{\alpha}|N_{\alpha}}(\lambda^{(\alpha)})$ for $L_{\mathfrak{q}(N_{\alpha})}(\lambda^{(\alpha)})$.

Applying these two cases to \cref{E:eVinitialdecomposition} yields the $\mathcal{e} \left(  \SS_{d} \wr A\right)  \mathcal{e} \cong \SS^{\delta(1)}_{d_{1}} \otimes \dotsb \otimes  \SS^{\delta(\ell)}_{d_{\ell}}$-supermodule isomorphism
\begin{equation}\label{E:tensorspacedecomposition}
\mathcal{e}\left(V_{n}^{\otimes d} \right) \cong \bigoplus_{\tuplambda = (\lambda^{(1)}, \dotsc , \lambda^{(\ell)}) \in \Lambda_{+}^{A}(n,\bd)} 2^{-h(\tuplambda )}d(\tuplambda) \left(   T^{\delta(1), \lambda^{(1)}} \otimes \dotsb \otimes T^{\delta (\ell), \lambda^{(\ell)}} \right),
\end{equation}
where 
\[
d(\tuplambda) =  \dim_{\k} \left( L^{\delta (1)}_{M_{1}|N_{1}}(\lambda^{(1)}) \otimes \dotsb \otimes L^{\delta(\ell)}_{M_{\ell}|N_{\ell}} (\lambda^{(\ell)}) \right)
\] and 
\[
h(\tuplambda ) = \sum_{i=1}^{\ell} \delta(\lambda^{(i)}).
\]  Recall that the function $\delta$ was defined on hook and strict partitions in \cref{SS:StrongMultiplicityFreeSWduality}.  Also note that the isomorphism in \cref{E:tensorspacedecomposition} is parity preserving.

Finally, applying the tensor product rule from \cref{SS:RepsofSemisimpleSuperalgebras} to \cref{E:tensorspacedecomposition} yields the following decomposition as $\mathcal{e} \left(  \SS_{d} \wr A\right)  \mathcal{e} \cong \SS^{\delta(1)}_{d_{1}} \otimes \dotsb \otimes  \SS^{\delta(\ell)}_{d_{\ell}}$-supermodules:
\[
\mathcal{e}\left(V_{n}^{\otimes d} \right) \cong \bigoplus_{\tuplambda = (\lambda^{(1)}, \dotsc , \lambda^{(\ell)}) \in \Lambda_{+}^{A}(n,\bd)} 2^{\lfloor h(\tuplambda )/2 \rfloor - h(\tuplambda )}d(\tuplambda) \left(    T^{\delta(1), \lambda^{(1)}} \star \dotsb \star T^{\delta (\ell), \lambda^{(\ell)}} \right).
\] 

Combined with the earlier reductions, we conclude that the multiplicity of $D^{\tuplambda, *}$ in $V_{n}^{\otimes d}$ is nonzero if and only if $\tuplambda \in \Lambda^{A}_{+}(n,d)$. When it is nonzero the multiplicity equals $ 2^{\lfloor h(\tuplambda )/2 \rfloor - h(\tuplambda )}d(\tuplambda)$.   The assertions of the theorem follow from this and \cref{L:SWMultiplicityFree}.
\end{proof}

\begin{remark}\label{R:gradedmultiplicites} Given a finite-dimensional semisimple superalgebra $B$, a $B$-supermodule $M$, and a simple $B$-supermodule $S$, the \emph{graded multiplicity} of $S$ in $M$ is defined to be $p|q$ if the $S$-isotypic component of $M$ is isomorphic via a parity preserving map to $S^{\oplus p} \oplus (\Pi S)^{\oplus q}$.  If the reader attends to the graded dimensions in the previous proof they will find that the graded multiplicity of $D^{\tuplambda}$ in $V_{n}^{\otimes d}$ is equal to
\begin{gather*}
2^{\lfloor h(\tuplambda )/2 \rfloor - h(\tuplambda )}  \sdim_{\k} \left( L^{\delta (1)}_{M_{1}|N_{1}}(\lambda^{(1)}) \otimes \dotsb \otimes L^{\delta(\ell)}_{M_{\ell}|N_{\ell}} (\lambda^{(\ell)}) \right) \\
= 2^{-\delta (\tuplambda )}\sdim_{\k} \left( L^{\delta (1)}_{M_{1}|N_{1}}(\lambda^{(1)}) \star \dotsb \star  L^{\delta(\ell)}_{M_{\ell}|N_{\ell}} (\lambda^{(\ell)}) \right).
\end{gather*}
Therefore, 
\[
\sdim_{\k} L^{A}_{n} (\tuplambda ) = \sdim_{\k} \left( L^{\delta (1)}_{M_{1}|N_{1}}(\lambda^{(1)}) \star \dotsb \star  L^{\delta(\ell)}_{M_{\ell}|N_{\ell}} (\lambda^{(\ell)}) \right).
\] for all $\tuplambda \in \Lambda_{+}^{A}(n,d)$.
\end{remark}

This is a numerical foreshadow of the fact $S^{A}(n,d)$ is isomorphic to direct sum of tensor products of Schur superalgebras of types $\mathtt{A}$ and $\mathtt{Q}$.  We show this next.  For $\alpha \in \Xi_{A}$ and $d \geq 0$, say $S^{\alpha}$ has graded dimension $m_{\alpha}|n_{\alpha}$ and let $M_{\alpha}=nm_{\alpha}$ and $N_{\alpha}=nn_{\alpha}$. Set
\[
S^{\alpha}(n,d) = \begin{cases} S(M_{\alpha}|N_{\alpha},d), &\text{ if $\delta(\alpha) = 0$};\\
                           Q(N_{\alpha},d), &\text{ if $\delta(\alpha)=1$}.
\end{cases} 
\]

\begin{proposition}\label{T:SchurAlgebraIsomorphism} Let $\k$ be an algebraically closed field of characteristic zero and let $A$ be a finite-dimensional semisimple $\k$-superalgebra.  Then, for all $n \geq 1$ and $d \geq 0$, there is an isomorphism of superalgebras 
\[
S^{A}(n,d) \cong \bigoplus_{\substack{\bd = (d_{1}, \dotsc , d_{\ell}) \in \Z_{\geq 0}^{\Xi_{A}}\\ \lVert \bd \rVert = d }} S^{1}(n,d_{1}) \otimes \dotsb \otimes S^{\ell}(n,d_{\ell}).
\] 
\end{proposition}

\begin{proof}  Let  $A = \oplus_{\alpha \in \Xi_{A}} A_{\alpha}$ be the decomposition given in \cref{E:SemisimpleDecomposition}.  By \cite[Lemma 4.8]{KM} there is a superalgebra isomorphism 
\begin{equation}\label{E:KMSchurIsom}
S^{A}(n,d) \cong \bigoplus_{\substack{\bd = (d_{1}, \dotsc , d_{\ell}) \in \Z_{\geq 0}^{\Xi_{A}}\\ \lVert \bd \rVert = d }} S^{A_{1}}(n,d_{1}) \otimes \dotsb \otimes S^{A_{\ell}}(n,d_{\ell}).
\end{equation}  Thus it suffices to describe $S^{A}(n,d)$ when $A$ is a semisimple superalgebra with a single simple module.  Assume $A$ is such a superalgebra.

If $\delta (\alpha)=0$, then the simple supermodule $S=S^{\alpha}$ is of type $\typeM$.  Say it has graded dimension $m_{\alpha}|n_{\alpha}$, and set $M_{\alpha}=m_{\alpha}n$ and $N_{\alpha}=n_{\alpha}n$. Let $e \in A$ be a minimal even idempotent and set $\mathcal{e}= 1_{\k \SS_{d}} \otimes e^{\otimes d} \in \SS_{d} \wr A$.  Using the theory of idempotent trunction (e.g., see \cite[Section 6.2]{Green} or see \cite[Section 2]{BK} for the super version), the functor $F_{\mathcal{e}}: \Awreathsmod  \to \eAewreathsmod$ given by $M \mapsto Me$ defines a superalgebra isomorphism

\[
\End_{\SS_{d}\wr A} \left(V_{n}^{\otimes d} \right) \to \End_{\mathcal{e}(\SS_{d}\wr A)\mathcal{e}} \left((V_{n}^{\otimes d})\mathcal{e} \right).
\]  However, $\mathcal{e}(\SS_{d}\wr A)\mathcal{e} \cong \SS_{d} \wr eAe \cong \k\SS_{d}$ as superalgebras and 
\[
(V_{n}^{\otimes d})\mathcal{e} \cong (V_{n}e)^{\otimes d} \cong ((Ae)^{\oplus n})^{\otimes d} \cong (\k^{m_{\alpha}|n_{\alpha}})^{\oplus n})^{\otimes d} \cong (\k^{M_{\alpha}|N_{\alpha}})^{\otimes d}
\] as right $\k \SS_{d}$-supermodules.  Thus 
\[
S^{A}(n,d) \cong \End_{\k \SS_{d}}\left((k^{M_{\alpha}|N_{\alpha}})^{\otimes d}  \right) = S^{\alpha}(M_{\alpha}|N_{\alpha},d).
\]

An identical argument applies when $\delta (\alpha)=1$ and the simple supermodule $S=S^{\alpha}$ is of type $\typeQ$. 
\end{proof}

We close this section with an observation that will be needed later.

\begin{remark}\label{R:duals}  For $B \in \left\{\SS_{d} \wr A, S^{A}(n,d), U(\gl_{n}(A)) \right\}$ there are natural left and right actions on $V_{n}^{\otimes d}$.   As discussed in the proof of \cref{T:LabellingSimpleSAndmodules} When $A$ is semisimple, there is an isomorphism of $(A,A)$-bimodules $ A^{*} \cong A$ and $\left( V_{n}^{\otimes d}\right)^{*} \cong V_{n}^{\otimes d}$.  For each superalgebra $B$, this isomorphism identifies the natural left action with the dual of the natural right action of $B$, and vice versa.  In particular, dualizing \cref{L:SWMultiplicityFree} yields
\begin{equation}\label{E:HopefullyRight!}
V_{n}^{\otimes d} \cong \bigoplus_{\tuplambda \in \Lambda_{+}^{A}(n,d)} D^{\tuplambda } \star L^{A}_{n}(\tuplambda)^{*},
\end{equation} where the isomorphism is as $(\SS_{d} \wr A, S^{A}(n,d))$-bisupermodules.
\end{remark}

\section{Howe duality}\label{S:HoweDuality}

 Throughout this section we assume $\k$ is a field and that $(A,a)_I$ is a good pair with \(I\) finite.  In particular, \(A\) is unital with identity \(1 = \sum_{i \in I} i\).

\subsection{The symmetric space \texorpdfstring{$\SAmn$}{SAmn}} 

Fix $m,n \geq 1$.  Let  $V_{m} = A^{\oplus m}$ and $V_{n}= A^{\oplus n}$ be viewed as column and row vectors, respectively.  Since $V_{m}$ (resp., $V_{n}$) is a right (resp., left) $A$-supermodule by multiplication, we can consider the $\k$-supermodule $V_{m}\otimes_{A} V_{n}$.  Given $f \in A$, we write $v_{r}^{f}$ for the element of $V_{m}$  which is a column vector of all zeros except for $f$ in the $r$th position.  Likewise, $v_{r}^{f}$ will also be used to denote the element of $V_{n}$ which is a row vector of all zeros except for $f$ in the $r$th position.   For brevity we  write $z_{p,q}^{f}$ for $v_{p}^{1_{A}} \otimes v_{q}^{f} = v_{p}^{f} \otimes v_{q}^{1_{A}}$ in $V_{m} \otimes_{A} V_{n}$.

Via matrix multiplication there is an obvious left (resp., right) action by $U(\gl_{m}(A))$ (resp., $U(\gl_{n}(A))$) on $V_{m} \otimes_{A} V_{n}$.  The actions of $U(\gl_{m}(A))$ and $U(\gl_{n}(A))$ mutually commute.  Via their coproducts, $U(\gl_{m}(A))$ and $U(\gl_{n}(A))$ have corresponding commuting actions on 
\[
\SAmn := S^{\bullet} \left(V_{m} \otimes_{A} V_{n} \right).
\]

The monomial 
\[
z_{p_{1}, q_{1}}^{f_{1}}\dotsb z_{p_{r}, q_{r}}^{f_{r}} \in \SAmn 
\] has $\gl_{m}(A)$-weight 
\[
d'_{1}\varepsilon_{1}+\dotsb + d'_{m}\varepsilon_{m}, 
\] where $d'_{k} = | \left\{ i \mid p_{i}=k \right\}|$.  It also has $\gl_{n}(A)$-weight 
\[
d''_{1}\varepsilon_{1}+\dotsb + d''_{n}\varepsilon_{n}, 
\] where $d''_{\ell} = | \left\{ j \mid q_{j}=\ell \right\}|$.  For $\lambda \in \Lambda_{m}$ and $\mu \in \Lambda_{n}$, let $\SAmn^{\lambda, \mu}$ denote the set of vectors of $\SAmn$ which are simultaneously vectors of weight $\lambda$ for $\gl_{m}(A)$ and weight $\mu$ for $\gl_{(A)}$.  The $\k$-basis for $\SAmn$ given by monomials shows there is a direct sum decomposition of $\k$-supermodules:
\begin{equation*}
\SAmn = \bigoplus_{\substack{\lambda \in \Lambda_{m} \\ \mu \in \Lambda_{n}}} \SAmn^{\lambda, \mu}.
\end{equation*}
For $\lambda \in \Lambda_{m}$ and $\mu \in \Lambda_{n}$, set 
\begin{equation}\label{E:weightdecomp2}
\SAmn^{\bullet, \mu} = \bigoplus_{\nu \in \Lambda_{m}} \SAmn^{\nu, \mu} \qquad \text{and} \qquad  \SAmn^{\lambda, \bullet} = \bigoplus_{\nu \in \Lambda_{n}} \SAmn^{\lambda, \nu}.
\end{equation}
Then,
\begin{equation}\label{E:weightdecomp4}
\SAmn = \bigoplus_{\mu \in \Lambda_{n}} \SAmn^{\bullet, \mu} \qquad \text{and} \qquad \SAmn = \bigoplus_{\lambda \in \Lambda_{m}} \SAmn^{\lambda, \bullet}. 
\end{equation}
Since the actions of $U(\gl_{m}(A))$ and $U(\gl_{n}(A))$ commute, the decompositions of $\SAmn$ given in \cref{E:weightdecomp4} are as $U(\gl_{m}(A))$- and $U(\gl_{n}(A))$-supermodules, respectively.

\begin{lemma}\label{L:weightspacesaresymmetricspaces}
For each $\lambda = \sum_{i=1}^{m}\lambda_{i}\varepsilon_{i} \in \Lambda_{m}$, there is an isomorphism of $U(\gl_{n}(A))$-supermodules,
\[
\mathcal{S}^{\lambda, \bullet}_{m,n} \cong S^{\lambda_{1}}(V_{n}) \otimes \dotsb \otimes S^{\lambda_{m}}(V_{n}),
\]
For each $\mu = \sum_{i=1}^{n}\mu_{i}\varepsilon_{i} \in \Lambda_{n}$, there is an isomorphism of $U(\gl_{m}(A))$-supermodules,
\[
\mathcal{S}^{\bullet, \mu}_{m,n} \cong S^{\mu_{1}}(V_{m}) \otimes \dotsb \otimes S^{\mu_{n}}(V_{m}),
\]

\end{lemma}

\begin{proof} The first isomorphism is the $\k$-linear map given on monomials by 
\begin{multline*}
z_{1, q_{1,1}}^{f_{1,1}}\dotsb z_{1, q_{1,r_{1}}}^{f_{1, r_{1}}}z_{2, q_{2,1}}^{f_{2,1}}\dotsb z_{2, q_{2,r_{2}}}^{f_{2, r_{2}}}\dotsb z_{m, q_{m,1}}^{f_{m,1}}\dotsb z_{m, q_{m,r_{m}}}^{f_{m, r_{m}}} \\
 \mapsto v_{q_{1,1}}^{f_{1,1}}\dotsb v_{q_{1,r_{1}}}^{f_{1, r_{1}}} \otimes v_{q_{2,1}}^{f_{2,1}}\dotsb v_{q_{2,r_{2}}}^{f_{2, r_{2}}} \otimes \dotsb \otimes v_{q_{m,1}}^{f_{m,1}}\dotsb v_{q_{m,r_{m}}}^{f_{m, r_{m}}}.
\end{multline*}
A check on monomials verifies this is an even isomorphism of $U(\gl_{n}(A))$-supermodules.  The obvious variant gives the second isomorphism.
\end{proof}

Let 
\[
\dot{U}(\gl_{n}(A)) = \bigoplus_{\lambda, \mu \in \Lambda_{m}} \UglnA (\lambda, \mu)
\]
be the locally unital superalgebra associated to the supercategory $\UglnA$.  Note that $\dot{U}(\gl_{n}(A))$ is an idempotented form of the enveloping superalgebra $U(\gl_{n}(A))$ whose weight idempotents lie in $\Lambda_{n}$.     Whenever $M$ is a supermodule for $U(\gl_{n}(A))$ with a direct sum decomposition into weight spaces whose weights lie in $\Lambda_{n}$ there is a well-defined action of  $\dot{U}(\gl_{n}(A))$ and, conversely, any supermodule for  $\dot{U}(\gl_{n}(A))$ defines a supermodule for $U(\gl_{n}(A))$ with a weight space decomposition whose weights lie in $\Lambda_{n}$. In particular, from \cref{E:weightdecomp2} it follows that  $\SAmn$ is a $(\dot{U}(\gl_{m}(A)), \dot{U}(\gl_{n}(A)))$-bisupermodule.

Given a left $U(\gl_{m}(A))$-supermodule homomorphism $f: \SAmn \to \SAmn$, we say $f$ is \emph{left finite} if there is a finite set $J \subseteq \Lambda_{n}$ such that $(\SAmn^{\bullet, \mu})f=0$ for all $\mu \not\in J$ and if the image of $f$ is contained in $\oplus_{\mu \in J} \SAmn^{\bullet, \mu}$.  Similarly, given a right $U(\gl_{n}(A))$-supermodule homomorphism $f: \SAmn \to \SAmn$, we say $f$ is \emph{right finite} if there is a finite set $J \subseteq \Lambda_{m}$ such that $f(\SAmn^{\lambda, \bullet})=0$ for all $\lambda \not\in J$ and if the image of $f$ is contained in $\oplus_{\lambda \in J} \SAmn^{\lambda, \bullet}$.  In both cases we write a superscript $\fin$ for the set of finite maps and it should be clear from context whether they are left finite or right finite.

Since the actions of $\dot{U}(\gl_{m}(A))$ and $\dot{U}(\gl_{n}(A))$ on $\SAmn$ commute, the action of an element of one superalgebra defines a supermodule homomorphism for the other.  Furthermore, the action of $\dot{U}(\gl_{m}(A))$ and $\dot{U}(\gl_{n}(A))$ is by left and right finite maps, respectively.  That is, there are superalgebra maps
\begin{align}
\rho^{\ell}_{m} &: \dot{U}(\gl_{m}(A)) \to \End_{\dot{U}(\gl_{n}(A))}\left(\SAmn \right)^{\fin } \subseteq \End_{\k}\left(\SAmn \right) \label{E:LeftUmodReponS},\\
\rho^{r}_{n} &: \dot{U}(\gl_{n}(A)) \to \End_{\dot{U}(\gl_{m}(A))}\left(\SAmn \right)^{\fin }\subseteq \End_{\k}\left(\SAmn \right) \label{E:RightUmodReponS}.
\end{align}

\subsection{Miscellanea}\label{SS:Miscellanea}

In this section we record various notions which will be needed in what follows.  Given $\k$-linear supercategories $\catC$ and $\catD$, let $\Fun(\catC, \catD)$ denote the category whose objects are even superfunctors $\catC  \to \catD$ and whose morphisms are supernatural transformations.  Given a $\k$-linear supercategory $\catC $, there is the \emph{opposite supercategory} $\catC ^{\sop}$ with the same objects as $\catC $, $\Hom_{\catC^{\sop}}(X,Y) := \Hom_{\catC}(Y,X)$, and composition $\bullet$ given by the rule $f \bullet g = (-1)^{\bar{f}\bar{g}} g \circ f$ for all composable homogenous morphisms in $\catC^{\sop}$.  
Given $\k$-linear supercategories $\catC$ and $\catD$, we can form the supercategory $\catC  \boxtimes \catD$.  The objects of this category are given by 
\[
\left\{ (x, y)\mid \text{$x$ is an object of $\catC $ and $y$ is an object of $\mathbf{D}$}\right\}
\]
 and the morphisms are given by 
\[
\Hom_{\catC \boxtimes \catD}\left((x_{1}, x_{2}), (y_{1}, y_{2}) \right) = \Hom_{\catC }(x_{1}, y_{1}) \otimes \Hom_{\catD}(x_{2}, y_{2}).
\] Composition is given on homogeneous morphisms by 
\begin{equation}\label{E:composition}
(f_{1} \otimes f_{2}) \circ (g_{1} \otimes g_{2}) = (-1)^{\bar{f}_{2}\bar{g}_{1}} (f_{1} \circ g_{1}) \otimes (f_{2} \circ g_{2}).
\end{equation}
See \cite{BE} for details on these constructions in the super setting.   The correspondence between $\k$-linear supercategories and locally unital $\k$-superalgebras takes opposites to opposites, and takes tensor products to tensor products.

A \emph{right representation} of a $\k$-linear supercategory $\catA$ is a $\k$-linear functor $F: \catA^{\sop} \to \ksMod$, and a \emph{birepresentation} of the pair $(\catA, \catB)$ is a $\k$-linear functor $F : \catA \boxtimes \catB^{\sop} \to \ksMod$.  We write $\catA\text{-Rep}$, $\text{Rep-}\catA$, and $\catA\text{-Rep-}\catB$ for the supercategories of representations of $\catA$, right representations of $\catA$, and birepresentations of $(\catA, \catB)$, respectively.  The correspondence with locally unital $\k$-superalgebras gives category equivalences $\catA \text{-Rep} \cong \Asmod$, $\text{Rep-}\catA \cong \smodA$, and $\catA \text{-Rep-}\catB \cong (A,B)$-supermodules.

The following lemma is straightforward.

\begin{lemma}\label{L:transposeisomorphisms} Let $A$ be a superalgebra and let $(A,a)_{I}$ be a good pair.  Then the following statements are true.  

\begin{enumerate}
\item The map which is the identity on objects and flips diagrams across their horizontal axis defines an isomorphism of supercategories,
\[
\left( \Web^{A,a}_{I}\right)^{\sop} \to \Web^{A^{\sop}, a^{\sop}}_{I}.
\]
\item The map which is the identity on objects and is the transpose, $E_{r,s}^{f} \mapsto E_{s,r}^{f}$, on generating morphisms defines an isomorphism of supercategories,
\[
\UglnA^{\sop} \to \UglnAsop.
\] 
\item The map given on generating elements  by $z_{p,q}^{f}\mapsto z_{q,p}^{f}$ defines a parity preserving isomorphism of $(\dot{U}(\gl_{n}(A^{\sop})), \dot{U}(\gl_{m}(A^{\sop})))$-bisupermodules,
\[
\mathcal{S}^{A}_{m,n} \cong \mathcal{S}^{A^{\sop}}_{n,m}.
\] 
\end{enumerate}

\end{lemma}

\subsection{Categorical representations from \texorpdfstring{$\SAmn$}{Smn} }\label{}
Let 
\begin{align}
R^{\ell}_{m}: \UglmA \to \ksMod \label{E:LeftUdotReponS}, \\
R^{r}_{n}: \UglnA^{\sop} \to \ksMod  \label{E:RightUdotReponS},
\end{align} be the representations corresponding to the maps $\rho^{\ell}_{m}$ and $\rho^{r}_{n}$, respectively.  Recall the functors $G_{m}$ and $W_{m}$ from \cref{Gthm,T:UdottoWebs} 

\begin{lemma}\label{L:repscoincide}  For all $m \geq 0$ the functors $R^{\ell}_{m}$ and $G_{m} \circ W_{m}$ are naturally isomorphic.
\end{lemma}

\begin{proof}  Let $\lambda = \sum_{i=1}^{m} \lambda_{i}\varepsilon_{i} \in \Lambda_{m}$.  Then $R^{\ell}_{m}(\lambda) = \SAmn^{\lambda, \bullet}$  and 
\[
(G_{m} \circ W_{m})(\lambda) = G_{m}(1^{\lambda_{1}} \otimes \dotsb \otimes 1^{\lambda_{m}}) = S^{\lambda_{1}}(V_{n}) \otimes \dotsb \otimes S^{\lambda_{m}}(V_{n}).
\]  These are isomorphic $\k$-modules via the map given in  \cref{L:weightspacesaresymmetricspaces}.

Given $1 \leq r, s \leq m$ and $f \in A$, consider the generating morphism $E_{r,s}^{f} \in \UglmA (\lambda, \gamma)$.  Then 
\[
R^{\ell}_{m}(E_{r,s}^{f}):  \SAmn^{\lambda, \bullet} \to  \SAmn^{\gamma, \bullet}
\] is the map given by the left action of the element $E_{r,s}^{f} \in \gl_{m}(A)$.

On the other hand, 
\[
(G_{m} \circ W_{m})(E_{r,s}^{f}) :  S^{\lambda_{1}}(V_{n}) \otimes \dotsb \otimes S^{\lambda_{m}}(V_{n}) \to  S^{\gamma_{1}}(V_{n}) \otimes \dotsb \otimes S^{\gamma_{m}}(V_{n})
\] is given by the $\k$-module homomorphism $G_{m}(e_{[r,s],\lambda}^{f,1})$.   Checking on a monomial confirms this map agrees with the one in the previous paragraph via the isomorphism given in  \cref{L:weightspacesaresymmetricspaces}.
\end{proof}

Since $\SAmn$ is a $(\dot{U}(\gl_{m}(A)),\dot{U}(\gl_{n}(A))_{}$-bisupermodule, the following result is immediate.
\begin{lemma}\label{L:BiRepCreation}  Fix $m,n \geq 1$. There is a birepresentation 
\[
F_{m,n}: \UglmA  \boxtimes \UglnA^{\sop}  \to \ksMod.
\] It is given on an object $(\lambda, \mu) \in \Lambda_{m} \times \Lambda_{n}$ by
\[
F_{m.n}(\lambda, \mu) = \SAmn^{\lambda, \mu}.
\] On homogeneous morphisms $f \in \Hom_{\UglmA}(\lambda_{1}, \lambda_{2})$, $ g \in  \Hom_{\UglnA^{\sop}}(\mu_{1}, \mu_{2})$ the map 
\[
F_{m.n}(f \otimes g): \SAmn^{\lambda_{1}, \mu_{1}} \to \SAmn^{\lambda_{2}, \mu_{2}}
\]
is given on a homogeneous  $x \in \SAmn^{\lambda_{1}, \mu_{1}}$ by
\[
F_{m,n}(f \otimes g) (x) = (-1)^{\bar{g}\bar{x}} f((x)g) =  (-1)^{\bar{g}\bar{x}} (f(x))g.
\]  
\end{lemma}

\subsection{Categorical and finitary Howe duality}\label{SS:HoweDuality}

For any supercategories $\mathbf{C}$ and $\mathbf{D}$, there is a functor 
\[
\iota^{\prime} : \mathbf{C} \to \Fun(\mathbf{D}, \mathbf{C} \boxtimes \mathbf{D})
\]
given as follows.  First, for an object $x$ in $\mathbf{C}$ define a functor $\iota^{\prime}(x) =  \iota^{\prime}_{x} : \mathbf{D} \to \mathbf{C} \boxtimes \mathbf{D}$  by $\iota^{\prime}_{x}(y) = (x, y)$ for all objects $y$ in $\mathbf{D}$ and by $\iota^{\prime}_{x}(g) = 1_{x} \otimes g$ for all morphisms $g \in \Hom_{\mathbf{D}}(y_1, y_2)$.  Next we define the functor $\iota^{\prime}$ on morphisms.  Given any morphism $f : x_1 \to x_2$ in $\mathbf{C}$ construct a supernatural transformation $\iota^{\prime}(f) : \iota^{\prime}_{x_{1}} \Rightarrow \iota^{\prime}_{x_{2}}$ by setting, for any object $y \in \mathbf{D}$, the morphism $\iota^{\prime}(f)_{y}=f \otimes 1_{y} : \iota^{\prime}_{x_1}(y)  \to \iota^{\prime}_{x_{2}}(y)$.  Making the obvious changes yields an analogous functor 
\[
\iota^{\prime \prime} : \mathbf{D} \to \Fun(\mathbf{C}, \mathbf{C} \boxtimes \mathbf{D}).
\]

Given a birepresentation $F : \mathbf{C}_1 \boxtimes \mathbf{C}^{\sop}_2 \to \ksMod$, there are functors $F_{*}:\Fun(\mathbf{C}^{\sop}_2 , \mathbf{C}_1 \boxtimes \mathbf C^{\sop}_{2}) \to \text{Rep-}\mathbf{C}_2$ and $F_{*}:\Fun(\mathbf{C}_1 , \mathbf{C}_1 \boxtimes \mathbf C_2^{\sop}) \to \mathbf{C}_1\text{-Rep}$ given in both cases by  $H \mapsto F \circ H$.  In turn, there are functors
\begin{align*}
{F}_1=F_* \circ \iota^{\prime} &: \mathbf{C}_1 \to \text{Rep-}\mathbf{C}_{2}, \\
{F}_2 =F_* \circ \iota^{\prime \prime}&: \mathbf{C}^{\sop}_2 \to \mathbf{C}_1\text{-Rep}.
\end{align*}

We say the birepresentation $F$ \emph{left centralizes} if the functor ${F}_{1}$ is full, it \emph{right centralizes} if the functor ${F}_{2}$ is full, and it \emph{has the double centralizer property} if both ${F}_{1}$ and ${F}_{2}$ are full.  As a sanity check, the reader may wish to verify that when $\mathbf{C}_{1}$ and $\mathbf{C}_{2}$ both have only a single object  these definitions specialize to the notion of two superalgebras having mutually centralizing actions on a bisupermodule.

\begin{theorem}\label{T:Centralizing}  Let $\k$ be a field of characteristic zero.  The following statements hold:
\begin{enumerate}
\item If for a given $m \geq 0$ the defining representation $G_{m}$ is full for $A$, then $F_{m,n}$ left centralizes for all $n \geq 0$.
\item If for a given $n \geq 0$ the defining representation $G_{n}$ is full for $A^{\sop}$, then $F_{m,n}$ right centralizes for all $m \geq 0$.
\end{enumerate}
\end{theorem}

\begin{proof} To prove that $F_{m,n}$ left centralizes we must show that $F_{m,n,1} = F_{m,n, *} \circ \iota' : \UglmA \to \text{Rep-}\UglnA$ is full.  Recall that the supercategory $\text{Rep-}\UglnA$ is equivalent to the supercategory of right $\dot{U}(\gl_{n}(A))$-supermodules.  Under this equivalence, for any object $\lambda \in \Lambda_{m}$ we have
\[
F_{m,n,1}(\lambda) = \mathcal{S}_{m.n}^{\lambda, \bullet}.
\] Given objects $\lambda, \gamma \in \Lambda_{m}$, the morphism $a \in \UglmA (\lambda, \mu)$ goes to the $\dot{U}(\gl_{n}(A))$-supermodule homomorphism 
\[
F_{m,n,1}(a): \mathcal{S}_{m.n}^{\lambda, \bullet} \to \mathcal{S}_{m.n}^{\gamma, \bullet}
\] given by acting on the left with the superalgebra element $a \in \dot{U}(\gl_{m}(A))$.  That is, under the equivalence $F_{m,n,1}$ goes to the functor $R^{\ell}_{m}$.  Therefore, by \cref{L:repscoincide} it suffices to show $G_{m} \circ W_{m}$ is full.  But $W_{m}$ is full by \cref{C:FullnessofWebFunctor} and $G_{m}$ is full by assumption.

For clarity, for the remainder of the proof we use a superscript $A$ to indicate the underlying algebra (e.g., $F^{A}_{m,n}$ for $F_{m,n}$). To prove that $F^{A}_{m,n}$ right centralizes we must show that 
\[
F^{A}_{m,n,2} = F^{A}_{m,n,*} \circ \iota^{'', A} : \UglnA^{\sop} \to \UglmA\text{-Rep}
\]
is full.  Recall that the supercategory $\UglmA\text{-Rep}$ is equivalent to the supercategory of left $\dot{U}(\gl_{m}(A))$-supermodules.   Under this equivalence, for any object $\mu \in \Lambda_{n}$ we have
\[
F^{A}_{m,n,2}(\mu) = \mathcal{S}_{m.n}^{\bullet, \mu}.
\] Given objects $\mu, \nu \in \Lambda_{n}$, the morphism $b \in \UglnA^{\sop} (\mu, \nu)$ goes to the $\dot{U}(\gl_{m}(A))$-supermodule homomorphism 
\[
F^{A }_{m,n,2}(b): \mathcal{S}_{m.n}^{\bullet, \mu} \to \mathcal{S}_{m.n}^{\bullet, \nu}
\] given by acting on the right with the superalgebra element $b \in \dot{U}(\gl_{n}(A))$.

That is, under the equivalence $F^{A }_{m,n,2}$ goes to the functor $R^{r, A }_{n}$ and fullness amounts to showing that any finite right $\dot{U}(\gl_{n}(A))$-supermodule endomorphism of $\mathcal{S}^{A}_{m,n}$ can be realized by the action on the left by an element of $\dot{U}(\gl_{m}(A))$.  That is, that any finite left  $\dot{U}(\gl_{n}(A))^{\sop}$-supermodule endomorphism of $\mathcal{S}^{A}_{m,n}$ can be realized by the action on the right by an element of $\dot{U}(\gl_{m}(A))^{\sop}$.  Applying the identifications given in \cref{L:transposeisomorphisms}  shows that fullness of $F^{A}_{m,n,2}$ follows from the fullness of $R^{\ell, A^{\sop}}_{n}$.  This was established in the previous paragraph (replacing $A$ with $A^{\sop}$, of course).
\end{proof}

Combining the above with the results from \cref{S:FunctorFullness} immediately implies the following categorical version of Howe duality.
\begin{proposition}\label{T:CategoricalHoweDuality}  Let $\k$ be a field of characteristic zero.  Assume that Schur--Weyl duality holds for $A$ and $A^{\sop}$.  Then the birepresentation
\[
F_{m,n}: \UglmA  \boxtimes \UglnA^{\sop}  \to \ksMod
\] has the double centralizer property.
\end{proposition}

These results can also be reformulated into the following finitary Howe duality.
\begin{proposition}\label{C:FinitaryHoweDuality}  Let $\k$ be a field of characteristic zero and let $A$ be a unital $\k$-linear superalgebra.  Assume that Schur--Weyl duality holds for $A$ or $A^{\sop}$.  Then, respectively, the superalgebra maps
\begin{align*}
\rho^{\ell}_{m} &: \dot{U}(\gl_{m}(A)) \to \End_{\dot{U}(\gl_{n}(A))}\left(\SAmn \right)^{\fin },\\
\rho^{r}_{n} &: \dot{U}(\gl_{n}(A)) \to \End_{\dot{U}(\gl_{m}(A))}\left(\SAmn \right)^{\fin },
\end{align*} are surjective.
\end{proposition}

\begin{proof} As discussed in the proof of \cref{T:Centralizing}, the assumptions of the theorem prove the functors $R^{\ell}_{m}$ and $R^{r}_{n}$ are full which, in turn, is equivalent to $\rho^{\ell}_{m}$ and $\rho^{r}_{n}$ being surjective.
\end{proof}

\subsection{Strong multiplicity-free decompositions}\label{SS:MultiplicityFreeDecompositions}  In this section we assume $\k$ is an algebraically closed field of characteristic zero, and that $A$ is finite-dimensional and semisimple.

For homogeneous $a_{1}, \dotsc , a_{d}, b_{1}, \dotsc , b_{d} \in A$, define 
\[
\varepsilon((a_{1}, \dotsc , a_{d}), (b_{1}, \dotsc , b_{d})) = \sum_{1 \leq q < p \leq d} \bar{a}_{p}\bar{b}_{q}.
\]  Then in $A^{\otimes d}$ we have 
\[
(a_{1} \otimes \dotsb \otimes a_{d})(b_{1} \otimes \dotsb \otimes b_{d}) = (-1)^{\varepsilon((a_{1}, \dotsc , a_{d}), (b_{1}, \dotsc , b_{d}))} (a_{1}b_{1}) \otimes \dotsb (a_{d}b_{d}).
\]  Also, for $\tau \in \SS_{d}$ recall the notation $\langle \tau ; (a_{1}, \dotsc , a_{d}) \rangle$ defined in \cref{twistsec}.  With these formulas a calculation shows that for all $\tau \in \SS_{d}$, 
\begin{equation}\label{E:intertwiningproduct}
\left[(a_{1} \otimes \dotsb \otimes a_{d})(b_{1} \otimes \dotsb \otimes b_{d}) \right] \cdot \tau = \left[ (a_{1} \otimes \dotsb \otimes a_{d}) \cdot \tau \right]\left[(b_{1} \otimes \dotsb \otimes b_{d}) \cdot \tau \right].
\end{equation}
\begin{theorem}\label{L:HDMultiplicityFree}  If $m,n \geq 1$ and $r$ is the minimum of $m$ and $n$, then for all $d \geq 0$,
\[
\mathcal{S}^{d}_{m,n} \cong \bigoplus_{\tuplambda \in \Lambda^{A}_{+}(r,d)} L^{A}_{m}(\tuplambda) \star L^{A}_{n}(\tuplambda)^{*}
\] as $(U(\gl_{m}(A)), U(\gl_{n}(A)))$-supermodules.
\end{theorem}

\begin{proof}  Let $A^{\otimes d}$ act by diagonal multiplication on the left on $V_{n}^{\otimes d}$ and on the right on $V_{m}^{\otimes d}$.  There is a $\k$-supermodule isomorphism 
\[
V_{m}^{\otimes d} \otimes_{A^{\otimes d}} V_{n}^{\otimes d} \xrightarrow{\cong} (V_{m} \otimes_{A} V_{n})^{\otimes d}
\] given on homogeneous elements by 
\begin{equation*}
(v_{i_{1}}^{a_{1}} \otimes \dotsb \otimes v_{i_{d}}^{ a_{d}}) \otimes (v_{j_{1}}^{b_{1}} \otimes \dotsb \otimes v_{j_{d}}^{ b_{d}})  \mapsto  (-1)^{\varepsilon((a_{1}, \dotsc , a_{d}), (b_{1}, \dotsc , b_{d}))} z_{i_{1}, j_{1}}^{a_{1}b_{1}} \otimes \dotsb \otimes z_{i_{d},j_{d}}^{a_{d}b_{d}}.
\end{equation*}

Let $\SS_{d}$ acts on the right on $ (V_{m} \otimes_{A} V_{n})^{\otimes d}$ via signed place permutation.  It also acts on the right on $V_{m}^{\otimes d} \otimes_{A^{\otimes d}} V_{n}^{\otimes d}$ by acting diagonally via simultaneous signed place permutations on $V_{m}^{\otimes d}$ and $V_{n}^{\otimes d}$.  That the above map is an isomorphism of $\SS_{d}$-supermodules follows from \cref{E:intertwiningproduct}.  Restricting this isomorphism to $\SS_{d}$-invariants yields the middle isomorphism:
\begin{equation}\label{E:tensorisomorphisms}
V_{m}^{\otimes d} \otimes_{\SS_{d} \wr A^{\otimes d}} V_{n}^{\otimes d} \cong  \left[V_{m}^{\otimes d} \otimes_{A^{\otimes d}} V_{n}^{\otimes d} \right]^{\SS_{d}} \cong  \left[(V_{m} \otimes_{A} V_{n})^{\otimes d} \right]^{\SS_{d}} = \mathcal{S}^{d}_{m,n}.
\end{equation}
The left isomorphism is straightforward. It is also straightforward to check that the maps in \cref{E:tensorisomorphisms}  are $(U(\gl_{m}(A)), U(\gl_{n}(A)))$-bisupermodule homomorphisms.

To further our analysis of \cref{E:tensorisomorphisms} we pause to consider, for $\tuplambda  \in \Lambda_{+}^{A}(n,d)$ and $\tupmu  \in \Lambda_{+}^{A}(m,d)$, 
\[
(L^{A}_{m}(\tupmu) \otimes D^{\tupmu, *} ) \otimes_{\SS_{d}\wr A} (D^{\tuplambda} \otimes L^{A}_{n}(\tuplambda)^{*} ).
\]  As a  $(U(\gl_{m}(A)), U(\gl_{n}(A)))$-bisupermodule it is isomorphic to a direct sum of copies of $L^{A}_{m}(\tupmu) \star L^{A}_{n}(\tuplambda)^{*}$.  To determine the number of summands will require knowing the dimension of $D^{\tupmu, *}  \otimes_{\SS_{d}\wr A} D^{\tuplambda}$.  By tensor-Hom adjunction there is a parity preserving isomorphism of $\k$-supermodules: 
\[
\Hom_{\k}\left( D^{\tupmu, *}  \otimes_{\SS_{d}\wr A} D^{\tuplambda}  , \k \right) \cong \Hom_{\SS_{d} \wr A}\left(D^{\tupmu, *},  D^{\tuplambda, *}\right) \cong \Hom_{\SS_{d} \wr A}\left(D^{\tuplambda},  D^{\tupmu} \right).
\]  By Schur's lemma the dimension of the right hand side is $2^{\delta (\tuplambda)}$ if $D^{\tuplambda} \cong  D^{\tupmu}$ as $\SS_{d} \wr A$-supermodules, and zero otherwise.  Note that when there is an isomorphism, $L^{A}_{n}(\tuplambda)^{*}$, $D^{\tuplambda, *}$, $D^{\tupmu, *}$, $D^{\tupmu}$, and $L^{A}_{m}(\tupmu)$ all have the same type.

Taken together with the tensor product rule from \cref{SS:RepsofSemisimpleSuperalgebras} this discussion shows
\[
(L^{A}_{m}(\tupmu) \otimes D^{\tupmu, *} ) \otimes_{\SS_{d}\wr A} (D^{\tuplambda} \otimes L^{A}_{n}(\tuplambda)^{*} ) \cong \begin{cases}   2^{\delta (\tuplambda)}L^{A}_{m}(\tupmu) \otimes L^{A}_{n}(\tuplambda)^{*}  \cong 2^{2\delta (\tuplambda)} L^{A}_{m}(\tupmu) \star L^{A}_{n}(\tuplambda)^{*}, & \text{ if }D^{\tuplambda} \cong  D^{\tupmu}; \\
0, & else.
\end{cases}
\]  Comparing multiplicities shows 
\[
(L^{A}_{m}(\tupmu) \star  D^{\tupmu, *} ) \otimes_{\SS_{d}\wr A} (D^{\tuplambda} \star L^{A}_{n}(\tuplambda)^{*} ) \cong \begin{cases}  L^{A}_{m}(\tupmu) \star L^{A}_{n}(\tuplambda)^{*}, & \text{ if } D^{\tuplambda} \cong  D^{\tupmu}; \\
0, & else.
\end{cases}
\]  

We are now prepared to prove the statement of the theorem.  Combining \cref{E:HopefullyRight!}, \cref{E:tensorisomorphisms},  \cref{L:SWMultiplicityFree}, and the previous calculation  yields the following isomorphisms of  $(U(\gl_{m}(A)), U(\gl_{n}(A)))$-bisupermodules:
\begin{align*}
\SAmn^{d}  & \cong V_{m}^{\otimes d} \otimes_{\SS_{d} \wr A^{\otimes d}} V_{n}^{\otimes d} \\
            & \cong  \left(\bigoplus_{\tupmu \in \Lambda_{+}^{A}(m,d)} L^{A}_{m}(\tupmu) \star D^{\tupmu, *}  \right) \otimes_{\SS_{d}\wr A}  \left(\bigoplus_{\tuplambda \in \Lambda_{+}^{A}(n,d)}   D^{\tuplambda} \star L^{A}_{n}(\tuplambda)^{*}   \right) \\
            & \cong \bigoplus_{\tupmu \in \Lambda_{+}^{A}(m,d)} \bigoplus_{\tuplambda \in \Lambda_{+}^{A}(n,d)}  \left( L^{A}_{m}(\tupmu) \star D^{\tupmu, *}  \right)  \otimes_{\SS_{d}\wr A} \left( D^{\tuplambda} \star L^{A}_{n}(\tuplambda)^{*}  \right) \\
	    & \cong \bigoplus_{\tuplambda \in \Lambda_{+}^{A}(r,d)}    L^{A}_{m}(\tuplambda) \star L^{A}_{n}(\tuplambda )^{*}.
\end{align*}
\end{proof}

The decomposition of $\mathcal{S}^{d}_{m,n}$ as a $(U(\gl_{m}(A)), U(\gl_{n}(A)))$-bimodule given in the previous theorem immediately implies that these actions are mutually centralizing. Thus we have the following corollary.

\begin{corollary}\label{T:HoweDuality}  Assume $\k$ is an algebraically closed field of characteristic zero and $A$ is a finite-dimensional semisimple superalgebra.  Then for all $m, n \geq 1$, the actions of $U(\gl_{m}(A))$ and $U(\gl_{n}(A))$ on $\mathcal{S}^{d}_{m,n}$ are mutually centralizing.
\end{corollary}

It is well-known that Howe dualities can be reformulated in the language of polynomial differential operators.  They also have a number of interesting applications to the study of $U(\gl_{n}(A))$ and its representation theory when, e.g., $A=\k$ or $A = \Cliff$.  See, for example, \cite{CW,Howe}.  We expect that most of these will have $A$-analogues and it would be interesting to develop them.

\section{Examples}\label{S:Examples}
We end with a discussion of some examples of interest.

\subsection{The category of \texorpdfstring{$\gl_n$}{gln}-webs}
If $ A = a = \k$ then \(\Web^{\k, \k}_{\{1\}}\) can be seen to be isomorphic to the category of \(\gl_n\)-webs defined in \cite{DKM}, the {\em Schur category} defined in \cite{BEPO}, and the web category introduced in \cite{CKM}. For \(n,d \in \Z_{\geq 0}\), the \(\k\)-superalgebra \(W^{\k, \k}_{n,d}\) is isomorphic to the well-studied classical Schur algebra \(S^\k(n,d)\).

\subsection{Cyclotomic Schur algebras} Let $C_{r}$ be a cyclic group of order $r$ and let $A= \k C_{r}$.  Then $\SS_{d} \wr \k C_{r}$ is a complex reflection group algebra. The generalized Schur algebra $S^{A}(n,d)$ is the cyclotomic Schur algebra.  It and related algebras have been studied in, for example, \cite{Ariki,ATY,BKLW,DDY,DJM,GeckHiss,GreenR,LNX,MS}.

\subsection{The category of \texorpdfstring{\(\mathfrak{q}_n\)}{q(n)}-webs}  \label{SS:typeQwebs}
Let \(A = \Cliff = \k[c]/(c^2-1)\) to be the Clifford superalgebra on one generator, where the $\Z_{2}$-grading is given by declaring \(\bar{c} = \bar{1} \).  Then $\SS_{d} \wr \Cliff$ is isomorphic to the Sergeev superalgebra $\Ser_{d}$ which, in turn, is related to the spin representations of the symmetric group.   It is not difficult to verify that \(\gl_n(\textup{Cl}_1) \cong \mathfrak{q}_n(\k )\) as Lie superalgebras. If we set \(a = \k1_{A}\) and \(I = \{1\}\), then there is an asymptotically faithful functor
\begin{equation*}
G_n: \Web^{\Cliff, \k}_{\left\{1 \right\}} \to \textup{mod}_{\mathcal{S}}\textup{-}\mathfrak{gl}_{n}(\Cliff).
\end{equation*}
When \(\k = \mathbb{C}\) a diagrammatic  supercategory  \(\mathfrak{q}\textup{-}\Web_{\uparrow}\) was introduced by Brown and the second author in \cite{BrKuWebs} and a similar functor 
\begin{equation*}
\mathfrak{q}\textup{-}\Web_{\uparrow} \to \mathfrak{q}_{n}(\C)\text{-supermodules}
\end{equation*} was defined and studied.  As we next explain, $\Web^{\Cliff, \k}_{{1}}$ and \(\mathfrak{q}\textup{-}\Web_{\uparrow}\) are in fact isomorphic as $\k$-linear monoidal supercategories and these two functors coincide.

The category \(\Web^{\Cliff, \C}_{\{1\}}\)  only has monotone generators and only needs the single coupon $c$ on the thin strands.  On the other hand, the category $\mathfrak{q}$-$\operatorname{\mathbf{Web}}_{\uparrow}$ was defined in \cite{BrKuWebs} as having generating objects $\uparrow_{k}$ for $k \geq 0$ and three kinds of (upward oriented) diagrammatic generators:  splits,  merges, and dots on $\uparrow_{k}$ for all $k$.  We skip the list of relations and ask the reader to refer to \cite[Section~4.1]{BrKuWebs}.  In \cite{BrKuWebs} the crossing of thin strands is defined by the same formula as in \cref{CrossingWebRel}, but for thick strands they are defined in terms of the single strand crossings by:
\[\beta_{\uparrow_{k}, \uparrow_{l}}=
\xy 
(0,0)*{
\begin{tikzpicture}[color=\clr, scale=1.25]
	\draw[thick, directed=1] (0,0) to (0.5,0.5);
	\draw[thick, directed=1] (0.5,0) to (0,0.5);
	\node at (0,-0.15) {\scriptsize $k$};
	\node at (0,0.65) {\scriptsize $l$};
	\node at (0.5,-0.15) {\scriptsize $l$};
	\node at (0.5,0.65) {\scriptsize $k$};
\end{tikzpicture}
};
\endxy :=\frac{1}{k!\,l!}
\xy
(0,0)*{
\begin{tikzpicture}[color=\clr, scale=.35]
	\draw [ thick, directed=1] (0, .75) to (0,1.5);
	\draw [ thick, directed=0.75] (1,-1) to [out=90,in=330] (0,.75);
	\draw [ thick, directed=0.75] (-1,-1) to [out=90,in=210] (0,.75);
	\draw [ thick, directed=1] (4, .75) to (4,1.5);
	\draw [ thick, directed=0.75] (5,-1) to [out=90,in=330] (4,.75);
	\draw [ thick, directed=0.75] (3,-1) to [out=90,in=210] (4,.75);
	\draw [ thick, directed=0.75] (0,-6.5) to (0,-5.75);
	\draw [ thick, ] (0,-5.75) to [out=30,in=270] (1,-4);
	\draw [ thick, ] (0,-5.75) to [out=150,in=270] (-1,-4); 
	\draw [ thick, directed=0.75] (4,-6.5) to (4,-5.75);
	\draw [ thick, ] (4,-5.75) to [out=30,in=270] (5,-4);
	\draw [ thick, ] (4,-5.75) to [out=150,in=270] (3,-4); 
	\draw [ thick ] (5,-1) to (1,-4);
	\draw [ thick ] (3,-1) to (-1,-4);
	\draw [ thick ] (1,-1) to (5,-4);
	\draw [ thick ] (-1,-1) to (3,-4);
	\node at (2.6, -0.5) {\scriptsize $1$};
	\node at (5.4, -0.5) {\scriptsize $1$};
	\node at (1.4, -0.5) {\scriptsize $1$};
	\node at (-1.4, -0.5) {\scriptsize $1$};
	\node at (2.6, -4.5) {\scriptsize $1$};
	\node at (5.4, -4.5) {\scriptsize $1$};
	\node at (1.4, -4.5) {\scriptsize $1$};
	\node at (-1.4, -4.5) {\scriptsize $1$};
	\node at (0.1, -0.65) { $\cdots$};
	\node at (4.1, -0.65) { $\cdots$};
	\node at (0.1, -4.6) { $\cdots$};
	\node at (4.1, -4.6) { $\cdots$};
	\node at (0,2) {\scriptsize $l$};
	\node at (4,2) {\scriptsize $k$};
	\node at (0,-7) {\scriptsize $k$};
	\node at (4,-7) {\scriptsize $l$};
\end{tikzpicture}
};
\endxy \ .
\]

\begin{proposition}  There is a  well-defined functor $\mathcal{F}:  \Web^{\Cliff, \C}_{\{1\}} \to \mathfrak{q}$-$\operatorname{\mathbf{Web}}_{\uparrow}$,  sending $1^{(k)}$ to $\uparrow_{k}$ for $k\in \mathbb{Z}_{\geq 0}$.  On morphisms it sends the merges and splits to their counterparts in $\mathfrak{q}$-$\operatorname{\mathbf{Web}}_{\uparrow}$,  the crossing in $\operatorname{Hom}(1^{(k)}1^{(\ell)},1^{(\ell)}1^{(k)})$ to the element $\beta_{\uparrow k, \uparrow \ell}$,  and, lastly,  it sends the coupon $c$ on the thin strand colored by $1$  to the dot on $\uparrow_{1}$.
\end{proposition}

\begin{proof}
Here we use the alternate presentation for $\WebAaQ$ given in \cref{alterpres}.  It is enough to show that the image of all defining relations for $\WebAaQ$  hold in $\mathfrak{q}$-$\operatorname{\mathbf{Web}}_{\uparrow}$.  Among them,  \cref{AssocRel},  \cref{Cox},  \cref{SplitIntertwineRel},  \cref{MergeIntertwineRel},  \cref{OddKnotholeRel}  and the first part of  \cref{AaRel2}  all have obvious analogues in $\mathfrak{q}$-$\operatorname{\mathbf{Web}}_{\uparrow}$.  In addition,  the second part in \cref{AaRel2},  as well as \cref{TAaSMRel} and \cref{AaIntertwine} are vacuous  since we only allow nontrivial coupons on thin strands. 

Among the remaining two relations,  we first argue that  \cref{CrossingWebRel} holds when $x=1$ by arguing by induction on $y$.  The base case when $x=y=1$ is immediate.  The induction step is given by
\begin{align}
\begin{tikzpicture}[color=\clr,  scale=.3,  baseline=0]
		\node (0) at (-3, 3) {};
		\node (1) at (3, -3) {};
		\node  (2) at (-3, -3) {};
		\node (3) at (3, 3) {};
		\draw [ thick, directed=1] (1.center) to (0.center);
		\draw [ thick, directed=1]  (2.center) to (3.center);
			\node at (-3,-4) {\scriptsize $1$};
	\node at (3,-4) {\scriptsize $\ell$};
\end{tikzpicture}
&=\frac{1}{\ell} \:
\begin{tikzpicture} [color=\clr,  scale=0.3,  baseline=0]
		\node (0) at (-3, 3) {};
		\node (1) at (-2, 2) {};
		\node (2) at (-3, -3) {};
		\node  (3) at (3, 3) {};
		\node (4) at (3, -3) {};
		\node (5) at (2, -2) {};
		\node (6) at (0, -2.5) {\scriptsize $1$};
		\node  (7) at (3.5, -0.75) {\scriptsize $\ell-1$};
		\draw [ thick, directed=1] (1.center) to (0.center);
		\draw [ thick, directed=1] (2.center) to (3.center);
		\draw [ thick, directed=1] (4.center) to (5.center);
		\draw  [thick, directed=0.75,  bend right=45, looseness=1.25] (5.center) to (1.center);
		\draw [thick, directed=0.75,  bend left=60] (5.center) to (1.center);
		\node at (-3,-4) {\scriptsize $1$};
	\node at (3,-4) {\scriptsize $\ell$};
\end{tikzpicture} \\
&=
-\frac{1}{\ell}\:
\begin{tikzpicture}[color=\clr,  scale=0.3,  baseline=0]
		\node (0) at (-3, 3) {};
		\node  (1) at (-3, -3) {};
		\node (2) at (3, -3) {};
		\node  (3) at (3, 3) {};
		\node (4) at (2.25, -2) {};
		\node  (5) at (-3, 2.5) {};
		\node  (6) at (0.5, 0) {};
		\node (7) at (-1, 1.75) {};
		\node  (8) at (0, -0.75) {\scriptsize $1$};
		\node  (9) at (3.5, 0) {\scriptsize $\ell-1$};
		\draw  [ thick, directed=1] (1.center) to (0.center);
		\draw   [ thick, directed=1] (2.center) to (4.center);
		\draw [ thick,  bend right=330] (4.center) to (6.center);
		\draw [ thick,  directed=.5,   bend left=45] (6.center) to (3.center);
		\draw  [ thick, directed=.5] (7.center) to (5.center);
		\draw [ thick, directed=1,  bend right=45, looseness=1.25] (4.center) to (7.center);
		\node at (-3,-4) {\scriptsize $1$};
	\node at (3,-4) {\scriptsize $\ell$};
\end{tikzpicture}+\frac{1}{\ell}\:
\begin{tikzpicture}[color=\clr,  scale=0.3,  baseline=0]
		\node (0) at (-3, 3) {};
		\node (1) at (-3, -3) {};
		\node  (2) at (3, -3) {};
		\node(3) at (3, 3) {};
		\node  (4) at (-2, 2) {};
		\node (5) at (-1, -0.25) {};
		\node(6) at (-1, -1.25) {};
		\node (7) at (1.5, 1) {};
		\node (8) at (1.5, -1.5) {};
		\node (9) at (-1.5, -0.75) {\scriptsize $2$};
		\node (10) at (3.5, 0) {\scriptsize $\ell-1$};
		\node  (11) at (0.5, 0.5) {\scriptsize $1$};
		\node (12) at (-2.25, 1) {\scriptsize $1$};
		\draw [thick, directed=.5] (1.center) to (6.center);
		\draw [thick, directed=.5,bend left, looseness=1.25] (8.center) to (6.center);
		\draw [thick, directed=.5] (2.center) to (8.center);
		\draw [thick, bend right=45] (8.center) to (7.center);
		\draw [thick, directed=.5,bend right, looseness=0.75] (7.center) to (4.center);
		\draw [thick, directed=1](4.center) to (0.center);
		\draw [thick, directed=.5](6.center) to (5.center);
		\draw [thick, directed=.5,bend left=15, looseness=1.25] (5.center) to (4.center);
		\draw [thick, directed=1] (5.center) to (3.center);
			\node at (-3,-4) {\scriptsize $1$};
	\node at (3,-4) {\scriptsize $\ell$};
\end{tikzpicture}  \label{thickcrossing1} \\
	&=-\frac{1}{\ell}\:
	\begin{tikzpicture}[color=\clr,  scale=0.3,  baseline=0]
		\node  (0) at (-3, 3) {};
		\node  (1) at (-3, 1) {};
		\node (2) at (-3, -3) {};
		\node(3) at (3, 3) {};
		\node  (4) at (3, -1) {};
		\node  (5) at (3, -3) {};
		\node (6) at (-0.25, 1) {\scriptsize $\ell-1$};
		\draw [thick, directed=1] (2.center) to (0.center);
		\draw [thick, directed=1] (5.center) to (3.center);
		\draw [thick, directed=.5] (4.center) to (1.center);
		\node at (-3,-4) {\scriptsize $1$};
	\node at (3,-4) {\scriptsize $\ell$};
\end{tikzpicture}
- \frac{1}{\ell}\: 
\begin{tikzpicture}[color=\clr,   scale=.3,  baseline=0]
		\node  (0) at (-3, 3) {};
		\node (1) at (-3, -3) {};
		\node  (2) at (3, 3) {};
		\node(3) at (3, -3) {};
		\node  (4) at (-1.75, 2.25) {};
		\node  (5) at (-1.25, -0.5) {};
		\node  (6) at (-0.5, 1) {};
		\node (7) at (-1.25, -1.5) {};
		\node (8) at (2, -2) {};
		\node  (9) at (2, 0.5) {};
		\node  (10) at (-0.75, 2.25) {};
		\node (11) at (1, 1.5) {\scriptsize$\ell-2$};
		\node  (12) at (-2, -1) {\scriptsize $2$};
		\node (13) at (3.5, -0.75) {\scriptsize$\ell-1$};
		\draw [thick, directed=.5] (1.center) to (7.center);
		\draw [thick, directed=.5] (7.center) to (5.center);
		\draw [thick, directed=.5,  bend left, looseness=1.25] (5.center) to (4.center);
		\draw [thick, directed=.5,  bend right] (5.center) to (6.center);
		\draw [thick, directed=.5,  bend right] (6.center) to (4.center);
		\draw [thick, directed=.5] (9.center) to (6.center);
		\draw [thick, directed=1] (9.center) to (2.center);
		\draw [thick, directed=.5] (8.center) to (9.center);
		\draw [thick, directed=.5] (3.center) to (8.center);
		\draw [thick, directed=.5] (8.center) to (7.center);
		\draw [thick, directed=1] (4.center) to (0.center);
		\node at (-3,-4) {\scriptsize $1$};
	\node at (3,-4) {\scriptsize $\ell$};
\end{tikzpicture}
+\frac{1}{\ell}\:
\begin{tikzpicture}[color=\clr,  scale=.3, baseline=0]
		\node  (0) at (-3, 3) {};
		\node(1) at (-3, -3) {};
		\node  (2) at (3, -3) {};
		\node  (3) at (3, 3) {};
		\node (4) at (-3, 2) {};
		\node (5) at (3, 1.25) {};
		\node (6) at (-3, -1.25) {};
		\node (6.5) at (-3.5, -1) {\scriptsize $2$};
		\node (7) at (3, -2) {};
		\node  (8) at (3, 0.5) {};
		\node (9) at (-3, -0.5) {};
		\node  (10) at (-0, 2.25) {\scriptsize $\ell-1$};
		\node  (11) at (3.5, 0.75) {\scriptsize $\ell$};
		\node  (12) at (4.25, -0.75) {\scriptsize $\ell-1$};
		\draw [thick, directed=1] (1.center) to (0.center);
		\draw [thick, directed=1] (2.center) to (3.center);
		\draw [thick, directed=.5] (5.center) to (4.center);
		\draw [thick, directed=.5] (9.center) to (8.center);
		\draw [thick, directed=.5]  (7.center) to (6.center);
			\node at (-3,-4) {\scriptsize $1$};
	\node at (3,-4) {\scriptsize $\ell$};
\end{tikzpicture}\\
&= - \begin{tikzpicture}[color=\clr,  scale=0.3,  baseline=0]
		\node  (0) at (-3, 3) {};
		\node  (1) at (-3, 1) {};
		\node (2) at (-3, -3) {};
		\node(3) at (3, 3) {};
		\node  (4) at (3, -1) {};
		\node  (5) at (3, -3) {};
		\node (6) at (-0.25, 1) {\scriptsize $\ell-1$};
		\draw [thick, directed=1] (2.center) to (0.center);
		\draw [thick, directed=1] (5.center) to (3.center);
		\draw [thick, directed=.5] (4.center) to (1.center);
		\node at (-3,-4) {\scriptsize $1$};
	\node at (3,-4) {\scriptsize $\ell$};
\end{tikzpicture}  + 
\begin{tikzpicture}[color=\clr,  scale=0.3,  baseline=0]
		\node  (0) at (-3, 3) {};
		\node (1) at (3, 3) {};
		\node  (2) at (-3, -3) {};
		\node  (3) at (3, -3) {};
		\node (4) at (0, 1) {};
		\node (5) at (0, -1) {};
		\node  (6) at (2, 0) {\scriptsize $\ell+1$};
		\draw  [thick, directed=1] (4.center) to (0.center);
		\draw  [thick, directed=1] (4.center) to (1.center);
		\draw   [thick, directed=.5](5.center) to (4.center);
		\draw   [thick, directed=.5](2.center) to (5.center);
		\draw  [thick, directed=.5] (3.center) to (5.center);
			\node at (-3,-4) {\scriptsize $1$};
	\node at (3,-4) {\scriptsize $\ell$};
\end{tikzpicture}\; . \label{thickcrossing2}
\end{align}
Note we omit the thickness labels when the strand is labeled by $1$.  Here,  the second equality follows from resolving the lower left crossing;  the third equality follows from \cite[(4.21)]{BrKuWebs},  generalized to the thick strand case,  as well as  applying the induction hypothesis to the last term in \cref{thickcrossing1}.   In the last equality,  the fourth term in \cref{thickcrossing2} appears three times in the previous step,  one from applying the association rule to the second term in \cref{thickcrossing2},   and another from applying the rung swap relation  \cite[(4.6)]{BrKuWebs} to the third term in \cref{thickcrossing2}.  These three terms appear with coefficients $-\frac{1}{\ell}$,  $-\frac{2(\ell-1)}{\ell}$, and $\frac{\ell-1}{\ell}$,  which sum to yield $-1$ in the last step.  We leave to the reader the general case where one fixes $y$ as an arbitrary integer and inducts on $x$.  The argument is similar and uses the above identity in the induction step.  

Finally,  we argue \cref{DiagSwitchRel} holds by induction on $s$.  The base case of $s=0$ is trivial.  Suppose  \cref{DiagSwitchRel} holds for $s-1$ and arbitrary $r$.   The induction step is given as follows.
\begin{align*}
\begin{tikzpicture}[color=\clr,  scale=0.3,  baseline=0]
		\node  (0) at (-2, 4) {};
		\node (1) at (-2, -4) {};
		\node  (2) at (2, -4) {};
		\node  (3) at (2, 4) {};
		\node  (4) at (-2, 2) {};
		\node(5) at (2, 1) {};
		\node (6) at (2, -1) {};
		\node  (7) at (-2, -2) {};
		\node  (8) at (-0.25, 2.25) {\scriptsize $r$};
		\node  (9) at (0.25, -2.25) {\scriptsize $s$};
		\draw [thick, directed=1] (1.center) to (0.center);
		\draw [thick, directed=1] (2.center) to (3.center);
		\draw [thick, directed=.5] (5.center) to (4.center);
		\draw [thick, directed=.5] (7.center) to (6.center);
		\node at (-2,-5) {\scriptsize $x$};
	\node at (2,-5) {\scriptsize $y$};
\end{tikzpicture} &=\frac{1}{s} \:
\begin{tikzpicture}[color=\clr,  scale=0.3,  baseline=0]
		\node  (0) at (-2, 4) {};
		\node  (1) at (-2, -4) {};
		\node  (2) at (2, -4) {};
		\node  (3) at (2, 4) {};
		\node  (4) at (-2, 2) {};
		\node  (5) at (2, 1) {};
		\node  (6) at (2, 0) {};
		\node  (7) at (-2, -1) {};
		\node  (8) at (-0.25, 2.25) {\scriptsize $r$};
		\node (9) at (0.4, -2.25) {\scriptsize $s-1$};
		\node  (10) at (-2, -2) {};
		\node (11) at (2, -1) {};
		\node (12) at (-4, -1.5) {\scriptsize $x-s+1$};
		\node  (13) at (4.5, 3.75) {\scriptsize $y+s-r$};
		\draw  [thick, directed=1] (1.center) to (0.center);
		\draw  [thick, directed=1] (2.center) to (3.center);
		\draw [thick, directed=.5] (5.center) to (4.center);
		\draw [thick, directed=.5] (7.center) to (6.center);
		\draw [thick, directed=.5] (10.center) to (11.center);
			\node at (-2,-5) {\scriptsize $x$};
	\node at (2,-5) {\scriptsize $y$};
\end{tikzpicture}\\
&= \frac{1}{s}\:
\begin{tikzpicture}[color=\clr,  scale=0.3,  baseline=0]
		\node  (0) at (-2, 4) {};
		\node(1) at (-2, -4) {};
		\node (2) at (2, -4) {};
		\node  (3) at (2, 4) {};
		\node (4) at (-2, 1) {};
		\node (5) at (2, -0.25) {};
		\node (6) at (2, 3) {};
		\node (7) at (-2, 1.5) {};
		\node (8) at (0, 1) {\scriptsize $r$};
		\node  (9) at (0, -2.25) {\scriptsize $s-1$};
		\node (10) at (-2, -2) {};
		\node(11) at (2, -1) {};
		\node (13) at (5, 1.5) {\scriptsize $y+s-1-r$};
		\draw [thick, directed=1] (1.center) to (0.center);
		\draw [thick, directed=1] (2.center) to (3.center);
		\draw [thick, directed=.5] (5.center) to (4.center);
		\draw [thick, directed=.5] (7.center) to (6.center);
		\draw [thick, directed=.5] (10.center) to (11.center);
			\node at (-2,-5) {\scriptsize $x$};
	\node at (2,-5) {\scriptsize $y$};
\end{tikzpicture}
+\frac{x+y-2s+r+1}{s}\:
\begin{tikzpicture}[color=\clr,  scale=0.3,  baseline=0]
		\node  (0) at (-2, 4) {};
		\node (1) at (-2, -4) {};
		\node  (2) at (2, -4) {};
		\node  (3) at (2, 4) {};
		\node  (4) at (-2, 2) {};
		\node(5) at (2, 1) {};
		\node (6) at (2, -1) {};
		\node  (7) at (-2, -2) {};
		\node  (8) at (-0.25, 2.25) {\scriptsize $r-1$};
		\node  (9) at (0.25, -2.25) {\scriptsize $s-1$};
		\draw [thick, directed=1] (1.center) to (0.center);
		\draw [thick, directed=1] (2.center) to (3.center);
		\draw [thick, directed=.5] (5.center) to (4.center);
		\draw [thick, directed=.5] (7.center) to (6.center);
		\node at (-2,-5) {\scriptsize $x$};
	\node at (2,-5) {\scriptsize $y$};
\end{tikzpicture}\\
&= \frac{s-t}{s}\binom{x-y+r-s+1}{t}\:
\begin{tikzpicture}[color=\clr,  scale=0.3,  baseline=0]
		\node  (0) at (-2, 4) {};
		\node (1) at (-2, -4) {};
		\node(2) at (2, -4) {};
		\node (3) at (2, 4) {};
		\node  (4) at (-2, 1) {};
		\node (5) at (2, 2) {};
		\node  (6) at (-2, -1) {};
		\node (7) at (2, -2) {};
		\node  (8) at (-0.5, 2.25) {\scriptsize $s-t$};
		\node(9) at (0.5, -2.5) {\scriptsize $r-t$};
		\draw [thick, directed=1] (1.center) to (0.center);
		\draw [thick, directed=1] (2.center) to (3.center);
		\draw [thick, directed=.5] (4.center) to (5.center);
		\draw [thick, directed=.5] (7.center) to (6.center);
		\node at (-2,-5) {\scriptsize $x$};
	\node at (2,-5) {\scriptsize $y$};
\end{tikzpicture}+
\frac{x-y-2s+r+1}{s}\binom{x-y+r-s}{t}\:
\begin{tikzpicture}[color=\clr,  scale=0.3,  baseline=0]
		\node  (0) at (-2, 4) {};
		\node (1) at (-2, -4) {};
		\node(2) at (2, -4) {};
		\node (3) at (2, 4) {};
		\node  (4) at (-2, 1) {};
		\node (5) at (2, 2) {};
		\node  (6) at (-2, -1) {};
		\node (7) at (2, -2) {};
		\node  (8) at (0, 2.25) {\scriptsize $s-t-1$};
		\node(9) at (0, -2.5) {\scriptsize $r-t-1$};
		\draw [thick, directed=1] (1.center) to (0.center);
		\draw [thick, directed=1] (2.center) to (3.center);
		\draw [thick, directed=.5] (4.center) to (5.center);
		\draw [thick, directed=.5] (7.center) to (6.center);
		\node at (-2,-5) {\scriptsize $x$};
	\node at (2,-5) {\scriptsize $y$};
\end{tikzpicture}
\end{align*}
Here,  after a change of variables $t\mapsto t-1$ in the second term,  one only needs to verify that 
\begin{align*}
\frac{s-t}{s}\binom{x-y+r-s+1}{t}+\frac{x-y-2s+r+1}{s}\binom{x-y+r-s}{t-1}=\binom{x-y+r-s}{t},
\end{align*}
which is straightforward.
\end{proof}

\begin{proposition}  There is a  well-defined functor $\mathcal{G}: \mathfrak{q}$-$\operatorname{\mathbf{Web}}_{\uparrow} \to \WebAaQ $,  sending  $\uparrow_{k}$ to $1^{(k)}$ for $k\in \mathbb{Z}_{\geq 0}$.  On morphisms it sends the merges and splits to their counterparts in $ \WebAaQ$,    and the dot in $\operatorname{End}_{\mathfrak{q}-\operatorname{\mathbf{Web}}_{\uparrow}}(\uparrow_k)$ to the following diagram in $\WebAaQ$:
\begin{align*}
\begin{tikzpicture}[scale=0.7]
  \draw[ultra thick, blue] (0,-0.1)--(0,0.1) .. controls ++(0,0.35) and ++(0,-0.35) .. (-0.4,0.6)--(-0.4,0.9);
    \draw[ultra thick, blue]  (-0.4,0.9)--(-0.4,1.2) 
  .. controls ++(0,0.35) and ++(0,-0.35) .. (0,1.7)--(0,1.9);
  \draw[ultra thick, blue] (0,-0.1)--(0,0.1) .. controls ++(0,0.35) and ++(0,-0.35) .. (0.4,0.6)--(0.4,0.9);
    \draw[ultra thick, blue]  (0.4,0.9)--(0.4,1.2) 
  .. controls ++(0,0.35) and ++(0,-0.35) .. (0,1.7)--(0,1.9);
  .. controls ++(0,0.35) and ++(0,-0.35) .. (0,1.7)--(0,1.9);
       \node[above] at (0,1.8) {$\scriptstyle 1^{\scriptstyle(k)}$};
         \node[below] at (0,-0.1) {$\scriptstyle 1^{\scriptstyle(k)}$};
             \node[right] at (0.4,0.9) {$\scriptstyle 1^{\scriptstyle(k-1)}$};
               \node[left] at (-0.5,0.9) {$\scriptstyle 1^{\scriptstyle(1)}$};
                \draw[thick, fill=yellow]  (-0.4,0.9) circle (7pt);  
                 \node at (-0.4,0.9) {$ \scriptstyle  c$};
\end{tikzpicture} \; .
\end{align*} 
\end{proposition}
\begin{proof}
It is enough to check the image of all relations in \cite[Section~4.1]{BrKuWebs} hold in $\WebAaQ$.  Here we omit the explicit computations.  Relations \cite[(4.1)-(4.2)]{BrKuWebs} are identical to the relations in $\WebAaQ$.  Relations \cite[(4.3)-(4.4)]{BrKuWebs} can be deduced from applying \cref{MSrel} to the middle portion of each diagram.  Relation \cite[(4.5)]{BrKuWebs} can be deduced from applying \cref{CrossingWebRel} to the middle portion.   Relation  \cite[(4.6)]{BrKuWebs} is  a special case of \cref{DiagSwitchRel}.  The first equality in \cite[(4.7)]{BrKuWebs} is the result of applying \cref{MSrel} to the middle right portion of each diagram on the left,  and similarly the second equality holds.  \cite[(4.8)]{BrKu} follows from applying \cref{MSrel} to the middle portion of the third picture.  Similarly \cite[(4.9)]{BrKuWebs} holds.
\end{proof}

By construction the functors $\mathcal{F}$ and $\mathcal{G}$ are $\k$-linear monoidal functors and are inverses of each other.

\begin{theorem}  There is an isomorphism of $\k$-linear monoidal supercategories,
\[
\Web^{\Cliff, \k}_{\left\{1 \right\}} \cong \mathfrak{q}\textup{-}\Web.
\]
\end{theorem}

Under the isomorphism $\gl_{n}(\Cliff) \cong \mathfrak{q}_{n}(\k )$ one readily checks that $V_{n}=\Cliff^{\oplus n} \cong \k^{n|n}$, the natural representation of $\mathfrak{q}_{n}(\k )$.  Given these identifications, one could reasonably expect that the functor $G_{m}$ coincides with the one given in \cite{BrKuWebs}.  This turns out to be true.  The only wrinkle is that we consider right $\gl_{n}(\Cliff)$-supermodules while in \cite{BrKuWebs} they considered left $\mathfrak{q}_{n}(\k )$-supermodules.  This can be resolved by suitably using opposites.  We leave the straightforward details to the interested reader.

Before leaving this example we should also mention that a Howe duality was already obtained for the pair $(\mathfrak{q}_{m}(\C ), \mathfrak{q}_{n}(\C ))$ by Sergeev \cite{Sergeev} (or see \cite{CW}).

\subsection{Zigzag superalgebras and RoCK blocks}\label{zigzagcat}
Let \(G = (I,E)\) be a simple connected graph with no loops. For each pair of connected vertices \(i\textup{\textemdash} j\), fix \(\varepsilon_{ij}, \varepsilon_{ji} \in \{\pm 1\}\). Then the associated locally unital {\em zigzag superalgebra} \(\overline{Z} = \bigoplus_{i,j \in I} j\overline{Z}i\) has distinguished set of orthogonal idempotents corresponding to the vertices \(I\) of \(G\), with the following \(\k\)-basis for each subspace:
\begin{align*}
j\overline{Z}i =
\begin{cases}
\k\{i, c_i\} & \textup{if } i =j;\\
\k\{a_{ji} \} & \textup{if } i\textup{\textemdash} j;\\
0 & \textup{otherwise},
\end{cases}
\end{align*} 
Multiplication is defined via the orthogonality of the idempotents \(I\), coupled with the additional rules:
\begin{align*}
a_{kj}a_{ji} =\delta_{ik} \varepsilon_{ji} c_i,
\qquad
c_ia_{ij} = a_{ij}c_j = 0,
\qquad
c_i^2 = 0.
\end{align*}
The parity of elements is defined by \(\bar i = \bar{c}_i = \bar 0\), \(\bar{a}_{ij} = \bar 1\). Taking \(\bar z= \k I\), we have a good pair \((\overline{Z}, \bar{z})_I\). It is straightforward to check that if \(G\) is a tree, then different choices of signs \(\varepsilon_{ij}\) nonetheless yield isomorphic zigzag superalgebras.

\subsubsection{The zigzag web category}\label{zigwebdef}
Following \cref{defweb} we may define the associated web category \(\Web^{\overline{Z}, \bar{z}}_{I}\). For the reader's benefit we describe a streamlined presentation of \(\Web^{\overline{Z}, \bar{z}}_{I}\) adapted to the particulars of the zigzag superalgebra.  
\(\Web^{\overline{Z}, \bar{z}}_{I}\) is a monoidal supercategory generated by objects \(i^{(x)}\), where \(i \in I\), \(x \in \Z_{\geq 0}\), with generating morphisms
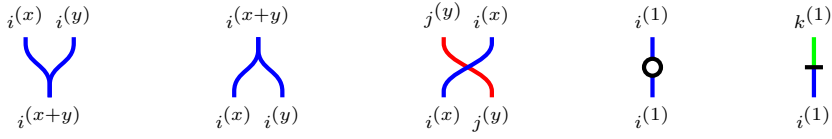
\begin{align*}
\hackcenter{
{}
}
\hackcenter{
\begin{tikzpicture}[scale=.8]
  \draw[ultra thick,blue] (0,0)--(0,0.2) .. controls ++(0,0.35) and ++(0,-0.35) .. (-0.4,0.9)--(-0.4,1);
  \draw[ultra thick,blue] (0,0)--(0,0.2) .. controls ++(0,0.35) and ++(0,-0.35) .. (0.4,0.9)--(0.4,1);
      \node[above] at (-0.4,1) {$ \scriptstyle i^{\scriptstyle (x)}$};
      \node[above] at (0.4,1) {$ \scriptstyle i^{\scriptstyle (y)}$};
      \node[below] at (0,0) {$ \scriptstyle i^{\scriptstyle (x+y)} $};
\end{tikzpicture}}
\qquad
\qquad
\hackcenter{
\begin{tikzpicture}[scale=.8]
  \draw[ultra thick,blue ] (-0.4,0)--(-0.4,0.1) .. controls ++(0,0.35) and ++(0,-0.35) .. (0,0.8)--(0,1);
\draw[ultra thick, blue] (0.4,0)--(0.4,0.1) .. controls ++(0,0.35) and ++(0,-0.35) .. (0,0.8)--(0,1);
      \node[below] at (-0.4,0) {$ \scriptstyle i^{ \scriptstyle (x)}$};
      \node[below] at (0.4,0) {$ \scriptstyle i^{ \scriptstyle (y)}$};
      \node[above] at (0,1) {$ \scriptstyle i^{ \scriptstyle (x+y)}$};
\end{tikzpicture}}
\qquad
\qquad
\hackcenter{
\begin{tikzpicture}[scale=.8]
  \draw[ultra thick,red] (0.4,0)--(0.4,0.1) .. controls ++(0,0.35) and ++(0,-0.35) .. (-0.4,0.9)--(-0.4,1);
  \draw[ultra thick,blue] (-0.4,0)--(-0.4,0.1) .. controls ++(0,0.35) and ++(0,-0.35) .. (0.4,0.9)--(0.4,1);
      \node[above] at (-0.4,1) {$ \scriptstyle j^{ \scriptstyle (y)}$};
      \node[above] at (0.4,1) {$ \scriptstyle i^{ \scriptstyle (x)}$};
       \node[below] at (-0.4,0) {$ \scriptstyle i^{ \scriptstyle (x)}$};
      \node[below] at (0.4,0) {$ \scriptstyle j^{ \scriptstyle (y)}$};
\end{tikzpicture}}
\qquad
\qquad
\hackcenter{
\begin{tikzpicture}[scale=.8]
  \draw[ultra thick, blue] (0,0)--(0,0.5);
   \draw[ultra thick, blue] (0,0.5)--(0,1);
   \draw[ultra thick, fill=white]  (0,0.5) circle (4pt);
     \node[below] at (0,0) {$ \scriptstyle i^{ \scriptstyle (1)}$};
      \node[above] at (0,1) {$ \scriptstyle i^{ \scriptstyle (1)}$};
\end{tikzpicture}}
\qquad
\qquad
\hackcenter{
\begin{tikzpicture}[scale=.8]
  \draw[ultra thick, blue] (0,0)--(0,0.5);
   \draw[ultra thick, green] (0,0.5)--(0,1);
     \draw[ultra thick, black] (-0.15,0.5)--(0.15,0.5);
     \node[below] at (0,0) {$ \scriptstyle i^{ \scriptstyle (1)}$};
      \node[above] at (0,1) {$ \scriptstyle k^{ \scriptstyle (1)}$};
\end{tikzpicture}}
\end{align*}
for \(i,j,k \in I\) with \(i\textup{\textemdash} k\), \(x,y \in \Z_{\geq 0}\), where all morphisms are even save the rightmost, which is odd. Here we are using the circle dot in lieu of the \(c_i\) coupon, and a color-change `tag' in lieu of the \(a_{ki}\) coupon. These morphisms are subject to relations (\ref{AssocRel}--\ref{MergeIntertwineRel}), together with the following streamlined analogues of relations (\ref{AaRel1}--\ref{AaIntertwine}):

\begin{align*}
\hackcenter{}
\hackcenter{
\begin{tikzpicture}[scale=.8]
  \draw[ultra thick, blue] (0,0)--(0,0.4);
   \draw[ultra thick, red] (0,0.4)--(0,1.1);
     \draw[ultra thick, green] (0,1.1)--(0,1.5);
     \draw[ultra thick, black] (-0.15,0.4)--(0.15,0.4);
          \draw[ultra thick, black] (-0.15,1.1)--(0.15,1.1);
     \node[below] at (0,0) {$ \scriptstyle i^{ \scriptstyle (1)}$};
      \node[left] at (0,0.8) {$ \scriptstyle j^{ \scriptstyle (1)}$};
        \node[above] at (0,1.5) {$ \scriptstyle k^{ \scriptstyle (1)}$};
\end{tikzpicture}}
=
\delta_{i,k}\varepsilon_{ji}
\hackcenter{
\begin{tikzpicture}[scale=.8]
  \draw[ultra thick, blue] (0,0)--(0,0.75);
   \draw[ultra thick, blue] (0,0.75)--(0,1.5);
   \draw[ultra thick, fill=white]  (0,0.75) circle (4pt);
     \node[below] at (0,0) {$ \scriptstyle i^{ \scriptstyle (1)}$};
      \node[above] at (0,1.5) {$ \scriptstyle i^{ \scriptstyle (1)}$};
\end{tikzpicture}}
\qquad
\qquad
\hackcenter{
\begin{tikzpicture}[scale=.8]
  \draw[ultra thick, blue] (0,0)--(0,0.4);
   \draw[ultra thick, blue] (0,0.4)--(0,1.1);
     \draw[ultra thick, red] (0,1.1)--(0,1.5);
        \draw[ultra thick, fill=white]  (0,0.5) circle (4pt);
          \draw[ultra thick, black] (-0.15,1.1)--(0.15,1.1);
     \node[below] at (0,0) {$ \scriptstyle i^{ \scriptstyle (1)}$};
        \node[above] at (0,1.5) {$ \scriptstyle j^{ \scriptstyle (1)}$};
\end{tikzpicture}}
=
\hackcenter{
\begin{tikzpicture}[scale=.8]
  \draw[ultra thick, blue] (0,0)--(0,0.4);
   \draw[ultra thick, red] (0,0.4)--(0,1.1);
     \draw[ultra thick, red] (0,1.1)--(0,1.5);
     \draw[ultra thick, black] (-0.15,0.4)--(0.15,0.4);
        \draw[ultra thick, fill=white]  (0,1) circle (4pt);
     \node[below] at (0,0) {$ \scriptstyle i^{ \scriptstyle (1)}$};
        \node[above] at (0,1.5) {$ \scriptstyle j^{ \scriptstyle (1)}$};
\end{tikzpicture}}
=
0
\qquad
\qquad
\hackcenter{
\begin{tikzpicture}[scale=.8]
  \draw[ultra thick, blue] (0,0)--(0,0.4);
   \draw[ultra thick, blue] (0,0.4)--(0,1.1);
     \draw[ultra thick, blue] (0,1.1)--(0,1.5);
        \draw[ultra thick, fill=white]  (0,1) circle (4pt);
         \draw[ultra thick, fill=white]  (0,0.5) circle (4pt);
     \node[below] at (0,0) {$ \scriptstyle i^{ \scriptstyle (1)}$};
        \node[above] at (0,1.5) {$ \scriptstyle j^{ \scriptstyle (1)}$};
\end{tikzpicture}}
=
0
\end{align*}
\begin{align*}
\hackcenter{}
\hackcenter{
\begin{tikzpicture}[scale=.8]
  \draw[ultra thick, blue] (0,0.2)--(0,0.7);
  \draw[ultra thick, red] (0,0.7)--(0,1) .. controls ++(0,0.35) and ++(0,-0.35)  .. (0.8,1.8)--(0.8,2); 
    \draw[ultra thick, green] (0.8,0.2)--(0.8,1) .. controls ++(0,0.35) and ++(0,-0.35)  .. (0,1.8)--(0,2); 
        \draw[ultra thick]  (-0.15,0.7)--(0.15,0.7);
     \node[below] at (0,0.2) {$ \scriptstyle i^{ \scriptstyle (1)}$};
     \node[below] at (0.8,0.2) {$  \scriptstyle  k^{ \scriptstyle (x)}$};
     \node[above] at (0,2) {$  \scriptstyle  k^{ \scriptstyle (x)}$};
      \node[above] at (0.8,2) {$ \scriptstyle j^{ \scriptstyle (1)}$};
\end{tikzpicture}}
=
\hackcenter{
\begin{tikzpicture}[scale=.8]
  \draw[ultra thick, red] (0,-0.2)--(0,-0.7);
  \draw[ultra thick, blue] (0,-0.7)--(0,-1) .. controls ++(0,-0.35) and ++(0,0.35)  .. (-0.8,-1.8)--(-0.8,-2); 
    \draw[ultra thick, green] (-0.8,-0.2)--(-0.8,-1) .. controls ++(0,-0.35) and ++(0,0.35)  .. (0,-1.8)--(0,-2); 
          \draw[ultra thick]  (-0.15,-0.7)--(0.15,-0.7);
     \node[above] at (0,-0.2) {$ \scriptstyle j^{ \scriptstyle (1)}$};
     \node[above] at (-0.8,-0.2) {$  \scriptstyle  k^{ \scriptstyle (x)}$};
     \node[below] at (0,-2) {$  \scriptstyle  k^{ \scriptstyle (x)}$};
      \node[below] at (-0.8,-2) {$ \scriptstyle i^{ \scriptstyle (1)}$};
\end{tikzpicture}}
\qquad
\qquad
\hackcenter{
\begin{tikzpicture}[scale=.8]
  \draw[ultra thick, blue] (0,0.2)--(0,0.5);
  \draw[ultra thick, blue] (0,0.5)--(0,1) .. controls ++(0,0.35) and ++(0,-0.35)  .. (0.8,1.8)--(0.8,2); 
    \draw[ultra thick, green] (0.8,0.2)--(0.8,1) .. controls ++(0,0.35) and ++(0,-0.35)  .. (0,1.8)--(0,2); 
        \draw[ultra thick, fill=white]  (0,0.7) circle (4pt);
     \node[below] at (0,0.2) {$ \scriptstyle i^{ \scriptstyle (1)}$};
     \node[below] at (0.8,0.2) {$  \scriptstyle  k^{ \scriptstyle (x)}$};
     \node[above] at (0,2) {$  \scriptstyle  k^{ \scriptstyle (x)}$};
      \node[above] at (0.8,2) {$ \scriptstyle i^{ \scriptstyle (1)}$};
\end{tikzpicture}}
=
\hackcenter{
\begin{tikzpicture}[scale=.8]
  \draw[ultra thick, blue] (0,-0.2)--(0,-0.5);
  \draw[ultra thick, blue] (0,-0.5)--(0,-1) .. controls ++(0,-0.35) and ++(0,0.35)  .. (-0.8,-1.8)--(-0.8,-2); 
    \draw[ultra thick, green] (-0.8,-0.2)--(-0.8,-1) .. controls ++(0,-0.35) and ++(0,0.35)  .. (0,-1.8)--(0,-2); 
        \draw[ultra thick, fill=white]  (0,-0.7) circle (4pt);
     \node[above] at (0,-0.2) {$ \scriptstyle i^{ \scriptstyle (1)}$};
     \node[above] at (-0.8,-0.2) {$  \scriptstyle  k^{ \scriptstyle (x)}$};
     \node[below] at (0,-2) {$  \scriptstyle  k^{ \scriptstyle (x)}$};
      \node[below] at (-0.8,-2) {$ \scriptstyle i^{ \scriptstyle (1)}$};
\end{tikzpicture}}
\qquad
\qquad
\hackcenter{
\begin{tikzpicture}[scale=.8]
  \draw[ultra thick, blue] (0,0)--(0,0.1) .. controls ++(0,0.35) and ++(0,-0.35) .. (-0.4,0.6)--(-0.4,0.9); 
 \draw[ultra thick, red] (-0.4,0.9)--(-0.4,1.2) .. controls ++(0,0.35) and ++(0,-0.35) .. (0,1.7)--(0,1.8);
  \draw[ultra thick, blue] (0,0)--(0,0.1) .. controls ++(0,0.35) and ++(0,-0.35) .. (0.4,0.6)--(0.4,0.9); 
 \draw[ultra thick, red] (0.4,0.9)--(0.4,1.2) .. controls ++(0,0.35) and ++(0,-0.35) .. (0,1.7)--(0,1.8);
 \draw[ultra thick]  (-0.55,0.9)--(-0.25,0.9);
 \draw[ultra thick]  (0.55,0.9)--(0.25,0.9);
       \node[above] at (0,1.8) {$  \scriptstyle j^{ \scriptstyle (2)}$};
         \node[below] at (0,0) {$  \scriptstyle i^{ \scriptstyle (2)}$};
             \node[right] at (0.3,0.3) {$ \scriptstyle  i^{ \scriptstyle (1)}$};
               \node[left] at (-0.3,0.3) {$ \scriptstyle  i^{ \scriptstyle (1)}$};
\end{tikzpicture}}
=
0.
\end{align*}

\subsubsection{RoCK blocks of symmetric groups}
Let \(H_N(q)\) be the Hecke algebra of the symmetric group \(\mathfrak{S}_N\) over a field \(\F\) with parameter \(q \in \F^\times\), and let \(e\) be the associated quantum characteristic. Let \(\mathcal{B}\) be a weight \(d\) {\em RoCK block} of \(H_N(q)\), see \cite{EKRoCK}.
Let \(G\) be the path graph \({\tt A}_{e-1}\) with vertices \(I = \{1,2,\ldots, e-1\}\):
\begin{align*}
\hackcenter{}
\hackcenter{
\begin{tikzpicture}[scale=.8]
  \draw[thick] (1,0)--(3.5,0);
  \draw[thick] (4.5,0)--(6,0);
         \node[] at (4,0) {$   \cdots $};
   \draw[ thick, fill=black]  (1,0) circle (3pt);
 \draw[ thick, fill=black]  (2,0) circle (3pt);
\draw[ thick, fill=black]  (3,0) circle (3pt);
 \draw[ thick, fill=black]  (5,0) circle (3pt);
\draw[ thick, fill=black]  (6,0) circle (3pt);
       \node[below] at (1,-0.1) {$ \scriptstyle 1$};
         \node[below] at (2,-0.1) {$ \scriptstyle 2$};
           \node[below] at (3,-0.1) {$ \scriptstyle 3$};
             \node[below] at (5,-0.1) {$ \scriptstyle e-2$};
               \node[below] at (6,-0.1) {$ \scriptstyle e-1$};
\end{tikzpicture}}
\end{align*}
Let \((\overline{Z},\bar{z})_I\) be the associated zigzag good pair as defined in \cref{zigzagcat}. Let \(\Web^{\overline{Z},\bar{z}}_I\) be the associated web category defined in \cref{zigwebdef}, and set:
\begin{align*}
\textup{Web}^{\overline{Z},\bar{z}}_{I,n,d} = \bigoplus_{ \bi^{(\bx)}, \bj^{(\by)} \in \hat{\Omega}_I(n,d)} \Web^{\overline{Z},\bar{z}}_I(\bi^{(\bx)}, \bj^{(\by)}).
\end{align*}
In other words, \(\Web^{\overline{Z},\bar{z}}_{I,n,d}\) is the full subcategory of \(\Web^{\overline{Z},\bar{z}}_I\) consisting of objects \(\bi^{(\bx)} \in \hat{\Omega}_I(n,d)\), and  \(\textup{Web}^{\overline{Z},\bar{z}}_{I,n,d}\) is the associated locally unital algebra, as described in \cref{locunicov}.

\begin{proposition}\label{zigdiagprop}
Let \(n \geq d\). Then \(\textup{Web}^{\overline{Z},\bar{z}}_{I,n,d}\) is Morita equivalent to \(\mathcal{B}\). 
\end{proposition}
\begin{proof}
It is shown in \cite[Theorem A]{EKRoCK} that the generalized Schur algebra \(T^{\overline{Z}}_{\bar{z}}(n,d)\) is Morita equivalent to \(\mathcal{B}\). 
For each \(\bi^{(\bx)} \in \hat{\Omega}_I(n,d)\), choose a representative \(\bi^{(\bx)}_0\) of this class in \(\Omega_I(n,d)\). Recalling the idempotents \(\xi_{\bi^{(\bx)}}\) defined in \cref{ixidems}, set  \(\xi_I = \sum_{\bi^{(\bx)}_0 \in \Omega_I(n,d)} \xi_{\bi^{(\bx)}_0} \in T^{\overline{Z}}_{\bar{z}}(n,d)\). 

Let \(\F\) be a field. As shown in \cite[\S7.9]{KMqh}, the simple \(T^{\overline{Z}}_{\bar{z}}(n,d)_{\F}\)-modules may be labeled as follows:
\begin{align*}
\left\{ L_{\blam, \F} \mid \blam = (\lambda^{(1)}, \ldots, \lambda^{(e-1)}) \textup{ an \((e-1)\)-multipartition of }d\right\},
\end{align*}
and there exist weight idempotents
\begin{align*}
\xi_{\blam} =\scalebox{2}{\textasteriskcentered}_{r=1}^n (E^{i}_{r,r})^{\otimes \lambda^{(i)}_{r}} \in T^{\overline{Z}}_{\bar{z}}(n,d)_{\F},
\end{align*}
such that \(\xi_{\blam} L_{\blam, \F} \neq 0\). 
It is straightforward to check that for any such multipartition \(\blam\) there exist elements \(\xi_1, \xi_2 \in T^{\overline{Z}}_{\bar{z}}(n,d)_{\F}\) such that \(\xi_{\blam} = \xi_1 \xi_I \xi_2\), so 
 we have \(\xi_I L_{\blam, \F} \neq 0\) for all multipartitions \(\blam\) as well.
 Hence by \cite[Lemma 8.26]{EKRoCK}, we have that 
 \begin{align*}
 \xi_IT_{\bar{z}}^{\overline{Z}}(n,d) \xi_I = \bigoplus_{\bi^{(\bx)}_0, \bj^{(\by)}_0 \in \Omega_I(n,d)} \xi_{\bj^{(\by)}_0} T^{\overline{Z}}_{\bar{z}}(n,d) \xi_{\bi^{(\bx)}_0}
 \end{align*} is Morita equivalent to \(T_{\bar{z}}^{\overline{Z}}(n,d)\). By \cref{moneq1}, the functor \(\chi\) induces an isomorphism of algebras \( \xi_IT_{\bar{z}}^{\overline{Z}}(n,d) \xi_I \cong \textup{Web}^{\overline{Z},\bar{z}}_{I,n,d}\), completing the proof.
\end{proof}

\subsubsection{Related zigzag/RoCK block conjectures}
Let \(n \geq d\). 
Taking \((Z,z)_I\) to be the good pair for the {\em extended} zigzag algebra \cite[\S7.9]{KMqh}, it is conjectured in \cite[Conjecture 7.58]{KMqh} that \(T^Z_z(n,d)\) is Morita equivalent to the level \(d\) RoCK block of the \(q\)-Schur algebra \(S_q(N,f)\). Similarly, taking \((A, a)_I\) to be the good pair for the Brauer tree superalgebra in \cite[\S2.2e]{KleLiv}, it is conjectured in \cite[Conjecture 1]{KleLiv} that \(T^{ A}_{ a}(n,d)\) is Morita (super)equivalent to the level \(d\) RoCK block of the Sergeev superalgebra \(\Ser_{n}\).
If these conjectures are found to hold, then a result analogous to \cref{zigdiagprop} should hold in these settings as well. If so, this will give diagrammatic models for these RoCK blocks which are similar to those in \cref{zigwebdef}.

\bibliographystyle{eprintamsplain}
\bibliography{Biblio}

\end{document}